\documentclass[reqno]{amsart}
\usepackage[utf8]{inputenc}
\usepackage{amssymb, amsthm}
\usepackage{mathrsfs}
\usepackage{bm,bbm}
\usepackage{microtype} %make spacing better, usually decrease the number of pages
\usepackage{comment} %introduce the comment environment which can hide large part of the text from displaying

\usepackage{amsfonts}
\DeclareMathAlphabet{\mathpgoth}{OT1}{pgoth}{m}{n}
\DeclareMathAlphabet{\mathesstixfrak}{U}{esstixfrak}{m}{n}
\DeclareMathAlphabet{\mathboondoxfrak}{U}{BOONDOX-frak}{m}{n}
\usepackage{amsmath,amssymb}
\usepackage[bbgreekl]{mathbbol}
\usepackage{yfonts,mathtools}

\usepackage{tikz-cd}
\usepackage[all,cmtip]{xy}

\numberwithin{equation}{section}
\usepackage{soul}

\usepackage{hyperref}
\hypersetup{colorlinks}
\definecolor{darkred}{rgb}{0.5,0,0}
\definecolor{darkgreen}{rgb}{0,0.5,0}
\definecolor{darkblue}{rgb}{0,0,0.5}
\hypersetup{colorlinks, linkcolor=darkblue, filecolor=darkgreen, urlcolor=darkred, citecolor=darkblue}
\makeatletter % `@' now normal "letter"
\@addtoreset{equation}{section}
\makeatother  % `@' is restored as "non-letter"

\numberwithin{equation}{section}

\newtheorem{thma}{Theorem}

\newtheorem{thm}{Theorem}[section]
\newtheorem{cor}[thm]{Corollary}

\newtheorem{prop}[thm]{Proposition}
\newtheorem{setup}[thm]{Setup} 
\newtheorem{lemma}[thm]{Lemma}
\theoremstyle{definition}
\newtheorem{defn}[thm]{Definition}
\theoremstyle{remark}
\newtheorem{rem}[thm]{Remark}
\newtheorem{hyp}[thm]{Hypothesis}

\usepackage{xcolor}
\newtheorem{example}[thm]{Example}
\newtheorem{notation}[thm]{Notation}

\newcommand{\beq}{\begin{equation}}
\newcommand{\eeq}{\end{equation}}
\newcommand{\beqn}{\begin{equation*}}
\newcommand{\eeqn}{\end{equation*}}
\newcommand{\ov}{\overline}
\newcommand{\mb}{\mathbb}
\newcommand{\mc}{\mathcal}
\newcommand{\mf}{\mathfrak}
\newcommand{\ms}{\mathscr}

\newcommand{\bA}{{\mathbb A}}
\newcommand{\bF}{{\mathbb F}}

\newcommand{\dep}{{\rm depth}}

\newcommand{\bpr}{{\boxplus r}}

\newcommand{\V}{{\bm V}}
\newcommand{\W}{{\bm W}}

\newcommand{\wembed}{\rightharpoonup}

\newcommand{\pss}{{\rm PSS}}
\newcommand{\ssp}{{\rm SSP}}
\newcommand{\hmtp}{{\rm hmtp}}
\newcommand{\per}{{\rm Per}}
\newcommand{\morse}{{\rm Morse}}
\newcommand{\floer}{{\rm Floer}}
\newcommand{\pearl}{{\rm pearl}}

\newcommand{\uds}[1]{\underline{\smash{#1}}}

\renewcommand{\i}{{\bm i}}
\renewcommand{\L}{{\rm L}}

\newcommand{\Gammait}{{\mathit{\Gamma}}}

\newcommand{\ev}{{\rm ev}}

\title{Arnold Conjecture over Integers}

\author{Shaoyun Bai}
\address{Mathematical Sciences Research Institute, Berkeley, CA 94720, USA}
\email{shaoyunb@msri.org}

\author{Guangbo Xu}
\address{Department of Mathematics, Texas A{\&}M University, College Station, Texas 77843, USA\\
and\\
Institute for Advanced Study, Princeton, New Jersey, 08540, USA}
\email{guangboxu@tamu.edu}
\thanks{This material is based upon work supported by the NSF under the Grant No. DMS-1928930, while the first author served as the McDuff postdoctoral fellow at the Simons Laufer Mathematical Sciences Institute (previously known as MSRI) Berkeley, California, during the Fall 2022 semester. }
\thanks{The second author is supported by NSF DMS-2204321.}

\date{\today}

\begin{document}

\maketitle

\begin{abstract}
    For any closed symplectic manifold, we show that the number of $1$-periodic orbits of a nondegenerate Hamiltonian thereon is bounded from below by a version of total Betti number over ${\mb Z}$ of the ambient space taking account of the total Betti number over ${\mb Q}$ and torsions of all characteristic. The proof is based on constructing a Hamiltonian Floer theory over the Novikov ring with integer coefficients, which generalizes our earlier work for constructing integer-valued Gromov--Witten type invariants. In the course of the construction, we build a Hamiltonian Floer flow category with compatible smooth global Kuranishi charts. This generalizes a recent work of Abouzaid--McLean--Smith, which might be of independent interest.
\end{abstract}

\setcounter{tocdepth}{1}
\tableofcontents

\section{Introduction}

Let $(M, \omega)$ be a closed symplectic manifold. Suppose $H: S^1 \times M \to {\mb R}$ is a smooth function, which is usually referred to as a $1$-periodic Hamiltonian function by identifying $S^1 = {\mb R} / {\mb Z}$. Denote by $H_t$ the smooth function obtained by restricting $H$ to $\{t\} \times M$. Then the Hamiltonian vector field $X_{H_t}$ of $H$ is a vector field on $M$ determined by the formula
$$
\omega(X_{H_t}, \cdot) = dH_t.
$$
A smooth map $x(t): S^1 \to M$ solving the ordinary differential equation
$$
\dot{x}(t) = X_{H_t}(x(t))
$$
is called a $1$-periodic orbit of $H$. Write $\phi_{t}: M \to M$ the time $t$ flow of $X_{H_t}$. Then then the set of $1$-periodic orbits of $H$ has a one-to-one correspondence with the set of fixed points $\phi_1 : M \to M$ by evaluating $x(t)$ at $t=0$. A periodic orbit $x(t)$ is called nondegenerate if the linear map 
$$
d \phi_1 : T_{x(0)}M \to T_{x(0)}M
$$
does not have $1$ as an eigenvalue. If all $1$-periodic orbits of $H$ are nondegenerate, the Hamiltonian $H$ is called nondegenerate.

Given the symplectic manifold $(M,\omega)$, there is a contractible choice of almost complex structures on $X$ which are compatible with $\omega$. The first Chern class of $(M, \omega)$, denoted by $c_1(M, \omega)=c_1(M)$, is defined to be the first Chern class of $TM$ endowed with a choice of, equivalently, any choice of almost complex structure compatible with $\omega$. The minimal Chern number of $(M, \omega)$ is defined to be the nonnegative integer $N \in \mathbb{Z}_{\geq 0}$ such that the range of the map
$$
\pi_2(M) \xrightarrow{\text{Hurewicz}} H_2(M;{\mb Z}) \xrightarrow{c_1(M)} \mathbb{Z}
$$
is $N{\mb Z} \subset {\mb Z}$. For any $i \in {\mb Z}/2N$, introduce the ${\mb Z}$-module
$$
H^{(2N)}_i(M;{\mb Z}) := \bigoplus_{j \equiv i \text{ mod } 2N} H_{j}(M;{\mb Z}).
$$
Namely, we collapse the natural ${\mb Z}$-grading on $H_{*}(M;\mathbb{Z})$ to a ${\mb Z} / 2N$-grading. Because $H^{(2N)}_i(M;{\mb Z})$ is a finitely generated ${\mb Z}$-module, there exist an integer $b_i \geq 0$ and a sequence of integers $a_1 | a_2 | \cdots | a_k$ such that
$$
H^{(2N)}_i(M;{\mb Z}) \cong {\mb Z}^{b_i} \oplus {\mb Z}/a_1 \oplus \cdots \oplus {\mb Z}/a_k,
$$
where the integers $a_1, \dots, a_k$ are the invariant factors of $H^{(2N)}_i(M;{\mb Z})$ as a finitely generated ${\mb Z}$-module. Using this decomposition, define the quantity
$$
\tau_{i}^{(2N)}(M) := \text{number of invariant factors of }H^{(2N)}_i(M;{\mb Z}).
$$

The main result of this paper is a solution to the homological Arnold conjecture over ${\mb Z}$ for arbitrary closed symplectic manifolds. This bound is similar to the Morse inequality over ${\mb Z}$ which takes into account the torsion part of the homology.

\begin{thma}\label{thm_intro_main}
Let $(M, \omega)$ be a closed symplectic manifold with minimal Chern number $N$. Suppose $H: S^1 \times M \to {\mb R}$ is a nondegenerate $1$-periodic Hamiltonian. Then the number of $1$-periodic orbits of $H$ is bounded below by
\beq\label{eqn:intro-low}
\text{rank } H_{*}(M;{\mb Q}) + 2\sum_{i \in {\mb Z}/2N} \tau_{i}^{(2N)}(M).
\eeq
\end{thma}

We briefly comment on the historical background of the Arnold conjecture. Arnold conjectured (\cite[Appendix 9]{Arnold_book}, \cite{Arnold_conjecture}) that the number of $1$-periodic orbits of any nondegenerate Hamiltonian is at least the minimal number of critical points of a Morse function on $M$ (the {\it strong} Arnold conjecture) and suggested that it is a consequence of a version of Morse inequality. While the strong Arnold conjecture is difficult to prove, there have been significant progresses towards the homological Arnold conjecture, namely, the number of periodic orbits being bounded from below by the total Betti number.\footnote{From now on, in this paper, the Arnold conjecture refers to the homological version.} These progresses started with the breakthrough of Conley--Zehnder \cite{Conley_Zehnder_1983}, who solved the strong Arnold conjecture for $T^{2n}$ using finite dimensional methods. Conley--Zehnder's method was later extended by Floer to solve the strong Arnold conjecture for surfaces with genus at least $2$ and for certain classes of K\"ahler manifolds \cite{Floer-Kahler}. It was then followed by the revolutionary work of Floer \cite{Floer_CMP} which invented the Floer homology and established the Arnold conjecture for $(M, \omega)$ satisfying $[\omega] = \lambda c_1(M)$ for some $\lambda \in {\mb R}$ (monotone symplectic manifolds). Floer's result was extended by Hofer--Salamon \cite{Hofer_Salamon} and Ono \cite{Ono_1995} to cover all semi-positive symplectic manifolds, i.e., $(M^{2n}, \omega)$ which does not have $A \in \pi_2(M)$ such that $\omega(A) > 0$ and $3-n \leq c_1(A) <0$. It is worth noting that these results hold over $\mathbb{Z}$. 

For general symplectic manifolds, Fukaya--Ono\cite{Fukaya_Ono}, Liu--Tian\cite{Liu_Tian_Floer}, and Ruan \cite{Ruan_virtual} proved the homological Arnold conjecture with lower bound coming from the \emph{rational} total Betti number. These papers are based on a kind of abstract machinery, generally called the ``virtual technique," which is designed for generalizing Floer's construction (along side with the mathematical theory of Gromov--Witten invariants). The most recent advancement towards the homological Arnold conjecture by Abouzaid--Blumberg \cite{Abouzaid_Blumberg}, which relies more heavily on stable homotopy theory, bounds the number of periodic orbits from below by the sum of Betti numbers in any finite field. In addition to the aforementioned works, using different versions of the virtual technique, the weak Arnold conjecture over rational numbers is reproved in \cite{pardon-VFC} and \cite{Filippenko_Wehrheim_2022}.

Theorem \ref{thm_intro_main} allows one to obtain a sharper lower bound for the number of $1$-periodic orbits of a given non-degenerate Hamiltonian by treating the torsion components of $H_{*}(M; \mathbb{Z})$ uniformly and simultaneously. The lower bound provided by ${\rm rank} H_{*}(M; \mathbb{F}_p)$ as from \cite{Abouzaid_Blumberg} is also recovered from Theorem \ref{thm_intro_main} by the universal coefficient theorem. Our bound is also strictly stronger than the bound of Abouzaid--Blumberg, for example, when $H_{\rm even} (M)$ has only $p$-torsion and $H_{\rm odd}(M)$ has only $q$-torsion and $p\neq q$.

\subsection{Proof strategy}

The proof of Theorem \ref{thm_intro_main} is based on constructing a version of Hamiltonian Floer homology over the Novikov ring $\Lambda$ with integer coefficients and exponents, where
\beqn
\Lambda:= \Big\{ \sum_{i=-m}^\infty a_i T^i\ |\ m \in {\mb Z},\ a_i \in {\mb Z} \Big\},
\eeqn
and a comparison with the Morse homology of $M$ with $\Lambda$-coefficients using the Piunikhin--Salamon--Schwarz (PSS) \cite{PSS} map. 

\begin{thma}\label{thm-intro-floer}
Let $(M, \omega)$ be a closed symplectic manifold such that $[\omega]$ is contained in the image of $H^2(M;{\mb Z}) \to H^2(M;{\mb R})$. Assume that the minimal Chern number of $(M, \omega)$ is $N$. Suppose $H: S^1 \times M \to {\mb R}$ is a non-degenerate $1$-periodic Hamiltonian such that the symplectic action of any capped $1$-periodic orbit of $H$ takes value in ${\mb Z}$. After choosing an almost complex structure $J$ compatible with $\omega$ and some other auxiliary data, there is a ${\mb Z}/2N$-graded chain complex 
$$
CF_{*}(H;\Lambda)
$$
freely generated over $\Lambda$ by all contractible $1$-periodic orbits of $H$, graded by the Conley--Zehnder index, with differential given by suitably counting stable Floer trajectories \emph{with trivial isotropy group} connecting the $1$-periodic orbits.
\end{thma}

\begin{thma}\label{thm-intro-pss}
Let $(M, \omega)$ and $H: S^1 \times M \to {\mb R}$ be the same as in Theorem \ref{thm-intro-floer}. Suppose $f: M \to {\mb R}$ is a smooth Morse function and let $CM_{*}(f;\Lambda)$ be the Morse chain complex associated with $f$ equipped with the reduced ${\mb Z}/2N$-grading. Then there exist a pair of $\Lambda$-linear chain maps
$$
\Psi^{\rm PSS}: CM_{*}(f;\Lambda) \to CF_{*}(H;\Lambda),
$$
$$
\Psi^{\rm SSP}: CF_{*}(H;\Lambda) \to CM_{*}(f;\Lambda),
$$
such that their composition satisfies
$$
\Psi^{\rm SSP} \circ \Psi^{\rm PSS} = Id + \text{terms with positive $T$-exponent}.
$$
\end{thma}

\begin{proof}[Proof of Theorem \ref{thm_intro_main}]
When $(M, \omega)$ and $H$ satisfy the assumptions in Theorem \ref{thm-intro-floer}, Theorem \ref{thm-intro-pss} implies that the induced map on homology
$$
\Psi^{\rm PSS}: H_{*}(M;\Lambda) \to HF_{*}(H;\Lambda)
$$
in an injection. Using the algebraic arguments in Section \ref{subsec-algebra}, we see that Theorem \ref{thm_intro_main} holds in this setting.

It is easy to see that if $[\omega]$ lies in the image of $H^2(M;{\mb Q}) \to H^2(M;{\mb R})$ and the symplectic action of $1$-periodic orbits of $H$ are all ${\mb Q}$-valued, Theorem \ref{thm_intro_main} also holds. Indeed, one can suitably rescale the symplectic form $\omega$ and the Hamiltonian $H$ by a positive integer to reduce to the integral setting, because such a rescaling process does not change the number of $1$-periodic orbits of $H$.

In general, we can choose a sequence of symplectic forms $\{ \omega_{k} \}$ which represent rational cohomology classes and converge to $\omega$ as $k \to \infty$. If $H: S^1 \times M \to {\mb R}$ is nondegenerate, then for $k$ sufficiently large, the $1$-periodic orbits of $H$ with respect to $\omega_k$ has a one-to-one correspondence with the $1$-periodic orbits of $H$ with respect to $\omega$. Therefore, without loss of generality, we can assume that $[\omega]$ represents a rational cohomology class. Based on the elementary discussion after Hypothesis \ref{hyp51}, we can modify $H$ to obtain a $1$-periodic Hamiltonian whose $1$-periodic orbits all have ${\mb Q}$-valued symplectic action without changing the number of $1$-periodic orbits. Thus the theorem follows from the previous discussions.
\end{proof}

The proof of Theorem \ref{thm-intro-floer} has two major steps, which account for the most important novelties of this paper. First, we show that any moduli space of stable Floer trajectories can be \emph{globally} presented as the zero locus of a continuous section on a \emph{smooth normally complex} orbifold vector bundle over a \emph{smooth normally complex} orbifold, by generalizing a recent result of Abouzaid--McLean--Smith  \cite{AMS} which works for the moduli space of closed genus $0$ $J$-holomorphic curves. Moreover, we show that these presentations are coherent, packaged using the language of \emph{flow categories} introduced by Cohen--Jones--Segal \cite{Cohen_Jones_Segal}. Second, we apply the \emph{Fukaya--Ono--Parker (FOP) perturbation scheme} introduced in our early work \cite{Bai_Xu_2022}, which was originally proposed by Fukaya--Ono \cite{Fukaya_Ono_integer}, to obtain ${\mb Z}$-valued virtual fundamental chains from the aforementioned presentations of the moduli spaces of stable Floer trajectories, and show that the induced algebraic counts from moduli spaces of virtual dimension $0$ can be organized to define a chain complex over $\Lambda$ freely generated by contractible $1$-periodic orbits of $H$.

The proof of Theorem \ref{thm-intro-pss} is quite similar to the proof of Theorem \ref{thm-intro-floer}, except that the relevant moduli spaces parametrize solutions of the PSS/SSP equations, which are also known as the ``spiked discs." We package these moduli spaces and their presentations using the language of \emph{flow bimodules}, which are investigated systematically in the symplectic context in the recent work of Abouzaid \cite{abouzaid2022axiomatic}.

Next, we give a slightly more detailed description of the technical ingredients.

\subsection{Flow categories and derived orbifold presentation}

In this subsection, we discuss about the geometric construction of compatible global charts for the moduli spaces arising from the proof of Theorem \ref{thm_intro_main}.

Given a closed symplectic manifold $(M, \omega)$ and a nondegenerate $1$-periodic Hamiltonian $H: S^1 \times M \to {\mb R}$, following \cite{Cohen_Jones_Segal}, one can introduce a topologically enriched category $T^{\floer}$ such that:
\begin{enumerate}

\item The objects are capped $1$-periodic orbits of $H$.

\item The morphism space from $p$ to $q$ is the compact Hausdorff topological space $\ov{\mc M}{}^\floer_{pq}$, the moduli space of stable Floer trajectories connecting $p$ and $q$.
\item For $p,r,q \in T^{\floer}$, the composition map
$$
\ov{\mc M}{}^\floer_{pr} \times \ov{\mc M}{}^\floer_{rq} \to \ov{\mc M}{}^\floer_{pq}
$$
is given by the concatenation which forms a broken Floer trajectory breaking at $r$. The composition maps are homeomorphisms onto their images and satisfy the natural associativity relation.
\end{enumerate}
The composition maps naturally equip the space $\ov{\mc M}{}^\floer_{pq}$ with a stratification structure, with strata indexed by words of the form $pr_1 \cdots r_l q$. The $pr_1 \cdots r_l q$-stratum of $\ov{\mc M}{}^\floer_{pq}$ is homeomorphic to the product
$$
\ov{\mc M}{}^\floer_{pr_1} \times \cdots \times \ov{\mc M}{}^\floer_{r_l q}.
$$
The moduli space $\ov{\mc M}{}^\floer_{pq}$ is only an orbispace in general, though the usual Kuranishi reduction process describes it \emph{locally} as the zero locus of a section on an orbifold vector bundle over an orbifold. Moreover, even if every $\ov{\mc M}{}^\floer_{pq}$ has a global Kuranishi model, its restriction to the boundary stratum $\ov{\mc M}{}^\floer_{pr_1} \times \cdots \times \ov{\mc M}{}^\floer_{r_l q}$ may not be the product of the global Kuranishi models on each factor. This would obstruct any meaningful inductive construction on $T^{\floer}$ from these Kuranishi models, leaving alone the issue concerning the \emph{smoothness} of the (thickened) moduli spaces. We solve all of these problems.

\begin{thm}(See Theorem \ref{thm:dorb-lift})
Let $(M, \omega)$ and $H: S^1 \times M \to {\mb R}$ be as in Theorem \ref{thm-intro-floer}. Then for any capped $1$-periodic orbits $p,q$ of $H$ such that $\ov{\mc M}{}^\floer_{pq} \neq \emptyset$, there exist a smooth effective normally complex orbifold with corners ${\mc U}_{pq}$ with the same stratification structure as $\ov{\mc M}{}^\floer_{pq}$, a smooth normally complex orbifold vector bundle ${\mc E}_{pq} \to {\mc U}_{pq}$, a continuous section ${\mc S}_{pq}: {\mc U}_{pq} \to {\mc E}_{pq}$, and a map
$$
\psi_{pq}: {\mc S}^{-1}(0) \to \ov{\mc M}{}^\floer_{pq}
$$
which defines an isomorphism of orbispaces.
\end{thm}

We call the quadruple $({\mc U}_{pq}, {\mc E}_{pq}, {\mc S}_{pq}, \psi_{pq})$ a \emph{derived orbifold chart} (D-chart for short) of $\ov{\mc M}{}^\floer_{pq}$. Given an index $pr_1 \cdots r_l q$, one can take the product
\begin{equation}\label{eqn-intro-prod}
({\mc U}_{pr_1} \times \cdots \times {\mc U}_{r_l q}, {\mc E}_{pr_1} \boxplus \cdots \boxplus {\mc U}_{r_l q}, {\mc S}_{pr_1} \boxplus \cdots \boxplus {\mc S}_{r_l q} , \psi_{pr_1} \times \cdots \times \psi_{r_l q})
\end{equation}
which defines a D-chart of $\ov{\mc M}{}^\floer_{pr_1} \times \cdots \times \ov{\mc M}{}^\floer_{r_l q}$.

\begin{thm}\label{thm-intro-1.2}(See Theorem \ref{thm:dorb-lift})
The restriction of $({\mc U}_{pq}, {\mc E}_{pq}, {\mc S}_{pq}, \psi_{pq})$ to the stratum of ${\mc U}_{pq}$ indexed by $pr_1 \cdots r_l q$ is isomorphic to the stabilization (Definition \ref{stabilization}) of the product chart \eqref{eqn-intro-prod} by a complex orbifold vector bundle. 
\end{thm}

If the above data satisfy some further compatibility conditions spelled out in Definition \ref{defn:flow-lift}, we call this system of derived orbifold charts of $\ov{\mc M}{}^\floer_{pq}$ a (normally complex) \emph{derived orbifold lift} of the flow category $T^{\floer}$. The full statement of Theorem \ref{thm:dorb-lift} exactly asserts the existence of an oriented and normally complex derived orbifold lift of $T^{\floer}$.

As mentioned above, the construction of derived orbifold chart for a single moduli space $\ov{\mc M}{}^\floer_{pq}$ is based on generalizing a recent result of Abouzaid--McLean--Smith \cite[Theorem 1.7]{AMS}. We replace their framed $J$-holomorphic spheres by framed $J$-holomorphic cylinders and modify the perturbation method accordingly. The main difficulty is the compatibility mentioned above. To this end, we introduce the \emph{multi-layered thickening} (see Section \ref{subsubsec:multi}) to make sure that the thickened moduli space ${\mc U}_{pr_1} \times \cdots \times {\mc U}_{r_l q}$ is truly embedded in ${\mc U}_{pq}$. To endow the space ${\mc U}_{pq}$ with a smooth structure, we follow the application of classical smoothing theory from \cite{AMS}. However, because ${\mc U}_{pq}$ is an orbifold with corners in general, the traditional smoothing theory does not suffice for our purpose. We develop a \emph{relative equivariant smoothing} technique to construct the smoothing in an inductive fashion, which also makes the resulting smooth structures on various moduli spaces compatible.

\begin{rem}
We expect that the construction of derived orbifold charts presented in this paper could be useful in the study of symplectic field theory (SFT) \cite{SFT} because the methods presented can be used to deal with punctures.
\end{rem}

\begin{rem}
The derived orbifold chart lift of the flow category $T^\floer$ should also be applied to construct certain Hamiltonian Floer homotopy theory for a complex-oriented generalized cohomology theory. As a first application, these geometric constructions should simplify the definition of Floer Morava $K$-theory \cite{Abouzaid_Blumberg} because the algebraic gluing of local virtual fundamental chains should be largely simplified given these global charts.
\end{rem}

\begin{rem}
We would like to remark that there are some other technical issues one needs to resolve in order to carry out our perturbation scheme on the derived orbifold charts. For instance, to make sure an inductive scheme could work, one needs to guarantee that any prescribed construction on the ``boundary statum" \eqref{eqn-intro-prod} can be extended to a neighborhood of it in $({\mc U}_{pq}, {\mc E}_{pq})$. To this end, a suitable collar structure is necessary and one has to keep track of the compatibility of such collar structures with the ``difference bundle" appearing in Theorem \ref{thm-intro-1.2}. Considerations of this form partially account for the length of the later sections, because we will need a more refined geometric construction compared to the one in \cite{AMS}.
\end{rem}

Given a Morse function $f: M \to {\mb R}$ and a Riemannian metric $g$ on $M$, if the pair $(f,g)$ satisfies certain properties stronger than the usual Morse--Smale condition, the work of Wehrheim \cite{Wehrheim-Morse} shows that the moduli spaces of unpamatrized gradient flow lines of $f$ with respect to $g$ are compact smooth manifolds with corners. Therefore, one can associate $(f,g)$ with a flow category $T^\morse$ such that:
\begin{enumerate}
    \item The objects are ``capped" critical points of $f$, see Section \ref{subsec:morse-flow}.
    \item The morphism space $\ov{\mc M}{}_{xy}^{\morse}$ is given by the moduli space of unparametrized gradient flow lines connecting the underlying critical points of $x$ and $y$ if the cappings are the same, otherwise it is the empty set.
    \item The composition
    $$
    \ov{\mc M}{}_{xz}^{\morse} \times \ov{\mc M}{}_{zy}^{\morse} \to \ov{\mc M}{}_{xy}^{\morse}
    $$
    is defined to be the concatenation of gradient flow lines, which is a diffeomorphism onto its image.
\end{enumerate}

To relate the two flow categories $T^\floer$ and $T^\morse$, we introduce the PSS and SSP \emph{flow bimodules} $M^\pss$ and $M^\ssp$. For instance, $M^\pss$ is described by the following data:
\begin{enumerate}

    \item For $x \in T^\morse$ and $p \in T^\floer$, we associate the pair $(x,p)$ with the moduli space of solutions to PSS equations $\ov{\mc M}{}^{\pss}_{xp}$ connecting $x$ and $p$. 
    
    \item The flow category $T^\morse$ acts on $M^\pss$ on the left, i.e., there is a map
    $$
    \ov{\mc M}{}^\morse_{xy} \times \ov{\mc M}{}^{\pss}_{yp} \to \ov{\mc M}{}^{\pss}_{xp}
    $$
    satisfying an associativity condition. Geometrically, this map is again given by concatenating solutions to gradient flow line equations and Floer-type equations, and it is a homeomorphism onto its image.
    \item The flow category $T^\floer$ acts on $M^\pss$ on the right in a similar fashion, i.e., there is a map
    $$
    \ov{\mc M}{}^\pss_{xq} \times \ov{\mc M}{}^{\floer}_{qp} \to \ov{\mc M}{}^{\pss}_{xp}
    $$
    satisfying an associativity condition.
    \item The right and left actions should satisfy certain associativity relation.
\end{enumerate}

The flow bimodule $M^\ssp$ is characterized similarly. Just like flow categories, one can define the notion of derived orbifold lift for flow bimodules (Definition \ref{defn:bimod-lift}). It turns out that we can indeed construct a derived orbifold lift for both $M^\pss$ and $M^\ssp$ via geometric arguments.

\begin{thm}(See Theorem \ref{thm:pss-ssp})
Each of the flow bimodules $M^\pss$ and $M^\ssp$ admits a derived orbifold lift (which extends the D-chart lift of $T^\floer$ and the trivial lift of $T^\morse$).
\end{thm}

In addition to these statements, the geometric underpinning of Theorem \ref{thm-intro-pss} also needs certain moduli spaces designed for interpolating between $\Psi^{\rm SSP} \circ \Psi^{\rm PSS}$ and the identity map on $H_{*}(M;\Lambda)$. Such a structure should be better thought of as a morphism between two flow bimodules ``parametrized" by the interval $[-1,1]$, as exploited in detail in \cite{abouzaid2022axiomatic}. We do not follow this viewpoint in this paper. Instead, we simply list out the necessary geometric input in Section \ref{sec-4} for the proof of Theorem \ref{thm-intro-pss}.

\subsection{Fukaya--Ono--Parker perturbations}
Now we discuss about the perturbation methods adapted in the course of the proof of Theorem \ref{thm_intro_main}.

The perturbation scheme to be used in our construction was proposed by Fukaya--Ono \cite{Fukaya_Ono_integer} and further developed by B. Parker \cite{BParker_integer}. So this scheme will be referred to as the FOP perturbation scheme. We give a detailed introduction to its idea. For complete details, the reader could consult \cite{Bai_Xu_2022}. In a moduli problem related to pseudo-holomorphic curves, objects (such as stable maps) may have nontrivial automorphism groups; in the case of Hamiltonian Floer theory, the sphere bubbles in a stable Floer trajectory may be multiple covers carrying nontrivial automorphisms. This means that the moduli spaces have certain orbifold-type behavior. It is well-known that transversality could not always be achieved in the orbifold setting. To define invariants using abstract perturbations (see \cite{Fukaya_Ono} \cite{Li_Tian}), one needs to use ``multi-valued perturbations'' and count of transverse zeroes with ${\mb Q}$-valued weights. Such constructions result in ${\mb Q}$-valued invariants or chain complexes with ${\mb Q}$-coefficients. Another more algebraic reasoning to explain the appearance of ${\mb Q}$ is that Poincar\'e duality for oriented orbifolds only holds over ${\mb Q}$ (see the discussions of Pardon \cite{pardon-VFC}). These two perspectives could account for the reasons why Gromov--Witten invariants and Hamiltonian Floer homology of a general symplectic manifold are only defined over the field of rational numbers. 

Let us set up a finite-dimensional model to illustrate the failure of equivariant transversality and the proposal of Fukaya--Ono. Consider a compact smooth manifold $U$ acted on by a finite group $\Gammait$ and ${\bm W}$ is a representation of $\Gammait$. We may regard the orbifold $U / \Gammait$ as the moduli space and the orbifold vector bundle $(U \times {\bm W}) /\Gammait$ as the obstruction bundle. To simplify the exposition even further, let us assume ${\rm dim } U = {\rm dim} {\bm W}$. In this case, a weighted count of the zeroes of an equivariant transverse maps $S: U \to {\bm W}$ will give an Euler number, which is an invariant of the obstruction bundle over the orbifold $U / \Gammait$. However, the equivariance of $S$ implies that $S^{-1}(0)$ may contain higher dimensional pieces; for example, when ${\bm W}$ contains no trivial subreprsentations, the fixed point locus $U^\Gammait$ must be contained in $S^{-1}(0)$. Hence ordinary transversality cannot be achieved equivariantly in general. 

The existence of a \emph{normal complex structure} allows one to consider a more delicate kind of section so that one form of equivariant transversality \emph{can} hold. In the above finite-dimensional model, consider the normal bundle $NU^\Gammait \to U^\Gammait$ over the $\Gammait$-fixed point set $U^\Gammait$. Meanwhile, consider the decomposition 
\beqn
{\bm W} = \mathring {\bm W}^\Gammait \oplus \check {\bm W}^\Gammait
\eeqn
where $\mathring {\bm W}^\Gammait$ is the direct sum of trivial subrepresentations and $\check {\bm W}^\Gammait$ is the direct sum of nontrivial irreducible subrepresentations. A normal complex structure in this case is a $\Gammait$-equivariant complex structure on both $NU^\Gammait$ and $\check {\bm W}^\Gammait$. To illustrate the FOP perturbation scheme, assume for simplicity that $NU^\Gammait$ is trivial with fiber a complex $\Gammait$-representation ${\bm V}$ and $\mathring {\bm W}^\Gammait = \{0\}$. A section $S$ defined near $U^\Gammait$ can then be regarded as a map 
\beqn
f: U^\Gammait \to C^\infty({\bm V}, {\bm W})^\Gammait,
\eeqn
where $C^\infty( {\bm V}, {\bm W} )^\Gammait$ denotes the space of smooth $\Gammait$-equivariant maps from ${\bm V}$ to ${\bm W}$. In \cite{Fukaya_Ono_integer} Fukaya--Ono proposed to consider sections corresponding to fiberwise polynomial maps 
\beqn
f: U^\Gammait \to {\rm Poly}_d^\Gammait({\bm V}, {\bm W}),
\eeqn
i.e., sections of $E$ whose restriction to each normal fiber is an equivariant complex polynomial map of degree at most $d$. Assume $U = U^\Gammait \times {\bm V}$ and denote the corresponding section by $S_f: U \to {\bm W}$. Then one has 
\beqn
S_f^{-1}(0) = U^\Gammait \cup \big( S_f^{-1}(0) \cap (U \setminus U^\Gammait)\big).
\eeqn
Although the zero locus of $S_f$ still contains $U^\Gammait$, Fukaya--Ono asserted that the count of the second component (modulo $\Gammait$), is an invariant, if $d$ is sufficiently large and $f$ is generic. If this is true, one could indeed define an integer-valued ``Euler number" by counting the zeroes of $S_f$ with trivial stabilizer.

A difficulty to implement Fukaya--Ono's idea is that we need to introduce a new notion of transversality for those fiberwise polynomial maps. Certain delicate properties are needed for this transversality notion, especially when we change the cut-off degree $d$ for polynomial maps and when we change the isotropy groups. The crucial idea in the preprint \cite{BParker_integer} of Brett Parker which addresses these difficulties played an important role in the recent construction of the authors \cite{Bai_Xu_2022}. In short, one can define a canonical notion of transversality for those normally polynomial perturbations which behaves well when we change the degree $d$ and the isotropy group and which is satisfifed by generic such perturbations such that the integral counting is well-defined as proposed in \cite{Fukaya_Ono_integer}. 

Going back to the Floer-theoretic construction, given a derived orbifold lift of the flow category $T^\floer$, we need to construct the FOP perturbation on all the thickened moduli spaces inductively. Although such a construction is cumbersome and somewhat routine, one novelty showing up is the \emph{multiplicativity} of the FOP perturbation.

\begin{thm}(See Corollary 2.8)
Let $({\mc U}_i, {\mc E}_i, {\mc S}_i)$, $i = 1, 2$ be normally complex derived orbifold charts (Definition \ref{defn21}) such that ${\mc S}_i$ is a strongly transverse FOP section (Definition \ref{defn:strongly_transverse}). Then the product section 
\beqn
{\mc S}_1 \boxplus {\mc S}_2: {\mc U}_1 \times {\mc U}_2 \to {\mc E}_1 \boxplus {\mc E}_2
\eeqn
is also a strongly transverse FOP section.
\end{thm}

The upshot of the inductive construction of FOP perturbations can be summarized schematically as follows, after defining the relevant counts using $0$-dimensional moduli spaces.

\begin{thm}(See Theorem \ref{thm_FOP_2})
Using the derived orbifold lift of $T^\floer$, one can define the chain complex $CF_{*}(H;\Lambda)$ from Theorem \ref{thm-intro-floer} after choosing a compatible family of strongly transverse FOP perturbations. Similarly, the chain maps $\Psi^{\rm PSS}$, $\Psi^{\rm SSP}$, and the homotopy
$$
\Psi^{\rm SSP} \circ \Psi^{\rm PSS} = Id + O(T)
$$
can be constructed by choosing a compatible family of strongly transverse FOP perturbations on the derived orbifold lifts of the relevant moduli spaces.
\end{thm}

As a consequence, Theorem \ref{thm-intro-floer} and Theorem \ref{thm-intro-pss} hold.

\begin{rem}
We would like to point out that the definition of FOP sections depend on an extra structure which is called a \emph{straightening} (see Definition \ref{defn:straightening_2}). Heuristically, it consists of a suitable Riemannian metric on ${\mc U}_{pq}$ and a special form of connection on ${\mc E}_{pq}$ for each of the derived orbifold charts appearing in the derived orbifold lift. Again, there is an issue of compatibility with various structures on our derived orbifold lift. Although the relevant compatibility conditions will be shown to hold after a somewhat routine and cumbersome induction construction, pinning down the correct formulations seems to be a nontrivial task.
\end{rem}

\subsection{Discussions}
We would like to comment on the implications of our result and some notable features of the techniques.

\subsubsection{Towards stronger forms of the Arnold conjecture}
One important quantity in the statement of Theorem \ref{thm_intro_main} is the minimal Chern number. In the extremal case when $N=1$, if the only torsion components of $H_{*}(M;{\mb Z})$ are ${\mb Z}/2$ and ${\mb Z}/3$, and both of them have odd degree, then they together contribute $2$ to the quantity \eqref{eqn:intro-low}. In other words, our lower bound is only as strong as the lower bound from \cite{Abouzaid_Blumberg} in this special case. This might suggest that the sharp lower bound in the Arnold conjecture might need to incorporate certain symplectic information from the ambient symplectic manifold, as witnessed by the minimal Chern number in our statement.

On the other hand, it is possible that the methods developed in this paper might eventually resolve the \emph{strong} Arnold conjecture for simply connected closed symplectic manifolds of dimension $\geq 6$. Indeed, following the bifurcation methods initiated in \cite{Floer_unregularized}, it should be possible to study the simple homotopy type of our integral Floer chain complex, which in turn is closely related to the \emph{stable} Morse number of the ambient manifold. When the ambient manifold satisfies the condition at the beginning of this paragraph, the stable Morse number actually coincides with the Morse number.

\subsubsection{Stable/normal complex structure}
As emphasized in the exposition of FOP perturbations, the normal complex structure of derived orbifold charts plays a crucial role even in the definition of these polynomial-like perturbations. The closely related notion of stable complex structures also plays an important role in the construction of Floer homotopy type in \cite{Abouzaid_Blumberg}, in which the stable complex structures are necessary for applying Poibcar\'e duality to the Morava K-theory of the classifying space of orbifolds. These structures are not necessary for the proof of Arnold conjecture over ${\mb Q}$. It is natural to expect more applications which make essential use of the normal/stable complex structures on the moduli spaces.

\subsubsection{Algebraic structures and operations}
Following the concept of ``flow multimodules" as developed in \cite{abouzaid2022axiomatic}, one should be able to define multiplicative structure on the integral Floer homology. For instance, it should be possible to generalize the definition of quantum Steenrod operations on Hamiltonian Floer theory of semi-positive symplectic manifolds \cite{wilkins2020construction} to all symplectic manifolds using the ${\mb F}_p$-reduction of our Hamiltonian Floer homology. Our construction of derived orbifold lifts should also be useful for regularizing moduli spaces of $J$-holomorphic curves originated from the algebraic structures.

\subsection{Outline}
The following describes the content of every section.

\begin{itemize}

\item In Section \ref{section2}, we review the differential topology related to the perturbation scheme used in this paper. Most notably, the multiplicativity of the FOP strongly transverse condition is derived in Corollary \ref{cor28}.

\item Section \ref{sec-3} is devoted to the discussions of flow categories, flow bimodules and their derived orbifold lifts. We show that for a flow category whose derived orbifold lift has a normal complex structure, it is possible to construct a family of compatible strongly transverse FOP perturbations so that one can define a chain complex over $\Lambda$ from these data if there is a further compatible orientation structure. The discussions culminate at Theorem \ref{thm_FOP_2}. A similar result for flow bimodules is presented as well.

\item We bring back concrete symplectic geometry in Section \ref{sec-4}. In particular, we describe the flow categories $T^\floer$, $T^\morse$, and flow bimodules $M^\pss$, $M^\ssp$, and the structures of the interpolating moduli spaces in detail. After presenting the main geometric statements in Section \ref{subsec-main} and the necessary input from a $1$-parameter family of PSS/SSP moduli spaces (``the chain homotopy moduli spaces") in Section \ref{subsec:hmtp}, we prove our main theorem.

\item In Section \ref{sec-5}, we show how to construct a compatible family of \emph{topological} global Kuranishi charts for the moduli spaces $\ov{\mc M}_{pq}$.

\item In Section \ref{sec-6}, We discuss how to use (relative) equivariant smoothing theory to endow the global Kuranishi charts constructed in Section \ref{sec-5} with smooth structures. Moreover, we describe how to construct a normally complex lift of all derived orbifold lifts.

\item The parallel constructions for PSS, SSP type moduli spaces as in Section \ref{sec-5} and Section \ref{sec-6} are presented in Section \ref{sec:pss}. Because most of the arguments are only a matter of cosmetic modification, most of the proofs are only sketched.

\end{itemize}

\begin{rem}
We were informed by Semon Rezchikov on September 14, 2022 that he had an independent approach towards similar results following our early work \cite{Bai_Xu_2022}.
\end{rem}

%\subsection{Conventions}

%\textcolor{red}{$2\pi$ as the period.}

\subsection*{Acknowledgements}

We thank Mohammed Abouzaid, Kenji Fukaya, Helmut Hofer, Suguru Ishikawa, Alexander Kupers, John Pardon, Mohan Swaminathan, and Weiwei Wu for interesting discussions.

This paper is completed during the second author's visit at the Institute for Advanced Study hosted by Professor Helmut Hofer. He would like to thank IAS and Professor Hofer for their hospitality. 

The second author would like to express his sincere gratitude to his family, especially his wife, Dr. Ning Lin, for her understanding, support, advice, encouragement, and love.

\section{Recap of the FOP natural transformation and multiplicativity}\label{section2}

In this section we briefly review the construction of \cite{Bai_Xu_2022} concerning the Fukaya--Ono--Parker perturbations and prove an additional property of FOP perturbations regarding the multiplicativity. Some necessary ingredients from the theory of Whitney stratifications are provided in Appendix \ref{appendixa}.

\subsection{Orbifolds and the canonical Whitney stratification on $\mathcal{Z}_d$}\label{subsec:orbi-setup}

Now we describe the setup for FOP perturbations. In this section we consider effective\footnote{A stabilization can make a non-effective orbifold effective.} orbifolds only. An $n$-dimensional {\bf orbifold} is a locally compact, Hausdorff, and second countable topological space ${\mc U}$ equipped with an atlas of orbifold charts: each chart is of the form $C = (U, \Gammait, \psi)$ where $\Gammait$ is a finite group with an effective linear action on ${\mb R}^n$, $U \subset {\mb R}^n$ is an invariant open subset, and $\psi: U \to {\mc U}$ is a $\Gammait$-invariant continuous map such that the induced map $U/\Gammait \to {\mc U}$ is a homeomorphism onto an open subset of ${\mc U}$. There is a compatibility requirement for overlapping charts which we will not recall here. In notation the map $\psi$ is often suppressed and we identify its image with $U/\Gammait$. Similarly, over an orbifold ${\mc U}$, an orbifold vector bundle consists of an orbifold ${\mc E}$, a continuous map $\pi_{\mc E}: {\mc E} \to {\mc U}$, and an atlas of bundle charts of ${\mc E}$. A bundle chart consists of a chart $C = (U, \Gammait)$ of the base ${\mc U}$, a $\Gammait$-equivariant vector bundle $\pi: E \to U$, and a homeomorphism $E /\Gammait \cong \pi_{\mc E}^{-1}(U/\Gammait)$ which is compatible with the projections. We denote by the triple $(\Gammait, E, U)$ a bundle chart. We ofter use charts which are centered at certain points of the orbifold. For each point $x \in {\mc U}$, an orbifold chart {\bf centered at} $x$ is a chart $C_x = (U_x, \Gammait_x, \psi_x)$ where $U_x \subset {\mb R}^n$ is an invariant open neighborhood of the origin, such that $\psi_x(0) = x$. A bundle chart for ${\mc E} \to {\mc U}$ centered at $x$ is then denoted by $(\Gammait_x, E_x, U_x)$ where $(\Gammait_x, U_x)$ is an orbifold chart centered at $x$ and $E_x \to U_x$ is a $\Gammait_x$-equivariant vector bundle.

We need a few other frequently used notations. Let $(\Gammait, E, U)$ be a bundle chart. For each subgroup $G \subset \Gammait$, denote by $U^G \subset U$ the fixed point locus of the induced $G$-action and denote by $NU^G$ the normal bundle of $U^G \hookrightarrow U$. Then the fibers of $NU^G$ are representations of $G$ whose decompositions into the direct sum of irreducible representations contain no trivial summand. We can also decompose the restriction
\beqn
E|_{U^G} = \mathring E^G \oplus \check E^G 
\eeqn
where $\mathring E^G \subset E$ is the subbundle whose fibers are the maximal trivial subrepresentations and $\check E^G$ is the complement, whose fibers are direct sums of nontrivial irreducible representations.

\subsubsection{Derived orbifold charts}

We recall the notion of derived orbifold charts introduced in \cite[Section 5]{pardon2020orbifold} and used in \cite{Bai_Xu_2022}, as well as a few related concepts. We will soon generalize this notion to the case of orbifolds with faces but this generalization is straightforward.

\begin{defn}\label{defn21}\hfill
\begin{enumerate}
    \item An (effective) {\bf derived orbifold chart} ({\bf D-chart} for short) is a triple $({\mc U}, {\mc E}, {\mc S})$ where ${\mc U}$ is an effective orbifold, ${\mc E}\to {\mc U}$ is an orbifold vector bundle, and ${\mc S}: {\mc U} \to {\mc E}$ is a continuous section. We say the triple $({\mc U}, {\mc E}, {\mc S})$ is compact if ${\mc S}^{-1}(0)$ is compact. 
    
    \item A (smooth) {\bf perturbation} of a compact derived orbifold chart $({\mc U}, {\mc E}, {\mc S})$ is a smooth section ${\mc S}': {\mc U} \to {\mc E}$ such that there exists a precompact open neighborhood ${\mc D}$ of ${\mc S}^{-1}(0)$ and a continuous norm on ${\mc E}$ such that
    \beqn
    \| {\mc S} - {\mc S}'\|_{C^0({\mc U}\setminus {\mc D})} < \inf_{x\in {\mc U}\setminus {\mc D}} | {\mc S}(x)|.
    \eeqn
    In particular, $({\mc S}')^{-1}(0)$ is still contained in ${\mc D}$ and hence $({\mc U}, {\mc E}, {\mc S}')$ is also a compact derived orbifold chart. 
    
    \item A {\bf normal complex structure} on a derived orbifold chart $({\mc U}, {\mc E}, {\mc S})$ consists of, for each bundle chart $(\Gammait, E, U)$ and for each subgroup $G \subseteq \Gammait$, a $G$-invariant complex structure $I_G$ on $NU^G$ and a $G$-invariant complex structure $J_G$ on $\check E^G$. Moreover, these complex structures are compatible in the following sense. \begin{itemize}
        
    \item Within the same chart, for each pair of subgroups $H \subset G\subset \Gammait$ for which we have $H$-equivariant inclusions
    \begin{align*}
    &\ NU^H|_{U^G} \subset NU^G,\ &\ \check E^H|_{U^G} \subset \check E^G
    \end{align*}
    we require that they are complex linear with respect to the complex structures $I_G$, $J_G$ and $I_H, J_H$. 
    
    \item The system of invariant complex structures are compatible with chart embeddings. 
    \end{itemize}
    
    \item A {\bf straightening} of a derived orbifold chart $({\mc U}, {\mc E}, {\mc S})$ consists of a Riemannian metric on ${\mc U}$ and a connection on ${\mc E}$ satisfying the following conditions.
    \begin{itemize}
        \item For each chart $(U, E, \Gammait)$, the pullback Riemannian metric $g_U$ is ``straightened.'' Namely, for each subgroup $G \subset \Gammait$, near $U^G$, the ambient Riemannian metric $g_U$ agrees with the bundle metric on $NU^G$ induced by $g_U$ via the exponential map along the normal directions.
        
        \item For each chart $(U, E, \Gammait)$, the pullback connection $\nabla^E$ on $E$ is ``straightened.'' Namely, for each subgroup $G \subset \Gammait$, we can identify a neighborhood of $U^G$ with a neighborhood of the zero section of $NU^G$ using the exponential map associated with $g_U$. After identifying $E|_{NU^G}$ with the pullback of $E|_{U^G}$ using the projection $NU^G \to U^G$ and the parallel transport along normal geodesics using $\nabla^E$, the connection $\nabla^E$ agrees with the pullback connection of the restriction of $\nabla^E$ to $U^G$.
    \end{itemize}
    \end{enumerate}
\end{defn}

\begin{lemma}\cite[Lemma 3.15, Lemma 3.20]{Bai_Xu_2022} For a compact derived orbifold chart $({\mc U}, {\mc E}, {\mc S})$ there exists a straightening in a neighborhood of ${\mc S}^{-1}(0)$.
\end{lemma}

There are several natural relations between derived orbifold charts. 

\begin{defn}\hfill\label{stabilization}
\begin{enumerate}
    \item An {\bf open embedding} from a derived orbifold chart $({\mc U}, {\mc E}, {\mc S})$ to $({\mc U}', {\mc E}', {\mc S}')$ consists of an open embedding $\phi: {\mc U} \to {\mc U}'$ of orbifolds and a bundle isomorphism $\widehat \phi: {\mc E} \to {\mc E}'|_{\phi({\mc U})}$ covering $\phi$ such that $\widehat \phi \circ {\mc S}= {\mc S}' \circ \phi$ and $({\mc S}')^{-1}(0) \subset \phi({\mc U})$.
    
    \item A {\bf germ} of open embeddings from a derived orbifold chart $({\mc U}, {\mc E}, {\mc S})$ to $({\mc U}', {\mc E}', {\mc S}')$ is an equivalence class of open embeddings from an open neighborhood of ${\mc S}^{-1}(0)$ with the restrictions of ${\mc E}$ and ${\mc S}$ over it, to $({\mc U}', {\mc E}', {\mc S}')$, where the equivalence relation is induced by shrinking the neighborhood. Such a germ is called to induce a \textbf{germ equivalence}.
    
    \item The {\bf product} of derived orbifold charts $({\mc U}_i, {\mc E}_i, {\mc S}_i)$, $i = 1, \ldots, k$, is the derived orbifold chart
    \beqn
    ({\mc U}_1\times \cdots \times {\mc U}_k, {\mc E}_1 \boxplus \cdots \boxplus {\mc E}_k, {\mc S}_1 \boxplus \cdots \boxplus {\mc S}_k).
    \eeqn
     
    \item Let ${\mc C}= ({\mc U}, {\mc E}, {\mc S})$ be a derived orbifold chart and $\pi_{\mc F}: {\mc F} \to {\mc U}$ be another orbifold vector bundle. The {\bf stabilization} of ${\mc C}$ by ${\mc F}$ is the chart
    \beqn
    {\rm Stab}_{\mc F}({\mc C}) = ( {\mc F}, \pi_{\mc F}^* {\mc E}\oplus \pi_{\mc F}^* {\mc F}, \pi_{\mc F}^* {\mc S} \oplus \tau_{\mc F})
    \eeqn
    where $\tau_{\mc F}: {\mc F} \to \pi_{\mc F}^* {\mc F}$ is the tautological section.
\end{enumerate}
\end{defn}

\begin{rem}
The stabilization operation intertwines with many other constructions. For example, suppose $\phi: {\mc C} \to {\mc C}'$ is an open embedding and ${\mc F}' \to {\mc U}'$ is a orbifold vector bundle. Denote by ${\mc F}$ the pullback bundle $\phi^* {\mc F}$. Then there is an obvious extension of $\phi$ which defines an open embedding from ${\rm Stab}_{\mc F}({\mc C})$ into ${\rm Stab}_{{\mc F}'}({\mc C}')$. On the other hand, if ${\mc C}$ is equipped with a straightening, one can endow the stabilization ${\rm Stab}_{\mc F} ({\mc C})$ with a straightening naturally once ${\mc F}$ is equipped with a bundle metric and a compatible connection.
\end{rem}

\subsubsection{Equivariant polynomial maps}

To describe the use of FOP sections we need to recall some basic properties of equivariant polynomial maps. Let $G$ be a finite group and ${\bm V}, {\bm W}$ be finite dimensional complex $G$-representations. Let ${\rm Poly}^G({\bm V}, {\bm W})$ be the space of equivariant complex polynomial maps from $V$ to $W$, and for each nonnegative integer $d$, let ${\rm Poly}_d^G({\bm V}, {\bm W}) \subset {\rm Poly}^G({\bm V}, {\bm W})$ be the subspace of maps with degree at most $d$. There is a natural $G$-equivariant evaluation map 
\beqn
{\rm ev}: {\bm V} \times {\rm Poly}_d^G({\bm V}, {\bm W}) \to {\bm W}
\eeqn
whose zero locus is denoted by 
\beq\label{eqn21}
Z:= Z_d^G:= Z_d^G({\bm V}, {\bm W}).
\eeq
This is a complex algebraic variety. 

The application of FOP sections crucially relies on the existence of certain Whitney stratifications on the above variety. By a classical theorem of Whitney \cite{Whitney_1965}, any complex algebraic subvariety inside a smooth variety admits a canonical Whitney stratification (see Theorem \ref{thm:Whitney_1965}) whose strata are smooth algebraic submanifolds. However, for the variety $Z$ \eqref{eqn21}, this canonical Whitney stratification may not respect the group action. One needs to use a ``more symmetric'' Whitney stratification with a few nice properties---this was the observation of B. Parker \cite{BParker_integer}. The necessity of having such nice properties comes from the consideration that we need to consider the Whitney stratifications on the variety $Z$ for different cut-off degrees $d$ and different groups $G$. In our previous work \cite[Theorem 4.3]{Bai_Xu_2022} we proved the following result showing the existence of certain canonical ``symmetric'' Whitney stratifications which have the desired nice properties.

\begin{thm}
There exists a unique Whitney stratification on $Z_d^G({\bm V}, {\bm W})$ subject to the following conditions.
\begin{enumerate}

    \item For each subgroup $H\subset G$, let ${\bm V}_H^* \subset {\bm V}$ be the subset of points whose stabilizer is exactly $H$. Then for each $x \in Z_d^G({\bm V}, {\bm W}) \cap ( {\bm V}_H^* \times {\rm Poly}_d^G({\bm V}, {\bm W}))$, the germ through $x$ is contained in ${\bm V}_H^* \times {\rm Poly}_d^G({\bm V}, {\bm W})$. 
    
    \item The Whitney stratification is the minimal one among all which satisfy the above condition.
\end{enumerate}
Moreover, this Whitney stratification enjoys the following additional properties.
\begin{enumerate}
    \item It is $G$-invariant.
    
    \item It is induced from a Whitney prestratification (see Definition \ref{defna1} and \ref{defna2}) on $Z$ whose strata are all algebraic submanifolds. 
    
    \item It is invariant under all $G$-equivariant diffeomorphisms of ${\bm V} \times {\rm Poly}_d^G({\bm V}, {\bm W})$ which preserve $Z$.
\end{enumerate}
\end{thm}

The nice behaviors of the canonical Whitney stratification allow us to extend our consideration to bundles. Let $B$ be a smooth manifold acted on trivially by a finite group $G$. Let $V, W \to B$ be smooth $G$-equivariant complex vector bundles with fibers isomorphic to representations ${\bm V}$ and ${\bm W}$ respectively. Such datum defines a locally trivial bundle 
\beqn
{\rm Poly}_d^G(V, W) \to B
\eeqn
whose fiber at $x \in B$ is the space ${\rm Poly}_d^G(V_x, W_x)$. Then there is a subbundle 
\beqn
{\mc Z}_d^G( V, W) \subset V \oplus {\rm Poly}_d^G(V, W)
\eeqn
whose fiber at $x$ is the zero locus associated with ${\rm ev}: V_x \times {\rm Poly}_d^G(V_x, W_x) \to W_x$. Using the invariance property of the canonical Whitney stratification of $Z_d^G({\bm V}, {\bm W})$ under $G$-equivariant diffeomorphisms, the canonical Whitney stratification on the fibers can be ``patched together" to define a canonical Whitney stratification on the fiber bundle ${\mc Z}_d^G(V, W)$ which is ``locally trivial.''

Another statement relevant to us, which was originally proved by Fukaya--Ono, says that when $d$ is sufficiently large, the variety $Z$ is a union of smooth pieces. For a proof, the readers could refer to \cite[Proposition 4.9]{Bai_Xu_2022}.

\begin{thm}
For $d$ sufficiently large, for each subgroup $H \subset G$, the locus 
\beqn
Z_d^G({\bm V}, {\bm W})_H:= Z_d^G({\bm V}, {\bm W}) \cap ( {\bm V}_H^* \cap {\rm Poly}_d^G({\bm V}, {\bm W}))
\eeqn
is a smooth algebraic submanifold of complex dimension
\beqn
\dim_{\mb C} {\rm Poly}_d^G({\bm V}, {\bm W}) + \dim_{\mb C} \mathring{\bm V}^H - \dim_{\mb C} \mathring{\bm W}^H.
\eeqn
\end{thm}

\subsubsection{FOP sections and strong transversality}

Now we recall the notion of normally complex sections (which we call by FOP sections, owing credit to Fukaya--Ono and B. Parker), which is Parker's generalization of the notion of normally polynomial sections.

We first consider the case for a single chart. Let $B$ be a smooth manifold, $G$ be a finite group trivially acting on $B$, and $\pi_V: V \to B$, $\pi_W: W \to B$ be $G$-equivariant complex vector bundles. Choose a positive integer $d$. Then there is a vector bundle 
\beqn
{\rm Poly}_d^G(V, W) \to B
\eeqn
whose fiber over each $b\in B$ is the vector space ${\rm Poly}_d^G(V_x, W_x)$ of $G$-equivariant polynomial maps with degree at most $d$. Let $V_\epsilon \subset V$ be an open $G$-invariant disk subbundle with respect to an auxiliary bundle metric. Consider smooth sections of the pullback bundle $\pi_V^* W \to V_\epsilon$. 

\begin{defn}\label{defn:strongly_transverse}
Let $s: V_\epsilon \to \pi_V^* W$ be a smooth $G$-equivariant section. 
\begin{enumerate}
    \item $s$ is called a {\bf normally polynomial section} of degree at most $d$ if its restriction to each fiber $V_x \cap V_\epsilon$ coincides with the restriction of an element of ${\rm Poly}_d^G(V_x, W_x)$. 
    
    \item $s$ is called a {\bf normally complex section} of degree at most $d$ if for each $(x, v) \in V_\epsilon$, there exists a smooth $G$-equivariant bundle map $f: V_\epsilon \to {\rm Poly}_d^G(V, W)$ such that for points $(x', v')$ near $(x, v)$, one has 
\beqn
s(x', v') = f(x', v')( v').
\eeqn
The map $f$ is called a {\bf local lift} of $s$ near $(x, v)$. We also call a normally complex section an {\bf FOP section}.

\item An FOP section $s$ is called {\bf strongly transverse} at $(x, v) \in V_\epsilon$ if for any local lift $f: V \to {\rm Poly}_d^G(V, W)$ near $(x, v)$, the graph of $f$, as a submanifold of the total space of $V_\epsilon \oplus {\rm Poly}_d^G(V, W)$, is transverse to the canonical Whitney stratification of ${\mc Z}_d^G(V, W)$ near $(x, v)$, i.e., the graph of $f$ is transverse to all the strata of the canonical Whitney (pre)stratification of ${\mc Z}_d^G(V, W)$.
\end{enumerate}
\end{defn}

With the above preparations, now we consider the global situation for a normally complex derived orbifold chart. 

\begin{defn}
Let $({\mc U},{\mc E}, {\mc S})$ be a compact normally complex derived orbifold chart equipped with a straightening. 

\begin{enumerate}

\item A section ${\mc S}': {\mc U} \to {\mc E}$ is called an {\bf FOP section} if for each bundle chart $(\Gammait, E, U)$ for which ${\mc S}'$ lifts to a $\Gammait$-equivariant section $S: U \to E$, the following condition is true. Over $U^\Gammait$ we can decompose $E|_{U^\Gammait} = \mathring E^\Gammait \oplus \check E^\Gammait$. The straightening induces an identification of a tubular neighborhood of $U^\Gammait$ with a disk bundle $N_\epsilon U^\Gammait$ inside the normal bundle $NU^\Gammait \to U^\Gammait$ as well as an equivariant bundle isomorphism
\beqn
E|_{N_\epsilon U^\Gammait} \cong \pi_{NU^\Gammait}^* \mathring E^\Gammait \oplus \pi_{NU^\Gammait}^* \check E^\Gammait.
\eeqn
Then with respect to this splitting we can decompose 
\beqn
S|_{N_\epsilon U^\Gammait} = (\mathring S, \check S).
\eeqn
We require that $\check S: N_\epsilon U^\Gammait \to \pi_{NU^\Gammait}^* \check E^\Gammait$ is an FOP section.

\item An {\bf FOP perturbation} of $({\mc U}, {\mc E}, {\mc S})$ is a smooth perturbation ${\mc S}': {\mc U} \to {\mc E}$ (see Definition \ref{defn21}) such that ${\mc S}'$ is an FOP section near its zero locus.

\item An FOP section ${\mc S}': {\mc U} \to {\mc E}$ is called {\bf strongly transverse} at $x \in {\mc U}$ if the following conditions are satisfied. Let $(\Gammait_x, U_x, E_x)$ be a bundle chart centered at $x$. By shrinking the chart we assume that $U_x$ is identified with a disk bundle $N_\epsilon U_x^{\Gammait_x}$ of the normal bundle $NU_x^{\Gammait_x}$ of $U_x^{\Gammait_x}$ and ${\mc S}'$ lifts to an equivariant section $S_x: U_x \to E_x$. With respect to the splitting (induced from the straightening)
\beqn
E_x\cong \pi_{NU_x^{\Gammait_x}}^* \mathring E_x \oplus \pi_{NU_x^{\Gammait_x}}^* \check E_x
\eeqn
we write $S_x = (\mathring S_x, \check S_x)$. Then we require that there exists an local lift $f: N_\epsilon U_x^{\Gammait_x} \to {\rm Poly}_d^{\Gammait_x}( NU_x^{\Gammait_x}, \check E_x^{\Gammait_x})$ of $\check S_x$ such that the induced bundle map
\beqn
(\mathring S_x, {\rm graph}f): N_\epsilon U_x^{\Gammait_x} \to \mathring E_x^{\Gammait_x} \oplus ( NU_x^{\Gammait_x} \oplus {\rm Poly}_d^{\Gammait_x}(NU_x^{\Gammait_x}, \check E_x) )
\eeqn
is transverse to the subbundle $\{0\}\oplus {\mc Z}_d^{\Gammait_x}(NU_x^{\Gammait_x}, \check E_x^{\Gammait_x})$ with respect to the canonical Whitney stratification at the point $0 \in U_x$. 
\end{enumerate}
\end{defn}

\begin{rem}
The strong transversality condition of FOP sections presented as above \emph{a priori} depends on the choices of bundle charts and the cut-off degree $d$ appearing in ${\rm Poly}_d^{\Gammait_x}( NU_x^{\Gammait_x}, \check E_x^{\Gammait_x})$. The most important output from \cite[Section 4]{Bai_Xu_2022} is showing that the canonical Whitney stratification on ${\mc Z}_d^G(V, W)$ is compatible with the change of the cut-off degree $d$ and the group $G$, when interpreted suitably. The upshot is, once fixing a straightening datum of $({\mc U},{\mc E}, {\mc S})$, the strong transversality condition is in fact intrinsic. The readers could refer to \cite{Bai_Xu_2022} for details, and we continue our discussions with such background in mind.
\end{rem}

It is a general fact that generic smooth maps are transverse to a given Whitney stratified object. As a consequence, a generic FOP section is strongly transverse. We formulate this fact as the following Proposition.

\begin{prop}\cite[Proposition 6.4]{Bai_Xu_2022}\label{prop:FOP_existence}
Let $({\mc U}, {\mc E}, {\mc S})$ be a compact normally complex derived orbifold chart equipped with a straightening. Fix a continuous norm on ${\mc E}$. Fix a precompact open neighborhood ${\mc D}\subset {\mc U}$ of ${\mc S}^{-1}(0)$. 

\begin{enumerate}

\item {\bf (Absolute version)} Given $\epsilon>0$, there exists a smooth section ${\mc S}_\epsilon: {\mc U} \to {\mc E}$ satisfying the following conditions. 
\begin{enumerate}
    \item ${\mc S}_\epsilon$ is an FOP section in a neighborhood of $\ov{\mc D}$ and it is strongly transverse near $\ov{\mc D}$.
    
    \item $\| {\mc S} - {\mc S}_\epsilon \|_{C^0({\mc D})} < \epsilon$. 
\end{enumerate}

\item {\bf (Relative version)} More generally, let $K \subset {\mc U}$ be a compact subset, ${\mc U}' \subset {\mc U}$ be an open neighborhood of $K$. Suppose we are given a section ${\mc S}_1': {\mc U}' \to {\mc E}|_{{\mc U}'}$ which is an FOP section and strongly transverse near $K \cap \ov{\mc D}$. Then there exists a smooth perturbation ${\mc S}': {\mc U} \to {\mc E}$ which is an FOP section and strongly transverse near $\ov{\mc D}$, such that ${\mc S}'$ coincides with ${\mc S}_1'$ near $K$. In addition, if $\epsilon>0$ is sufficiently small and $\| {\mc S} - {\mc S}_1'\|_{C^0} < \epsilon$, then we can choose ${\mc S}'$ such that $\| {\mc S} - {\mc S}'\|_{C^0} < 2\epsilon$. 

\end{enumerate}
\end{prop}

The relative version of the above proposition is often referred to as a ``CUDV'' type statement. This means that a good perturbation has been constructed on an open neighborhood $U$ of a closed subset $C$ and we would like to find a good perturbation on a neighborhood of $C \cup D$ where $D$ is another closed subset, while we want to maintain the original perturbation near $C$ and do not change anything outside an open neighborhood $V$ of $D \setminus C$. 

Moreover, as proposed by Fukaya--Ono, the isotropy free part of the zero locus of a strongly transverse FOP section should induce a homology class. 

\begin{prop}\label{prop:FOP}
Let $({\mc U}, {\mc E}, {\mc S})$ be an oriented (i.e., both ${\mc U}$ and ${\mc E}$ are oriented) and compact normally complex derived orbifold chart equipped with a straightening. Let ${\mc U}^* \subset {\mc U}$ be the manifold part, i.e., the open and dense subset of points whose isotropy groups are trivial. Then the following is true. 

\begin{enumerate}
\item Let ${\mc S}'$ be a strongly transverse FOP perturbation. Then the set $({\mc S}')^{-1}(0) \cap {\mc U}^*$ is an oriented smooth submanifold of ${\mc U}^*$ of real dimension being ${\rm dim} {\mc U} - {\rm rank} {\mc E}$ and the inclusion map $({\mc S}')^{-1}(0) \cap {\mc U}^* \hookrightarrow {\mc U}$ is an oriented pseudocycle, hence represents an integral homology class. \footnote{In \cite{Bai_Xu_2022} we extend the notion of pseudocycles in manifolds to general Thom--Mather stratified spaces (including orbifolds) and proved that they represent integral homology classes.}

\item The resulting homology class, called the {\bf FOP Euler class}, denoted by 
\beqn
\chi^{\rm FOP}({\mc U}, {\mc E}, {\mc S})\in H_*( {\mc U}; {\mb Z}),
\eeqn
is independent of the choice of strongly transverse FOP perturbations and is independent of the choice of straightening, and hence is an invariant of the normally complex derived orbifold chart. 
\end{enumerate}
\end{prop}

\begin{rem}
In fact there exist a collection of homology classes associated to a derived orbifold chart indexed by a finite group and a pair of complex representations. The FOP Euler class in Proposition \ref{prop:FOP} is the leading one in this collection. 
\end{rem}

\subsection{FOP sections and products}

In this subsection we prove that the product of strongly transverse FOP sections is still strongly transverse. This is a necessary ingredient for the inductive construction of perturbations in Floer theory. Moreover, we show that the natural transformation ${\mc FOP}_{\gamma}$ from the stably complex derived orbifold bordism $\ov\Omega{}_*^{{\mb C}, {\rm der}}$ to the integral homology constructed by \cite[Theorem 1.4]{Bai_Xu_2022} is multiplicative.

\subsubsection{Products of transverse FOP sections}

We fix our notations. Let $G_i$, $i = 1, 2$ be finite groups. Let $\V_i, \W_i$ be complex $G_i$-representations. Then $\V_1 \oplus \V_2$ and $\W_1 \oplus \W_2$ are $G_1 \times G_2$ representations under the product action. Choose nonnegative integers $d, d_1, d_2$ such that $d \geq d_1, d_2$. Consider the space
\beqn
{\rm Poly}_d^{G_1 \times G_2} (\V_1 \oplus \V_2, \W_1 \oplus \W_2).
\eeqn
It has a subspace 
\beqn
{\rm Poly}_{d_1}^{G_1}(\V_1, \W_1) \times {\rm Poly}_{d_2}^{G_2}(\V_2, \W_2).
\eeqn
One also has the inclusion of $Z$-varieties
\beqn
Z_{d_1}^{G_1}(\V_1, \W_1) \times Z_{d_2}^{G_2} (\V_2, \W_2) \subset Z_d^{G_1\times G_2}(\V_1 \oplus \V_2, \W_1\oplus \W_2).
\eeqn
Abbreviate the three $Z$-varieties as $Z_{d_1}^{G_1}$, $Z_{d_2}^{G_2}$, and $Z_d^{G_1 \times G_2}$ respectively.

\begin{prop}\label{prop27}
When $d_1, d_2$ are sufficiently large, the inclusion 
\begin{multline*}
\phi: \big( \V_1 \times {\rm Poly}_{d_1}^{G_1} ( \V_1, \W_1) \big) \times \big( \V_2 \times {\rm Poly}_{d_2}^{G_2}(\V_2, \W_2) \big)\\
\hookrightarrow (\V_1 \oplus \V_2) \times  {\rm Poly}_d^{G_1\times G_2}( \V_1\oplus \V_2, \W_1 \oplus \W_2)
\end{multline*}
is transverse to all the strata of the canonical Whitney stratification on $Z_d^{G_1 \times G_2}$. Moreover, the inclusion pulls back the canonical Whitney stratification on the target to the canonical (product) Whitney stratification on the domain.
\end{prop}

The proof is given in Subsection \ref{subsection24}. 

\begin{cor}\label{cor28}
Let $({\mc U}_i, {\mc E}_i, {\mc S}_i)$, $i = 1, 2$ be normally complex derived orbifold charts such that ${\mc S}_i$ is a strongly transverse FOP section. Then the product section 
\beqn
{\mc S}_1 \boxplus {\mc S}_2: {\mc U}_1 \times {\mc U}_2 \to {\mc E}_1 \boxplus {\mc E}_2
\eeqn
is also a strongly transverse FOP section.
\end{cor}

\begin{proof}
As transversality is defined locally, it suffices to restrict our consideration to local charts. Let $(U_i, E_i, G_i)$ be a bundle chart of ${\mc E}_i$ centered at $x_i\in {\mc U}_i$ such that ${\mc S}_i$ is pulled back to a $G_i$-equivariant section 
\beqn
S_i: U_i \to E_i.
\eeqn
Suppose $NU_i^{G_i}$ is trivial with fiber $\V_i$ and $E_i$ is trivial with fiber $\W_i$. Decompose $\W_i = \mathring \W_i^{G_i} \oplus \check \W_i^{G_i}$. Then near $NU_i^{G_i}$ we can write 
\beqn
S_i = (\mathring S_i, \check S_i).
\eeqn
By the definition of FOP sections, there exists a local lift of $\check S_i$ near the origin
\beqn
f_i: U_i^{G_i} \times \V_i \to {\rm Poly}_{d_i}^{G_i}( \V_i, \check \W_i^{G_i}).
\eeqn

By assumption, $\mathring S_i$ is transverse to $0\in \mathring \W_i$ (in the usual sense). Denote $U = U_1 \times U_2$, $\W = \W_1 \oplus \W_2$. Then 
\beqn
\mathring \W^{G_1 \times G_2} = \mathring \W_1^{G_1}\oplus \mathring \W_2^{G_2}.
\eeqn
Hence $\mathring S = (\mathring S_1, \mathring S_2)$ is transverse to $0\in \mathring \W^{G_1 \times G_2}$ and 
\beqn
\mathring S^{-1}(0) = \mathring S_1^{-1}(0) \times \mathring S_2^{-1}(0).
\eeqn
Hence we may assume that $\mathring \W_1^{G_1} = 0$, $\mathring \W_2^{G_2} = 0$. Next, the graph of $f_i$ is transverse to $U_i^{G_i}\times Z_{d_i}^{G_i} ( \V_i, \check \W_i^{G_i})$. Notice that 
\beqn
{\rm graph}(f) = {\rm graph} (f_1 \times f_2) = {\rm graph}(f_1) \times {\rm graph}(f_2).
\eeqn
It follows that the graph of $f$ is transverse to the product Whitney stratification on $U^{G_1 \times G_2} \times Z_{d_1}^{G_1} ( \V_1, \check \W_1^{G_2}) \times Z_{d_2}^{G_2} ( \V_2, \check \W_2^{G_2})$. If we view $f$ as a map 
\beqn
f: U^{G_1\times G_2} \times \V_1 \times \V_2 \to {\rm Poly}_{{\rm max}(d_1, d_2)}^{G_1 \times G_2}( \V_1 \oplus \V_2, \check \W_1^{G_1} \oplus \check \W_2^{G_2}),
\eeqn
it is easy to see that the intersection between ${\rm graph}(f)$ and $U^{G_1 \times G_2} \times Z_{{\rm max}(d_1, d_2)}^{G_1 \times G_2}$ is contained in $U^{G_1 \times G_2} \times Z_{d_1}^{G_1} \times Z_{d_2}^{G_2}$. By Proposition \ref{prop27}, as the inclusion $Z_{d_1}^{G_1} \times Z_{d_2}^{G_2} \hookrightarrow Z_{{\rm max}(d_1, d_2)}^{G_1 \times G_2}$ respects the canonical Whitney stratifications, we see that ${\rm graph}(f)$ is transverse to all the strata of $U^{G_1 \times G_2} \times Z_{{\rm max}(d_1, d_2)}^{G_1 \times G_2}$. By definition, this means that $S = S_1\times S_2$ is strongly transverse at $x = (x_1, x_2)$. 
\end{proof}

\subsubsection{Multiplicativity of the FOP natural transformation}

A consequence of previous discussions is that the FOP natural transformation defined in \cite{Bai_Xu_2022} is natural with respect to products. Although this result will not be used in the setting of Floer theory, we include it here. 

Let us recall the relevant notations. An {\bf isotropy type} is a triple $(G, {\bm V}, {\bm W})$ where $G$ is a finite group, ${\bm V}, {\bm W}$ are finite-dimensional complex $G$-representations which do not contain trivial $G$-summands. A {\bf stable isotropy type} is an equivalence classes of isotropy types with respect to the equivalence relation generated by 
\beqn
(G, {\bm V}, {\bm W}) \sim (G, {\bm V} \oplus {\bm R}, {\bm W} \oplus {\bm R} )
\eeqn
where ${\bm R}$ is a nontrivial irreducible complex representation of $G$. A stable isotropy type is denoted by $\gamma$ or $[G, {\bm V}, {\bm W}]$ if $(G, {\bm V}, {\bm W})$ represents it. To discuss products, we define the multiplication of (stable) isotropy types in the obvious way: 
\beqn
[G_1, {\bm V}_1, {\bm W}_1] \times [G_2, {\bm V}_2, {\bm W}_2]:= [G_1\times G_2, {\bm V}_1 \oplus {\bm V}_2, {\bm W}_1 \oplus {\bm W}_2].
\eeqn

The pseudocycles defined by strongly transverse FOP sections induce natural transformations of generalized homology theories. First, for a topological space $Y$ one can define the stably complex derived orbifold bordism group
\beqn
\overline{\Omega}{}_*^{{\mb C}, \rm der}(Y)
\eeqn
generated by isomorphism classed of quadruples $({\mc U}, {\mc E}, {\mc S}, f)$ where $({\mc U}, {\mc E}, {\mc S})$ is a compact stably complex derived orbifold chart and $f: {\mc U} \to Y$ is a continuous map, modulo the equivalence relations generated by 
\begin{enumerate}
    \item ({\it Restriction}) $({\mc U}, {\mc E}, {\mc S}, f)\sim ({\mc U}', {\mc E}', {\mc S}', f')$ if ${\mc U}' \subset {\mc U}$ is an open neighborhood of ${\mc S}^{-1}(0)$ and ${\mc E}' = {\mc E}|_{{\mc U}'}$, ${\mc S}' = {\mc S}|_{{\mc U}'}$, and $f' = f|_{{\mc U}'}$.
    
    \item ({\it Stabilization}) $({\mc U}, {\mc E}, {\mc S}, f)\sim ({\mc U}', {\mc E}', {\mc S}', f')$ if ${\mc U}'$ is equal to the total space of a vector bundle $\pi_{\mc F}: {\mc F} \to {\mc U}$, ${\mc E}' = \pi_{\mc F}^* {\mc E} \oplus \pi_{\mc F}^* {\mc F}$, ${\mc S}' = \pi_{\mc F}^* {\mc S} \oplus \tau_{\mc F}$ where $\tau_{\mc F}: {\mc F} \to \pi_{\mc F}^* {\mc F}$ is the tautological section, and $f' = f\circ \pi_{\mc F}$. 
    
    \item ({\it Cobordism}) $({\mc U}, {\mc E}, {\mc S}, f)\sim ({\mc U}', {\mc E}', {\mc S}', f')$ if there is a bordism between them extending the stable complex structures.
\end{enumerate}
A stable complex structure on $({\mc U}, {\mc E}, {\mc S})$ is roughly a lift of the virtual bundle $T{\mc U} - {\mc E}$ to a complex virtual bundle. Disjoint union of derived orbifold charts induces the group structure on $\overline{\Omega}{}_*^{{\mb C}, \rm der}(Y)$. For details of the relevant terminologies, the readers could refer to \cite[Section 7]{Bai_Xu_2022}. For each stable isotropy type $\gamma$ represented by $(G, {\bm V}, {\bm W})$, denote $n_\gamma = {\rm dim}_{\mb R} {\bm  V} - {\rm dim}_{\mb R} {\bm W}$. Then in \cite{Bai_Xu_2022} we constructed a natural transformation of generalized homology theories made of linear maps (see \cite[Theorem 1.4]{Bai_Xu_2022})
\beqn
\mc{FOP}_\gamma^Y: \overline{\Omega}{}_*^{{\mb C}, \rm der}(Y) \to H_{* - n_\gamma}(Y; {\mb Z}).
\eeqn

The main consequence of the fact that strong transversality of FOP perturbations is preserved after taking products regarding the stably complex derived bordism theory is the following theorem.

\begin{thm}
Let $Y_1$ and $Y_2$ be topological spaces and $\gamma_1, \gamma_2$ be stable isotropy types. Then the follow diagram is commutative: 
\beq\label{eqn23}
\xymatrix{
\overline{\Omega}{}_*^{{\mb C}, \rm der} (Y_1) \times \overline{\Omega}{}_*^{{\mb C}, \rm der} (Y_2) \ar[d] \ar[rrr]^-{\mc{FOP}_{\gamma_1}^{Y_1} \times \mc{FOP}_{\gamma_2}^{Y_2}} & & & H_{* - n_{\gamma_1}}(Y_1; {\mb Z}) \times H_{* - n_{\gamma_2}}(Y_2; {\mb Z}) \ar[d] \\
{\overline{\Omega}{}_{*}^{\mathbb{C}, {\rm der}}(Y_1 \times Y_2)} \ar[rrr]^-{\mc{FOP}^{Y_1 \times Y_2}_{\gamma_1\times \gamma_2}}  & & & H_{* - n_{\gamma_1} - n_{\gamma_2} }(Y_1 \times Y_2; {\mb Z}). 
}
\eeq
Here the left vertical arrow is the map induced by product of derived orbifold charts:
\beqn
({\mc U}_1, {\mc E}_1, {\mc S}_1, f_1) \times ({\mc U}_2, {\mc E}_2, {\mc S}_2, f_2) \mapsto ({\mc U}_1 \times {\mc U}_2, {\mc E}_1\boxplus {\mc E}_2, {\mc S}_1 \boxplus {\mc S}_2, f_1 \times f_2)
\eeqn
and the right vertical arrow is the Eilenberg--Zilber map.
\end{thm}

\begin{proof}
Let $({\mc U}_i, {\mc E}_i, {\mc S}_i, f_i)$, $i = 1, 2$ be a representative of an element of $\overline{\Omega}{}_*^{{\mb C}, \rm der}(Y_i)$. By the definition of $\mc{FOP}_{\gamma_i}^{Y_i}$, we choose straightenings on $({\mc U}_i, {\mc E}_i)$ and choose strongly transverse FOP sections ${\mc S}_i': {\mc U}_i \to {\mc E}_i$ which is $C^0$-close to ${\mc S}_i$. Then ${\mc S}_i'$ defines a pseudocycle in ${\mc U}_i$ which is contained in the closure of the stratum ${\mc U}_{i, \gamma_i} \subset {\mc U}_i$, i.e., the set of points $x_i \in{\mc U}_i$ whose stabilizers are isomorphic to $G_i$ and for a bundle chart $(G_i, E_i, U_i)$ centered at $x_i$, the stable isotropy type defined by $(G_i, (N U^{G_i})_{x_i}, (\check{E}_i)_{x_i} )$ lies in the class $\gamma$. Then 
\beqn
\mc{FOP}_{\gamma_i}^{Y_i} ([ {\mc U}_i, {\mc E}_i, {\mc S}_i, f_i]) = (f_i)_* [\ov{({\mc S}_i')^{-1}(0) \cap \mc U_{i, \gamma_i}}].
\eeqn
Now consider the product chart $({\mc U}_1 \times {\mc U}_2, {\mc E}_1 \boxplus {\mc E}_2, {\mc S}_1 \boxplus {\mc S}_2, f_1 \times f_2)$. The chosen straightenings produce a straightening on the product, with respect to which the product ${\mc S}_1' \boxplus {\mc S}_2'$ is an FOP section. Corollary \ref{cor28} implies that ${\mc S}_1' \boxplus {\mc S}_2'$ is also a strongly transverse FOP section and $({\mc S}_1' \boxplus {\mc S}_2')^{-1}(0) = ({\mc S}_1')^{-1}(0) \times ({\mc S}_2')^{-1}(0)$. Restricting the product ${\mc U}_{1, \gamma_1}\times {\mc U}_{2, \gamma_2}$, one has, as sets 
\beqn
({\mc S}_1' \boxplus {\mc S}_2')^{-1}(0) \cap ({\mc U}_1 \times {\mc U}_2)_{\gamma_1 \times \gamma_2} = \big( ({\mc S}_1')^{-1}(0) \cap {\mc U}_{1, \gamma_1} \big) \times \big( ({\mc S}_2')^{-1}(0) \cap {\mc U}_{2, \gamma_2} \big).
\eeqn
As the product of pseudocycles is still a pseudocycle and the homology classes represented by pseudocycles respect such product structures, %(one needs to verify this for Thom--Mather stratified spaces), 
therefore the commutativity of the diagram \eqref{eqn23} follows. %\textcolor{red}{because the Eilenberg--Zilber map is represented by the product between pseudocycles - seems to have no reference!}
\end{proof}

\subsection{Proof of Proposition \ref{prop27}}\label{subsection24}

First we consider the canonical Whitney stratification on $Z_{d_1}^{G_1} \times Z_{d_2}^{G_2}$. Abbreviate 
\beqn
{\mc Y}_i = {\bm V}_i \times {\rm Poly}_{d_i}^{G_i} ( {\bm V}_i, {\bm W}_i ).
\eeqn
${\mc Y}_i$ is endowed with the action prestratification, i.e., 
\beqn
{\mc Y}_i = \bigsqcup_{H_i \in {\mf A}_i} {\bm V}_{i, H_i}^*
\eeqn
We explain the notations. For each subgroup $H_i \subset G_i$, ${\bm V}_{i, H_i}^* \subset  {\bm V}_i$ is the set of points whose stabilizers are exactly $H_i$. The symbol ${\mf A}_i$ denotes the set of all subgroups of $G_i$ for which ${\bm V}_{i, H_i}^* \neq \emptyset$. Then on the product ${\mc Y}:= {\mc Y}_1 \times {\mc Y}_2$ which has the $G_1 \times G_2$-action, the strata of the action prestratification is indexed exactly by ${\mf A}_1 \times {\mf A}_2$ and 
\beqn
{\mc Y} = \bigsqcup_{(H_1, H_2) \in {\mf A}_1 \times {\mf A}_2 } {\bm V}_{1, H_1}^* \times {\bm V}_{2, H_2}^*.
\eeqn
The strata of this prestratification are all algebraic submanifolds. Hence by \cite[Theorem A.21]{Bai_Xu_2022}, the variety $Z_{d_1}^{G_1}\times Z_{d_2}^{G_2}$ has a canonical Whitney stratification, which is the minimal Whitney stratification respecting the action prestratification; by Proposition \ref{propa5} of Appendix \ref{appendixa}, this Whitney stratification is the product of the canonical Whitney stratifications on $Z_{d_1}^{G_1}$ and $Z_{d_2}^{G_2}$. 

Now we prove Proposition \ref{prop27}. As the canonical Whitney stratification on $Z_d^G$ respects the inclusion $Z_d^G \hookrightarrow Z_{d'}^G$ for $d \leq d'$ (see \cite[Theorem 4.12]{Bai_Xu_2022}), one may assume that $d_1 = d_2 = d$. The inclusion map sends $Z_d^{G_1}\times Z_d^{G_2}$ into $Z_d^{G_1\times G_2}$. 

The proof of the following lemma is analogous to that of \cite[Lemma 4.14]{Bai_Xu_2022}.

\begin{lemma}
There exists a map  
\beqn
\psi: ({\bm V}_1 \oplus {\bm V}_2) \times {\rm Poly}_d^{G_1\times G_2} ( {\bm V}_1 \oplus {\bm V}_2, {\bm W}_1 \oplus {\bm W}_2) \to {\mc Y}_1 \times {\mc Y}_2
\eeqn
satisfying the following conditions. 
\begin{enumerate}
    \item $\psi \circ \phi$ is the identity map on ${\mc Y}_1 \times {\mc Y}_2$.
    
    \item For each $(v, P) \in ({\bm V}_1 \oplus {\bm V}_2) \times {\rm Poly}_d^{G_1\times G_2}( {\bm V}_1 \oplus {\bm V}_2, {\bm W}_1 \oplus {\bm W}_2)$, ${\rm ev}( \psi(v, P)) = {\rm ev}(v, P)$.
    
    \item $\phi \circ \psi$ is transverse to $Z_d^{G_1\times G_2}$ and pulls back the canonical Whitney stratification to itself.
\end{enumerate}
\end{lemma}

\begin{proof}
For each $P \in {\rm Poly}_d^{G_1\times G_2}( {\bm V}_1 \oplus {\bm V}_2,  {\bm W}_1 \oplus {\bm W}_2)$, denote its ${\bm W}_1$-component by $P_1$ and its ${\bm W}_2$-component by $P_2$. Then we can regard $P_1$ as a $G_2$-invariant polynomial map 
\beqn
P_1 \in {\rm Poly}_d^{G_2} ( {\bm V}_2, {\rm Poly}_d^{G_1}( {\bm V}_1, {\bm W}_1))
\eeqn
and regard $P_2$ as a $G_1$-invariant polynomial map
\beqn
P_2 \in {\rm Poly}_d^{G_1} ({\bm V}_1, {\rm Poly}_d^{G_2} ({\bm V}_2, {\bm W}_2)).
\eeqn
Then for $v = (v_1, v_2) \in {\bm V}_1 \oplus {\bm V}_2$, define 
\beqn
\psi(v, P) = \psi(v_1, v_2, P_1, P_2) =\Big( (v_1, P_1(\cdot, v_2)), (v_2, P_2(v_1, \cdot)) \Big) \in {\mc Y}_1 \times {\mc Y}_2.
\eeqn
Then it is easy to verify that $\psi \circ \phi= {\rm Id}$ and that 
\beq\label{eqn22}
{\rm ev}(\psi(v, P)) = {\rm ev}(v, P).
\eeq

Now we prove the last property. Consider the manifold $B = {\rm Poly}_d^{G_1\times G_2} ( {\bm V}_1 \oplus {\bm V}_2, {\bm W}_1 \oplus  {\bm W}_2 ) $ over which there are the trivial bundles $V = B \times ({\bm V}_1 \oplus {\bm V}_2)$ and $W = B \times ( {\bm W}_1 \oplus {\bm W}_2)$. Then there are two bundle maps
\beqn
\uds f_1, \uds f_2: V \to {\rm Poly}_d^{G_1\times G_2}( V, W) \cong B \times {\rm Poly}_d^{G_1\times G_2}( {\bm V}_1 \oplus {\bm V}_2, {\bm W}_1 \oplus  {\bm W}_2)
\eeqn
where (here the first variable is the fiber coordinate and the second variable is the base coordinate)
\begin{align*}
&\ \uds f_1( v, P) = (P, P),\ &\ \uds f_2(v, P) = ( \uds \phi (\psi(v, P)), P)
\end{align*}
where $\uds\phi(v, P) = P$. Then \eqref{eqn22} implies that 
\beqn
{\rm graph}( \uds f_1 - \uds f_2) \in {\mc Z}_d^{G_1 \times G_2}(V, W).
\eeqn
As the identity map of $( {\bm V}_1 \oplus {\bm V}_2) \times {\rm Poly}_d^{G_1\times G_2} ( {\bm V}_1 \oplus {\bm V}_2,  {\bm W}_1 \oplus  {\bm W}_2)$ is transverse to $Z_d^{G_1\times G_2}$, which means that the graph of $\uds f_1$ is transverse to ${\mc Z}_d^{G_1\times G_2}$, by a lemma of B. Parker (see \cite[Lemma 4.10]{BParker_integer}, also \cite[Lemma 4.7]{Bai_Xu_2022}), the graph of $\uds f_2$ is also transverse to ${\mc Z}_d^{G_1\times G_2}$, implying that $\phi \circ \psi$ is transverse to $Z_d^{G_1\times G_2}$. Moreover, as the identity map pulls back any Whitney stratification to itself, by \cite[Lemma 4.8]{Bai_Xu_2022}, the map $\phi\circ\psi$ also pulls back the canonical Whitney stratification on $Z_d^{G_1\times G_2}$ to itself.
\end{proof}

\begin{proof}[Proof of Proposition \ref{prop27}]
As $\phi \circ \psi$ is transverse to $Z_d^{G_1 \times G_2}$, it follows that $\phi$ is transverse to $Z_d^{G_1 \times G_2}$ along the image of $\psi$. As $\psi$ is surjective, it follows that $\phi$ is transverse to $Z_d^{G_1\times G_2}$ everywhere. Moreover, as $\psi \circ \phi$ is the identity, it follows that $\psi$ is transverse to $Z_d^{G_1}\times Z_d^{G_2}$ along the image of $\phi$. Moreover, as $\psi = \psi \circ \phi \circ \psi$, which implies that the image of $d\psi$ at any point is equal to the image of $d\psi$ at some point in ${\rm Im}(\phi)$. Hence $\psi$ is transverse to $Z_d^{G_1}\times Z_d^{G_2}$ everywhere. 

The proof of the claim that $\phi$ resp. $\psi$ pulls back the canonical Whitney stratification to the canonical one is similar to the proof of \cite[Theorem 4.12]{Bai_Xu_2022}, which relies crucially on a property of minimal Whitney stratifications (see \cite[Lemma A.11]{Bai_Xu_2022}).
\end{proof}

\section{Abstract constructions of chain complexes and maps over the integers}\label{sec-3}

In this section we provide an abstract recipe of constructing chain complexes associated to flow categories and chain maps associated to flow bimodules. We explain the list of necessary structures on flow categories and flow bimodules which allow one to use FOP perturbations to define the algebraic counts over the integers. This section also serves as a source of notations. In Subsection \ref{subsection31} we set up the notations for partially ordered sets and abstract stratified spaces. In Subsection \ref{subsection32} we introduce the abstract notion of topological flow categories and flow bimodules. In Subsection \ref{subsection33} we define the notion of derived orbifold lifts of flow categoreis and bimodules which are abstract frameworks for regularizing the moduli spaces. In Subsection \ref{subsection34} we lift certain auxiliary structures which will be necessary to carry out the FOP perturbation scheme. In Subsection \ref{subsection35} we consider the important notion of stable normal complex structures. In Subsection \ref{subsection36} and Subsection \ref{subsection37} we explain the recipe of inductively constructing FOP perturbations associated to derived orbifold lifts of flow categories and flow bimodules and the recipe of extracting chain complexes and chain maps from the countings.

%The structure is 1) introduce basic notions about manifolds with corners 2) introduce the moduli space of trajectories and the moduli flow category  3) introduce smooth Kuranishi flow category (topological) 4) Perturbation.

%\textcolor{red}{Maybe we can avoid introducing any Floer-theoretic notion in this section and just discuss things on the level of flow categories. We should also introduce the definition of bimodules between flow categories (``1-cell" in the category of flow categories) and homotopies between bimodules (``2-cell" in the category of flow categories). Abouzaid--Blumberg didn't seem to do this in their paper. Instead, they introduced some auxiliary flow category whose morphisms include all the data needed to prove the Arnold conjecture.}

\subsection{Stratified spaces}\label{subsec:strat}\label{subsection31}

\subsubsection{Partially ordered sets}

Many objects in Floer theory are indexed by certain partially ordered sets. We abbreviate the phrase ``partially ordered set'' by the word {\bf poset}. In this paper posets are always countable. We use different symbols such as $\leq$, $\preceq$, etc. to denote the partial order relations. For a poset $\bA$, let $\bA^{\max}\subseteq \bA$ be the subset of maximal elements. A poset $\bA$ has a canonical {\bf Alexandrov topology}: a subset $U \subset \bA$ is open if $\alpha \in U$ and $\alpha \leq \beta$ imply that $\beta \in U$. 

The product of finitely many posets carries a canonically induced partial order. Indeed, if $\bA_1, \ldots, \bA_k$ are posets, then the relation 
\beqn
(\alpha_1, \ldots, \alpha_k) \leq ( \beta_1, \ldots, \beta_k )\ \text{if and only if}\ \alpha_i \leq \beta_i\ \forall i = 1, \ldots, k
\eeqn
is a partial order on $\bA_1 \times \cdots \times \bA_k$. 

We often consider posets with a well-defined ``depth'' function. We introduce the following notion of homogeneous posets. In fact all moduli spaces considered in this paper are stratified by the following kind of posets. 

\begin{defn}\label{defn:graded_poset}
A poset ${\mb A}$ is called {\bf homogeneous} if for each $\alpha \in \bA$, the length of a maximal sequence of elements $\alpha = \alpha_0 < \alpha_1 < \cdots < \alpha_k$ such that $\alpha_k \in \bA^{\max}$ is finite and only depends on $\alpha$. This length is called the {\bf depth} or {\bf codimension} of $\alpha$, denoted by $\dep( \alpha )$. In particular, 
\beqn
\alpha \in \bA^{\max} \Longleftrightarrow \dep(\alpha) = 0.
\eeqn
A poset map between homogeneous posets is a called a {\bf homogeneous} poset map if it  preserves the depth.
\end{defn}

\begin{defn}\label{defn_adjacent_face}
Given a homogeneous poset $\bA$, the {\bf set of adjacent faces} of $\alpha \in \bA$ is defined to be
$$
{\mb F}_{\alpha} := \{ \beta \in \bA | \alpha \leq \beta \text{ and } \dep(\beta)=1 \}.
$$
\end{defn}

It is straightforward to check that homogeneous posets also admit finite products with depth function being
\beqn
\dep(\alpha_1, \ldots, \alpha_k) = \dep (\alpha_1) + \cdots + \dep (\alpha_k).
\eeqn
%\textcolor{red}{The depth function on $\bA_1 \times \cdots \times \bA_k$ is defined to be}
%$$
%(\alpha_1, \ldots, \alpha_k) \mapsto \dep(\alpha_1) + \cdots \dep(\alpha_k) + k-1.
%$$

We introudce the following ``boundary stratum'' notation. Given a poset $\bA$ and an element $\alpha \in \bA$, denote
\beqn
\partial^\alpha \bA:= \{ \alpha' \in \bA\ |\ \alpha' \leq \alpha \}.
\eeqn
More generally, if ${\mb B} \subset {\mb A}$ is a subset, denote 
\beqn
\partial^{\mb B} {\mb A}:= \bigcup_{\beta \in {\mb B}} \partial^\beta \bA.
\eeqn
It has the induced partial order. If $\bA$ is a homogeneous poset, then $\partial^\alpha \bA$ is also homogeneous with depth function being shifted by $\dep(\alpha)$. Moreover, for each nonnegative integer $k$, denote 
\beqn
\partial^{[k]} \bA:= \partial^{\dep^{-1}(k)} \bA
\eeqn
which is still homogeneous with the depth function shifted by $k$.

%The moduli spaces considered in this paper are always parametrized by a special kind of posets. We first consider a prototypical one. Let $S$ be a finite set. Let $\bA^{(S)}$ be the set of all subsets of $S$ with partial order induced by inclusion:
%\beqn
%\alpha \leq \beta \Longleftrightarrow \beta \subseteq \alpha.
%\eeqn
%Then $\bA^{(S)}$ is a homogeneous poset with a unique maximal element $\emptyset$ and its depth function being the cardinality. The poset $\bA^{(S)}$ can be used to model the local stratification of a manifold with corners near a stratum. More generally, let $\bA$ be a homogeneous poset. For each $\alpha \in \bA$, denote
%\beqn
%\bA_{\geq \alpha}:= \{ \beta\in \bA\ |\ \beta\geq \alpha \}
%\eeqn
%and 
%\beqn
%{\mb F}_\alpha:= \{ \beta \in \bA_{\geq \alpha}\ |\ \dep (\beta) = 1\}
%\eeqn
%which are called the set of {\bf adjacent faces} of $\alpha$. 

%\begin{defn}
%A {\bf poset with faces} is a homogeneous poset $\bA$ such that there exist a collection of homogeneous bijective poset maps
%\beqn
%\iota_\alpha: \bA^{({\mb F}_\alpha)} \cong \bA_{\geq \alpha}
%\eeqn
%such that for all pairs $\alpha \leq \beta$, the following diagram commutes
%\beqn
%\xymatrix{  \bA^{({\mb F}_\beta)}  \ar[r] \ar[d]     &  \bA_{\geq \beta} \ar[d] \\
%            \bA^{({\mb F}_\alpha)} \ar[r]            &  \bA_{\geq \alpha}   }.
%            \eeqn
%In particular, for each $\alpha \in \bA$, $\bA_{\geq \alpha}$ has a unique maximal element and the depth of $\alpha$ is equal to the cardinality of $\mb{F}_\alpha$.
%\end{defn}

\subsubsection{Stratified spaces}

We introduce our notion of stratified topological spaces. We emphasize here that stratified spaces always refer to a poset. Moreover, the notion of stratified spaces should not be confused with the notion of prestratified spaces discussed in Appendix \ref{appendixa}. 

\begin{defn}
Let $\bA$ be a (countable) poset. An \textbf{$\bA$-stratified space} ($\bA$-space for short) is a locally compact, Hausdorff and second countable topological space $X$ endowed with a continuous map
\beqn
s: X \to \bA
\eeqn
with respect to the Alexandrov topology on $\bA$ such that the range of $s$ is finite. In particular, we can write
\beqn
X = \bigsqcup_{\alpha \in \bA} X_\alpha,\ {\rm where}\ X_\alpha:= s^{-1}(\alpha)
\eeqn
satisfying the following conditions. 
\begin{enumerate}
    \item Each $X_\alpha$ (called a {\bf stratum}) is locally closed (which can be empty\footnote{For example, in a moduli space of stable Floer cylinders, the subset of smooth Floer cylinders, which should be the top stratum, could be empty.}).
    
    \item All but finitely many strata are empty.
    
    \item For each $\alpha\in \bA$, the subset
    \beqn
    \partial^\alpha X:= \bigsqcup_{\beta \leq \alpha} X_\beta
    \eeqn
    is a closed set (which may contain the closure of $X_\alpha$ properly). Note that this condition follows from the continuity of $s$.
\end{enumerate}
%\beqn
%\alpha \leq \beta \Longleftrightarrow X_\alpha \subset \ov{X_\beta}.
%\eeqn
%For any $\alpha \in \mathcal{Q}$, define the poset $\partial^\alpha \mathcal{Q} := \{ \alpha' \leq \alpha | \alpha' \in \mathcal{Q} \}$ with the induced partial order. 
%For each $\alpha \in \bA$, define
%\beqn
%\partial^\alpha X := \bigsqcup_{\alpha' \in \partial^\alpha \bA} X_{\alpha'} ,
%\eeqn
%which should be thought of as the ``boundary stratum" of $X$ associated with $\alpha \in \bA$, which is a $\partial^\alpha \bA$-space.
\end{defn}

We introduce the following notions for stratified spaces.

\begin{defn}\hfill
\begin{enumerate}
    \item A map from an $\bA_1$-space $X_1$ to an $\bA_2$-space $X_2$ is a commutative diagram
    \beqn
    \xymatrix{ X_1 \ar[r]^f \ar[d]_{s_1} & X_2 \ar[d]^{s_2} \\
               \bA_1 \ar[r]_i & \bA_2} 
               \eeqn
    where $i: \bA_1 \to \bA_2$ is a poset map and $f$ is a continuous map. If $\bA_1$ and $\bA_2$ are both homogeneous (see Definition \ref{defn:graded_poset}), then we require that $i$ is a homogeneous map. We usually call such a map a {\bf stratified map} to emphasize that it respects the stratifications.
            
    \item A stratified map $f: X_1 \to X_2$ as above is called an {\bf embedding} if $f$ is a homeomorphism onto its image and $i: \bA_1 \to \bA_2$ is an injection. %It is easy to see that the natural inclusion $\partial^\alpha X \hookrightarrow X$ defines an embedding with closed image. 
    An open embedding is an embedding with an open image.
    
    \item A stratified map $f: X_1 \to X_2$ is called a homeomorphism if $f$ is a homeomorphism of topological spaces and the underlying poset map is an isomorphism.
    
    \item The product of $\bA_i$-spaces $X_i$ ($i = 1, \ldots, k$) is the product topological space $X_1 \times \cdots \times X_k$ stratified by the product poset $\bA_1 \times \cdots \times \bA_k$. It is easy to see 
    \beqn
    \partial^{(\alpha_1, \ldots, \alpha_k)} (X_1 \times \cdots \times  X_k) = \partial^{\alpha_1} X_1 \times \cdots \times \partial^{\alpha_k} X_k.
    \eeqn
    
    \item Let $G$ be a topological group. A $G$-action on an $\bA$-space $X$ is a continuous $G$-action on $X$ which preserves each stratum. In this case, the $\partial^\alpha \bA$-space $\partial^\alpha X$ has an induced $G$-action.
\end{enumerate}
\end{defn}

%\textcolor{red}{Do we need to introduce collar systems?}

%\begin{defn}
%Let $X$ be an ${\mc Q}$-space with a $G$-action.

%\begin{enumerate}
%    \item For each $P \in {\mc Q}$, a collar near $X_P$ consists of a $G$-equivariant retraction $\varphi_P: W_P \to X_P$ where $W_P$ is a $G$-invariant open neighborhood of $X_P$.
    
%    \item A {\bf collar system} on $X$ is a collection of collars $(\varphi_P: W_P \to X_P)$ satisfying the following compatibility condition. For all $P \leq P'$, one has 
%    \beqn
%    \varphi_P = \varphi_P \circ \varphi_{P'}.
%    \eeqn
%    in a neighborhood of $X_P$.
%\end{enumerate}
%\end

%\begin{lemma}
%The outer collaring of a ${\mc Q}$-space with a $G$-action admits a collar system.
%\end{lemma}
%\begin{proof}
%\textcolor{red}{Introduce outer collaring here or in the smoothing process or the in the perturbation process?}
%\end{proof}

\subsection{Topological flow categories and bimodules}\label{subsec:top-flow-cat}\label{subsection32}

The concept of flow categories was introduced by Cohen--Jones--Segal \cite{Cohen_Jones_Segal}. We need a variant of the original construction similar to \cite[Section 7]{pardon-VFC} and \cite[Section 7,8]{abouzaid2022axiomatic}.

\begin{setup}\label{setup:poset}
Let $N$ be a nonnegative integer, $\Pi$ be an infinite cyclic group, and $\omega: \Pi \to {\mb Z}$ be a group injection. 

Let ${\mc P}$ be a countable poset equipped with the following extra data: a free $\Pi$-action and two functions (called the {\bf action} and the {\bf index})
\begin{align*}
&\ {\mc A}^{\mc P}: {\mc P} \to {\mb R},\ &\ {\rm ind}^{\mc P}: {\mc P} \to {\mb Z}/ 2N.
\end{align*}
Assume the following conditions. 
\begin{enumerate}
    \item The $\Pi$-action is order-preserving. Namely, for all $p, q \in {\mc P}$ and $a \in \Pi$
    \beqn
    p \leq q \Longleftrightarrow a \cdot p \leq a \cdot q.
    \eeqn
    
    \item For all $p \in {\mc P}$ and $a \in \Pi$,
\begin{equation}\label{eqn:action-shift}
{\mc A}^{\mc P} (a \cdot p) = {\mc A}^{\mc P} (p) + \omega(a)
\end{equation}
and 
\begin{equation}\label{eqn:CZ-shift}
{\rm ind}^{\mc P} (a \cdot p) = \rm{ind}^{\mc P} (p).
\end{equation}

\item For all $p, q \in {\mc P}$, 
\beqn
p < q \Longrightarrow {\mc A}^{\mc P} (p) < {\mc A}^{\mc P} (q).
\eeqn

\item The quotient set $\uds{\mc P}:= {\mc P}/ \Pi$ is finite. 

\end{enumerate}

%Let ${\mc P}$ be a poset. Suppose $\Pi$ is a discrete group acting on ${\mc P}$ which preserves the partial order: if $p, p' \in {\mc P}$ and $p \leq p'$, then for any $a \in \Pi$, we have $a \cdot p \leq a \cdot p'$. Furthermore, we assume that there exist group homomorphisms
%\begin{align*}
%{\mc A}_{\Pi}: \Pi \rightarrow {\mb R}, \\
%2 c_1 : \Pi \rightarrow 2 N{\mb Z}
%\end{align*}
%for some $N \in {\mb Z}_{>0}$, and maps
%\begin{align*}
%{\mc A}_{\mc P}: {\mc P} \rightarrow {\mb R}, \\
%\text{gr}: {\mc P} \rightarrow {\mb Z}
%\end{align*}
%such that ${\mc A}_{\mc P}$ {\it preserves} the partial order and for any $p \in {\mc P}, a \in \Pi$,
%\begin{align}\label{eqn:action-compatible}
%{\mc A}_{\mc P} (a \cdot p) = {\mc A}_{\Pi}(a) + {\mc A}_{\mc P} (a), \nonumber \\
%\text{gr}(a \cdot p) = 2c_1 (a) + \text{gr}(p).
%\end{align}
%For our purpose, the poset ${\mc P}$ is required to be locally finite-dimensional: for any $p<q$, there are only finitely many elements of ${\mc P}$ lying between $p$ and $q$.
\end{setup}

The conditions that $\omega: \Pi \to {\mb Z}$ is injective and that $\uds{\mc P}$ is finite imply that ${\mc P}$ is ``locally finite-dimensional,'' namely, for any pair of elements $p < q$ of ${\mc P}$, there are at most finitely many elements lying between them.

\begin{notation}\label{notation35}
Given a pair of elements $p < q$ in ${\mc P}$, we define a poset 
\beqn
\bA_{pq}^{\mc P}:= \big\{ \alpha = pr_1 \cdots r_l q\ |\ p< r_1 < \cdots < r_l < q,\ r_1, \ldots, r_l \in {\mc P} \big\}
\eeqn
whose partial order is induced by inclusion
\beqn
ps_1 \cdots s_m q \leq p r_1 \cdots r_l q\Longleftrightarrow \{r_1 \cdots r_l \} \subseteq \{ s_1, \ldots, s_m \}.
\eeqn
The poset $\bA_{pq}^{\mc P}$ has a unique maximal element $pq$ and is homogeneous whose depth function is 
\beqn
\dep(pr_1 \cdots r_l q) = l.
\eeqn

The following is a formal characterization of the feature that boundary strata of moduli spaces consist of broken trajectories. Namely, given a triple $prq \in \bA_{pq}^{\mc P}$, it is easy to see that there is an isomorphism of homogeneous posets
\beq\label{eqn33}
\bA_{pr}^{\mc P} \times \bA_{rq}^{\mc P} \cong \partial^{prq} \bA_{pq}^{\mc P}
\eeq
under the concatenation of strings.
\end{notation}

%\textcolor{red}{Do we use homological or cohomological convention? %Does the differential increase or decrease the energy?

Now we introduce the notion of flow categories under the setting of Setup \ref{setup:poset}. 

\begin{defn}\label{def:flow-cat}
Let ${\mc P}$ be as in Setup \ref{setup:poset}. A \textbf{flow category} $T^{\mc P}$ over ${\mc P}$ is a topologically enriched category\footnote{Namely, the set of morphisms are topological spaces and composition maps are continuous.} with the set of objects given by ${\mc P}$, with morphism spaces $T_{pq}$ satisfying the following conditions.
\begin{enumerate}
    \item $T_{pq} \neq \emptyset$ only if $p \leq q$ in ${\mc P}$.\footnote{In the Morse or Floer case, it is indeed true that $T_{pq} \neq \emptyset$ if and only if $p \leq q$.}
    
    \item $T_{pp}$ is the singleton.
    
    \item $T_{pq}$ is a compact $\bA_{pq}^{\mc P}$-space. %$\langle P \rangle_{p,q}$-space for $T_{pq} \neq \emptyset$.
    
    \item Given a triple $p<r<q$ in ${\mc P}$, the composition map factors through a stratified homeomorphism
    \beqn
    \xymatrix{ T_{pr}\times T_{rq} \ar[r] \ar[d] & \partial^{prq} T_{pq} \ar[d] \\
              \bA_{pr}^{\mc P} \times \bA_{rq}^{\mc P} \ar[r]  & \partial^{prq} \bA_{pq}^{\mc P}}
              \eeqn
              where the underlying poset isomorphism is the map \eqref{eqn33}. We require that whenever $p<r<s<q$, the following associativity diagram holds:
    \begin{equation*}
        \begin{tikzcd}
T_{pr} \times T_{rs} \times T_{sq} \arrow[r] \arrow[d] & T_{ps} \times T_{sq} \arrow[d] \\
T_{pr} \times T_{rq} \arrow[r]                         & T_{pq}.                        
\end{tikzcd}
    \end{equation*}
    
    \item $\Pi$ defines a strict action on $T^{\mc P}$: for any $a \in \Pi$ and $p, q\in {\mc P}$, there is a stratified homeomorphism 
    \beqn
    \xymatrix{  T_{pq} \ar[r]^{\phi_a} \ar[d] & T_{a\cdot p \ a \cdot q } \ar[d]\\
                \bA_{pq}^{\mc P} \ar[r] & \bA_{a\cdot p\  a\cdot q}^{\mc P} }
    \eeqn
    where the underlying poset map is the natural isomorphism. Moreover, when $a_1, a_2 \in \Pi$, we require that the equation $\phi_{a_1 \cdot a_2} = \phi_{a_1} \circ \phi_{a_2}$ holds and $\phi_0$ is the identity map for $a = 0 \in \Pi$.
\end{enumerate}
\end{defn}

\begin{lemma}\label{lem:bound-prod}
Given $\alpha = pr_1 \cdots r_l q \in \mb{A}^{\mc P}_{pq}$, the space $\partial^\alpha T_{pq}$ is homeomorphic to $T_{p r_1} \times \cdots \times T_{r_l q}$ as $\mb{A}^{\mc P}_{pr_1} \times \cdots \times \mb{A}^{\mc P}_{r_l q}$-spaces.
\end{lemma}
\begin{proof}
We prove the statement by induction on $\dep(\alpha)$. For $\dep(\alpha) = 0$, this is tautology, and for $\dep(\alpha) = 1$, the assertion follows from Definition \ref{def:flow-cat}. Suppose the lemma holds for all $\dep(\alpha) \leq l-1$. Now suppose $\alpha = pr_1 \cdots r_l q$. Consider the homeomorphism between $\mb{A}^{\mc P}_{pr_1} \times \mb{A}^{\mc P}_{r_1 q}$ spaces $T_{p r_1} \times T_{r_1 q} \to \partial^{p r_1 q} T_{pq}$. Restricting the homeomorphism along the closed stratum $\mb{A}^{\mc P}_{pr_1} \times \partial^{r_1 \cdots r_l q} \mb{A}^{\mc P}_{r_1 q}$ and using the induction hypothesis, we obtain a homeomorphism of $\mb{A}^{\mc P}_{pr_1} \times \cdots \times \mb{A}^{\mc P}_{r_l q}$-spaces
$$ T_{p r_1} \times \cdots \times T_{r_l q} \to \partial^{\alpha} T_{pq}. $$
By associativity, if we construct such a homeomorphism by decomposing $\alpha$ as $pr_1 \cdots r_k$ and $r_k \cdots r_l q$ for some $1 \leq k \leq l$, the resulting homeomorphism between the stratified spaces is the same.
\end{proof}

Definition \ref{def:flow-cat} does not impose any regularity conditions or tangential structures on the morphism spaces. The purpose of such a definition is to single out the stratification structures and we will introduce several enhancements by putting various structures on the morphism spaces in the sequel.

\subsubsection{Flow bimodules}

We first introduce the posets indexing strata in flow bimodules. 

\begin{notation}
Suppose ${\mc P}$ and ${\mc P}'$ are two posets as in Setup \ref{setup:poset} equipped with own action and index functions 
\begin{align*}
&\ ({\mc A}^{\mc P}, {\rm ind}^{\mc P}): {\mc P} \to {\mb R} \times ({\mb Z}/ 2N),\ &\ ({\mc A}^{{\mc P}'}, {\rm ind}^{{\mc P}'}): {\mc P}' \to {\mb R}\times ( {\mb Z}/2N).
\end{align*}
For $p \in {\mc P}$ and $p' \in {\mc P}'$, define a poset 
\beqn
\bA_{pp'}:= \big\{ \alpha = p q_1 \cdots q_k q^{\prime}_{k'} \cdots q^{\prime}_{1} p' \ |\ p < q_1 < \cdots < q_k,\ q^{\prime}_{k'} < \cdots < q^{\prime}_{1} < p'\big\}
\eeqn
The partial order is again induced by inclusion: %except for the identity relation, 
\begin{multline*}
p q_1 \cdots q_k q^{\prime}_{k'} \cdots q^{\prime}_{1} p' \leq p \tilde{q}_1 \cdots \tilde{q}_{\tilde{k}} \tilde{q}^{\prime}_{\tilde{k}'} \cdots \tilde{q}^{\prime}_{1} p'\\
\Longleftrightarrow \{ \tilde{q}_1, \dots, \tilde{q}_{\tilde{k}} \} \subseteq \{ q_1, \dots, q_k \}\ {\rm and}\ \{ \tilde{q}^{\prime}_{1} \cdots \tilde{q}^{\prime}_{\tilde{k}'} \} \subseteq \{ q^{\prime}_{1} \cdots q^{\prime}_{k'} \}.
\end{multline*}
$\bA_{pp'}$ is a homogeneous poset with a unique maximal element $pp'$ and depth function
\beqn
\dep(pq_1 \cdots q_k q_{k'}'\cdots q_1' p') = k + k'.
\eeqn

There are similar characterizations of ``broken configurations.'' If $p<q$ are elements in ${\mc P}$, it is easy to see that there is a natural isomorphism of homogeneous posets
\beqn
\bA_{pq}^{\mc P}\times \bA_{qp'} \cong \partial^{pqp'} \bA_{pp'}
\eeqn
by concatenation of words. Such an isomorphism makes the following diagram commute:
\beqn
\xymatrix{  &  \bA_{pq_1}^{\mc P}\times \bA^{\mc P}_{q_1 q_2} \times \bA_{q_2 p'}   \ar[rd] \ar[ld] & \\ 
\bA_{p q_2}^{\mc P}\times \bA_{q_2 p'} \ar[rd] & & \bA^{\mc P}_{pq_1}\times \bA_{q_1 p'} \ar[ld] \\
 &   \bA_{pp'} & }
\eeqn
Similarly, if $q' < p'$ in ${\mc P}'$, one has
\beqn
\bA_{pq'}\times \bA_{q'p'}^{{\mc P}'} \cong \partial^{pq'p'} \bA_{pp'},
\eeqn
which satisfies a similar commutative relation as above, and in this case the poset $\bA^{{\mc P}'}$ acts on the right. Moreover, these two types of isomorphisms are compatible in the following sense. Namely, the following diagram is commutative for which the arrows are induced by the obvious concatenation of words.
\beqn
\xymatrix{  &  \bA_{pq}^{\mc P}\times \bA_{qq'} \times \bA_{q'p'}^{{\mc P}'}   \ar[rd] \ar[ld] & \\ 
\bA_{pq}^{\mc P}\times \bA_{qp'} \ar[rd] & & \bA_{pq'}\times \bA_{q'p'}^{{\mc P}'} \ar[ld] \\
 &   \bA_{pp'} & }
\eeqn
\end{notation}

\begin{defn}\label{defn:flow-bimod}
Let $T^{\mc P}$ and $T^{{\mc P}'}$ be flow categories over ${\mc P}$ and ${\mc P}'$ respectively. A \textbf{flow bimodule} $M$ from $T^{\mc P}$ to $T^{{\mc P}'}$ consists of the following data.
\begin{enumerate}
    \item A compact $\bA_{pp'}$-space $M_{pp'}$ (which could be empty) for all $p \in {\mc P}$ and $p' \in {\mc P}'$.%if it is nonempty.
    
    \item For $p<q$, a homeomorphism of stratified spaces
    \beqn
    \xymatrix{ T^{\mc P}_{pq} \times M_{qp'} \ar[r] \ar[d] & \partial^{pqp'} M_{pp'} \ar[d] \\
               \bA_{pq}^{\mc P} \times \bA_{qp'} \ar[r]_{\cong} & \partial^{pqp'} \bA_{pp'} }.
               \eeqn
               
    \item For $q' < p'$, a homeomorphism of stratified spaces
    \beqn
    \xymatrix{ M_{pq'} \times T^{{\mc P}'}_{q'p'}  \ar[r] \ar[d] & \partial^{pq'p'}  M_{pp'} \ar[d] \\
              \bA_{pq'}\times \bA_{q'p'}^{{\mc P}'} \ar[r]_{\cong} & \partial^{pq'p'} \bA_{pp'} }.  
              \eeqn
              \end{enumerate}
These data should be subject to the following conditions.
\begin{enumerate}
    \item There is a constant $ C > 0$ such that for all $p\in {\mc P}$, $p'\in {\mc P}'$, 
    \beq\label{eqn34}
    M_{pp'} \neq \emptyset \Longrightarrow {\mc A}^{\mc P}(p) < {\mc A}^{{\mc P}'}(p') + C.\footnote{The constant will be revealed to be related to the Hofer-type norm of a given Hamiltonian.}
    \eeq
    
    \item For $p < q_1 < q_2$ in ${\mc P}$ and $p' \in {\mc P}'$, the following diagram commutes:
    \begin{equation*}
        \begin{tikzcd}
T^{\mc P}_{pq_1} \times T^{\mc P}_{q_1 q_2} \times M_{q_2 p'} \arrow[d] \arrow[r] & T^{\mc P}_{pq_1} \times M_{q_1p'} \arrow[d] \\
T^{\mc P}_{pq_2} \times M_{q_2p'} \arrow[r]                               & M_{pp'}                            
\end{tikzcd}
    \end{equation*}
    where $T^{\mc P}_{pq_1} \times M_{q_1 p'} \rightarrow M_{pp'}$ is induced by the composition of the homeomorphism $T^{\mc P}_{pq_1} \times M_{q_1 p'} \rightarrow \partial^{p q_1 p'} M_{pp'}$ and the inclusion $\partial^{p q_1 p'} M_{pp'} \hookrightarrow M_{pp'}$ and so forth.
    \item Similarly, for $p \in {\mc P}$ and $q_2' < q_1' < p'$ in ${\mc P}'$, we have a commutative diagram
    \begin{equation*}
    \begin{tikzcd}
M_{pq_2'} \times T^{{\mc P}'}_{q_2' q_1'} \times T^{{\mc P}'}_{q_1' p'} \arrow[d] \arrow[r] & M_{p q_1'} \times T^{{\mc P}'}_{q_1' p} \arrow[d] \\
M_{pq_2'} \times T^{{\mc P}'}_{q_2' p'} \arrow[r]                                & M_{pp'}.                           
\end{tikzcd}
    \end{equation*}
    \item For $p<q$ in ${\mc P}$ and $q' < p'$ in ${\mc P}'$, we have a commutative diagram
        \begin{equation*}
    \begin{tikzcd}
T^{\mc P}_{pq} \times M_{qq'} \times T^{{\mc P}'}_{q' p'} \arrow[d] \arrow[r] & M_{p q'} \times T^{{\mc P}'}_{q' p'} \arrow[d] \\
T^{\mc P}_{pq} \times M_{qp'} \arrow[r]                                & M_{pp'}.                           
\end{tikzcd}
    \end{equation*}
    \item Strict $\Pi$-action: for any $a \in \Pi$, there is a stratified homeomorphism \beqn
    \xymatrix{ M_{pp'} \ar[r] \ar[d] & M_{a\cdot p \ a \cdot p'} \ar[d] \\
               \bA_{pp'} \ar[r] & \bA_{a\cdot p \ a \cdot p'} }
               \eeqn
 such that for $a_1, a_2 \in \Pi$ the equation $\phi^M_{a_1 \cdot a_2} = \phi^M_{a_1} \circ \phi^M_{a_2}$ holds, and such that $\phi^M_{id}$ is the identity map. Moreover, we require that the actions
$$
T^{\mc P}_{pq_1} \times M_{q_1 p'} \rightarrow M_{pp'}, M_{pq_1'} \times T^{{\mc P}'}_{q_1' p'} \to M_{pp'}
$$
are $\Pi$-equivariant.
\end{enumerate}
\end{defn}

The following statement is the analog of Lemma \ref{lem:bound-prod} for flow bimodules. The associativity conditions from Definition \ref{defn:flow-bimod} guarantees that the maps between the stratified spaces are well-defined.

\begin{lemma}
Suppose $M_{pp'}$ is nonempty. Given an element $\alpha = pq_1 \cdots q_k q^{\prime}_{k'} \cdots q^{\prime}_1 p' \in \bA_{pp'}$, we have a stratified homeomorphism
\beqn
\vcenter{ \xymatrix{ T^{\mc P}_{pq_1} \times \cdots \times M_{q_k q^{\prime}_{k'}} \times \cdots \times T^{{\mc P}'}_{q^{\prime}_1 p'} \ar[r] \ar[d] & \partial^\alpha M_{pp'} \ar[d] \\
           \bA^{\mc P}_{p q_1} \times \cdots \times \bA_{q_k,q^{\prime}_{k'}} \times \cdots \times \bA^{{\mc P}'}_{q^{\prime}_1, p'} \ar[r]^-{\cong} & \partial^\alpha \bA_{pp'} }  }.
           \eeqn
\end{lemma}

As spelled out in \cite[Section 8]{abouzaid2022axiomatic}, one can define ``higher" homotopies between flow bimodules parametrized cubes $[0,1]^n$ for any $n \in {\mb Z}_{>0}$ and construct a cubically enriched category of flow bimodules. Such notions are useful for proving invariance of Floer homology/homotopy groups, but establishing such invariance is beyond the scope of this paper.

\subsection{Outer-collaring}\label{subsection:outer_collar}

In Floer theory one often needs to inductively construct structures (such as perturbations) on infinitely many moduli spaces. Each moduli space is an orbifold with corners in a suitable sense. One wishes to construct the structures such that near the boundary or corner they are of the product type. This requires various compatible collar structures near the boundary and corners. The existence of a compatible system of collar structures is difficult to construct directly. Instead, following \cite[Chapter 17]{FOOO_Kuranishi}, we take a short-cut by constructing the collars ``outside.''

We will first consider the outer-collaring construction for individual stratified spaces, flow categories, and bimodules. They are operated in the topological category. This discussion is not immediately used in this section. Later we will also discuss outer-collaring of stratified objects in the smooth category. 

\subsubsection{Outer-collaring of stratified spaces}

We first discuss the outer-collaring of a single moduli space which may appear in either a flow category or a bimodule. We consider the first case (flow category) and the second case is completely similar. Abbreviate $\bA_{pq}^{\mc P}$ by $\bA$ where ${\mc P}$ satisfies conditions of Setup \ref{setup:poset} and $p, q\in {\mc P}$. Let $X$ be an $\bA$-space whose strata are indexed by words $\alpha = pr_1 \cdots r_l q$. Choose $r \geq 0$ and we will define a new $\bA$-space denoted by $X^\bpr$. As a set, 
\beqn
X^\bpr = \left( \bigsqcup_{\alpha \in \bA} \partial^\alpha X \times [-r, 0]^{\bF_\alpha} \right)/ \sim,
\eeqn
where $\bF_\alpha \subset \bA$ is the set of adjacent faces and where the equivalence relation $\sim$ is generated by the following relation: if $\alpha \leq \beta$ (which implies the inclusion of the sets of adjacent faces $\bF_\beta \subset \bF_\alpha$), we identify 
\beqn
(x, (t_i)_{i \in \bF_\beta}) \in \partial^\beta X \times [-r, 0]^{\bF_\beta}
\eeqn
with
\beqn
(y, (s_j)_{j\in \bF_\alpha}) \in \partial^\alpha X \times [-r, 0]^{\bF_\alpha}
\eeqn
if $x = y \in \partial^\alpha X$, $s_j = 0$ when $j \notin \bF_\beta$ and $t_j = s_j$ when $j \in \bF_\beta$. We call $X^\bpr$ the {\bf outer collaring} of $X$ of width $r$ (see Figure \ref{outercollar}).

\begin{center}
\begin{figure}[h]
\begin{tikzpicture}

\node at (1, 1) {\scriptsize $X$};
\node at (0, 2.2) {\scriptsize $\partial^{\beta_2} X$};
\node at (2.4, -0.05) {\scriptsize $\partial^{\beta_1} X$};

\draw (0, 2) -- (0, 0) -- (2, 0);

\draw [very thick] (-2, 2) -- (-2, -2) -- (2, -2);

\draw [dotted] (-2, 0) -- (0, 0) -- (0, -2);

\node at (-1, -1) {\scriptsize $\partial^\alpha X \times [-r, 0]^2$};
\node at (1, -1) {\scriptsize $\partial^{\beta_1} X \times [-r, 0]$};
\node at (-1, 1) {\scriptsize $\partial^{\beta_2} X \times [-r, 0]$};
\end{tikzpicture}
\caption{The local picture of an outer collaring of a prestratified space for $\alpha < \beta_1, \beta_2$.}\label{outercollar}
\end{figure}
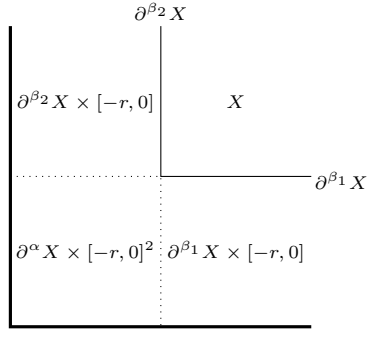
\end{center}

One has a natural identification $X = X^{\boxplus 0}$ and the natural inclusion $X^\bpr \hookrightarrow X^{\boxplus r'}$ when $r \leq r'$. Moreover, $X^\bpr$ has a structure of an $\bA$-space as follows. For each $\beta \in \bA$, define the boundary stratum
\beqn
\partial^\beta X^\bpr = \bigcup_{\alpha \leq \beta} \Big( \partial^{\alpha} X \times \{ (t_j)_{j\in \bF_\alpha} \ |\ t_j  = - r \ \forall j \in \bF_\beta \} \Big).
\eeqn
Then the corresponding stratum for $\beta \in \bA$ is given by
\beqn
X_\beta^\bpr = \partial^\beta X^\bpr \setminus \bigcup_{\alpha < \beta} \partial^\alpha  X^\bpr.
\eeqn
It is easy to see that the induced map $X^\bpr \to \bA$ is a stratification.

The following lemma shows that the outer-collaring construction respects products.

\begin{lemma}\label{outer_product}
\begin{enumerate}
    \item For any $\alpha \in \bA$, we have $(\partial^\alpha X)^{\bpr} = \partial^\alpha X^\bpr$ as $\partial^\alpha \bA$-spaces. 
    
    \item If $X_i$ are $\bA_i$-spaces for $i=1,2$, then $X_1^\bpr \times X_2^\bpr = (X_1 \times X_2)^{\bpr}$ as $\bA_1 \times \bA_2$-spaces.
    \end{enumerate}
\end{lemma}
\begin{proof}
These statements are true by inspecting the definition.
\end{proof}

\begin{lemma}\label{lem:collar-map}
Given a stratified map
\beqn
\xymatrix{ X_1 \ar[r]^f \ar[d] & X_2 \ar[d] \\
           \bA_1 \ar[r]_{\iota} & \bA_2},
           \eeqn
          there is a canonical extension 
          \beqn
          \xymatrix{ X_1^\bpr \ar[r]^{f^\bpr} \ar[d] & X_2^\bpr\ar[d] \\
                     \bA_1 \ar[r]_\iota & \bA_2}.
                     \eeqn
\end{lemma}

\begin{proof}
Using the presentation
\beqn
X_i^\bpr = \left( \bigsqcup_{\alpha_i \in \bA_i} \partial^\alpha X_i \times [-r, 0]^{\bF_{\alpha_i}} \right)/ \sim \text{ for } i=1,2,
\eeqn
we define the restriction of $f^{\bpr}$ over $\partial^{\alpha_1} X_1 \times [-r, 0]^{\bF_{\alpha_1}} \to \partial^{\iota(\alpha_1)} X_2 \times [-r, 0]^{\bF_{\iota(\alpha_1)}}$ to be $f \times \text{id}$. It is easy to see that equivalent points are mapped to equivalent points, therefore we can define $f^{\bpr}$ to be the induced map.
\end{proof}

The outer collar construction is compatible with group actions. Let $G$ be a topological group acting on $X$ via stratum-preserving homeomorphisms. For each $r>0$, define the $G$-action on $X^\bpr$ as follows. For each $g \in G$ and $(x, (t_i)_{i \in \bF_\alpha})\in \partial^\alpha X \times [-r, 0]^{\bF_\alpha}$, define
\beqn
g \cdot (x, (t_i)_{i \in \bF_\alpha} ) = (gx, (t_i)_{i \in \bF_\alpha} ) \in \partial^\alpha X \times [-r, 0]^{\bF_\alpha} \subset X^\bpr.
\eeqn
It is easy to check that the action is well-defined, continuous, and preserves strata.

\subsubsection{Outer-collaring of flow categories and bimodules}\label{sec:outer-collar}

The outer-collaring construction described above can be naturally carried over to a system of stratified spaces included in a flow category or bimodule. Let ${\mc P}$ satisfy Setup \ref{setup:poset} and let $T^{\mc P}$ be a topological flow category over ${\mc P}$ (the $\Pi$-action in this discussion is trivial hence we omit the reference to it). Fix $r\geq 0$. All outer-collaring in this discussion will be of a fixed width $r$ and will be labelled by $+$ instead of $\bpr$. Then we can apply the outer-collaring construction to each individual space $T_{pq}$ for all $p< q$, obtaining $\bA_{pq}^{\mc P}$-spaces $T_{pq}^+$. Define ``composition maps''
\beqn
T_{pr}^+ \times T_{rq}^+ \cong \partial^{prq} T_{pq}^+ \hookrightarrow T_{pq}^+
\eeqn
as the composition
\beqn
\xymatrix{ T_{pr}^+ \times T_{rq}^+ \ar[r] & (T_{pr} \times T_{rq})^+ \ar[r] & (\partial^{prq} T_{pq})^+ \ar[r] & \partial^{prq} (T_{pq}^+)}.
\eeqn
Here the first map comes from item (2) of Lemma \ref{outer_product}, the second map is the canonical outer-collaring of the original composition map $T_{pr}\times T_{rq} \to \partial^{prq} T_{pq}$, and the third map comes from item (1) of Lemma \ref{outer_product}. It is easy to check that the composition maps are still associative. Hence we obtained a new topological flow category over ${\mc P}$, which we call the {\bf outer-collaring} of $T^{\mc P}$, denoted by 
\beqn
(T^{\mc P})^+
\eeqn
whose morphisms spaces are $T_{pq}^+$.

Similar happens when we apply outer-collaring to flow bimodules. Let $T^{\mc P}$ and $T^{{\mc P}'}$ be flow categories over ${\mc P}$ and ${\mc P}'$ respectively and let $M_{{\mc P}{\mc P}'}$ be a flow bimodule from $T^{{\mc P}}$ to $T^{{\mc P}'}$. By doing outer-collaring to each individual space $M_{pp'}$ for $p\in {\mc P}$ and $p'\in {\mc P}'$, obtaining new spaces $M_{pp'}^+$, one obtains a flow bimodule $M_{{\mc P}{\mc P}'}^+$ from $T^{{\mc P}+}$ to $T^{{\mc P}'+}$. This new flow bimodule is called the {\bf outer-collaring} of $M_{{\mc P}{\mc P}'}$.

\subsection{Derived orbifold charts with stratifications}\label{subsection33}

As mentioned before, the morphism spaces of flow categories and flow bimodules are \emph{a priori} only topological spaces. We would like to define the notion of \textbf{derived orbifold lift} for (topological) flow categories and flow bimodules.

\subsubsection{Stratified topological manifolds corners}

Before discussing smooth structures, we introduce a refined structure on stratified spaces in a special setting. It is well-known that ``manifolds with corners" are identical to manifolds with boundary in the topological category. One way to define topological $\bA$-manifolds is given by \cite[Definition 2.4]{abouzaid2022axiomatic}.  We specialize \emph{loc. cit.} which suffices for our purpose.

First we introduce some notations. The standard model for a manifold with corners is the space $[0, +\infty)^k$ for $k \geq 0$. It is stratified by subspaces where a subset of coordinates are equal to $0$. 

\begin{notation}\label{notation314}
For any finite set ${\mb F}$, let $\bA^{({\mb F})}$ be the set of all subsets of ${\mb F}$ where the partial order is induced by inclusion:
\beqn
\alpha \leq \beta \Longleftrightarrow \beta \subseteq \alpha.
\eeqn
$\bA^{({\mb F})}$ is homogeneous with a unique maximal element $\emptyset$ and its depth function being
\beqn
\dep(\alpha ) = \# \alpha.
\eeqn
We abbeviate $\bA^{(k)}:= \bA^{(\{ 1, \ldots, k\})}$. Then $[0, +\infty)^k$ is an $\bA^{(k)}$-space. 
\end{notation}

%finite poset with minimal element $p$ and maximal element $q$. Let $\A$ be the partially ordered set consisting of strings $\alpha = p < r_1 < \cdots < r_l < q$ (abbreviated by a word $\alpha = pr_1 \cdots r_l q$) of elements in $\bA$, with partial order induced by inclusion: $ps_1 \cdots s_m q \leq pr_1 \cdots r_l q$ if $\{s_1, \cdots, s_m\}\subseteq \{r_1, \dots, r_l\}$. In practice, $\A$ is of the form $\bA^\floer_{pq}$ or $\langle {\mc P}-{\mc P}' \rangle_{p,p'}$. For $\langle {\mc P}-{\mc P}' \rangle_{p,p'}$, we declare that every element in ${\mc P}'$ is greater than any element in ${\mc P}$, so $\langle {\mc P}-{\mc P}' \rangle_{p,p'}$ comes from a poset with minimal element $p$ and maximal element $p'$. 

%\textcolor{red}{How to make notations consistent to include bimodule homotopies/sets of the form $\A \times {\mc Q}_{[0,1]}$?}

%\textcolor{blue}{I think in the new definition it is unnecessary to discuss these notations.}

%\begin{notation}
%Let $\bA$ be a finite homogeneous poset. %with the codimension function $\dep: \bA \to {\mb Z}_{\geq 0}$. 
%For each element $\alpha \in \bA$, denote 
%\beqn
%{\rm Face}(\alpha) = \{ \beta \in \partial^{[1]} \bA\ |\ \alpha \leq \beta\} \subseteq \bA,
%\eeqn
%which is the set of adjacent faces to the stratum $\alpha$. %If $\alpha = pq$, define ${\rm Face}(\alpha)= % \emptyset$.
%\end{notation}

\begin{defn}\label{defn312}
Let $\bA$ be a (countable) homogeneous poset. An $\bA$-space is called a (topological) \textbf{$\bA$-manifold} if the following is true.
\begin{enumerate}

\item For any $\alpha \in \bA$, the space $\partial^\alpha X$ is a topological manifold (possibly with boundary) whose interior is $X_\alpha$. %such that the inclusion map
%\beqn
%\bigcup_{\beta < \alpha} \partial^\beta X \to \partial^\alpha X
%\eeqn
%is a homeomorphism onto the boundary of $\partial^\alpha X$.

    \item %For any pair $\alpha < \beta$ of elements of $\bA$, %define 
    %\beqn
    %c_{\alpha\beta}:= \# \Big( {\rm Face}(\alpha) \setminus {\rm Face}(\beta) \Big).
    %\eeqn
%    and any $x \in \partial^\alpha X \subset \partial^\beta X$, there is a $G$-invariant neighborhood $U_x \subset \partial^\alpha X$ such that there exists an $G$-equivariant embedding
 %   \beqn
 %   U_x \times [0,1)^{\dep(\alpha) - \dep(\beta)} \rightarrow \partial^\beta X
 %   \eeqn
 %   which defines a $G$-invariant neighborhood of $x \in \partial^\beta X$, where $G$ acts trivially on the $[0,1)^{\dep(\alpha) - \dep(\beta)}$-factor and $U_x \subset \partial^\alpha X$ is identified with $U_x \times \{0\}^{\dep(\alpha) - \dep(\beta)}$.
 
Each $x \in X_\alpha \subset X$ has a {\bf corner chart}, i.e., a stratified open embedding 
\beqn
\xymatrix{  U_{x, \alpha} \times [0, 1)^{\dep(\alpha)} \ar[r]^-\varphi \ar[d] & X \ar[d] \\
             \{ \alpha \} \times \bA^{(\dep(\alpha))} \ar[r] &  \bA }
 \eeqn
 where $U_{x, \alpha}$ is an open subset of $x$ in $X_\alpha$ (with the trivial stratification indexed by the singleton $\{\alpha\}$). 
 
\item For each corner chart $\varphi: U_x \times [0, 1)^{\dep(\alpha)} \to X$ near $x \in X_\alpha$, one has 
\beqn
\alpha \leq \beta \Longrightarrow {\rm Im}\varphi \cap X_\beta \neq \emptyset.
\eeqn
In other words, corner charts touch all strata above $\alpha$.

\end{enumerate}
\end{defn}

\begin{rem}
Conditions of Definition \ref{defn312} impose more restrictions on the homogeneous poset $\bA$. For example, for each $\alpha \in \bA$, the set $\{ \beta\in \bA\ |\ \alpha \leq \beta \}$ is isomorphic to $\bA^{(\dep(\alpha))}$ and hence contains a unique maximal element. Moreover, each $\alpha \in \bA$ has a definite set of elements $\alpha \leq \beta$ with $\dep(\beta)=1$ whose cardinality is equal to its depth. The posets $\bA_{pq}^{\mc P}$ and $\bA_{pp'}$ appeared in the disucssions of flow categories and flow bimodules all satisfy these additional restrictions. 
\end{rem}

\begin{defn}
Let $G$ be a topological group acting on an $\bA$-manifold $X$ viewing $X$ as an $\bA$-space. We say that $X$ is an $\bA$-stratified $G$-manifold with corners if each $x \in X_\alpha$ has a $G$-equivariant corner chart, i.e., a corner chart $U_{x, \alpha} \times [0,1)^{\dep(\alpha)} \to X$ which is a $G$-equivariant map. Here $U_{x, \alpha} \subset U_\alpha$ is a $G$-invariant open subset and $G$ acts trivially on the factor $[0, 1)^{\dep(\alpha)}$.
\end{defn}

\subsubsection{Stratified smooth manifolds with corners}

Next we recall some notions related to smooth manifolds with corners. We use the following notion of smoothness for functions defined on arbitrary subsets of Euclidean spaces. Let $A \subseteq {\mb R}^n$ be a subset. A function $f: A \to {\mb R}$ is called smooth if for each $a \in A$ there is an open neighborhood $U_a \subset {\mb R}^n$ such that $f|_{A \cap U_a}$ can be extended to a smooth function $f_a: U_a \to {\mb R}$. Then one can define the notion of smooth maps $f: A \to B$ for subsets $A \subseteq {\mb R}^n$ and $B \subseteq {\mb R}^m$. Then compositions of smooth maps are still smooth. A smooth map $f: A \to B$ is a diffeomorphism if it admits a smooth inverse. 

The concept of smooth manifolds with corners is defined as follows. Let $M$ be a topological space. A chart (of an $n$-dimensional smooth manifold with corners) consists of an open subset $U\subset M$ and a homeomorphism $\varphi: U \to [0, +\infty)^n$ onto an open subset. Two charts $\varphi_i:U_i \to [0, +\infty)^n$, $i = 1, 2$, are compatible if $\varphi_1\circ \varphi_2^{-1}: \varphi_2(U_1 \cap U_2) \to \varphi_1(U_1 \cap U_2)$ is a diffeomorphism. A structure of smooth $n$-dimensional manifold with corners is a maximal compatible atlas on $M$. If $M$ is equipped with such a structure, we say that $M$ is a smooth $n$-dimensional manifold with corners. Over smooth manifolds with corners, one can define the notions of smooth vector bundles, smooth connections, and smooth Riemannian metrics, etc.

The concept of manifold with faces is one which can exclude certain pathological scenarios. Let $M$ be a smooth $n$-dimensional manifold with corners. For each point $x \in M$ and each chart $\varphi: U \to [0, +\infty)^n$ defined over an open neighborhood of $x$, the number $\dep (x)$ of coordinates of $\varphi(x)$ which are zero is independent of the choice of charts. Given $k \leq n$, the subset 
\beqn
\partial^k M:= \{ x \in M\ |\ \dep(x) = k \}\subset M
\eeqn
is naturally an $n-k$-dimensional smooth manifold. A {\bf connected face} of $M$ is a connected component of the set $\{ x\in M\ |\ \dep (x) = 1\}$. $M$ is called a {\bf smooth manifold with faces} if every $x \in M$ belongs to the closure of exactly $\dep(x)$ different connected faces. %We assume that $M$ has only finitely many connected faces.

As we are treating various manifolds with faces which have more refined structure of stratifications, we would like to introduce the following concepts. %The notion of $\A$-manifolds is a generalization of the notion of $\langle n\rangle$-manifolds introduced in \cite{Laures_2000}. 

\begin{defn}\label{defn31}
Let $\bA$ be a countable homogeneous poset.
\begin{enumerate}

\item A \textbf{smooth $\bA$-manifold} is a smooth manifold with faces $M$ which is also a topological $\bA$-manifold such that for each $\alpha \in \bA$, $M_\alpha$ is an open and closed subset of $\partial^{\dep(\alpha)} M$ 
%together with a bijection between the set of connected faces with the set of three-letter words. If $M$ is a smooth $\A$-manifold, then $M$ is an $\A$-space as from Section \ref{subsec:strat}. 
and such that each $x \in M$ has a smooth corner chart.

\item Let $G$ be a compact Lie group. A smooth $G$-action on a smooth $\bA$-manifold $M$ is a $G$-action on the underlying topological $\bA$-manifold of $M$ which acts on $M$ by smooth maps.
%\item \textcolor{red}{Each stratum admits a collar $\partial^{pr_1\cdots r_k q}M \times [0,1)^{k+1}$ near each point for which $G$ acts by identity on the $[0,1)^{k+1}$ coordinates.}
\end{enumerate}
\end{defn}

Note that any smooth $\bA$-manifold $M$ has a natural structure of a topological $\bA$-manifold by forgetting the smooth structure.% and declaring the (topological) boundary of $\partial^\alpha M$ to be the union of strata $\bigcup_{\beta < \alpha} \partial^\beta M$.

\begin{defn}\label{defn:strong-sub}
A smooth map $f$ from a smooth $\bA$-manifold $M$ to a smooth manifold without boundary $N$ is called \emph{stratified submersive} if for any (open) stratum $M_\alpha$, the restriction 
\beqn
f|_{M_\alpha}: M_\alpha \rightarrow N
\eeqn
is a smooth submersion.
\end{defn}

\begin{lemma}\label{lem:regular-level}
Suppose $f: M \to N$ is stratified submersive. Then for any $y \in N$, the space $f^{-1}(y)$ is a smooth $\bA$-manifold. Moreover, if $G$ is a compact Lie group acting on $M$ and $f$ is $G$-invariant, then $f^{-1}(y)$ has a smooth $G$-action. %if $(f|_{M_\alpha})^{-1}(y)$ is connected for each $\alpha \in \bA$. ??????
\end{lemma}

\begin{proof}
By Definition \ref{defn:strong-sub}, the subsets $f^{-1}(y) \cap M_\alpha$ are smooth manifolds. Using the collars near each stratum, we see that $f^{-1}(y)$ is a smooth manifold with corners. We now verify that $f^{-1}(y)$ is a manifold with faces, namely, each $x\in f^{-1}(y)$ is contained in exactly $\dep(x)$ connected faces. Indeed, viewing $x$ as a point in $M$, by definition, $x$ is contained in exactly $\dep(x)$ connected faces of $M$. The implicit function theorem implies that $x$ is contained locally in exactly $\dep(x)$ faces and these faces are still disconnected globally in $f^{-1}(y)$. Hence $f^{-1}(y)$ is a smooth manifold with faces. 

Next, we see that the decomposition
\beqn
f^{-1}(y) = \bigsqcup_{\alpha \in \bA} f^{-1}(y) \cap M_\alpha
\eeqn
induces an $\bA$-space structure on $f^{-1}(y)$. Conditions of Definition \ref{defn312} can then be verified using the implicit function theorem. Transversality implies that $f^{-1}(y)$ has smooth corner charts. Hence $f^{-1}(y)$ is a smooth $\bA$-manifold. %by the connectedness assumption after checking the intersection pattern of faces. 
The equivariant case follows immediately from the definition.
\end{proof}

\subsubsection{Stratified smooth orbifolds with corners and derived orbifold charts}

Smooth orbifolds with corners can be defined as in Section \ref{subsec:orbi-setup}, with the modification that for an $n$-dimensional orbifold with corners ${\mc U}$, each local chart $C = (U, \Gammait, \psi)$ has the following property: consider ${\mb R}^{n-k} \times {\mb R}_{\geq 0}^k$ where $\Gammait$ acts linearly on the first factor and trivially on the second factor. Then $U \subset {\mb R}^{n-k} \times {\mb R}_{\geq 0}^k$ is a $\Gammait$-invariant open subset and $\psi: U / \Gammait \to {\mc U}$ is a homeomorphism onto an open subset of ${\mc U}$. Just as the manifold case, one can define the notion of smooth orbifold with faces and the notion of smooth $\bA$-orbifolds for any homogeneous poset $\bA$. We omit the discussions of bundles, sections, and embeddings in the stratified sense. Note that if $\alpha \in \bA$ and ${\mc U}$ is an $\bA$-orbifold, the $\partial^\alpha \bA$-space $\partial^\alpha {\mc U}$ is naturally a $\partial^\alpha \bA$-orbifold.

Now we can generalize the definition of derived orbifold charts (see Definition \ref{defn21}) to the case when the domains have corners or faces. 

\begin{defn}[Stratified derived orbifold charts]\hfill
\begin{enumerate}
    \item An $\bA$-stratified derived orbifold chart is a triple $({\mc U}, {\mc E}, {\mc S})$, where ${\mc U}$ is a smooth $\bA$-orbifold, ${\mc E} \to {\mc U}$ is an orbifold vector bundle, and ${\mc S}: {\mc U} \to {\mc E}$ is a continuous section. We say that $({\mc U}, {\mc E}, {\mc S})$ is compact if ${\mc S}^{-1}(0)$ is compact.
    
    \item Given $\bA_i$-stratified derived orbifold charts $C_i = ({\mc U}_i, {\mc E}_i, {\mc S}_i)$, $i = 1, 2$, a chart embedding ${\bm \iota}_{21}: C_1\hookrightarrow C_2$ consists of a smooth stratified embedding 
    \beqn
    \xymatrix{ {\mc U}_1 \ar[r]^{\iota_{21}} \ar[d] &  {\mc U}_2 \ar[d] \\
            \bA_1 \ar[r] & \bA_2 }
            \eeqn
    and a bundle embedding $\widehat \iota_{21}: {\mc E}_1 \to {\mc E}_2$ covering $\iota_{21}$ which intertwines with the sections ${\mc S}_1$ and ${\mc S}_2$.

    \item Given an $\bA$-space $X$, a \textbf{derived orbifold presentation} (or called a {\bf D-chart presentation}) of $X$ is a quadruple $({\mc U}, {\mc E}, {\mc S}, \psi)$ consisting of an $\bA$-stratified derived orbifold chart $({\mc U}, {\mc E}, {\mc S})$ with a homeomorphism of $\bA$-spaces $\psi: {\mc S}^{-1}(0) \to X$.
\end{enumerate}
\end{defn}

The following set of notations are frequently used. If $({\mc U}, {\mc E}, {\mc S}, \psi)$ is a derived orbifold presentation of the $\bA$-space $X$, for any stratum $\alpha$, the datum $( \partial^\alpha {\mc U}, {\mc E}|_{\partial^\alpha{\mc U}},  {\mc S}|_{\partial^\alpha {\mc U}}, \psi|_{{\mc S}^{-1}(0)\cap \partial^\alpha {\mc U}} )$ is a derived orbifold presentation of $\partial^\alpha X$. In practice, we might have a derived orbifold presentation of the boundary stratum $\partial^\alpha X$ different from the restriction of a derived orbifold presentation of $X$. To carry out our construction of perturbations, we would like to relate the possibly different derived orbifold presentations of the boundary strata. This motivates the following definition, which should be thought of as a special case of the \emph{Kuranishi diagram} from \cite[Definition 6.6]{abouzaid2022axiomatic}. Some notions appear in the following were defined in Definition \ref{stabilization}.

\begin{defn}\label{defn:d-chart-pre}
A \textbf{system of D-chart presentation} of an $\bA$-space $X$ consists of the following objects.%\footnote{\textcolor{blue}{Now I feel it is still better not to include the stabilization map as part of the data. These data should be part of a structure in another more technical level. It is also more consistent with the meaning of flow category: only composition maps but no more details about the ``normal direction'' of the composition maps. Moreover, the normal complex structure should be also be stated on the current level without specifying the ``normal behavior.''}}
\begin{enumerate}
\item A collection of derived orbifold presentations 
\beqn
\big\{ C_\alpha= ({\mc U}_\alpha, {\mc E}_\alpha, {\mc S}_\alpha, \psi_\alpha)  \big\}_{\alpha \in \bA}
\eeqn
of the collection of boundary strata $\{ \partial^\alpha X \}_{\alpha \in \bA}$ as stratified spaces. 

\item A collection of chart embeddings 
\beqn
\big\{ {\bm \iota}_{\beta\alpha}: C_\alpha \to \partial^\alpha C_\beta  \big\}_{\alpha\leq \beta}.
\eeqn
\end{enumerate}
These objects need to satisfy the following conditions.

\begin{enumerate}

\item[(A)] The collection of chart embeddings satisfy the cocycle condition. More precisely, for any triple of strata $\alpha \leq \beta \leq \gamma$, there holds
\beqn
\iota_{\gamma\beta} \circ \iota_{\beta\alpha} = \iota_{\gamma\alpha}.
\eeqn

\item[(B)] Adjacent strata differ by a stabilization. More precisely, for any pair of strata $\alpha \leq \beta$, there exist an orbifold vector bundle ${\mc F}_{\beta \alpha} \to {\mc U}_\alpha$ and a germ equivalence 
\beqn
 {\rm Stab}_{{\mc F}_{\beta\alpha}}( C_\alpha) \simeq \partial^\alpha C_\beta
\eeqn

%The chart embedding $\iota_{\beta\alpha}$ induces a bundle isomorphism
%\beqn
%{\mc E}_\alpha \oplus {\mc F}_{\beta\alpha} \cong \iota_{\beta\alpha}^* ( {\mc E}_\beta|_{\partial^\alpha {\mc U}_\beta}).
%\eeqn
%Then we require that for any triple of stratum $\alpha \leq \beta \leq \gamma$, one has
%\beqn
%{\mc F}_{\beta\alpha} \oplus {\mc F}_{\gamma\beta}|_{{\mc U}_\alpha} = {\mc F}_{\gamma\alpha}
%\eeqn
%where the two sides are regarded as subbundles of $\iota_{\gamma\alpha}^* {\mc E}_\gamma$. 

% chart embeddings satisfy the cocycle condition. Namely, for any triple of stratum $\alpha \leq \beta \leq \gamma$, one has  For $P < P'$ in $\A$, there is a orbifold vector bundle ${\mc E}_{P P'} \rightarrow {\mc U}_{P}$, such that the derived orbifold presentation of $\partial^P X$ obtained from applying the ${\mc E}_{P P'}$-stabilization to $({\mc U}_{P}, {\mc E}_{P}, {\mc S}_{P}, \psi_{P})$ is germ equivalent to the restriction presentation of $\partial^P X$
%    $$ (\partial^P {\mc U}_{P'}, {\mc E}_{P}|_{\partial^P {\mc U}_{P'}}, {\mc S}_{P}|_{\partial^P {\mc U}_{P'}}, \psi_{P}|_{{\mc S}_{P}^{-1}(0)|_{\partial^P {\mc U}_{P'}}}) $$
%    via an open embedding $\phi_{P P'}$ (\textcolor{red}{we also need $\psi_P = \psi_{P'} \circ \phi_{P P'}$}).
    
%    \item For elements $P < P' < P''$ in $\A$, there is an isomorphism ${\mc E}_{PP'} \oplus \phi_{PP'}^* {\mc E}_{P' P''} \rightarrow {\mc E}_{P P''}$ and $\phi_{PP''} = \phi_{P' P''} \circ \phi_{P P'}$, where the latter should understood as the germ equivalence.
\end{enumerate}
\end{defn}

In certain special cases the chart embeddings between different strata are open. For example in the case of Morse flow category and the pearly bimodule. We introduce the following notion.

\begin{defn}\label{defn:single-layer}
A system of derived orbifold presentations of an $\bA$-space $X$ is called {\bf single-layered} if all chart embeddings ${\bm \iota}_{\beta\alpha}: C_\alpha \to \partial^\alpha C_\beta$ are open embeddings, or equivalently, one can take ${\mc F}_{\beta\alpha} = 0$ for all $\alpha \leq \beta$.
\end{defn}

%\begin{defn}\label{defn:d-chart-pre}
%A \textbf{system of derived orbifold chart presentations} of an $\A$-space $X$ is a collection of derived orbifold presentations 
%\beqn
%\big\{ ({\mc U}_\alpha, {\mc E}_\alpha, {\mc S}_\alpha, \psi_\alpha)  \big\}_{\alpha \in \A}
%\eeqn
%of the collection of boundary strata $\{ \partial^\alpha X \}_{\alpha \in \A}$ satisfying the %following compatibility conditions.

%\begin{enumerate}

%    \item For $P < P'$ in $\A$, there is a orbifold vector bundle ${\mc E}_{P P'} \rightarrow {\mc U}_{P}$, such that the derived orbifold presentation of $\partial^P X$ obtained from applying the ${\mc E}_{P P'}$-stabilization to $({\mc U}_{P}, {\mc E}_{P}, {\mc S}_{P}, \psi_{P})$ is germ equivalent to the restriction presentation of $\partial^P X$
%    $$ (\partial^P {\mc U}_{P'}, {\mc E}_{P}|_{\partial^P {\mc U}_{P'}}, {\mc S}_{P}|_{\partial^P {\mc U}_{P'}}, \psi_{P}|_{{\mc S}_{P}^{-1}(0)|_{\partial^P {\mc U}_{P'}}}) $$
%    via an open embedding $\phi_{P P'}$ (\textcolor{red}{we also need $\psi_P = \psi_{P'} \circ \phi_{P P'}$}).
    
%    \item For elements $P < P' < P''$ in $\A$, there is an isomorphism ${\mc E}_{PP'} \oplus \phi_{PP'}^* {\mc E}_{P' P''} \rightarrow {\mc E}_{P P''}$ and $\phi_{PP''} = \phi_{P' P''} \circ \phi_{P P'}$, where the latter should understood as the germ equivalence.
%\end{enumerate}
%\end{defn}

\subsubsection{Derived orbifold lifts of flow categories and bimodules}

Now we can go back to the setup in Section \ref{subsec:top-flow-cat}. The following definition imposes certain regular structures on morphism spaces of flow categories and  flow bimodules.

\begin{defn}\label{defn:flow-lift}
Let $T^{\mc P}$ be a flow category over the poset ${\mc P}$. A \textbf{derived orbifold lift} of $T^{\mc P}$, denoted by ${\mf D}^{\mc P}$, consists of the following objects. 
\begin{enumerate}
    \item A collection
\beqn
\big\{ C_{pq} = ({\mc U}_{pq}, {\mc E}_{pq}, {\mc S}_{pq}, \psi_{pq}) \big\}_{p\leq q}
\eeqn
of derived orbifold presentations of the $\bA^{\mc P}_{pq}$-space $T_{pq}$ such that for each connected component ${\mc U}_{pq, j}\subset {\mc U}_{pq}$, one has
\beq\label{index_formula_1}
{\rm dim}_{\mb R} {\mc U}_{pq, j} - {\rm rank}_{\mb R} {\mc E}_{pq}|_{{\mc U}_{pq, j}} \equiv {\rm ind}^{\mc P}(p) - {\rm ind}^{\mc P}(q) - 1\ {\rm mod}\ 2N.
\eeq

\item A collection of chart embeddings 
\beqn
\big\{ {\bm \iota}_{prq}: C_{pr}\times C_{rq} \hookrightarrow \partial^{prq} C_{pq} \big\}_{p \leq r \leq q}
\eeqn
(with the underlying poset identification $\bA^{\mc P}_{pr}\times \bA^{\mc P}_{rq} \cong \partial^{prq} \bA^{\mc P}_{pq}$). In particular, if $\iota_{prq}: {\mc U}_{pr}\times {\mc U}_{rq} \hookrightarrow {\mc U}_{pq}$ is the associated domain embedding and $\widehat \iota_{prq}: {\mc E}_{pr} \boxplus {\mc E}_{rq} \hookrightarrow {\mc E}_{pq}$ is the associated bundle embedding, then the following diagram commutes.
\beq\label{comm_diag_32}
\vcenter{ \xymatrix{  {\mc E}_{pr}\boxplus {\mc E}_{rq} \ar[rr]^{\widehat \iota_{prq}} \ar[d] & & {\mc E}_{pq}\ar[d]  \\
            {\mc U}_{pr}\times {\mc U}_{rq} \ar[rr]_{\iota_{prq}}  \ar@/^1.0pc/@[][u]^{{\mc S}_{pr}\times {\mc S}_{rq}} & & {\mc U}_{pq} \ar@/_1.0pc/@[][u]_{{\mc S}_{pq}}}   }
\eeq
\end{enumerate}
These objects need to satisfy the following conditions. 
\begin{enumerate}

\item[(A)] For $p=q$, the space ${\mc U}_{pp}$ is a singleton with trivial isotropy and ${\mc E}_{pp} = \{0\}$ is the trivial bundle.%, ${\mc S}_{pp} = 0$, and $\psi_{pp}: {\mc U}_{pp} \to T_{pp}^{\mc P}$ is the only map.

\item[(B)] The chart embeddings satisfy the associativity. More precisely, whenever $p< r < s < q$, the following diagram commutes.
\beqn
\xymatrix{ &   C_{pr}\times C_{rs} \times C_{rq}  \ar[rd]^-{{\rm id}\times {\bm \iota}_{rsq}} \ar[ld]_-{{\bm \iota}_{prs}\times {\rm id}}  &  \\
            \partial^{prs} C_{ps} \times C_{rq}  \ar[rd]_{{\bm \iota}_{psq}} & &   C_{pr}\times \partial^{rsq} C_{rq} \ar[ld]^{{\bm \iota}_{prq}} \\
            &     \partial^{prsq} C_{pq} & }
\eeqn

\item[(C)] For each $pq$ and all $\alpha = pr_1 \cdots r_l q \in \bA_{pq}^{\mc P}$, define 
\beqn
C_\alpha:= C_{pr_1}\times \cdots \times C_{r_l q}.
\eeqn
Then condition (B) implies that for each pair of elements $\alpha \leq \beta$ in $\bA_{pq}^{\mc P}$, there is a well-defined chart embedding 
\beqn
{\bm \iota}_{\beta\alpha}: C_\alpha \hookrightarrow \partial^\alpha C_\beta.
\eeqn
Then we require that the collection $\{C_\alpha\}_{\alpha \in  \bA^{\mc P}_{pq}}$ of derived orbifold presentations and the collection $\{ {\bm \iota}_{\beta\alpha}\}_{\alpha \leq \beta}$ of chart embeddings constitute a system derived orbifold presentations of $T_{pq}$.

\item[(D)] The strict $\Pi$-equivariance condition: for any $a \in \Pi$ and $p<q$, there is an isomorphism between derived orbifold charts (in the obvious sense) $\tilde{\phi}_a : C_{pq} \to C_{a \cdot p \ a \cdot q}$ satisfying $\tilde{\phi}_{a_1 \cdot a_2} = \tilde{\phi}_{a_1} \circ \tilde{\phi}_{a_2}$, and $\tilde{\phi}_{a}$ restricts to $\phi_a$ along the zero locus ${\mc S}_{pq}^{-1}(0)$ to the map $\phi_a$ from Definition \ref{def:flow-cat}. Furthermore, $\tilde{\phi}_0 = {\rm Id}$ for $a = 0 \in \Pi$ should be the identity map.
\end{enumerate}
\end{defn}

Now consider derived orbifold lifts of bimodules.

\begin{defn}\label{defn:bimod-lift}
Let $M$ be a flow bimodule from a flow category $T^{\mc P}$ to $T^{{\mc P}'}$ as in Definition \ref{defn:flow-bimod}. Suppose $T^{\mc P}$ resp. $T^{{\mc P}'}$ is endowed with a derived orbifold lift 
\beqn
{\mf D}^{\mc P} = \Big( \big\{ C^{\mc P}_{pq} = ({\mc U}^{\mc P}_{pq}, {\mc E}^{\mc P}_{pq}, {\mc S}^{\mc P}_{pq}, \psi^{\mc P}_{pq}) \big\}_{p< q}, \big\{{\bm \iota}_{\beta \alpha}^{\mc P} \big\}_{\alpha \leq \beta} \Big) \text{    resp. } 
\eeqn
\beqn
{\mf D}^{{\mc P}'} = \Big( \big\{ C^{{\mc P}'}_{p' q'} = ({\mc U}^{{\mc P}'}_{p' q'}, {\mc E}^{{\mc P}'}_{p' q'}, {\mc S}^{{\mc P}'}_{p' q'}, \psi^{{\mc P}'}_{p' q'}) \big\}_{p' < q'}, \big\{{\bm \iota}_{\beta' \alpha'}^{{\mc P}'}\big\}_{\alpha' \leq \beta'}\Big).
\eeqn
A \textbf{derived orbifold lift} of $M$ \textbf{compatible with (or extending)} ${\mf D}^{\mc P}$ and ${\mf D}^{{\mc P}'}$ consists of the following objects.
\begin{enumerate}
    \item A collection
    \beqn
    \big\{ C^{M}_{pp'} = ({\mc U}^{M}_{pp'}, {\mc E}^{M}_{pp'}, {\mc S}^{M}_{pp'}, \psi^{M}_{pp'}) \big\}_{p\in {\mc P},p'\in {\mc P}'}
    \eeqn
    of derived orbifold presentations of $M_{pp'}$ for $M_{pp'}$ as an $\bA_{pp'}$-space such that for each connected component ${\mc U}_{pp', j}^M \subset {\mc U}_{pp'}^M$, one has 
    \beq\label{index_formula_2}
    {\rm dim}_{\mb R} {\mc U}_{pp', j}^M - {\rm rank}_{\mb R} {\mc E}_{pp'}^M|_{{\mc U}_{pp', j}^M} \equiv {\rm ind}^{\mc P}(p) - {\rm ind}^{{\mc P}'}(p')\ {\rm mod}\ 2N.
    \eeq
    
    \item Given $p \in {\mc P}$ and $p' \in {{\mc P}'}$ for $M_{pp'} \neq \emptyset$, for any $p \leq q$, a chart embedding
    $$ {\bm \iota}^{M}_{pqp'}: C^{\mc P}_{pq} \times C^{M}_{qp'} \hookrightarrow \partial^{pqp'}C^{M}_{pp'}, $$
    and for any $q' \leq p'$, a chart embedding
    $$ {\bm \iota}^{M}_{p q' p'}: C^{M}_{p q'} \times C^{{\mc P}'}_{q' p'} \hookrightarrow \partial^{pq' p'}C^{M}_{pp'}. $$
    The precise meaning can be spelled out as in Equation \eqref{comm_diag_32}.
\end{enumerate}
These objects are required to satisfy the following conditions.
\begin{enumerate}
    \item[(A)] For $p=q$, the chart embedding ${\bm \iota}_{pqp'}^{M}$ is the identity map after identifying $C^{\mc P}_{pq}$ with the trivial chart for the singleton. Similarly, for $q'=p'$, the chart embedding ${\bm \iota}_{p q' p'}^{M}$ is also the identity map.
    \item[(B)] The chart embeddings satisfy the associativity. Namely, given $M_{pp'} \neq \emptyset$, the following three diagrams commute if the relevant topological spaces are nonempty.
    
    For $p < q_1 < q_2$, we have
    \beqn
\xymatrix{ &   C^{\mc P}_{p q_1}\times C^{\mc P}_{q_1 q_2} \times C^M_{q_2 p'}  \ar[rd]^-{{\rm id}\times {\bm \iota}^{M}_{q_1 q_2 p'}} \ar[ld]_-{{\bm \iota}^{\mc P}_{p q_1 q_2}\times {\rm id}}  &  \\
            \partial^{pq_1 q_2} C^{\mc P}_{pq_2} \times C^M_{q_2 p'}  \ar[rd]_{{\bm \iota}^M_{p q_2 p'}} & &   C^{\mc P}_{p q_1}\times \partial^{q_1 q_2 p'} C^M_{q_1 p'} \ar[ld]^{{\bm \iota}^M_{p q_1 p'}} \\
            &     \partial^{p q_1 q_2 p'} C^M_{pp'}. & }
\eeqn
   For $q_2' < q_1' < p'$, we have
       \beqn
\xymatrix{ &   C^{M}_{p q_2'}\times C^{{\mc P}'}_{q_2' q_1'} \times C^{{\mc P}'}_{q_1' p'}  \ar[rd]^-{{\rm id}\times {\bm \iota}^{{\mc P}'}_{q_1' q_2' p'}} \ar[ld]_-{{\bm \iota}^{M}_{p q_2' q_1'}\times {\rm id}}  &  \\
            \partial^{pq_2' q_1'} C^{M}_{p q_1'} \times C^{{\mc P}'}_{q_1' p'}  \ar[rd]_{{\bm \iota}^M_{p q_1' p'}} & &   C^{M}_{pq_2'}\times \partial^{q_2' q_1' p'} C^{{\mc P}'}_{q_2' p'} \ar[ld]^{{\bm \iota}^M_{p q_2' p'}} \\
            &     \partial^{p q_2' q_1' p'} C^M_{pp'}. & }
\eeqn
Finally, for $p<q$ and $q' < p'$, we have
       \beqn
\xymatrix{ &   C^{{\mc P}}_{p q}\times C^M_{q q'} \times C^{{\mc P}'}_{q' p'}  \ar[rd]^-{{\rm id}\times {\bm \iota}^M_{qq' p'}} \ar[ld]_-{{\bm \iota}^{M}_{p q q'}\times {\rm id}}  &  \\
            \partial^{pq q'} C^{M}_{p q'} \times C^{{\mc P}'}_{q' p'}  \ar[rd]_{{\bm \iota}^M_{p q' p'}} & &   C^{{\mc P}}_{pq}\times \partial^{qq' p'} C^M_{q p'} \ar[ld]^{{\bm \iota}^M_{p q p'}} \\
            &     \partial^{p q q' p'} C^M_{pp'}. & }
\eeqn

\item[(C)] For each $p, p'$ and all $\alpha = pr_1 \cdots r_l r_l' \cdots r_1' p'$, define
\beqn
C_{\alpha}^M = C^{\mc P}_{pr_1} \times \cdots \times C^M_{r_l r_l'} \times \cdots \times C^{{\mc P}'}_{r_1' p'}.
\eeqn
As in Definition \ref{defn:flow-lift}, for each pair $\alpha \leq \beta$ in $\bA_{pp'}$, there is a well-defined chart embedding
$$ {\bm \iota}_{\beta \alpha}^{M}: C_{\alpha}^M \hookrightarrow C_{\beta}^M. $$
It is required that the derived orbifold presentations $\{ C_{\alpha}^M \}_{\alpha \in \bA_{pp'}}$ and the collection of chart embeddings $\{ {\bm \iota}_{\beta \alpha} \}_{\alpha \leq \beta}$ constitute a system of derived orbifold presentations of $M_{pp'}$.

\item[(D)] The strict $\Pi$-equivariance condition: for any $a \in \Pi$ and $M_{pp'} \neq \emptyset$, there is an isomorphism between derived orbifold charts $\tilde{\phi}^{M}_{a}: C^M_{pp'} \to C^M_{a \cdot p \ a \cdot p'}$ satisfying $\tilde{\phi}^{M}_{a_1 \cdot a_2} = \tilde{\phi}^{M}_{a_1} \circ \tilde{\phi}^{M}_{a_2}$, and $\tilde{\phi}^{M}_{a}$ restricts to the map $\phi^M_a$ from Definition \ref{defn:flow-bimod} along the zero locus $({\mc S}^{M}_{pp'})^{-1}(0)$. Moreover, $\tilde{\phi}^{M}_{id}$ for $a = id \in \Pi$ is the identity map. Moreover, the left and right actions of the charts from $T^{\mc P}$ and $T^{\mc P'}$ on the charts of $M$ should be $\Pi$-equivariant.
\end{enumerate}
\end{defn}

\subsection{Additional structures}\label{subsection34}

In order to construct a coherent system of FOP perturbations, we need more structures on the derived orbifold lifts. They are the {\it collar structure}, {\it scaffolding}, and {\it straightening}.

\subsubsection{Collar structure}

\begin{defn}\label{defn:collar}
\begin{enumerate}
    \item A {\bf collar structure} on an $\bA$-stratified derived orbifold chart $({\mc U}, {\mc E}, {\mc S})$ is a collection of open chart embeddings
\beqn
\vcenter{\xymatrix{ {\mc E}|_{\partial^\alpha {\mc U}} \times [0, \epsilon)^{\bF_\alpha} \ar[d] \ar[rr]^-{\widehat \theta_\alpha^{\rm collar}} & & {\mc E} \ar[d] \\
\partial^\alpha {\mc U} \times [0, \epsilon)^{\bF_\alpha} \ar[rr]_-{\theta_\alpha^{\rm collar}} & & {\mc U} } }\ \ \forall \alpha \in \bA
\eeqn
(where $\bF_\alpha \subset \bA$ is the set of adjacent faces to $\alpha$, see Definition \ref{defn_adjacent_face} and Notation \ref{notation314}) covering the poset map 
$$
\partial^{\alpha}\bA \times \bA^{(|{\mb F}_{\alpha}|)} \to \bA
$$
satisfying the following conditions. 
\begin{enumerate}

    \item For a pair of stratum $\alpha \leq \beta$, $v \in {\mc E}|_{\partial^\alpha {\mc U}}$, $t_\beta \in [0, \epsilon)^{\bF_\beta}$ and $t_{\beta\alpha}\in [0, \epsilon)^{\bF_\alpha \setminus \bF_\beta}$ (so that $(t_\beta, t_{\beta\alpha})$ is identified with an element of $[0, \epsilon)^{\bF_\alpha}$), one has 
    \beqn
    \widehat\theta_\alpha^{\rm collar}(v, t_\beta, t_{\beta\alpha}) = \widehat\theta_\beta^{\rm collar}(\widehat\theta_\alpha^{\rm collar}(v, 0, t_{\beta\alpha}), t_\beta).
    \eeqn
    
    \item For each $\alpha \in \bA$, $x\in \partial^\alpha {\mc U}$ and $t_\alpha \in [0, \epsilon)^{\bF_\alpha}$, one has 
    \beqn
    \widehat\theta_\alpha^{\rm collar} ({\mc S}(x), t_\alpha) = {\mc S}( \theta_\alpha^{\rm collar} ( x, t_\alpha) ).
    \eeqn
\end{enumerate}

\item A collar structure on a system of D-chart presentations $((C_\alpha)_{\alpha \in \bA}, ({\bm \iota}_{\beta\alpha})_{\alpha \leq \beta})$ of an $\bA$-stratified space $X$ (see Definition \ref{defn:d-chart-pre}) consists of collar structures on all derived orbifold charts $C_\alpha = ({\mc U}_\alpha, {\mc E}_\alpha, {\mc S}_\alpha)$ for all $\alpha \in \bA$. Namely, for each $\beta\in \bA$, a collection of open chart embedding 
\beqn
\xymatrix{ {\mc E}_\beta|_{\partial^\alpha {\mc U}_\beta} \times [0, \epsilon)^{\bF_\alpha \setminus \bF_\beta} \ar[rr]^-{\widehat\theta_{\beta\alpha}^{\rm collar}} \ar[d] & & {\mc E}_\beta \ar[d] \\
           \partial^\alpha {\mc U}_\beta \times [0, \epsilon)^{\bF_\alpha \setminus \bF_\beta} \ar[rr]_-{\theta_{\beta\alpha}^{\rm collar}} & & {\mc U}_\beta }
\eeqn
for all $\alpha < \beta$ which satisfy conditions in item (1) above. We require the following condition is satisfied. For each triple $\alpha < \beta < \gamma$
\beqn
\widehat \iota_{\gamma\beta} \Big( \widehat \theta_{\beta\alpha}^{\rm collar}( v_\beta, t_{\beta\alpha}) \Big) = \widehat \theta_{\gamma\alpha}^{\rm collar} \Big( \widehat\iota_{\gamma\beta}( v_\beta), t_{\beta\alpha}, 0_{\gamma\beta} \Big)
\eeqn
for all $v_\beta \in {\mc E}_\beta|_{\partial^\alpha {\mc U}_\beta}$ and $t_{\beta\alpha}\in [0, \epsilon)^{\bF_\alpha \setminus \bF_\beta}$. Here $0_{\gamma\beta} = (0, \ldots, 0)\in [0, \epsilon)^{\bF_\beta\setminus \bF_\gamma}$ and hence $(t_{\beta\alpha}, 0_{\gamma\beta})$ is regarded as a point of $[0, \epsilon)^{\bF_\alpha \setminus \bF_\gamma}$. 

\item A collar structure on a derived orbifold lift of a flow category (using the same notations as in Definition \ref{defn:flow-lift}), is a collection of collar structures on the chart $C_{pq}$ such that the following is satisfied. 
\begin{enumerate}

\item For each $\alpha = pr_1 \cdots r_l q \in \bA_{pq}^{\mc P}$, the collar structures on $C_{pr_1}, \ldots, C_{r_l q}$ define a product collar structure on the product chart $C_\alpha$. Then the collection of collar structures on $C_\alpha$ is a collar structure on the derived orbifold presentation of $T_{pq}$ as defined in above item (2).

\item The collar structure is strictly $\Pi$-equivariant. 
\end{enumerate}

\item Suppose ${\mf D}_{{\mc P}{\mc P}'}$ is a derived orbifold lift of a flow bimodule $M$ from $T^{\mc P}$ to $T^{{\mc P}'}$ which extends ${\mf D}^{\mc P}$ and ${\mf D}^{{\mc P}'}$. Suppose ${\mf D}^{\mc P}$ and ${\mf D}^{{\mc P}'}$ are equipped with collar structures. Then a collar structure on ${\mf D}_{{\mc P} {\mc P}'}$ which extends the collar structure on ${\mf D}^{\mc P}$ and ${\mf D}^{{\mc P}'}$ is a collection of collar structures on all the derived orbifold charts $({\mc U}_{pp'}^M, {\mc E}_{pp'}^M, {\mc S}_{pp'}^M, \psi_{pp'}^M)$ of the form 
\beqn
\vcenter{ \xymatrix{ {\mc E}^M_{pp'}|_{\partial^\alpha {\mc U}^M_{pp'}} \times [0, \epsilon)^{\bF_\alpha} \ar[d] \ar[rr]^-{\widehat\theta_{pp', \alpha}^{\rm collar}} & & {\mc E}_{pp'}^M \ar[d] \\
\partial^\alpha {\mc U}_{pp'}^M \times [0, \epsilon)^{\bF_\alpha} \ar[rr]_-{\theta_{pp', \alpha}^{\rm collar}} & & {\mc U}^M_{pp'} } }\ \ \forall p\in {\mc P},\ p' \in {\mc P'}, \ \alpha \in \bA_{pp'}
\eeqn
which satisfy similar compatibility conditions as the case of flow categories using the factorizations of boundary strata. We omit the details.

\end{enumerate}
\end{defn}

\begin{rem}
Although the description of collar structures seems to be very complicated, they are constructed using the outer-collaring method which implies the required conditions almost automatically. 
\end{rem}

\subsubsection{Scaffolding}

\begin{defn}\label{defn_scaffolding}
A {\bf scaffolding} of a system of D-chart presentations $((C_\alpha)_{\alpha \in \bA}, ({\bm \iota}_{\beta\alpha})_{\alpha \leq \beta})$ of an $\bA$-space $X$ is a collection of data
\beqn
\big( {\mc F}_{\beta\alpha}, {\bm \theta}_{\beta\alpha} \big)_{\alpha \leq \beta}
\eeqn
where for each pair $\alpha \leq \beta$
\begin{enumerate}
\item the {\bf difference bundle} ${\mc F}_{\beta\alpha} \to {\mc U}_\alpha$ is an orbifold vector bundle ${\mc F}_{\beta\alpha}\to {\mc U}_\alpha$. In notation, when $\bA = \bA_{pq}^{\mc P}$ resp. $\bA_{pp'}$ and $\beta$ is the maximal element $pq$ resp. $pp'$, denote ${\mc F}_{\beta\alpha}$ by ${\mc F}_{pq, \alpha}$ resp. ${\mc F}_{pp', \alpha}$. 

\item the {\bf stabilization map} ${\bm \theta}_{\beta\alpha}$ is a germ equivalence
\begin{equation}\label{eqn:germ-equiv}
{\bm \theta}_{\beta\alpha} = (\theta_{\beta\alpha}, \widehat\theta_{\beta\alpha}): {\rm Stab}_{{\mc F}_{\beta\alpha}}( C_\alpha) \simeq \partial^\alpha C_\beta
\end{equation}
which extends the chart embedding ${\bm \iota}_{\beta\alpha}: C_\alpha \hookrightarrow \partial^\alpha C_\beta$. This germ equivalence induces a projection map 
\beqn
\pi_{\beta\alpha}: \partial^\alpha C_\beta \to C_\alpha
\eeqn
as well as a bundle splitting
\beq\label{scaffolding_splitting}
{\mc E}_\beta|_{{\mc U}_\alpha} = {\mc E}_\alpha \oplus {\mc F}_{\beta\alpha}
\eeq
where ${\mc E}_\beta|_{{\mc U}_\alpha} = {\bm \iota}_{\beta \alpha}^* {\mc E}_{\beta}$ and a bundle isomorphism
\beq\label{vartheta}
\vartheta_{\beta\alpha}: \pi_{\beta\alpha}^* ( {\mc E}_\beta|_{{\mc U}_\alpha} ) \cong {\mc E}_\beta|_{\partial^\alpha {\mc U}_\beta}.
\eeq
\end{enumerate}
These objects need to satisfy the following conditions. 
\begin{enumerate}
    \item[(A)] For any triple of stratum $\alpha \leq \beta \leq \gamma$, as subbundles of ${\mc E}_\gamma|_{{\mc U}_\alpha}$ there holds
    \begin{equation}\label{eqn:sca-split}
    {\mc F}_{\gamma\alpha} = {\mc F}_{\beta\alpha} \oplus {\mc F}_{\gamma\beta}|_{{\mc U}_\alpha} .
    \end{equation}
    
    \item[(B)] The stabilization map preserves stratum. More precisely, the following diagram commutes
    \beqn
    \vcenter{ \xymatrix{  {\rm Stab}_{{\mc F}_{\beta\alpha}}(C_\alpha) \ar[r]^-{{\bm \theta}_{\beta\alpha}} \ar[d] & \partial^\alpha C_\beta\ar[d]^{{\bm \iota}_{\gamma\beta}} \\
                          {\rm Stab}_{{\mc F}_{\gamma\alpha}} (C_\alpha) \ar[r]_-{{\bm \theta}_{\gamma\alpha}}    &  \partial^\alpha C_\gamma }  }.
    \eeqn
    It follows that 
    \beq\label{eqn311}
    \pi_{\gamma\alpha} \circ (\iota_{\gamma\beta}|_{\partial^\alpha U_\beta}) = \pi_{\beta\alpha}.
    \eeq
    
    \item[(C)] The bundle isomorphism \eqref{vartheta} preserves stratum. More precisely, for any triple of strata $\alpha \leq \beta \leq \gamma$, consider the following diagram
    \beqn
    \vcenter{ \xymatrix{    \pi_{\gamma\alpha}^* \left( {\mc E}_\gamma|_{{\mc U}_\alpha} \right)|_{\partial^\alpha {\mc U}_\beta} \ar[rrr]^{\vartheta_{\gamma\alpha}|_{\partial^\alpha {\mc U}_\beta}} \ar@{=}[d] & & &   {\mc E}_\gamma|_{\partial^\alpha {\mc U}_\beta} \ar@{=}[dd] \\
                             \pi_{\beta\alpha}^* \left( {\mc E}_\gamma|_{{\mc U}_\alpha} \right)  \ar@{=}[d] & & & \\
                             \pi_{\beta\alpha}^* \left( {\mc E}_\beta|_{{\mc U}_\alpha} \oplus {\mc F}_{\gamma\beta}|_{{\mc U}_\alpha} \right) \ar[rrr]_{ \vartheta_{\beta\alpha} \oplus \vartheta_{\gamma\beta\alpha}}  & &  &  {\mc E}_\beta|_{\partial^\alpha {\mc U}_\beta} \oplus {\mc F}_{\gamma\beta}|_{\partial^\alpha {\mc U}_\beta}    }
                        }.
    \eeqn
    We explain the notations here. The first vertical equal arrow on the left is due to \eqref{eqn311} and the second one is due to \eqref{scaffolding_splitting}. Here the requirement is that, we require 
    \beqn
    \vartheta_{\gamma\alpha}\left( \pi_{\beta\alpha}^* ( {\mc F}_{\gamma\beta}|_{{\mc U}_\alpha} )\right) = {\mc F}_{\gamma\beta}|_{\partial^\alpha {\mc U}_\beta}
    \eeqn 
    and the restriction to $\vartheta_{\gamma\alpha}$ to $\pi_{\beta\alpha}^* ( {\mc F}_{\gamma\beta}|_{{\mc U}_\alpha} )$ is equal to a linear isomorphism $   \vartheta_{\gamma\beta\alpha}$ (which is in the above commutative diagram). And we require that the above diagram commutes. 
    
    \item[(D)] The stabilization maps satisfy the cocycle condition. Namely, for each triple of strata $\alpha \leq \beta \leq \gamma$, the following diagram commutes.
    \beq\label{comm_diag_36}
    \vcenter{ \xymatrix{ {\rm Stab}_{{\mc F}_{\gamma\alpha}}(C_\alpha) \ar@{=}[r] \ar[dd]_-{{\bm \theta}_{\gamma\alpha}} & {\rm Stab}_{{\mc F}_{\gamma\beta}|_{{\mc U}_\alpha} \oplus {\mc F}_{\beta\alpha}}(C_\alpha) \ar[d]^-{{\bm \theta}_{\beta\alpha}} \\
     & {\rm Stab}_{\pi_{\beta\alpha}^* ({\mc F}_{\gamma\beta}|_{{\mc U}_\alpha})}(\partial^\alpha C_\beta) \ar[d]^{\vartheta_{\gamma\alpha}} \\
                \partial^\alpha C_\gamma  &     {\rm Stab}_{{\mc F}_{\gamma\beta}}( \partial^\alpha C_\beta) \ar[l]^-{{\bm \theta}_{\gamma\beta}} }   }.
    \eeq
    Here the ``$=$'' arrow is induced from the identity ${\mc F}_{\gamma\alpha} = {\mc F}_{\gamma\beta}|_{{\mc U}_\alpha} \oplus {\mc F}_{\beta\alpha}$.
\end{enumerate}
\end{defn}

Now consider a derived orbifold lift of a flow category. 

\begin{defn}\label{scaffolding_2}
A {\bf scaffolding} of a derived orbifold lift of a flow category $T^{\mc P}$ consists of a collection of scaffoldings for the induced system of derived orbifold chart presentations of $T_{pq}$
\beqn
\Big( \big( {\mc F}_{\beta\alpha}, {\bm\theta}_{\beta\alpha} \big)_{\alpha \leq \beta}\Big)_{p < q}
\eeqn
satisfying 
 \begin{enumerate}

\item Suppose $pr_1 \cdots r_l q = \alpha \leq \beta = ps_1 \cdots s_m q$. Denote $\alpha_0 = pr_1 \cdots s_1$, $\ldots$, $\alpha_m = s_m \cdots r_l q$. Then as subbundles of ${\mc E}_\beta|_{{\mc U}_\alpha}$, in view of the identification ${\mc U}_\alpha = {\mc U}_{\alpha_0}\times \cdots \times {\mc U}_{\alpha_m}$, one has
\beqn
{\mc F}_{\beta\alpha} = {\mc F}_{ps_1, \alpha_0} \boxplus \cdots \boxplus {\mc F}_{s_m q, \alpha_m}.
\eeqn

\item With respect to the last identity, one has (as germs of maps) 
\beqn
{\bm \theta}_{\beta\alpha} = {\bm\theta}_{ps_1, \alpha_0}\times \cdots \times {\bm \theta}_{s_m q, \alpha_m}.
\eeqn
\end{enumerate}
\end{defn}

\begin{defn}\label{defn:bimod-sca}
Given two flow categories $T^{\mc P}$ and $T^{{\mc P}'}$ endowed with derived orbifold lifts ${\mf D}^{\mc P}$, ${\mf D}^{{\mc P}'}$ respectively, let $M$ be a flow bimodule from $T^{\mc P}$ to $T^{{\mc P}'}$ endowed with a compatible derived orbifold lift ${\mf D}^M$. Suppose ${\mf D}^{\mc P}$ and ${\mf D}^{{\mc P}'}$ come with scaffoldings
\beqn
\big(  \big\{ {\mc F}^{\mc P}_{\beta \alpha}, {\bm\theta}^{\mc P}_{\beta \alpha} \big\}_{\alpha \leq \beta} \big)_{p<q}, \text{        }  \big(  \big\{ {\mc F}^{{\mc P}'}_{\beta' \alpha'}, {\bm\theta}^{{\mc P}'}_{\beta' \alpha'} \big\}_{\alpha' \leq \beta'} \big)_{p'<q'}.
\eeqn
A \textbf{scaffolding} of such a derived orbifold lift \textbf{compatible with} the given scaffoldings is given by a collection of scaffoldings for the induced system of derived orbifold chart presentations
\beqn
\big( \big\{ {\mc F}_{\tilde \beta \tilde \alpha}, {\bm\theta}_{\tilde \beta\tilde\alpha} \big\}_{\tilde\alpha \leq \tilde \beta} \big)_{\tilde\alpha, \tilde\beta \in \bA_{pp'}^M}
\eeqn
satisfying similar conditions as in Definition \ref{scaffolding_2} using the factorization of the boundary strata (see Definition \ref{defn:bimod-lift} (C)).
\end{defn}

\subsubsection{Straightening}

As we will perform FOP perturbations over stratified charts, we need to consider straightenings in the stratified situation.

\begin{defn}[Straightenings]\label{defn:straightening_2}
\begin{enumerate}

\item Let $({\mc U}, {\mc E}, {\mc S})$ be an $\bA$-stratified derived orbifold chart. A {\bf straightening} consists of a Riemannian metric on ${\mc U}$ and a connection on ${\mc E}$
such that for each $\alpha \in \bA$, the restriction of the metric and the connection onto ${\mc U}_\alpha$ (which is itself a smooth orbifold) is a straightening of $({\mc U}_\alpha, {\mc E}|_{{\mc U}_\alpha}, {\mc S}|_{{\mc U}_\alpha})$ (see Definition \ref{defn21}).

\item Consider a system of derived orbifold presentations on an $\bA$-stratified space $X$. A {\bf straightening} consists of a collection of straightenings of $C_\alpha$ for all $\alpha \in \bA$ such that for each pair $\alpha \leq \beta$ (with associated chart embedding ${\bm \iota}_{\beta\alpha} = (\iota_{\beta\alpha}, \widehat \iota_{\beta\alpha}): C_\alpha \hookrightarrow \partial^\alpha C_\beta$), the domain embedding $\iota_{\beta\alpha}$ is isometric and the bundle embedding preserves the connection, i.e., the image of the bundle embedding is preserved by the connection $\nabla^{{\mc E}_\beta}$ and the restriction of $\nabla^{{\mc E}_\alpha}$ to the image of the bundle embedding coincides with $\nabla^{{\mc E}_\alpha}$.

\item Consider a derived orbifold lift ${\mf D}^{\mc P}$ of a topological flow category $T^{\mc P}$ over ${\mc P}$. A {\bf straightening} on ${\mf D}^{\mc P}$ consists of a straightening on all charts $C_{pq}$ satisfying the following condition. Fix $p< q$. For each $\alpha = pr_1 \cdots r_l q$, there is a product straightening on the associated chart $C_\alpha = C_{pr_1}\times \cdots \times C_{r_l q}$. Then we require that the collection of these straightenings for all $\alpha \in \bA_{pq}^{\mc P}$ is a straightening on the system of derived orbifold presentations of $T_{pq}^{\mc P}$.

\item It is straightforward to define the notion of straightenings on a derived orbifold lift of a flow bimodule $M$ compatible with existing straightenings on the lifts of the two flow categories $T^{\mc P}$ and $T^{\mc P'}$ with derived orbifold lifts ${\mf D}^{\mc P}$ and ${\mf D}^{\mc P'}$. Namely, it consists of a straightening on all charts $C^M_{pp'}$ such that: for $\alpha = pq_1 \cdots q_k q^{\prime}_{k'} \cdots q^{\prime}_1 p' \in \bA_{pp'}$ with $M_\alpha \neq \emptyset$, there is a product straightening on the product chart $C^{\mc P}_{p q_1} \times \cdots \times C^M_{q_k q^{\prime}_{k'}} \times \cdots \times C^{\mc P'}_{q^{\prime}_1 p'}$; we require that the collection of these straightenings for all $\alpha \in \bA_{pp'}$ is a straightening on the system of derived orbifold presentations of $M_{pp'}$.
\end{enumerate}
\end{defn}

\subsubsection{Compatibility}

We need the three kinds of additional structures (collar structure, scaffolding, and straightening) to be compatible in a certain sense. These compatibility conditions are necessary for the inductive construction of FOP perturbations. For example, once a perturbation is chosen on all boundary strata, it can be canonically extended to a neighborhood of the boundary using the collar structure; meanwhile, there is another automatic extension using the scaffolding. Hence a compatibility requirement between these two structures is necessary. We also need the automatic extensions remain in the class of FOP perturbations, which further requires that the straightening is compatible with the collar structure and the scaffolding. 

\begin{rem}
In practice the three kinds of structures are constructed in different ways. The collar structure is constructed by the outer-collaring construction, which is very easy and straightforward. Then (see Subsection \ref{subsection59}) we construct scaffolding in the topological category. The scaffolding will play a role in the smoothing process. After smoothing, the topological charts become smooth. Then finally we construct the compatible straightenings.
\end{rem}

\begin{rem}
In the following, we just define the notions of compatibility of a derived orbifold presentation. The relevant compatibility requirements for flow categories and flow bimodules are then reduced to the corresponding requirements for the induced derived orbifold presentations of all relevant topological spaces.
\end{rem}

\begin{defn}\label{defn329}
Consider a system of derived orbifold presentations $\{ C_\alpha \}_{\alpha \in \bA}$ of an $\bA$-stratified space $X$. Suppose it is equipped with a collar structure and a scaffolding. We say that they are compatible if the following conditions are satisfied.
\begin{enumerate}
    \item Given a triple $\alpha < \beta < \gamma$, we require that the following diagram commutes. 
    \beqn
    \xymatrix{   {\mc E}_\gamma|_{\partial^\alpha {\mc U}_\gamma} \times [0, \epsilon)^{\bF_\alpha \setminus \bF_\beta} \times \{0\}^{\bF_{\beta} \setminus \bF_{\gamma}} \ar[rr]^-{\widehat\theta_{\gamma\alpha}^{\rm collar}} & & {\mc E}_\gamma |_{\partial^\beta {\mc U}_\gamma} \\
                        {\mc E}_\gamma|_{\partial^\alpha {\mc U}_\beta} \times [0, \epsilon)^{\bF_\alpha \setminus \bF_\beta} \times \{0\}^{\bF_{\beta} \setminus \bF_{\gamma}} \ar[u] \ar[rr]^-{\widehat\theta_{\gamma\alpha}^{\rm collar}}  & & {\mc E}_\gamma|_{{\mc U}_\beta}   \ar[u] \\
                        {\mc F}_{\gamma\beta}|_{\partial^\alpha {\mc U}_\beta} \times [0, \epsilon)^{\bF_\alpha \setminus \bF_\beta} \times \{0\}^{\bF_{\beta} \setminus \bF_{\gamma}} \ar[rr]^-{\widehat \theta_{\gamma\alpha}^{\rm collar}} \ar[u] & &  {\mc F}_{\gamma\beta} \ar[u] }
    \eeqn
    Here the commutativity of the upper square is a requirement of the collar structure. Moreover, the bottom horizontal arrow uses the splitting \eqref{scaffolding_splitting} and $\widehat \theta_{\gamma\alpha}^{\rm collar}$ is the restriction to the ${\mc F}_{\gamma \beta}$ component. 
    
    \item The stablization maps also respect the collars. More precisely, for $\alpha < \beta < \gamma$, we require that
    \beqn
    \xymatrix{  \Big( {\rm Stab}_{{\mc F}_{\gamma\beta}|_{\partial^\alpha {\mc U}_\beta}} (\partial^\alpha C_\beta)\Big)  \times [0, \epsilon)^{\bF_\alpha \setminus \bF_\beta} \ar[rr] \ar[d] & &   {\rm Stab}_{{\mc F}_{\gamma\beta}}(C_\beta) \ar[d] \\
                \partial^\alpha C_\gamma \times [0, \epsilon)^{\bF_\alpha \setminus \bF_\beta} \ar[rr]^-{{\bm \theta}_{\gamma\alpha}^{\rm collar}} & & \partial^\beta C_\gamma            }.
    \eeqn
Here the horizontal arrows are induced from the collar structure and the vertical arrows are restrictions of the open chart embeddings associated to the stabilization and scaffolding.

\end{enumerate}
\end{defn}

Finally we spell out the meaning of compatible straightenings. 

\begin{defn}\label{defn330}
Consider a system of D-chart presentations equipped with a collar structure and a scaffolding which are compatible. Then a straightening (see Definition \ref{defn:straightening_2}) is compatible with the collar structure and the scaffolding if the following conditions are satisfied.
\begin{enumerate}
    \item For each pair $\alpha< \beta$, near $\partial^\alpha {\mc U}_\beta$, using the corner coordinates induced from the map $\theta_{\beta\alpha}^{\rm collar}$, near $\partial^\alpha {\mc U}_\beta$, one has
    \begin{equation}\label{eqn:metric-collar}
    g_\beta = g_\beta|_{\partial^\alpha {\mc U}_\beta} + \sum_{i \in \bF_\alpha \setminus \bF_\beta} dt_i^2.
    \end{equation}
    Moreover, with respect to the bundle isomorphism $\widehat \theta_{\beta\alpha}^{\rm collar}$, near $\partial^\alpha {\mc U}_\beta$ one has 
    \beqn
    \nabla^{{\mc E}_\beta} = \left( \widehat \pi_{\beta\alpha}^{\rm collar} \right)^* \nabla^{{\mc E}_\beta}|_{\partial^\alpha {\mc U}_\beta} .
    \eeqn

    \item For each pair $\alpha \leq \beta$, the domain embedding 
    \beqn
    \iota_{\beta\alpha}: {\mc U}_\alpha \hookrightarrow \partial^\alpha {\mc U}_\beta
    \eeqn
    is isometric and totally geodesic. The normal bundle of this embedding has an induced metric and metric connection. Via the stabilization map ${\bm \theta}_{\beta\alpha}$ which embeds ${\mc F}_{\beta\alpha}$ into ${\mc E}_\beta$, the bundle ${\mc F}_{\beta\alpha}$ then carries a metric and metric connection.
    
    \item The bundle embedding ${\mc F}_{\beta\alpha}\hookrightarrow {\mc E}_\beta$ preserves the connection. 
    
    \item The stabilization map on the domain is isometric.
    
    \item Via the stabilization map (on the bundle), the connection on ${\mc E}_\beta|_{\partial^\alpha {\mc U}_\beta}$ near the embedding image is the pullback connection, i.e., the direct sum of the pullback connections on ${\mc F}_{\beta\alpha}$ and ${\mc E}_\alpha$.

\end{enumerate}
\end{defn}

\subsection{Normal complex structure and orientations on flow categories}\label{subsection35}

In order to carry out the FOP perturbations and define counts over ${\mb Z}$, we need to introduce normal complex structures on the derived orbifold lifts and certain orientation structures.

\begin{defn}\label{defn:normal-C}
Let $T^{\mc P}$ be a flow category over ${\mc P}$ equipped with a derived orbifold lift ${\mf D}^{\mc P}$, with the collection of derived orbifold presentations 
\beqn
\big\{ C_{pq} = ({\mc U}_{pq}, {\mc E}_{pq}, {\mc S}_{pq}, \psi_{pq}) \big\}_{p\leq q}.
\eeqn
Moreover, assume that ${\mf D}^{\mc P}$ has a scaffolding
\beqn
{\mf F}^{\mc P} = \Big( \big( {\mc F}_{\beta\alpha}, {\bm\theta}_{\beta\alpha} \big)_{\alpha \leq \beta}\Big)_{p<q}.
\eeqn
Then a {\bf normal complex structure} (or normally complex lift) on the pair ${\mf D}^{\mc P}$ and ${\mf F}^{\mc P}$ consists of the following data.
\begin{enumerate}
    \item A normal complex structure on each of the derived orbifold chart $({\mc U}_{pq}, {\mc E}_{pq}, {\mc S}_{pq})$.
    \item A complex structure on each of the orbifold vector bundles ${\mc F}_{\beta \alpha} \to {\mc U}_{\alpha}$.
    \item For any pair $p, q \in {\mc P}$ with $T^{\mc P}_{pq} \neq \emptyset$ and any pair of strata $\alpha \leq \beta$ in $\bA^{\mc P}_{pq}$, the underlying map of the germ equivalence \eqref{eqn:germ-equiv}
    $$
    \theta_{\beta \alpha}: {\rm Stab}_{{\mc F}_{\beta\alpha}}( {\mc U}_\alpha) \simeq \partial^\alpha {\mc U}_\beta
    $$
    respects the normal complex structures.
    \item Under the same assumptions as the previous item, the map
    $$
    \widehat\theta_{\beta\alpha}: \pi^*_{\beta\alpha} {\mc E}_{\alpha} \oplus \pi^*_{\beta\alpha} {\mc F}_{\beta \alpha} \to {\mc E}_{\beta}|_{\partial^{\alpha}U_{\beta}}
    $$
    intertwines with the normal complex structures.
    \item The splitting of orbifold vector bundles \eqref{scaffolding_splitting} over ${\mc U}_{\alpha}$, the bundle isomorphism \eqref{vartheta}, and the splitting \eqref{eqn:sca-split} respect the normal complex structures.
    \item These (normal) complex structures admit a strict $\Pi$-action in the obvious sense.
\end{enumerate}
\end{defn}

\begin{defn}
Let $T^{\mc P}$, ${\mf D}^{\mc P}$, and ${\mf F}^{\mc P}$ be as in Definition \ref{defn:normal-C}. Suppose we are given a normal complex structure on the pair ${\mf D}^{\mc P}$ and ${\mf F}^{\mc P}$. Then a collar structure compatibile with ${\mf F}^{\mc P}$ (see Definition \ref{defn329}) is said to be compatible with the normal complex structure if the collar maps intertwine with the normal complex structures.
\end{defn}

\begin{rem}
Note that the complex structure on ${\mc F}_{\beta \alpha}$ and the normal complex structure on ${\mc U}_{\alpha}$ induce a normal complex structure on ${\rm Stab}_{{\mc F}_{\beta\alpha}}( {\mc U}_\alpha)$. Moreover, due to the behavior of orbifold charts near the boundary and corners, $\partial^\alpha {\mc U}_\beta$ indeed has a normal complex structure.
\end{rem}

\begin{rem}
We could define the notion of normal complex structure in a more intrinsic way which does not rely on the choice of scaffoldings.
\end{rem}

\begin{rem}
In our applications, the notion of normal complex structure is mostly useful for considering the space ${\mc U}_{\alpha}$, as the orbifold vector bundles ${\mc E}_{\alpha}$ and ${\mc F}_{\beta \alpha}$ would have fiberwise complex structures in practice.
\end{rem}

The normal complex structure is concerned with the tangential structure on the normal directions to strata specified by the isomorphism classes of the isotropy groups on an orbifold. In particular, an unorientable manifold has a (trivial) normal complex structure when being viewed as an orbifold. To define algebraic counts, we need to define orientation structures on the flow categories and flow bimodules, which is an extra structure beyond the normal complex structure.

\begin{defn}
Let $V$ be a vector space over ${\mb R}$. The \textbf{orientation line} $\mathfrak{o}_V$ of $V$ is defined to be the free ${\mb Z}$-module of rank $1$ given by $H_{\dim V}(V,V\setminus \{0\};{\mb Z})$ with ${\mb Z}/2$-grading the mod $2$ reduction of $\dim V$. More generally, if $(V^+, V^-)$ is a virtual vector space, its orientation line $\mathfrak{o}_{(V^+, V^-)}$ is defined to be the tensor product of orientation lines $\mathfrak{o}_{V^+} \otimes \mathfrak{o}_{V^-}^{\vee}$.
\end{defn}

\begin{defn}
An $\bA$-stratified derived orbifold chart $C = ({\mc U}, {\mc E}, {\mc S})$ is called \textbf{orientable} if both ${\mc U}$ and ${\mc E}$ are orientable. Given a connected component of an orientable derived orbifold chart, the \textbf{orientation line} $\mathfrak{o}_C$ of $({\mc U}, {\mc E}, {\mc S})$ is defined to be the orientation line of the virtual vector space
$$ (T_x {\mc U},{\mc E}_x) $$
for arbitrary $x \in {\mc U}$ lying in the top stratum of ${\mc U}$.
\end{defn}

If $C=({\mc U}, {\mc E}, {\mc S})$ is an orientable $\bA$-stratified derived orbifold chart, given an index $\alpha \in \bA$, the $\partial^{\alpha} \bA$-stratified derived orbifold chart $\partial^{\alpha}C = (\partial^{\alpha}{\mc U}, {\mc E}|_{\partial^{\alpha}{\mc U}}, {\mc S}|_{\partial^{\alpha}{\mc U}})$ is also orientable. Moreover, it is easy to see that there is a natural isomorphism
$$ \mathfrak{o}_{\partial^\alpha C} \xrightarrow{\sim} \mathfrak{o}_C \otimes (\mathfrak{o}_{\mb R}^{\vee})^{\otimes \dep(\alpha)}. $$

\begin{defn}\label{defn:orient}
Suppose $T^{\mc P}$, ${\mf D}^{\mc P}$ and ${\mf F}^{\mc P}$ are the same in Definition \ref{defn:normal-C} An \textbf{orientation} of such a normally complex derived orbifold lift is given by:
\begin{enumerate}
    \item A virtual vector space $(V_p^+, V_p^-)$ for each $p \in {\mc P}$. Define $\mathfrak{o}_{p} := \mathfrak{o}_{(V_p^+, V_p^-)}$.
    \item An isomorphism of orientation lines
    \begin{equation}\label{eqn:orient-line}
    \mathfrak{o}_{C_{pq}} \xrightarrow{\sim} \mathfrak{o}_p^{\vee} \otimes \mathfrak{o}_q^.
    \end{equation}
    \item For $p<r<q$, an isomorphism of orientation lines
    \begin{equation}\label{eqn:orient-coh}
    \mathfrak{o}_{C_{pr}} \otimes \mathfrak{o}_{C_{rq}} \otimes (\mathfrak{o}_{\mc F_{prq, pq}} \otimes \mathfrak{o}_{\mc F_{prq, pq}}^{\vee}) \xrightarrow{\sim} \mathfrak{o}_{\partial^{prq} C_{pq}}.
    \end{equation}
    \item The above isomorphisms are preserved by the $\Pi$-action.
\end{enumerate}
\end{defn}
Note that if the above derived orbifold lift is given an orientation, then for any $\alpha \leq \beta$ in $\bA^\floer_{pq}$, there is an isomorphism
$$ \mathfrak{o}_{C_\alpha} \xrightarrow{\sim} \mathfrak{o}_{\partial^{\alpha} C_\beta}.$$

\begin{defn}\label{defn:orient-complex-lift}
Given a flow category $T^{\mc P}$ endowed with a derived orbifold lift ${\mf D}^{\mc P}$, a scaffolding ${\mf F}^{\mc P}$, and a compatible collar structure, if they are further endowed with a normal complex structure and orientation, we call such a datum an {\bf oriented and normally complex derived orbifold lift} and abbreviate it as ${\mf D}^{\mc P}$.
\end{defn}

If we have a flow bimodule $M$ between $T^{\mc P}$ and $T^{\mc P'}$ such that these three objects admit a compatible derived orbifold lift, as well as a compatible scaffolding structure, one can spell out the meaning of a normal complex structure and an orientation on $M$ similarly to Definition \ref{defn:normal-C}.

\begin{defn}\label{defn:normal-C-bimod}
Let $T^{\mc P}$ (resp. $T^{\mc P'}$) be flow categories endowed with a derived orbifold lift ${\mf D}^{\mc P} = \{C^{\mc P}_{pq} \}_{p\leq q}$ (resp. ${\mf D}^{\mc P'} = \{C^{\mc P'}_{p' q'} \}_{p'\leq q'}$) and a scaffolding 
$$
{\mf F}^{\mc P} =\Big( \big( {\mc F}_{\beta\alpha}^{\mc P}, {\bm\theta}_{\beta\alpha} \big)_{\alpha \leq \beta} \Big)_{p<q} \text{ resp. } {\mf F}^{\mc P'} = \Big( \big( {\mc F}_{\beta' \alpha'}^{\mc P'}, {\bm\theta}_{\beta' \alpha'} \big)_{\alpha' \leq \beta'} \Big)_{p'<q'}.
$$
Moreover, both the pair $({\mf D}^{\mc P}, {\mf F}^{\mc P})$ and the pair $({\mf D}^{\mc P'}, {\mf F}^{\mc P'})$ have a normal complex structure in the sense of Definition \ref{defn:normal-C}. Suppose $M$ is a flow bimodule from $T^{\mc P}$ to $T^{\mc P'}$ with a derived orbifold lift extending ${\mf D}^{\mc P}$ and ${\mf D}^{\mc P'}$ (Definition \ref{defn:bimod-lift})
$$
{\mf D}^M = \Big(\big\{ C^{M}_{pp'} = ({\mc U}^{M}_{pp'}, {\mc E}^{M}_{pp'}, {\mc S}^{M}_{pp'}, \psi^{M}_{pp'}) \big\}_{p\in {\mc P},p'\in {\mc P}'}, {\bm \iota}^M_{pp'} \Big)
$$
and a scaffolding
$$
{\mf F}^M = \big( \big\{ {\mc F}_{\tilde \beta \tilde \alpha}^M, {\bm\theta}_{\tilde \beta\tilde\alpha}^M \big\}_{\tilde\alpha \leq \tilde \beta} \big)_{\tilde\alpha, \tilde\beta \in \bA_{pp'}^M}
$$
compatible with ${\mf F}^{\mc P}$ and ${\mf F}^{\mc P'}$ (Definition \ref{defn:bimod-sca}). Then a {\bf normal complex structure} (or normally complex lift) on this datum extending or compatible with the given ones on $({\mf D}^{\mc P}, {\mf F}^{\mc P})$ and $({\mf D}^{\mc P'}, {\mf F}^{\mc P'})$ consists of the following data.
\begin{enumerate}
    \item A normal complex structure on each of the derived orbifold chart $({\mc U}^{M}_{pp'}, {\mc E}^{M}_{pp'}, {\mc S}^{M}_{pp'})$.
    \item A complex structure on the each of the orbifold vector bundles ${\mc F}_{\tilde \beta \tilde \alpha}^M \to {\mc U}^{M}_{\alpha}$.
    \item The germ equivalence $\theta^{M}_{\tilde\alpha, \tilde\beta}$ and the bundle map covering it $\widehat\theta^{M}_{\tilde\alpha, \tilde\beta}$ all intertwine with the normal complex structures.
    \item Splittings or isomorphism of the form \eqref{scaffolding_splitting}, \eqref{vartheta}, and \eqref{eqn:sca-split} all respect the (normal) complex structures.
    \item The respective (normal) complex structures are identified under the $\Pi$-action.
\end{enumerate}
\end{defn}

\begin{defn}
Under the same setting as in Definition \ref{defn:normal-C-bimod}, suppose we are given a compatible normal complex structure on the pair ${\mf D}^{M}$ and ${\mf F}^{\mc P}$. Then a collar structure compatibile with ${\mf F}^M$ (see Definition \ref{defn329}) is said to be compatible with the normal complex structure if the collar maps intertwine with the normal complex structures.
\end{defn}

\begin{defn}\label{defn:orient-bimod}
Suppose $(T^{\mc P}, {\mf D}^{\mc P}, {\mf F}^{\mc P})$, $(T^{\mc P'}, {\mf D}^{\mc P'}, {\mf F}^{\mc P'})$, and $(M, {\mf D}^{M}, {\mf F}^{M})$ be the same as in Definition \ref{defn:normal-C-bimod}. Moreover, assume that $({\mf D}^{\mc P}, {\mf F}^{\mc P})$ and $({\mf D}^{\mc P'}, {\mf F}^{\mc P'})$ are equipped with an orientation in the sense of Definition \ref{defn:orient}. Then an {\bf orientation} on the given normally complex lift of ${\mf D}^{M}$ and ${\mf F}^{M}$ is given by
\begin{enumerate}
\item An isomorphism of orientation lines
\begin{equation}\label{eqn:sign-bimod}
\mathfrak{o}_{C_{pp'}^M} \xrightarrow{\sim} \mathfrak{o}_p^{\vee} \otimes \mathfrak{o}_{p'}.
\end{equation}
\item For $p<q$ and $p'$, an isomorphism of orientation lines
\begin{equation}\label{eqn: coh-bimod-1}
\mathfrak{o}_{C^{\mc P}_{pq}} \otimes \mathfrak{o}_{C^M_{q p'}} \otimes (\mathfrak{o}_{\mc F^M_{pqp', pq}} \otimes \mathfrak{o}_{\mc F^M_{pqp', pq}}^{\vee}) \xrightarrow{\sim} \mathfrak{o}_{\partial^{pqp'} C_{pp'}^M}.
\end{equation}
\item For $p$ and $q' < p'$, an isomorphism of orientation lines
\begin{equation}\label{eqn:coh-bimod-2}
\mathfrak{o}_{C^{M}_{pq'}} \otimes \mathfrak{o}_{C^{\mc P'}_{q' p'}} \otimes (\mathfrak{o}_{\mc F^M_{pq'p', pq}} \otimes \mathfrak{o}_{\mc F^M_{pq'p', pq}}^{\vee}) \xrightarrow{\sim} \mathfrak{o}_{\partial^{pq'p'} C_{pp'}^M}.
\end{equation}
\item The above isomorphisms are preserved by the $\Pi$-action.
\end{enumerate}
\end{defn}

For a flow bimodule $M$ from $T^{\mc P}$ to $T^{\mc P'}$ as in Definition \ref{defn:normal-C-bimod} (with ${\mf D}^{M}$, ${\mf F}^M$ and a compatible collar structure), if it is further endowed with a orientation structure, we say it has an {\bf oriented and normally complex derived orbifold lift} and abbreviate as ${\mf D}^M$.

\subsection{Constructing FOP perturbations}\label{subsection36}

We extend the notion of FOP perturbations from the case of a single derived orbifold chart to the case of a derived orbifold lift of a flow category. The case of a flow bimodule is similar.

\begin{defn}[Perturbations on a derived orbifold lift]\label{defn_perturbation}\hfill
Given a derived orbifold lift ${\mf D}^{\mc P}$ of a flow category $T^{\mc P}$ (see Definition \ref{defn:flow-lift}). 
\begin{enumerate}
\item A {\bf perturbation} on ${\mf D}^{\mc P}$ consists of a system of smooth perturbations  
\beqn
{\mf S}' = \Big( {\mc S}_{pq}': {\mc U}_{pq} \to {\mc E}_{pq}\Big)_{p< q}
\eeqn
of the compact derived orbifold chart $({\mc U}_{pq}, {\mc E}_{pq}, {\mc S}_{pq})$ (see Definition \ref{defn21}) such that the diagram \eqref{comm_diag_32} still commutes if we replace all ${\mc S}_{pq}$ by ${\mc S}_{pq}'$.

\item Suppose the derived orbifold lift is equipped with a scaffolding (see Definition \ref{defn_scaffolding}). Then we say that a perturbation ${\mf S}'$ {\bf respects the scaffolding} if the following is true. For each pair of strata $\alpha \leq \beta$ of $\bA^{\mc P}_{pq}$, the datum ${\mf S}'$ induces a perturbation ${\mc S}_\alpha'$ of $C_\alpha$ and a perturbation ${\mc S}_\beta'$ of $C_\beta$ by taking the product. Then we require 
\beqn
{\mc S}_\beta'|_{\partial^\alpha {\mc U}_\beta} = {\rm Stab}_{{\mc F}_{\beta\alpha}}({\mc S}_\alpha').
\eeqn

\item Suppose the derived orbifold lift is equipped with a collar structure. Then we say that a perturbation ${\mf S}'$ {\bf respects the collar structure} if for each $pq$ and each stratum $\alpha \in \bA_{pq}^{\mc P}$, the following diagram commutes.
\beqn
\xymatrix{  (\pi_{pq, \alpha}^{\rm collar})^* ( {\mc E}_{pq}|_{\partial^\alpha {\mc U}_{pq}}) \ar[rr]^-{\vartheta_{pq, \alpha}^{\rm collar}} & &  {\mc E}_{\partial^\alpha {\mc U}_{pq}} \\
            \partial^\alpha {\mc U}_{pq}\times [0, \epsilon)^{\bF_\alpha} \ar[u]^{(\pi_{pq, \alpha}^{\rm collar})^* ({\mc S}_{pq}'|_{\partial^\alpha {\mc U}_{pq}})}  \ar[rr]_-{\theta_{pq, \alpha}^{\rm collar}} & &  {\mc U}_{pq}                                         \ar[u]_{{\mc S}_{pq}'}    }
\eeqn

\item Suppose the derived orbifold lift is normally complex and is equipped with a package of compatible structures (scaffolding, collar structure, and straightening). Then a perturbation ${\mf S}'$ is called an {\bf FOP perturbation} if respects the scaffolding, respects the collar structure, and for each $pq$, with respect to the straightening, the restriction of ${\mc S}_{pq}'$ to the interior of each stratum $\partial^\alpha {\mc U}_{pq}$ is an FOP section of ${\mc E}_{pq}$.\footnote{Indeed being an FOP perturbation only needs to refer to the straightening and the normal complex structure.}

\item Under the assumptions of the last item, an FOP perturbation is called {\bf strongly transverse} if the restriction of each ${\mc S}_{pq}'$ to the interior of each stratum of ${\mc U}_{pq}$ is strongly transverse (see Definition \ref{defn:strongly_transverse}).
\end{enumerate}
\end{defn}

\begin{thm}\label{thm_FOP_2}
Given a flow category $T^{\mc P}$, a derived orbifold lift ${\mf D}^{\mc P}$ with a normally complex structure, together with a package of compatible extra structures (scaffolding, collaring, and straightening), there exists a strongly transverse FOP perturbation. Moreover, we can make the perturbation $\Pi$-invariant, i.e.,
\beqn
{\mc S}_{pq}' = {\mc S}_{a\cdot p\ a\cdot q}'\ \forall a \in \Pi\ p, q \in {\mc P}
\eeqn
after identifying $C_{pq}$ with $C_{a \cdot p\ a \cdot q}$. 
\end{thm}

\begin{proof}
The construction is based on the same induction strategy as defining the Hamiltonian Floer homology (as well as continuation maps etc.) using abstract perturbations as in \cite{Fukaya_Ono} and \cite{Liu_Tian_Floer}. The package of additional structures is necessary because the FOP perturbation scheme is more rigid than the traditional smooth or continuous multivalued perturbation scheme. We start with a pair $p< q$ with minimal (nonzero) energy ${\mc A}^{\mc P}(q) - {\mc A}^{\mc P} (p)$. Then by definition, $\bA_{pq}^{\mc P}$ is a poset with a single element and hence the chart $C_{pq} = ({\mc U}_{pq}, {\mc E}_{pq}, {\mc S}_{pq})$ is a derived orbifold chart with no boundary or corners. By the absolute version of Proposition \ref{prop:FOP_existence}, with respect to the normal complex structure and the straightening, there exists a strongly transverse FOP perturbation ${\mc S}_{pq}': {\mc U}_{pq} \to {\mc E}_{pq}$ which can be arbitrarily close (measured in $C^0$) to ${\mc S}_{pq}$. In particular, ${\mc S}_{pq}'$ has a compact zero locus. We can also make such perturbations $\Pi$-invariant. 

Now for a given pair $p< q$, we state our induction hypothesis: for pair $r< s$ with $p \leq r < s \leq q$ and $(p, q) \neq (r, s)$, we have constructed a strongly transverse FOP section with respect to the straightening and normal complex structure
\beqn
{\mc S}_{rs}': {\mc U}_{rs} \to {\mc E}_{rs}
\eeqn
satisfying the compatibility condition for a perturbation (with respect to the collar structure and scaffolding, see Definition \ref{defn_perturbation}). We would like to construct a strongly transverse FOP perturbation ${\mc S}_{pq}'$ which extends the existing ones and which still satisfies the compatibility conditions. 

For each proper stratum $\alpha = pr_1 \cdots r_l q$ of $pq$, the product of ${\mc S}_{pr_1}'$, $\ldots$, ${\mc S}_{r_l q}'$ provides a section 
\beqn
{\mc S}_\alpha':= {\mc S}_{pr_1}'\times \cdots \times {\mc S}_{r_l q}': {\mc U}_\alpha \to {\mc E}_\alpha.
\eeqn
Then via the stabilization map 
\beqn
{\bm \theta}_{pq, \alpha}: {\rm Stab}_{{\mc F}_{pq, \alpha}} ( C_\alpha) \simeq \partial^\alpha C_{pq}
\eeqn
one obtains a section 
\beqn
{\mc S}_{pq, \alpha}': \partial^\alpha {\mc U}_{pq} \to {\mc E}_{pq} |_{\partial^\alpha {\mc U}_{pq}}
\eeqn
which is the stabilization of ${\mc S}_\alpha'$. We check the following conditions. 
\begin{enumerate}
    \item The collection of ${\mc S}_{pq, \alpha}'$ agree on overlaps. Indeed, for each pair of stratum $\alpha \leq \beta$, by the compatibility condition (see item (2) of Definition \ref{defn_perturbation}) satisfied by the existing perturbations (which is assumed as induction hypothesis), one has 
    \beqn
    {\mc S}_\beta'|_{\partial^\alpha {\mc U}_\beta} = {\rm Stab}_{{\mc F}_{\beta\alpha}}( {\mc S}_\alpha').
    \eeqn
    Then by the definition of scaffolding (see Definition \ref{defn_scaffolding} and \eqref{comm_diag_36}), one has 
    \begin{multline*}
    {\mc S}_{pq, \beta}' |_{\partial^\alpha {\mc U}_{pq}} = {\rm Stab}_{{\mc F}_{pq, \beta}}( {\mc S}_\beta'|_{\partial^\alpha {\mc U}_\beta} )  =  {\rm Stab}_{{\mc F}_{pq, \beta}} ( {\rm Stab}_{ {\mc F}_{\beta\alpha} }( {\mc S}_\alpha') ) \\
     = {\rm Stab}_{ {\mc F}_{pq, \alpha}} ( {\mc S}_{\alpha}') = {\mc S}_{pq, \alpha}'.
    \end{multline*}
    
    \item The collection of ${\mc S}_{pq, \alpha}'$ respect the collar structure. This is a consequence of the induction hypothesis and the fact that the scaffolding respects the collar structure (see Definition \ref{defn329}). 
    
    \item Each ${\mc S}_{pq, \alpha}'$ is an FOP perturbation. This is a consequence of the compatibility condition on the additional structures (see Definition \ref{defn330}) which guarantees that the stabilization of an FOP section is still an FOP section. 
    
    \item Each ${\mc S}_{pq, \alpha}'$ is strongly transverse within the interior of $\partial^\alpha {\mc U}_{pq}$. Notice that 
    \beqn
    {\rm Int} {\mc U}_\alpha = {\rm Int} {\mc U}_{pr_1}\times \cdots \times {\rm Int} {\mc U}_{r_l q}.
    \eeqn
    Corollary \ref{cor28} implies that the restriction of ${\mc S}_\alpha'$ to the interior of ${\mc U}_\alpha$ is an FOP section. Moreover, as strong transversality is preserved under stabilization, we know that ${\mc S}_{pq, \alpha}'$ is strongly transverse within ${\rm Int} \partial^\alpha {\mc U}_{pq} = {\rm Stab}_{{\mc F}_{pq, \alpha}}({\rm Int} {\mc U}_\alpha)$.
\end{enumerate}
Therefore, we can extend the collection of ${\mc S}_{pq, \alpha}'$ to a neighborhood of $\partial {\mc U}_{pq}$ using the collar structure, simply by pulling back the existing perturbations using the projection maps. Hence we have obtained a section ${\mc S}_{pq}'$ of ${\mc E}_{pq}$ defined in an open neighborhood of $\partial {\mc U}_{pq}$. The conditions on the collar structure imply that the extension is well-defined. The compatibility between straightenings and the collar structure implies that the extension is still an FOP section. As we can extend strongly transverse FOP sections in the standard ``CUDV'' fashion (see the relative version of Proposition \ref{prop:FOP_existence}), one can construct a strongly FOP perturbation ${\mc S}_{pq}'$ which extends the existing ones near the boundary. The compactness assumption on ${\mc S}_{pq}^{-1}(0)$ implies that one can make the perturbed zero locus $({\mc S}_{pq}')^{-1}(0)$ compact. The inductive construction can then be carried on. Moreover, the $\Pi$-equivariance of the perturbation can be maintained in the induction process.
\end{proof}

\subsection{Chain complexes and maps over the integers}\label{subsection37}

We recall the notion of Novikov coefficient ring. 

\begin{defn}\label{defn:Novikov}
\begin{enumerate}

\item The integral {\bf Novikov ring} is the ring of formal Laurent series in a single variable $T$ with integer coefficients, i.e.
\beqn
\Lambda:= {\mb Z}[[T]][T^{-1}] =  \Big\{ {\mf x} = \sum_{i=-m}^\infty a_i T^i,\ {\rm where}\ m \in {\mb Z}\ {\rm and}\ a_i \in {\mb Z}\Big\}.
\eeqn

\item The {\bf valuation} on $\Lambda$ is the map 
\beqn
{\rm val}: \Lambda \to {\mb Z},\ {\rm val}(\sum a_i T^i ) = {\rm min} \left\{ i\ |\ a_i \neq 0 \right\}.
\eeqn
Denote
\beqn
\Lambda_0:= \{ {\mf x} \in \Lambda\ |\ {\rm val}({\mf x}) = 0\}.
\eeqn
and
\beqn
\Lambda_+:= \{ {\mf x} \in \Lambda\ |\ {\rm val}({\mf x}) > 0 \}.
\eeqn
\end{enumerate}
\end{defn}

\subsubsection{The chain complexes}\label{sec:chain}

Now we start to build the chain complexes and chain maps. First, associated to a poset ${\mc P}$ as in Setup \ref{setup:poset}, one can define a ${\mb Z} / 2N$-graded free $\Lambda$-module as follows. Define  
\beqn
C_*^{\mc P}:= \left\{ \sum_{p \in {\mc P}} m_p p\ |\ m_p\in {\mb Z},\ \forall c>0, \#\{ p\ |\ m_p \neq 0, {\mc A}^{\mc P} (p) < c \} < \infty \right\}.
\eeqn
This is a free abelian group graded by the index function ${\rm ind}^{\mc P}: {\mc P} \to {\mb Z}/ 2N$. Define a $\Lambda$-module structure on $C_*$ by
\beqn
T^a\Big( \sum_{p \in {\mc P}} m_p p \Big):= \sum_{p\in {\mc P}} m_p (a\cdot p).
\eeqn
Conditions of Setup \ref{setup:poset} implies that $C_*$ is a ${\mb Z}/2N$-graded free $\Lambda$-module whose rank is equal to the cardinality of ${\mc P}/ \Pi$. 

Isolated zeroes in the free locus of each oriented derived orbifold chart induce integer counts. More precisely, fix a strongly transverse FOP perturbation 
\beqn
{\mf S}':= \Big( {\mc S}_{pq}': {\mc U}_{pq} \to {\mc E}_{pq}\Big)_{p<q}.
\eeqn
For each pair $p< q$, we can write 
\beqn
{\mc U}_{pq} = \bigsqcup_{i=-\infty}^{+\infty} {\mc U}_{pq}^{[i]}
\eeqn
where ${\mc U}_{pq}^{[i]} \subset {\mc U}_{pq}$ is the open and closed subset of points whose local virtual dimension is $i$. It follows from \eqref{index_formula_1} that
\beqn
{\mc U}_{pq}^{[i]} \neq \emptyset \Longrightarrow i \equiv {\rm ind}^{\mc P}(p) - {\rm ind}^{\mc P}(q) - 1\ {\rm mod}\ 2N.
\eeqn
Then let $n_{pq}$ be the count (with signs) of zeroes of ${\mc S}_{pq}'$ in the free locus of the zero-dimensional component ${\mc U}_{pq}^{[0]}$. The sign of an isolated zero point of ${\mc S}_{pq}'$ is uniquely determined by the sign read off from the isomorphism \eqref{eqn:orient-line}. The compactness assumption (see item (3) of Definition \ref{def:flow-cat}) implies that $n_{pq}$ is finite. Moreover, by the finiteness of ${\mc P}/\Pi$ and the compactness condition, these counts  define a ${\mb Z}$-linear map 
\beqn
d^{\mc P}: C_*^{\mc P} \to C_{*-1}^{\mc P}.
\eeqn
The $\Pi$-equivariance of the perturbation implies that 
\beqn
n_{pq} = n_{a\cdot p\ a \cdot q},\ \forall a \in {\mb Z}\ and\ p, q \in {\mc P}
\eeqn
which further implies that $d^{\mc P}$ is $\Lambda$-linear. By looking at 1-dimensional components of all derived orbifold charts, using the coherence of the orientations \eqref{eqn:orient-coh}, one can see that $d^{\mc P}$ is a differential map, i.e., $d^{\mc P}\circ d^{\mc P} = 0$. Therefore, one obtains a chain complex 
\beqn
(C_*^{\mc P}, d^{\mc P})
\eeqn
of $\Lambda$-modules. One hence obtains the homology
\beqn
H_*^{\mc P} = \bigoplus_{i \in {\mb Z}/ 2N} H_i^{\mc P}
\eeqn
which has a natural $\Lambda$-module structure. One can show that the chain homotopy equivalence class of the chain complex does not depend on the choice of the perturbations. We omit the details because such a fact is not needed in our application.

\subsubsection{The chain maps}

Now suppose we have two topological flow categories, $T^{\mc P}$ over ${\mc P}$ and $T^{{\mc P}'}$ over ${\mc P}'$ and a flow bimodule $M$ from $T^{\mc P}$ to $T^{{\mc P}'}$. 

\begin{thm}\label{thm_FOP_3}
Assume the following conditions. 
\begin{enumerate}
    \item There is an oriented and normally complex derived orbifold lift ${\mf D}^{\mc P}$ resp. ${\mf D}^{{\mc P}'}$ of $T^{\mc P}$ resp. $T^{{\mc P}'}$.
    
    \item There is an oriented and normally complex derived orbifold lift ${\mf D}_{{\mc P}{\mc P}'}$ of $M$ which extends ${\mf D}^{\mc P}$ and ${\mf D}^{{\mc P}'}$. 
    
    \item On these derived orbifold lifts there exist compatible collar structures, scaffoldings, and straightenings.
\end{enumerate}
Suppose we are given a strongly transverse FOP perturbation on ${\mf D}^{\mc P}$ and a strongly transverse FOP perturbation on ${\mf D}^{{\mc P}'}$. Then there exists a strongly transverse FOP perturbation on ${\mf D}_{{\mc P}{\mc P}'}$ which extends the existing ones which is $\Pi$-invariant.
\end{thm}

\begin{proof}
This theorem is essentially a relative version of Theorem \ref{thm_FOP_2} and the proof is the same as we can always extend strongly transverse FOP perturbations from local to global using the standard ``CUDV'' fashion (see the relative version of Proposition \ref{prop:FOP_existence}). 
\end{proof}

Now under the assumptions of Theorem \ref{thm_FOP_3}, the existing strongly transverse FOP perturbations induce chain complexes 
\begin{align*}
    &\ (C_*^{\mc P}, d^{\mc P}),\ &\ (C_*^{{\mc P}'}, d^{{\mc P}'}).
\end{align*}
We would like to define a chain map using the extended FOP perturbations on ${\mf D}_{{\mc P}{\mc P}'}$. Indeed, for $p\in {\mc P}$ and $p'\in {\mc P}'$, consider the perturbation
\beqn
{\mc S}_{pp'}': {\mc U}_{pp'} \to {\mc E}_{pp'}.
\eeqn
Let $n_{pp'}$ be the count (with signs read off from the isomorphism between orientation lines \eqref{eqn:sign-bimod}) of zeros of ${\mc S}_{pp'}'$ in the free locus of all components of ${\mc U}_{pp'}$ with local virtual dimension zero. The compactness condition on $M_{pp'}$ (see Definition \ref{defn:flow-bimod}) implies that $n_{pp'}$ is finite. Moreover, \eqref{index_formula_2} implies that 
\beqn
n_{pp'} \neq 0 \Longrightarrow {\rm ind}^{\mc P}(p) \equiv {\rm ind}^{{\mc P}'}(p')\ {\rm mod}\ 2N.
\eeqn
Then we formally define 
\beqn
\Psi^M: C_*^{\mc P} \to C_*^{{\mc P'}}
\eeqn
by linearly extending
\beqn
\Psi^M(p) = \sum_{p'\in {\mc P}'} n_{pp'} p'.
\eeqn
We claim that the above is a legitimate element of $C_*^{{\mc P}'}$. Indeed, this is a consequence of the condition \eqref{eqn34} of Definition \ref{defn:flow-bimod} and the finiteness of ${\mc P}'/\Pi$. Hence $\Psi^M$ is a well-defined morphisms of graded abelian group. Moreover, by the $\Pi$-equivariance of all structures, one has 
\beqn
n_{pp'} = n_{a \cdot p\ n\cdot p'},\ \forall a\in \Pi.
\eeqn
Hence $\Psi^M$ is $\Lambda$-linear. Lastly, by the description of codimension one stratum of all bimodule moduli spaces and the coherence of orientations \eqref{eqn: coh-bimod-1} and \eqref{eqn:coh-bimod-2}, similar to the case of $d^{\mc P} \circ d^{\mc P} = 0$, one has that $\Psi^M$ is a chain map, i.e., 
\beqn
\Psi^M \circ d^{\mc P} = d^{{\mc P}'} \circ \Psi^M.
\eeqn

\begin{rem}
We would like to remark that the definitions of $d^{\mc P}$ and $\Psi^M$ from counting points in moduli spaces of virtual dimension $0$, and the usual proof for $(d^{\mc P})^2 = 0$ and $d^{\mc P'} \circ \Psi^M = \Psi^M \circ d^{\mc P}$ from counting the boundary points of moduli spaces of virtual dimension $1$ works in our setting, because the pseudocycle condition on the free locus of $({\mc S}_{pq}')^{-1}(0)$ guarantees that the topological boundary $\partial ( ({\mc S}_{pq}')^{-1}(0))^{\rm free}$ has codimension at least $2$ so the desired compactness property, therefore the finiteness of algebraic counts holds.
\end{rem}

\section{Proof of the integral Arnold conjecture}\label{sec-4}

In this section, we define the topological flow category associated with a non-degenerate Hamiltonian and the topological flow category constructed from a Morse--Smale function. Then we define two bimodules between these two flow categories, which are respectively the space-level lifts of the well-known Piunikhin--Salamon--Schwarz map (the PSS map) and its inverse (the SSP map). We state in Theorem \ref{thm:dorb-lift} the most important technical result in this paper, namely, the existence of derived orbifold lifts of the aforementioned objects, whose proof is contained in later sections. Finally, we explain how to use the perturbation scheme developed in Section \ref{section2} and Section \ref{sec-3} to prove Theorem \ref{thm_intro_main}.

The following is the standing assumption of our discussions later on.

\begin{hyp}\label{hyp51}
\begin{enumerate}
    \item The symplectic manifold $(M, \omega)$ is integral, i.e., the de Rham cohomology class $[\omega] \in H_{dR}^2(M) \cong H^2(M; {\mb R})$ lies in the image of $H^2(M; {\mb Z}) \to H^2(M; {\mb R})$.
    
    \item $H$ is a nondegenerate Hamiltonian on $(M, \omega)$.
    
    \item All 1-periodic orbits of $H$ are embedded and any two of them are disjoint. 
    
    \item The symplectic actions \eqref{eqn:symp-action} of all capped 1-periodic orbits are integral.
    
    \item $J$ is a time-independent $\omega$-compatible almost complex structure such that for all Floer trajectories with smooth domains (i.e., no sphere bubbles are attached and the trajectories are not broken) $u: \Theta \to X$, the linearized operator is surjective.
\end{enumerate}
\end{hyp}

We explain why it suffices to prove the Arnold conjecture under these assumptions in order to establish it in full generality. 

\begin{enumerate}
    \item As we explained in the introduction of this paper, the case of Arnold conjecture for which $[\omega]$ is rational implies the general case, and this is equivalent to the case when $[\omega] \in H^2(M;\mathbb{Z})$ by suitably rescaling $\omega$ and $H$ using a common integer-valued factor.
    \item Because any $C^2$-small perturbation of a nondegenerate Hamiltonian $H$ will not change the number of periodic orbits, we can freely perturb $H$ to guarantee that all $2 \pi$-periodic orbits are nondegenerate. 
    \item As the 2-dimensional case of the integral Arnold conjecture is known, we can restrict ourselves to the case when ${\rm dim} M\geq 4$. Then one can slightly perturb $H$ so that distinct periodic orbits do not intersect and all periodic orbits are embedded.
    
    \item Given a nondegenerate Hamiltonian $H$ (whose capped 1-periodic orbits are all discrete), by adding a $C^2$-small, $t$-independent function $f: M \to {\mb R}$ whose restriction to a neighborhood of the image of each periodic orbit is a constant does not change the set of capped 1-periodic orbits but will shift their symplectic actions by constants. Then we can add a sufficiently small function $f$ which makes the symplectic actions of all capped periodic orbits rational. Then the symplectic actions can be made integral by a further rescaling.
    
    \item For any given $J$, one can slightly perturb $H$ to achieve transversality for all smooth Floer trajectories. Moreover, the perturbed Hamiltonian can be chosen to agree with $H$ up to second order on each 1-periodic orbits of $H$ (see \cite[Theorem 5.1]{Floer_Hofer_Salamon}).
    
\end{enumerate}

\subsection{Hamiltonian Floer flow categories}\label{subsec:floer-flow}

\subsubsection{The moduli spaces of Floer trajectories}

We first review the basics about the moduli spaces of stable Floer trajectories. Let $(M, \omega)$ be a compact symplectic manifold and $H = (H_t)_{t\in S^1}$ be a smooth 1-periodic Hamiltonian on $M$. A {\bf capped 1-periodic orbit} is a pair $p = ([u], \uds p)$ where $\uds p: S^1 \to M$ is a 1-periodic orbit of the Hamiltonian vector field $X_{H_t}$ and $[u]$ is an equivalence class of maps $u: {\mb D} \to M$ such that (we view ${\mb D}$ as the unit disk in ${\mb C}$)
\beqn
u(e^{ 2\pi {\bm i} t}) = \uds p (t);
\eeqn
two such maps $u_1, u_2: {\mb D} \to M$ are equivalent if 
\beqn
\int_{\mb D} u_1^* \omega = \int_{{\mb D}} u_2^* \omega.
\eeqn
Define the {\bf symplectic action} of a capped orbit $p = ([u], \uds p)$ to be 
\beq\label{eqn:symp-action}
{\mc A}_H(p) =  \int_{\mb D} u^* \omega + \int_{S^1} H_t(\uds p(t)) dt.
\eeq
We denote by $\per(H)$ the set of contractible 1-periodic orbits of $H$, whose elements are denoted by $\uds p, \uds q$, etc. Denote by $\widetilde{\per}(H)$ the set of capped 1-periodic orbits of $H$, whose elements are denoted by $p, q$, etc. It follows from the integrality assumption on the symplectic class that $\widetilde{\per}(H)$ is a ${\mb Z}$-covering of $\per(H)$.

Choose an $\omega$-compatible almost complex structure $J$ on $M$. Let $\Theta = {\mb R}\times S^1$ be the infinite cylinder with standard coordinates $z = s + \i t$. The Floer equation is the first-order equation for smooth maps $u: \Theta \to M$
\beq\label{Floereqn}
\frac{\partial u}{\partial s} + J \left( \frac{\partial u}{\partial t} - X_{H_t}(u) \right) = 0.
\eeq
The associated energy for a map $u: \Theta \to M$ is defined to be 
\beqn
E_H(u):= \frac{1}{2} \int_\Theta \left( \left| \frac{\partial u}{\partial s} \right|^2 + \left| \frac{\partial u}{\partial t} - X_{H_t}(u) \right|^2 \right) ds dt.
\eeqn
Any solution to \eqref{Floereqn} with finite energy necessarily converges as $s \to \pm \infty$ to periodic orbits of $H$. Then one can use a pair of (capped) orbits to label solutions. Let ${\mc M}_{\uds p \uds q}^\floer$ be the set of solutions to \eqref{Floereqn} which converge to $\uds p$ resp. $\uds q$ as $s \to -\infty$ resp. $s \to +\infty$, modulo the obvious time translation. Denote by 
\beqn
{\mc M}_{pq}^\floer \subset {\mc M}_{\uds p \uds q}^\floer
\eeqn
the subset of equivalence classes of solutions $u$ such that the concatenation of $p$ and $u$ is equivalent to $q$ as capped 1-periodic orbits. Then for any solution $u$ to \eqref{Floereqn} representing an element of ${\mc M}_{pq}^\floer$, it is standard that
\beq\label{eqn:floer-energy}
E_H(u) = {\mc A}_H(q) - {\mc A}_H(p).
\eeq
As a result the Floer differential {\it increases} the symplectic action.

The moduli spaces in general are not compact with respect to the $C^\infty_{\rm loc}$-topology. Indeed, the space ${\mc M}_{pq}^\floer$ admits a natural compactification called the Gromov--Floer compactification, denoted by $\ov{\mc M}{}_{pq}^\floer$, incorporating both bubbling of holomorphic spheres and breaking of Floer trajectories, see e.g. \cite[Section 18]{Fukaya_Ono}. It is standard knowledge that $\ov{\mc M}{}_{pq}^\floer$ is a compact Hausdorff topological space.\footnote{It is possible that ${\mc M}_{pq}^\floer = \emptyset$ while $\ov{\mc M}{}_{pq}^\floer \neq \emptyset$.}

\subsubsection{The flow category}

We now package the collection of Floer moduli spaces into a flow category. First we see how this system of moduli spaces fits into Setup \ref{setup:poset}. Let $N\in {\mb Z}_{\geq 0}$ be the minimal Chern number of $(M, \omega)$. The integral symplectic form $\omega$ defines a homomorphism
\beqn
\omega: \pi_2(M) \to {\mb Z}.
\eeqn
Define 
\beqn
\Pi:= \pi_2(M)/ {\rm ker} \omega.
\eeqn
Then $\Pi$ is an infinite cyclic group and $\omega$ induces an injection 
\beqn
\omega: \Pi \to {\mb Z}
\eeqn
Then we define 
\beqn
{\mc P}^\floer:= \widetilde{\rm Per}(H).
\eeqn
Define the partial order by the existence of nonempty Floer moduli spaces, i.e., 
\beqn
p\leq q \Longrightarrow \ov{\mc M}{}_{pq}^\floer \neq \emptyset.
\eeqn
It has a free $\Pi$-action defined by taking the connected sum between a representative of a capped 1-periodic orbit and a representative of an element $a \in \Pi$. Define
\begin{align*}
&\ {\mc A}^{{\mc P}^\floer}:= {\mc A}_H,\ &\ {\rm ind}^{{\mc P}^\floer}:= \text{ Conley--Zehnder\ index}.
\end{align*}
Then ${\mc P}^\floer$ satisfies conditions of Setup \ref{setup:poset}. 

\begin{notation}
As a convention, objects labelled by ${\mc P}^\floer$ are often also labelled by $\floer$. For example, we abbreviate
\begin{align*}
    &\ {\mc A}^\floer:= {\mc A}^{{\mc P}^\floer},\ &\ {\rm ind}^\floer:= {\rm ind}^{{\mc P}^\floer}.
\end{align*}
\end{notation}

One can then describe the stratifications on Floer moduli spaces. As in Notation \ref{notation35}, for each pair $p, q\in {\mc P}^\floer$, one has the homogeneous poset
\beqn
\bA_{pq}^\floer:= \big\{ \alpha = pr_1 \cdots r_l q\ |\ p< r_1 < \cdots < r_l < q\big\}.
\eeqn
Then the moduli space $\ov{\mc M}{}_{pq}^\floer$ an $\bA_{pq}^\floer$-stratified topological space. Indeed, for each $\alpha = pr_1 \cdots r_l q$, one has 
\beqn
\partial^\alpha \ov{\mc M}{}_{pq}^\floer \cong \ov{\mc M}{}_{pr_1}^\floer\times \cdots \times \ov{\mc M}{}_{r_l q}^\floer
\eeqn
as the $\alpha$-stratum is the subset of ``broken trajectories'' of a type described by $\alpha$. 

%The concept of flow category was introduced by Cohen--Jones--Segal \cite{Cohen_Jones_Segal}. The original definition, which only consider the compactified moduli spaces of Floer trajectories, is relatively simple. Because of the failure of transversality and the need of developing the perturbation scheme, we will introduce several enhancements of the flow category. 

%\begin{defn}
%Let ${\mc P}$ be a countable partially ordered set. We say that ${\mc P}$ is locally finite-dimensional if for each pair $p< q$, there are only finitely many elements of ${\mc P}$ which are between $p$ and $q$. 
%\end{defn}

%\begin{example}
%\begin{enumerate}
%    \item Let ${\mc P} = {\mc P}_H$ be the set of capped 1-periodic orbits of the Hamiltonian $H$.  The system of moduli spaces $\ov{\mc M}_{pq}$ induce a partial order on ${\mc P}_H$ which is locally finite-dimensional: $p \leq q$ if and only if $\ov{\mc M}_{pq} \neq \emptyset$. 

%    \item The system of PSS moduli spaces and the system of moduli spaces of Floer trajectories induce a partial order on the set ${\mc P}\cup \{\bullet\}$.
%\end{enumerate}
%\end{example}

\begin{defn}
The {\bf Hamiltonian Floer flow category} (associated to $H$ and $J$), denoted by $T^{\rm Floer}$, is the topological flow category over ${\mc P}^\floer$ whose morphism space between $p, q \in \widetilde{\per}(H)$ is the moduli space $\ov{\mc M}{}_{pq}^{\rm Floer}$, and whose composition maps are the natural inclusions
\beqn
\ov{\mc M}{}_{pr}^{\rm Floer}\times \ov{\mc M}{}_{rq}^{\rm Floer} \to \ov{\mc M}{}_{pq}^{\rm Floer},\ \forall p\leq r \leq q.
\eeqn
\end{defn}

\subsection{Morse flow category}\label{subsec:morse-flow}

We describe the definition of the more classical Morse flow category and explain how it fits into the general framework of this paper. We declare that in this paper Morse flows are always the ascending flow, i.e., the flow generated by the gradient vector field. Suppose $(f, g)$ is a Morse--Smale pair on $M$, namely, $f$ is a Morse function and $g$ is a Riemannian metric such that the unstable manifold of any critical point intersects transversely with any stable manifold. We use $\uds x, \uds y$ to denote the critical points of $f$. To ensure the moduli spaces of (unparametrized) gradient flow lines to have smooth structures, we assume that near each $\uds x \in {\rm crit}(f)$, there exists a coordinate chart $(x_1, \dots, x_{2n})$ such that $f = \pm x_1^2 \pm \cdots \pm x_{2n}^2$ and $g = dx_1 \otimes dx_1 + \cdots dx_{2n} \otimes dx_{2n}$. Then by \cite{Wehrheim-Morse}, for any pair of cricial points $\uds x, \uds y \in {\rm crit}(f)$, the moduli space 
\beqn
\ov{\mc M}{}^{\rm Morse}_{\uds x \uds y}
\eeqn
of unparametrized broken flow lines connecting $x$ and $y$ is a smooth manifold with faces. It is standard knowledge that upon choosing orientations on the unstable manifolds of all critical points, one can count rigid Morse flow lines (with signs) and define a ${\mb Z}$-graded chain complex over ${\mb Z}$ and its homology coincides with $H_*(M; {\mb Z})$.

We want to fit the Morse flow category into the abstract framework. Define
\beqn
{\mc P}^\morse:= \Pi \times {\rm crit}(f)
\eeqn
with a partial order defined by  
\beqn
(a, \uds x) \leq (b, \uds y) \Longleftrightarrow a = b\ {\rm and}\ \ov{\mc M}{}_{\uds x \uds y}^\morse \neq \emptyset.
\eeqn
$\Pi$ naturally acts freely on ${\mc P}^\morse$. The action function is defined by
\beqn
{\mc A}^\morse(a, \uds x) = \omega(a) + \epsilon f( \uds x )
\eeqn
where $\epsilon$ is a sufficiently small positive number. The index function is defined by  
\beqn
{\rm ind}^\morse(a, \uds x) \equiv \frac{1}{2} {\rm dim} M - {\rm Morse\ index\ of\ }\uds x\ {\rm mod}\ 2N.
\eeqn
One can easily check that the triple $({\mc P}^\morse, {\mc A}^\morse, {\rm ind}^\morse)$ satisfies conditions of Setup \ref{setup:poset}. Then following Notation \ref{notation35} one obtains a collection of homogeneous posets $\bA_{xy}^\morse$ for all $x, y \in {\mc P}^\morse$. Define
\beqn
\ov{\mc M}{}_{xy}^\morse = \left\{ \begin{array}{cc} \ov{\mc M}{}_{\uds x \uds y}^\morse,\ &\ {\rm if}\ x = (a, \uds x )\ {\rm and}\ y = (a, \uds y),\\
 \emptyset,\ &\ {\rm otherwise}. \end{array} \right.
\eeqn
Then each $\ov{\mc M}{}_{xy}^\morse$ is stratified by the poset $\bA_{xy}^\morse$. We can then define the flow category $T^{\rm Morse}:= T^{\rm Morse}(f, g)$ associated with $(f, g)$ as follows. 

\begin{defn}
The {\bf Morse flow category} $T^\morse$ is the topological flow category over ${\mc P}^\morse$ whose morphism spaces are $T_{xy}^\morse = \ov{\mc M}{}_{xy}^\morse$ and whose composition maps are the natural inclusions
\beqn
\ov{\mc M}{}_{xz}^\morse \times \ov{\mc M}{}_{zy}^\morse \cong \partial^{xzy} \ov{\mc M}{}_{xy}^\morse \hookrightarrow \ov{\mc M}{}_{xy}^\morse.
\eeqn
It is obvious that $T^\morse$ is a strict $\Pi$-equivariant flow category.
\end{defn}

As transversality is already achieved for the Morse flow category and there is no orbifold behavior, one does not need to use derived orbifold lift nor normal complex structures to define the resulting chain complex. However we would like to formally put this case into the general framework of the previous section as it will be necessary when we connect the Morse flow category to objects where transversality fails and orbifold behavior appears. 

First, following the abstract outer-collaring recipe, one can construct an outer-collaring of the Morse flow category. Fix the width $r = 1$. The outer-collared Morse flow category $(T^{\morse})^+$ has morphism spaces being $(\ov{\mc M}{}_{xy}^{\morse})^+$. A priori this is only a topological flow category. However, one can equip the morphism spaces and composition maps with smooth structure. Indeed, by \cite[Theorem 1.4]{Wehrheim-Morse} which proves the associativity of gluing maps for the special kind of Morse--Smale pair $(f, g)$ (see also \cite{Qin_2018}), the original Morse flow category has compatible ``interior'' collars. Then the outer-collaring construction does not alter the feature and put a smooth structure on each space $(\ov{\mc M}{}_{xy}^\morse)^+$ such that the composition maps are smooth.

Second, as transversality is already achieved and the moduli spaces are manifolds but not orbifolds, the collection of outer-collared moduli spaces together with the zero obstruction bundle and the zero Kuranishi map form a derived orbifold lift of $(T^\morse)^+$, denoted by ${\mf D}^\morse$. The outer-collaring provides a collar structure on this lift. There is also the trivial scaffolding and the trivial straightening, which are obviously compatible with the collar structure. In the manifold case, there is only the trivial normal complex structure. Lastly, it is a classical knowledge that upon choosing orientations on all unstable manifolds the Morse moduli spaces inherit coherent orientations. Here the orientation line ${\mf o}_x$ associated with a capped critical point $x = (a, \uds{x})$ is defined to be the orientation line of the tangent space of the stable submanifold of $\nabla_{g}f$ at $\uds{x}$. We summarize these observations as follows.

\begin{lemma}
A choice of orientations on all unstable manifolds of $\nabla^g f$ makes ${\mf D}^\morse$ an oriented and normally complex drived orbifold lift of $(T^\morse)^+$ equipped with a compatible package of additional structures. Moreover, the $0$-perturbation is a strongly transverse FOP perturbation on ${\mf D}^\morse$ which defines the ${\mb Z}/2N$-graded chain complex $(C_*(f;\Lambda), d^\morse)$ of $\Lambda$-modules, whose homology is isomorphic to the $\Lambda$-module
\beqn
H^{(2N)}_{n-*}(M; \Lambda) = \bigoplus_{i\in {\mb Z}/2N } H^{(2N)}_{n-i}(M; {\mb Z})\otimes_{\mb Z} \Lambda.
\eeqn 
%A choice of trivializations of the tangent spaces of the stable submanifolds of the vector field $-\nabla_{g} f$ lifts the flow category $T^\morse$ (or the outer-collaring $T^{\morse+}$) to a normally complex derived orbifold flow category, whose derived orbifold charts have underlying spaces given by smooth manifolds with corners and the obstruction bundles are trivial. Moreover, the $0$-perturbation is an strongly transverse FOP perturbation and the defined chain complex is the same as the Morse--Smale--Witten complex 
%\beqn
%CM_{n-*}(f, g)\otimes \Lambda
%\eeqn
Here
\beqn
H_i^{(2N)}(M; {\mb Z}):= \bigoplus_{j \equiv i\ {\rm mod}\ 2N} H_j(M; {\mb Z}).
\eeqn
\end{lemma}

%\begin{proof}
%As mentioned above, the moduli space $\ov{\mc M}{}_{xy}^\morse$ (or its outer-collaring space $(\ov{\mc M}{}_{xy}^\morse)^+$) is a compact smooth manifold with corners. Hence the system of moduli spaces themselves provide a derived orbifold lift of the Morse flow category. Moreover, the outer-collaring also provides a collar structure this lift. As transversality is already achieved, there is only the trivial scaffolding. Mover, note that when an orbifold is actually a manifold, there is a trivial normal complex structure and any smooth perturbation is by definition an FOP perturbation. As the obstruction bundles are all trivial, the trivial section is an FOP section and automatically respect the collar structure (and the trivial scaffolding). Therefore, the chain complex associated with $T^\morse$ (or its outer-collaring) is the standard Morse chain complex and this concludes the proof.
%\end{proof}

\subsection{PSS and SSP bimodules}\label{subsection_pss}

\subsubsection{Moduli spaces}

Now we describe the moduli spaces which allow us to interpolate between Floer theory and Morse theory. We first set up a convention: PSS moduli spaces are defined by objects with input from Morse critical points and output from 1-periodic (capped) Hamiltonian orbits and SSP moduli spaces parametrizes objects in the reversed direction.

We only describe PSS moduli spaces in detail. The case of SSP moduli spaces is similar. To define the equation and hence the moduli space, we make the following choices.
\begin{enumerate}
    \item An integer $C^{\rm PSS}\in {\mb Z}$ satisfying 
    \beq\label{eqn44x}
    \inf_{(t, x) \in S^1 \times M} H_t (x) >  C^\pss.
    \eeq
    
    \item A smooth, {\it non-decreasing} cut-off function $\rho: {\mb R} \to [0, 1]$ such that 
    \begin{equation}\label{eqn:cut-off}
\rho(s) = \left\{ \begin{array}{cc} 0,\ &\ s \leq 0,\\
                                    1,\ &\ s \geq 1.
                                    \end{array}\right.
\end{equation}
\end{enumerate}
Define a $2$-parameter family of functions 
\beqn
H_{s,t}^{\pss}(x) = (1-\rho(s)) C^\pss + \rho(s) H_t(x).
\eeqn
Then over the cylinder $\Theta$ one has the vector-field-valued $1$-form 
\beqn
X_H^\pss \otimes dt
\eeqn
which is defined by the Hamiltonian vector field of $H_{s,t}^{\pss}(x)$.
Note that it vanishes for $s \leq 0$. The PSS equation is 
\beq\label{eqn:pss}
\frac{\partial u}{\partial s} + J \left( \frac{\partial u}{\partial t} - X_H^\pss (u) \right) = 0,\ {\rm where}\ u \in C^\infty(\Theta, M).
\eeq
The energy of a solution is defined to be and is computed as
\beqn
\begin{split}
E^\pss (u) : = &\  \int_\Theta \| \partial_s u \|^2 ds dt \\
= &\ \int_\Theta \omega( \partial_s u, \partial_t u - X_H^\pss (u)) ds dt\\
= &\ \int_\Theta u^* \omega + \int_\Theta d H_{s, t}^\pss (\partial_s u) ds dt \\
= &\ \int_\Theta u^* \omega + \int_\Theta \frac{\partial}{\partial s} (H_{s,t}^\pss (u)) ds dt - \int_\Theta \frac{\partial H^\pss_{s,t}}{\partial s} ds dt.
\end{split}
\eeqn
If the energy is finite, then by Gromov's removal of singularity theorem, $u$ extends to a $J$-holomorphic map near $-\infty$ and hence automatically determines a cap for the periodic orbit at $+\infty$. Hence one can use a capped 1-periodic orbit $p\in \widetilde{\per}(H)$ to label the moduli space of PSS solutions. It is similar to the case of Floer trajectories that ${\mc M}{}_{\bullet p}^{\rm PSS}$ admits a natural compactification consisting of ``stable PSS trajectories,'' i.e. configurations with both cylindrical and spherical components. Denote the compactification by 
\beq\label{eqn:pss-moduli}
\ov{\mc M}{}_{\bullet p}^{\rm PSS}.
\eeq

For each map $u$ representing an element of $\ov{\mc M}{}_{\bullet p}^{\rm PSS}$, one has the following energy identity
\beqn
E^{\rm PSS}(u) = {\mc A}_H(p) - C^{\rm PSS} - \int_\Theta \rho'(s) (H(u) - C^{\rm PSS}) ds dt.
\eeqn
An important consequence of \eqref{eqn44x} is that 
\beq
\ov{\mc M}{}_{\bullet p}^{\rm PSS} \neq \emptyset \Longrightarrow {\mc A}_H(p) - C^{\rm PSS} > 0.
\eeq
We introduce
\beqn
d_{\bullet p}:= {\mc A}_H(p) - C^{\rm PSS}
\eeqn
and call it the {\bf topological energy} of a solution $u$. 

Now we bring in the Morse theory of a Morse--Smale pair $(f, g)$. Given $\uds x \in {\rm crit}(f)$, consider the unstable manifold $W^u(\uds x)\subset M$ of the flow of $\nabla^g f$  and its compactification $\ov{W^u(\uds x)}$. Then define 
\beqn
\ov{\mc M}{}_{\uds x p}^\pss:= \ev_\bullet^{-1}(\ov{W^u(\uds x)}),
\eeqn
where $\ev_\bullet: \ov{\mc M}{}_{\bullet p}^{\rm PSS} \to M$ is the evaluation map at $-\infty$.
Moreover, if $x = (a, \uds x)\in {\mc P}^\morse$ and $p \in {\mc P}^\floer$, then define 
\beqn
\ov{\mc M}{}_{xp}^\pss:= \ov{\mc M}{}_{\uds x\ (-a) \cdot p}^\pss.
\eeqn

The stratifications on the PSS moduli spaces fit into the general framework of bimodules described in the previous section. Recall that to $x \in {\mc P} = {\mc P}^\morse$ and $p \in {\mc P}' = {\mc P}^\floer$ one associates a homogeneous poset
\beqn
\bA_{xp}:= \bA_{xp}^\pss:= \{ xy_1 \cdots y_k q_l \cdots q_1 p\ |\ x< y_1 < \cdots < y_k, q_l < \cdots < q_1 < p \}
\eeqn
which has the natural partial order induced from inclusion and the depth function 
\beqn
\dep(x y_1 \cdots y_k q_l \cdots q_1 p) = k + l .
\eeqn
The PSS moduli space $\ov{\mc M}{}_{xp}^\pss$ is stratified by $\bA_{xp}^\pss$ with corresponding strata
\beqn
\left( \ov{\mc M}{}_{xp}^\pss \right)_{xy_1 \cdots y_k q_l \cdots q_1 } \cong \ov{\mc M}{}_{xy_1}^\morse\times \cdots \times \ov{\mc M}{}_{y_{k-1}y_k}^\morse \times \ov{\mc M}{}_{y_k q_l}^\pss \times \ov{\mc M}{}_{q_l q_{l-1}}^\floer \times \cdots \times \ov{\mc M}{}_{q_1 p}^\floer
\eeqn
i.e., the subset of configurations with breakings described by the word $xy_1 \cdots y_k q_l \cdots q_1 p$. Gromov compactness shows that there are at most finitely many nonempty stratum. 

To define the SSP moduli spaces, let us choose an integer $C^{\rm SSP}\in {\mb Z}$ such that 
\begin{equation}\label{eqn:SSP+}
\sup_{(t,x) \in S^1 \times M} H_t(x) < C^{\rm SSP}.
\end{equation}
Using the same cut-off function as in \eqref{eqn:cut-off}, we can write down another $2$-parameter family of functions 
\beq\label{eqn:ssp}
H_{s,t}^{\ssp}(x) = (1-\rho(s))H_t(x) + \rho(s) C^{\ssp}.
\eeq
Denote the associated vector-field-valued $1$-form on $\Theta$ by $X_{H}^{\ssp} \otimes dt$, then the SSP equation is written as
$$
\frac{\partial u}{\partial s} + J \left( \frac{\partial u}{\partial t} - X_H^\ssp (u) \right) = 0,\ {\rm where}\ u \in C^\infty(\Theta, M).
$$
Because $X_{H}^{\ssp}dt$ vanishes for $s \geq 0$, any solution $u$ extends to a $J$-holomorphic map near $+ \infty$, which determines a cap for the periodic orbit at $- \infty$. Given $p\in \widetilde{\per}(H)$, we cam similarly consider ``stable SSP trajectories" which constitute a compact moduli space
$$ \ov{\mc{M}}{}_{p \bullet}^{\ssp}. $$
For any representative $u$ of a point in $\ov{\mc{M}}_{p \bullet}^{\ssp}$, its topological energy is defined to be
$$ d_{p \bullet} := C^\ssp - \mc{A}_{H}(p). $$
Due to the choice \eqref{eqn:SSP+}, we see that $d_{p \bullet} > 0$ as long as $\ov{\mc{M}}_{p \bullet}^{\ssp} \neq \emptyset$. Now suppose $\uds x \in {\rm crit}(f)$, let $W^s(\uds x) \subset M$ be the stable submanifold of $\nabla^g f$ and denote by $\ov{W^s(\uds x)}$ its compactification. Introduce the moduli space 
$$ \ov{\mathcal{M}}^{\ssp}_{p \uds x} := \ev_{\bullet}^{-1}(\ov{W^s(\uds x)}),$$
where $\ev_{\bullet}: \ov{\mc{M}}_{p \bullet}^{\ssp} \to M$ is the evaluation map at $+ \infty$. If $x = (a, \uds x)\in {\mc P}^\morse$ and $p \in {\mc P}^\floer$, then define 
\beqn
\ov{\mc M}{}_{px}^\ssp:= \ov{\mc M}{}_{a \cdot p \ \uds x}^\ssp.
\eeqn
Similar to the PSS case, the homogeneous posets associated with SSP spaces are denoted by
\beqn
\bA_{px}^\ssp := \{p q_1 \cdots q_l y_k \cdots y_1 x |\ p < q_1 < \cdots < q_l, y_k < \cdots < y_1 < x \}
\eeqn
for $p \in {\mc P}^\floer$ and $x \in {\mc P}^\morse$, endowed with depth function
$$
\dep(p q_1 \cdots q_l y_k \cdots y_1 x) = k + l .
$$
$\ov{\mc M}{}_{px}^\ssp$ is stratified by $\bA_{px}^\ssp$ with corresponding strata
\beqn
\left( \ov{\mc M}{}_{px}^\ssp \right)_{p q_1 \cdots q_l y_k \cdots y_1 x} \cong \ov{\mc M}{}_{p q_1}^\floer \times \cdots \times \ov{\mc M}{}_{q_{l-1} q_l}^\floer \times \ov{\mc M}{}_{q_l y_k}^\ssp \times \ov{\mc M}{}_{y_k y_{k-1}}^\morse \times \cdots \times \ov{\mc M}{}_{y_1 x}^\morse.
\eeqn

\subsubsection{The bimodule structure}

The PSS and SSP moduli spaces can be packaged into two flow bimodules (see Definition \ref{defn:flow-bimod}). Recall that one has the Floer flow category $T^{\rm Floer}$ and the Morse flow category $T^\morse$. Essentially by the way we compactify the PSS resp. SSP moduli spaces we see that one can define a flow bimodule $M^\pss$ from $T^\morse$ to $T^{\rm Floer}$ and a flow bimodule $M^\ssp$ from $T^\floer$ to $T^\morse$. We explain the specific terms for the PSS bimodule; the SSP case is completely symmetric.

\begin{prop}
For $x \in {\mc P}^{\morse}$ and $p \in {\mc P}^\floer$, define $M^{\pss}_{xp} := \ov{\mc M}{}_{xp}^\pss$. Then together with the natural inclusion of the boundary strata
$$ \ov{\mc M}{}^{\morse}_{xy} \times \ov{\mc M}{}_{yp}^\pss \to \ov{\mc M}{}_{xp}^\pss,$$
$$
\ov{\mc M}{}_{xq}^\pss \times \ov{\mc M}{}_{qp}^\floer \to \ov{\mc M}{}_{xp}^\pss,
$$
the spaces $M^{\pss}_{xp}$ define a flow bimodule from $T^\morse$ to $T^\floer$.
\end{prop}
\begin{proof}
For any $x \in {\mc P}^{\morse}$ and $p \in {\mc P}^\morse$ the moduli space $M^{\pss}_{xp}$ is compact due to Gromov compactness and the bound on the topological energy. For such $x$ and $p$, indeed the difference of energies ${\mc A}^\morse (x) - {\mc A}^\floer (p)$ is uniformly bounded. The associaticity of inclusions of boundary strata follows from the construction, so is the strict $\Pi$-equivariance property.
\end{proof}

We refer to such a bimodule as the {\bf PSS bimodule}, and the version for the SSP spaces as {\bf SSP bimodule}, denoted by $M^\pss$ and $M^\ssp$ respectively.

%This PSS moduli spaces is stratified as follows. For each pair $x$ and $p$, define 
%\beqn
%{\mf m}_{xp}^\pss:= \left\{ xy_1 \cdots y_l q_1 \cdots q_k p\ |\ \ov{\mc M}{}_{xy_1}^\morse, \ldots, \ov{\mc M}{}_{y_{l-1} y_l}^\morse, \ov{\mc M}{}_{y_l q_1}^\pss, \ov{\mc M}{}_{q_1 q_2}^\floer, \ldots, \ov{\mc M}{}_{q_k p}^\floer \neq \emptyset \right\}.
%\eeqn
%There are also natural maps of homogeneous posets
%\begin{align*}
%&\ {\mf b}_-^\pss: {\mf t}{}_{xy}^\morse \times {\mf m}{}_{yp}^\pss \to \partial^{[1]} {\mf m}{}_{xp}^\pss,\ &\ {\mf b}_+^\pss: {\mf m}{}_{xq}^\pss \times {\mf t}{}_{qp}^\floer \to \partial^{[1]} {\mf m}{}_{xp}^\pss.
%\end{align*}
%The collection $\{ ({\mf m}{}_{xp}^\pss), {\mf b}_-^\pss, {\mf b}_+^\pss\}$ form a continuation relation (see Definition \ref{defn_continuation_relation}). Moreover, one can easily see that there are natural homeomorphisms 
%\beqn
%\ov{\mc M}{}_{xy}^\morse \times \ov{\mc M}{}_{yp}^\pss \cong \partial^{xyp} \ov{\mc M}{}_{xp}^\pss
%\eeqn
%and 
%\beqn
%\ov{\mc M}{}_{xq}^\pss \times \ov{\mc M}{}_{qp}^\floer \to \partial^{xqp} \ov{\mc M}{}_{xp}^\pss
%\eeqn
%of stratified spaces. Then the collection of moduli spaces $\ov{\mc M}{}_{xp}^\pss$ form a $(T^\morse, T^\floer)$-bimodule, called the {\bf PSS bimodule}. In a completely symmetric way, the collection of SSP moduli spaces $\ov{\mc M}{}_{px}^\ssp$ form a $(T^{\rm Floer}, T^\morse)$-bimodule, called the {\bf SSP bimodule}.

\subsection{The pearly bimodule}

We consider the moduli space of (parametrized) $J$-holomorphic maps $u: \mb{CP}^1 \to M$ with two marked points $z_- = 0$ and $z_+ =\infty$. Given a homology class $A \in H_2(M; {\mb Z})$, let ${\mc M}{}_{0,2}^{\rm pearl} (M, J; A)$ be the moduli space of (parametrized) $J$-holomorphic maps %\textcolor{red}{with a lateral line} 
representing the class $A$ whose domains are smooth. It has a Gromov compactification $
\ov{\mc M}{}_{0,2}^\pearl (M, J; A)$. Each element of this compactification is represented by a stable map whose domain has a distinguished component whose parametrization is fixed. Moreover, by intersecting with the unstable manifold of $\uds x \in {\rm crit} f$ at $z_-$ and with the stable manifold of $\uds y \in {\rm crit} f$ at $z_+$, we have a moduli space 
\beqn
\ov{\mc M}{}_{\uds x \uds y}^\pearl(M, J; A)
\eeqn
For $a \in \Pi$, define 
\beqn
\ov{\mc M}{}_{\uds x \uds y}^\pearl(a):= \bigcup_{\omega(A) = \omega(a)} \ov{\mc M}{}_{\uds x\uds y}^\pearl(M, J; A).
\eeqn
Then given $x, y \in {\mc P}^\morse$ which can be written as $(a, \uds x)$ and $(b, \uds y)$ where $a, b \in \Pi$ and $\uds x, \uds y \in {\rm crit} f$, define 
\beqn
\ov{\mc M}{}_{xy}^\pearl:= \ov{\mc M}{}_{\uds x \uds y}^\pearl (b-a).
\eeqn

The pearly moduli spaces are stratified by configurations which have broken Morse trajectories either on the incoming edge or on the out-going edge. Indeed, for ${\mc P} = {\mc P}' = {\mc P}^\morse$, as in Notation \ref{notation35} there is a system of homogeneous posets indexed by pairs of $x, y \in {\mc P}^\morse$. More explicitly, given $x, y \in {\mc P}^\morse$, define
\beqn
\bA_{xy}^\pearl:= \left\{ x x_1 \cdots x_k y_l \cdots y_1 y \ \left| \ \begin{array}{c} x_i = (a, \uds x_i),\ y_j = (b, \uds y_j),\\
\uds x < \uds x_1 < \cdots < \uds x_k,\ \uds y_l < \cdots < \uds y_1 < \uds y  \end{array} \right. \right\}.
\eeqn
Note that all the $x_i$'s resp. $y_j$'s are marked with the same class $a \in \Pi$ resp. $b \in \Pi$, because otherwise, the morphism space $T^\morse_{x_i x_{i+1}}$ or $T^\morse_{y_j y_{j+1}}$ is empty by the construction in Section \ref{subsec:morse-flow}. Given $\alpha = xx_1 \cdots x_k y_l \cdots y_1 y\in \bA_{xy}^\pearl$, the corresponding stratum in the pearly moduli space is 
\beqn
\Big( \ov{\mc M}{}_{xy}^\pearl \Big)_\alpha \cong \ov{\mc M}{}_{xx_1}^\morse \times \cdots \times \ov{\mc M}{}_{x_{k-1} x_k}^\morse \times \ov{\mc M}{}_{x_k y_l}^\pearl \times \ov{\mc M}_{y_l y_{l-1}}^\morse \times \cdots \times \ov{\mc M}{}_{y_1 y}^\morse.
\eeqn
There are again only finitely many nonempty strata, thanks to Gromov compactness. Using the formulation of flow bimodules, we can see that the collection of pearly moduli spaces and the product structures of various strata provide a flow bimodule from $T^\morse$ to $T^\morse$, which we call the {\bf pearly bimodule} and denote it by $M^\pearl$. The strict $\Pi$-action on $M^\pearl$ follows from the construction. The outer-collaring (of width 1) of $M^\pearl$ provides a bimodule $(M^\pearl)^+$ from $(T^\morse)^+$ to $(T^\morse)^+$.

\subsection{Main Theorems}\label{subsec-main}
Using the notations introduced above, we can state the output of Sections \ref{sec-5}, \ref{sec-6}, and \ref{sec:pss}.

\begin{thm}\label{thm:dorb-lift}
Denote by $(T^\floer)^+$ the outer-collaring (see Section \ref{sec:outer-collar}) of the Hamiltonian Floer flow category $T^\floer$. Then $(T^\floer)^+$ admits a derived orbifold lift (Definition \ref{defn:flow-lift})
\beqn
{\mf D}^\floer = \Big( \big\{ C^\floer_{pq} = ({\mc U}^\floer_{pq}, {\mc E}^\floer_{pq}, {\mc S}^\floer_{pq}, \psi^\floer_{pq}) \big\}_{p< q}, \big\{{\bm \iota}_{\beta \alpha}^\floer \big\}_{\alpha \leq \beta} \Big),
\eeqn
with a collar structure (Definition \ref{defn:collar})
\beqn
\{ \widehat\theta_{\beta\alpha}^{{\rm collar},\floer} \}_{\alpha \leq \beta}, 
\eeqn
and a scaffolding (Definition \ref{defn_scaffolding})
\beqn
\big( {\mc F}_{\beta\alpha}^\floer, {\bm \theta}_{\beta\alpha}^\floer \big)_{\alpha \leq \beta}
\eeqn
such that they are compatible (Definition \ref{defn329}). Moreover, such a lift can be upgraded to an oriented and normally complex derived orbifold lift (Definition \ref{defn:orient-complex-lift}). Further, one can equip the D-chart lift a straightening which is compatible with the collar and scaffoldings.
\end{thm}

The proof is completed in Theorem \ref{smoothing_theorem}.

\begin{proof}[Proof of Theorem \ref{thm-intro-floer}]
From Theorem \ref{thm:dorb-lift}
% and Lemma \ref{lem:straight}, 
this is a corollary of Theorem \ref{thm_FOP_2} and the discussion in Section \ref{sec:chain}.
\end{proof}

\begin{thm}\label{thm:pss-ssp}
Let $(M^\pss)^+$ be the outer-collaring of the flow bimodule $M^\pss$ from $(T^\morse)^+$ to $(T^\floer)^+$. Then $(M^\pss)^+$ has an oriented and normally complex derived orbifold lift 
\beqn
{\mf D}^\pss = \Big( \big\{ C^\pss_{xp} = ({\mc U}^\pss_{xp}, {\mc E}^\pss_{xp}, {\mc S}^\pss_{xp}, \psi^\pss_{xp}) \big\}_{x,p}, \big\{{\bm \iota}_{\beta \alpha}^\pss \big\}_{\alpha \leq \beta} \Big),
\eeqn
with a collar structure
\beqn
\{ \widehat\theta_{\beta\alpha}^{{\rm collar},\pss} \}_{\alpha \leq \beta}, 
\eeqn
and a scaffolding (Definition \ref{defn_scaffolding})
\beqn
\big( {\mc F}_{\beta\alpha}^\pss, {\bm \theta}_{\beta\alpha}^\pss \big)_{\alpha \leq \beta}
\eeqn
such that they are compatible, and they extend the given structures on ${\mf D}^\floer$ and ${\mf D}^\morse$. The same statement holds by reversing the role of  $(T^\floer)^+$ to $(T^\morse)^+$ and replace ${\rm PSS}$ by ${\rm SSP}$.
\end{thm}

The proof is provided in Section \ref{sec:pss}. The following assertion is proved similarly as the arguments in Section \ref{subsubsec:straight}.

\begin{lemma}
Both ${\mf D}^\pss$ and ${\mf D}^\ssp$ have a straightening extending the straightening on ${\mf D}^\floer$, and they are compatible with the respective collar structure and scaffolding. \qed
\end{lemma}

As a consequence, Theorem \ref{thm_FOP_3} and the discussion in Section \ref{sec:chain} define the $\Lambda$-linear chain maps
$$
\Psi^{\rm PSS}: CM_{*}(f;\Lambda) \to CF_{*}(H;\Lambda),
$$
$$
\Psi^{\rm SSP}: CF_{*}(H;\Lambda) \to CM_{*}(f;\Lambda).
$$

\begin{thm}
Let $(M^\pearl)^+$ be the flow bimodule from $(T^\morse)^+$ to $(T^\morse)^+$ obtained from the outer-collaring of the pearly bimodule $M^\pearl$. Then $(M^\pearl)^+$ has an oriented and normally complex derived orbifold lift
\beqn
{\mf D}^\pearl = \Big( \big\{ C^\pearl_{xy} = ({\mc U}^\pearl_{xy}, {\mc E}^\pearl_{xy}, {\mc S}^\pearl_{xy}, \psi^\pearl_{xy}) \big\}_{x,y}, \big\{{\bm \iota}_{\beta \alpha}^\pearl \big\}_{\alpha \leq \beta} \Big),
\eeqn
with a collar structure
\beqn
\{ \widehat\theta_{\beta\alpha}^{{\rm collar},\pearl} \}_{\alpha \leq \beta}, 
\eeqn
such that the induced derived orbifold presentation on each $\ov{\mc M}{}^\pearl_{xy}$ is single-layered (Definition \ref{defn:single-layer}), i.e., ${\mf D}^\pearl$ has a trivial scaffolding. Moreover, the orientation structure, normally complex structure, and the collar structure extend the existing ones on ${\mf D}^\morse$.
\end{thm}

The proof is provided in Section \ref{sec:pss}. By the argument in Section \ref{subsubsec:straight} (but simpler because the scaffolding is absent in this case), we can find a straightening of ${\mf D}^\pearl$ which is compatible with the collar structure. We fix it once for all.

\begin{cor}\label{cor:unitriangular}
The FOP countings associated to the oriented and normally complex derived orbifold lift ${\mf D}^\pearl$ and the chosen straightening thereon defines a chain map 
\beqn
\Psi^\pearl: CM_*(f;\Lambda) \to CM_*(f;\Lambda)
\eeqn
which is unitriangular, i.e., 
\beqn
\Psi^\pearl - {\rm Id} \in ({\rm End}_\Lambda (CM_*(f;\Lambda)))_+.
\eeqn
In particular, the induced map on homology
\beqn
\Psi^\pearl: H_*(M; \Lambda) \to H_*(M; \Lambda)
\eeqn
is invertible.
\end{cor}
\begin{proof}
The construction of $\Psi^\pearl$ follows from Theorem \ref{thm_FOP_3} and discussions in Section \ref{subsection37}. To prove that $\Psi^\pearl$ is unitriangular, observe that for a pair of capped orbits $x=(\uds{x}, a)$ and $y=(\uds{y},b)$, the moduli space $\ov{\mc M}^\pearl_{xy}$ is nonempty only if $\omega(b-a) \geq 0$. Moreover, when $\omega(b-a) = 0$, the only nonempty moduli space contributing to the counting in $\Psi^\pearl(x)$ is the moduli space of parametrized gradient flow lines from $x$ to itself, which is a single point. It implies that the incidence coefficient $n^\pearl_{xx} = 1$. As a consequence, the statement is proved.
\end{proof}

\subsection{The homotopy}\label{subsec:hmtp}

In this subsection, we describe how to interpolate between the chain maps $\Psi^\ssp \circ \Psi^\pss$ and $\Psi^\pearl$ and prove Theorem \ref{thm-intro-pss}.

We choose a 3-parameter family of Hamiltonians parametrized by $t \in S^1$, $s \in {\mb R}$, and $\tau \in (-1, 1]$ which satisfy the following conditions.
\begin{enumerate}
    \item For each $\tau$, $H_{\tau, s, t}^\hmtp$ is equal to $C^\pss$ for $s$ near $-\infty$ and is equal to $C^\ssp$ for $s$ near $+\infty$.

    \item For all $\tau, s, t, x$ there holds 
    \beqn
    \frac{\partial H_{\tau, s, t}^\hmtp}{\partial s}(x) \geq 0.
    \eeqn
    
    \item As $\tau \to -1$, $H_{\tau, s, t}^\hmtp$ converges to the concatenation of $H_{s, t}^\pss$ and $H_{s, t}^\ssp$. 
    
    \item For $\tau$ near $+1$, $H_{\tau, s, t}^\hmtp$ does not depend on $t\in S^1$ and $x \in M$. In particular, the associated Hamiltonian vector field is zero.
\end{enumerate}

Then consider pairs $(\tau, u)$ where $\tau \in (-1, 1]$ and $u: \Theta \to M$ solving the equation 
\beq\label{eqn_homotopy}
\frac{\partial u}{\partial s} +  J \left( \frac{\partial u}{\partial t}  - X_{H_{\tau, s, t}^\hmtp}(u) \right) = 0.
\eeq
The energy of a solution is defined by 
\beqn
E^\hmtp(\tau, u):= \left\| \partial_s u \right\|_{L^2(\Theta)}^2.
\eeqn
As the Hamiltonian perturbation vanishes near $s = \pm \infty$, any finite energy solution converges at the infinities. Hence one can use an element of $\Pi$ to label solutions. Given $a \in \Pi$, let 
\beqn
{\mc M}{}_{\bullet\bullet}^\hmtp(a)
\eeqn
to be the set of solutions to \eqref{eqn_homotopy} whose $\Pi$-class is $a$. Moreover, by intersecting with (un)stable manifolds of Morse critical points, we can define 
\beqn
{\mc M}{}_{\uds x \uds y}^\hmtp(a)
\eeqn
to be the set of such solutions whose limit at $-\infty$ resp. $+\infty$ lies in the unstable resp. stable manifold of $\uds x$ resp. $\uds y$. Lastly, given $x = (a, \uds x), y = (b, \uds y) \in {\mc P}^\morse$, define 
\beqn
{\mc M}{}_{xy}^\hmtp:= {\mc M}{}_{\uds x \uds y}^\hmtp( b-a).
\eeqn

We compactify this space by adding configurations with sphere bubbles, broken Morse trajectories, and (when $\tau \to -1$) breakings at 1-periodic orbits of $H$. Denote the compactification by
\beqn
\ov{\mc M}{}_{xy}^\hmtp.
\eeqn
We call such moduli spaces the {\bf homotopy moduli spaces}.

\subsubsection{Stratifications on homotopy moduli spaces}

The system of the homotopy moduli spaces can be put in a more abstract narrative to include {\it bimodule compositions} and {\it bimodule homotopies}. However, as its role in the proof of the Arnold conjecture is technical rather than conceptual, we refrain from introducing such frameworks. Here we give a more concrete description of the stratifications on these moduli spaces. Fix $x, y \in {\mc P}^\morse$. We define a poset
\beqn
\bA_{xy}^\hmtp:= \bA_{xy}^{\pss + \ssp} \sqcup \mathring \bA_{xy}^\hmtp \sqcup \bA_{xy}^\pearl.
\eeqn
Here as sets
\beqn
\mathring \bA_{xy}^\hmtp = \bA_{xy}^\pearl
\eeqn
while as sets
\beqn
\bA_{xy}^{\pss + \ssp}:= \Big\{ x x_1 \cdots x_k p_1 \cdots p_s y_l \cdots y_1 y\ |\ xx_1 \cdots x_k p_1 \cdots p_s \in \bA_{xp_s}^\pss,\ p_s y_l \cdots y_1 y \in \bA_{p_s y}^\ssp \Big\}.
\eeqn
Then one can see that the moduli space $\ov{\mc M}{}_{xy}^\hmtp$ has a natural stratification whose strata are indexed by the set $\bA_{xy}^\hmtp$. This then naturally induces a partial order on $\bA_{xy}^\hmtp$ according to the breaking of Morse or Floer trajectories as well as whether the parameter $\tau$ hits $-1$ or $1$. We use $\kappa, \nu$ etc. instead of $\alpha, \beta$ to denote elements of $\bA_{xy}^\hmtp$. Then it is straightforward to check that $\bA_{xy}^\hmtp$ is a homogeneous poset with a unique maximal element and depth function being
\beqn
\dep(\kappa) = \left\{ \begin{array}{ll} {\rm number\ of\ breakings},\ &\ \kappa \in \mathring \bA_{xy}^\hmtp,\\
                                           {\rm number\ of\ breakings}+1,\ &\ \kappa \in \bA_{xy}^\pearl,\\
                                           {\rm number\ of\ breakings},\ &\ \kappa \in \bA_{xy}^{\pss + \ssp}. \end{array} \right.
\eeqn

The types of codimension one (i.e., strata on which the depth function takes value $1$) degenerations of configurations in the homotopy moduli spaces can be described by the following types of stratified embeddings.
\begin{enumerate}
    \item When the parameter $\tau$ hits $+1$, there is a stratified embedding
    \beq\label{hmtp_map_1}
    \vcenter{ \xymatrix{ \ov{\mc M}{}_{xy}^\morse \ar[r] \ar[d] & \partial^+ \ov{\mc M}{}_{xy}^\hmtp \ar[d] \\
                \bA_{xy}^\pearl \ar[r] &                \partial^+ \bA_{xy}^\hmtp}}.
                \eeq
                
                \item When the parameter $\tau$ hits $-1$, for each $p \in {\mc P}^\floer$, there is a stratified embedding 
                \beq\label{hmtp_map_2}
                \vcenter{ \xymatrix{ \ov{\mc M}{}_{xp}^\pss \times \ov{\mc M}{}_{py}^\ssp \ar[r] \ar[d] &  \partial^{xpy} \ov{\mc M}{}_{xy}^\hmtp\ar[d]\\
                           \bA_{xp}^\pss \times \bA_{py}^\ssp \ar[r] &  \partial^{xpy} \bA_{xy}^\hmtp} }.
                           \eeq
                           
                           \item When a Morse trajectory breaks off ``on the left'' there is a stratified embedding 
                           \beq\label{hmtp_map_3}
                           \vcenter{  \xymatrix{  \ov{\mc M}{}_{xx'}^\morse \times \ov{\mc M}{}_{x'y}^\hmtp \ar[r] \ar[d] &    \partial^{xx'y} \ov{\mc M}{}_{xy}^\hmtp \ar[d]\\
                                        \bA_{xx'}^\morse \times \bA_{x'y}^\hmtp \ar[r] & \partial^{xx'y} \bA_{xy}^\hmtp }  }.
                                        \eeq
                                        
    \item When a Morse trajectory breaks off ``on the right'' there is a stratified embedding 
                           \beq\label{hmtp_map_4}
                           \vcenter{  \xymatrix{  \ov{\mc M}{}_{xy'}^\hmtp \times \ov{\mc M}{}_{y'y}^\morse \ar[r] \ar[d] &    \partial^{xy'y} \ov{\mc M}{}_{xy}^\hmtp \ar[d]\\
                                        \bA_{xy'}^\hmtp \times \bA_{y'y}^\morse \ar[r] & \partial^{xy'y} \bA_{xy}^\hmtp }  }.
                                        \eeq
\end{enumerate}
The above maps satisfy a list of obvious associativity properties, which we do not describe explicitly here.

\subsubsection{Outer collaring}

We can also apply the general outer-collaring construction on the homotopy moduli spaces so that it naturally extends the outer-collaring of the Floer, Morse, PSS, and SPP moduli spaces. The same type of structure maps as listed in \eqref{hmtp_map_1}---\eqref{hmtp_map_4} are still present and satisfy the same associativity properties.

\subsubsection{Derived orbifold lift}

\begin{defn}\label{defn_hmtp_lift}
Assume the following objects are given.
\begin{enumerate}
\item A derived orbifold lift ${\mf D}^\floer$ of the outer-collared Hamiltonian Floer flow category $(T^\floer)^+$. (Remember there is also the trivial derived orbifold lift ${\mf D}^\morse$ of the outer-collared Morse flow category $(T^\morse)^+$). 

\item A derived orbifold lift ${\mf D}^\pss$ resp. ${\mf D}^\ssp$ of the outer-collared PSS resp. SSP bimodule which extends ${\mf D}^\floer$ and ${\mf D}^\morse$. 

\item A derived orbifold lift ${\mf D}^\pearl$ of the outer-collared pearly bimodule $(M^\pearl)^+$.

\end{enumerate}
Then a {\bf derived orbifold lift} of the system of the outer-collared homotopy moduli spaces $(\ov{\mc M}{}_{xy}^\hmtp)^+$ consists of the following objects. 
\begin{enumerate}
    \item For each $x, y$, an $\bA_{xy}^\hmtp$-stratified derived orbifold chart $C_{xy}^\hmtp$ of $(\ov{\mc M}{}_{xy}^\hmtp)^+$.
    
    \item Stratified chart embeddings
    \beq
    \vcenter{ \xymatrix{ C_{xy}^\pearl  \ar[r] \ar[d] & \partial^+ C_{xy}^\hmtp\ar[d]\\
              \bA_{xy}^\pearl \ar[r] &     \partial^+ \bA_{xy}^\hmtp} },
    \eeq
    \beq
    \vcenter{ \xymatrix{ C_{xp}^\pss \times C_{py}^\ssp \ar[r] \ar[d] & \partial^{xpy} C_{xy}^\hmtp \ar[d]\\
                         \bA_{xp}^\pss \times \bA_{py}^\pss \ar[r] & \partial^{xpy} \bA_{xy}^\hmtp} },
    \eeq
    \beq
    \vcenter{  \xymatrix{   C_{xx'}^\morse \times C_{x'y}^\hmtp \ar[r] \ar[d]  & \partial^{xx'y} C_{xy}^\hmtp \ar[d] \\
                            \bA_{xx'}^\morse \times \bA_{x'y}^\hmtp \ar[r] & \partial^{xx'y} \bA_{xy}^\hmtp} },
    \eeq
    and 
    \beq
    \vcenter{ \xymatrix{ C_{xy'}^\hmtp \times C_{y'y}^\morse \ar[r] \ar[d] &  \partial^{xy'y} C_{xy}^\hmtp \ar[d] \\
                          \bA_{xy'}^\hmtp \times \bA_{y'y}^\morse \ar[r] & \partial^{xy'y} \bA_{xy}^\hmtp } }.
    \eeq
\end{enumerate}
These objects need to satisfy the following conditions.
\begin{enumerate}
    \item The obvious associativity relations.
    
    \item For each $\kappa \in \bA_{xy}^\hmtp$, there is a product chart $C_\kappa^\hmtp$ obtained from taking corresponding product of certain Morse, Floer, PSS, SSP, pearly, and/or homotopy derived orbifold charts. The above associativity relations induce chart embeddings 
    \beqn
    {\bm \iota}_{\nu\kappa}^\hmtp: C_\kappa^\hmtp \to \partial^\kappa C_\nu^\hmtp\ \forall \kappa \leq \nu.
    \eeqn
    Then $((C_\kappa^\hmtp)_{\kappa \in \bA_{xy}^\hmtp}, ({\bm \iota}_{\nu\kappa})_{\kappa \leq \nu})$ form a derived orbifold presentation of the space $(\ov{\mc M}{}_{xy}^\hmtp)^+$. 
    
    \item When $\kappa \notin \partial^- \bA_{xy}^\hmtp$, the chart embedding ${\bm \iota}_{\nu\kappa}^\hmtp$ is an open embedding. 
    
    \item The obvious strict $\Pi$-equivariance condition. 
\end{enumerate}
\end{defn}

We also need to discuss the additional structures (collar, scaffolding, and straightening) on the derived orbifold lift of the system of homotopy moduli spaces. 

\begin{defn}
In addition to the assumption of Definition \ref{defn_hmtp_lift}, assume the following objects are given.
\begin{enumerate}
    \item A compatible package of additional structures on ${\mf D}^\floer$. Remember that there is also a package of additional structures on ${\mf D}^\morse$ which essentially contains only the collar structure.
    
    \item A compatible package of additional structures on ${\mf D}^\pss$ resp. ${\mf D}^\ssp$ which extend the ones on ${\mf D}^\floer$ and ${\mf D}^\morse$.
    
    \item A compatible package of additional structures on ${\mf D}^\pearl$ which extends the package on ${\mf D}^\morse$.
\end{enumerate}
Let ${\mf D}^\hmtp$ be a derived orbifold lift of the system of homotopy moduli spaces which extends all given derived orbifold lifts. 
\begin{enumerate}

\item A {\bf collar structure} resp. {\bf straightening} on ${\mf D}^\hmtp$ is a collection of collar structures resp. straightenings on the derived orbifold charts $C_{xy}^\hmtp$ satisfying the following condition. Fix $x, y$. For each $\kappa \in \bA_{xy}^\hmtp$, the product of collar structures resp. straightenings on all derived orbifold chart factors defines a collar structure resp. straightening on the product chart $C_\kappa^\hmtp$. Then the collection of collar structures resp. straightenings for all $\kappa \in \bA_{xy}^\hmtp$ is a collar structure resp. straightening of the derived orbifold presentation of $(\ov{\mc M}{}_{xy}^\hmtp)^+$.

\item A {\bf scaffolding} on ${\mf D}^\hmtp$ consists of a collection of scaffoldings
\beqn
\Big( \big( {\mc F}_{\nu\kappa}, {\bm \theta}_{\nu\kappa}\big)_{\kappa \leq \nu} \Big)_{x, y \in {\mc P}^\morse}
\eeqn
on the collection of derived orbifold presentations $((C_\kappa^\hmtp)_{\kappa \in \bA_{xy}^\hmtp}, ({\bm \iota}_{\kappa \leq \nu}))$ which satisfy the following conditions.
\begin{enumerate}
    \item If $\nu \in \partial^- \bA_{xy}^\hmtp$ (i.e. a stratum in the $\tau=-1$ slice of the moduli, i.e., a stratum of configurations with breakings at 1-periodic orbits of $H$), then $({\mc F}_{\nu\kappa}, {\bm \theta}_{\nu\kappa})$ is the product from corresponding difference bundles and stabilization maps on ${\mf D}^\floer$, ${\mf D}^\pss$, ${\mf D}^\ssp$, and/or the trivial one on ${\mf D}^\morse$.
    
    \item If $\kappa \notin \partial^- \bA_{xy}^\hmtp$, then ${\mc F}_{\nu\kappa} = 0$.
    
\end{enumerate}
\end{enumerate}
\end{defn}

Lastly we discuss orientations and normal complex structures on a system of derived orbifold lifts on homotopy moduli spaces. 

\begin{defn}\label{defn:homotopy-normal-C}
Given a derived orbifold lift ${\mf D}^\hmtp$, and a collar structure compatible with a scaffolding $\Big( \big( {\mc F}_{\nu\kappa}, {\bm \theta}_{\nu\kappa}\big)_{\kappa \leq \nu} \Big)_{x, y \in {\mc P}^\morse}$, assume that ${\mf D}^\floer$, ${\mf D}^\pss$, ${\mf D}^\ssp$, ${\mf D}^\morse$ and their scaffoldings and outer-collarings are endowed with a {\bf normal complex structure}. Then a normal complex structure on ${\mf D}^\hmtp$ with the given scaffolding and outer-collaring consists of
\begin{enumerate}
    \item A normal complex structure on all the underlying derived orbifold charts of $(C_\kappa^\hmtp)_{\kappa \in \bA_{xy}^\hmtp}$ such that the embeddings $({\bm \iota}_{\kappa \leq \nu})$ all intertwine with the normal complex structures.
    \item A complex structure on the vector bundles ${\mc F}_{\nu\kappa}$ such that the embeddings underlying ${\bm \theta}_{\nu\kappa}$ all respects the induced normal complex structures.
    \item The structural maps from the outer-collaring and the compatibility equations with the scaffolding intertwine with the normal complex structure as in Definition \ref{defn:normal-C}.
    \item The strict $\Pi$-action preserves the normal complex structures.
\end{enumerate}
\end{defn}

\begin{defn}
Let ${\mf D}^\floer$, ${\mf D}^\pss$, ${\mf D}^\ssp$, ${\mf D}^\morse, {\mf D}^\hmtp$ be as in Definition \ref{defn:homotopy-normal-C} with the respective collared structure, scaffolding, and normal complex structure. Moreover, assume that each of ${\mf D}^\floer$, ${\mf D}^\pss$, ${\mf D}^\ssp$, ${\mf D}^\morse$ is equipped with an orientation in the sense of Definition \ref{defn:orient} and Definition \ref{defn:orient-bimod}. Then an {\bf orientation} on the given normally complex lift of ${\mf D}^\hmtp$ consists of
\begin{enumerate}
\item For all $x,y$, an isomorphism of orientation lines
\begin{equation}\label{eqn:hmtp-sign}
\mathfrak{o}_{C_{xy}^\hmtp} \xrightarrow{\sim} \mathfrak{o}_x^{\vee} \otimes \mathfrak{o}_{y}.
\end{equation}
\item Isomorphisms of orientation lines
\begin{equation}\label{eqn:coh-hmtp}
\begin{aligned}
{\mf o}_{C_{xy}^\pearl} \xrightarrow{\sim} {\mf o}_{\partial^+ C^\hmtp_{xy}}, \\
{\mf o}_{C_{xp}^\pss} \otimes {\mf o}_{C_{py}^\ssp} \xrightarrow{\sim} {\mf o}_{\partial^{xpy}C_{xy}^\hmtp}, \\
{\mf o}_{C_{xx'}^\morse} \otimes {\mf o}_{C_{x'y}^\hmtp} \xrightarrow{\sim} {\mf o}_{\partial^{xx'y}C_{xy}^\hmtp}, \\
{\mf o}_{C_{xy'}^\hmtp} \otimes {\mf o}_{y' y}^\morse \xrightarrow{\sim} {\mf o}_{\partial^{xy'y}C^\hmtp_{xy}}.
\end{aligned}
\end{equation}
Here we omit the cancellation of the orientation lines of the complex scaffolding vector bundles and its dual.
\item The above isomorphisms are preserved by the $\Pi$-action.
\end{enumerate}
\end{defn}

\begin{thm}
There is a $\Lambda$-linear chain homotopy
$$
\Psi^\hmtp: CM_*(f;\Lambda) \to CM_{*-1}(f;\Lambda)
$$
such that 
\begin{equation}\label{eqn:hmtp-equiv}
\Psi^\pearl - \Psi^\ssp \circ \Psi^\pss = d^\morse \circ \Psi^\hmtp - \Psi^\hmtp \circ d^\morse.
\end{equation}
In particular, there holds
\begin{equation}\label{eqn:homology-equal}
\Psi^\ssp \circ \Psi^\pss = \Psi^\pearl
\end{equation}
as graded linear maps on $H_*(M; \Lambda)$.
\end{thm}

\begin{proof}
Using the description presented in Section \ref{sec:pss}, there exists a compatible oriented and normally complex derived orbifold lift ${\mf D}^\hmtp$ of the outer-collared homotopy moduli spaces $(\ov{\mc M}{}_{xy}^\hmtp)^+$, with collared structure and scaffolding. Then the relative version of the arguments in Section \ref{subsubsec:straight} can be used to equip ${\mf D}^\hmtp$ with a straightening compatible with the given straightenings on ${\mf D}^\floer$, ${\mf D}^\pss$, ${\mf D}^\ssp$, ${\mf D}^\morse$. Therefore, a relative version of Theorem \ref{thm_FOP_2} allows us to find a compatible system of strongly transverse FOP perturbations of the derived orbifold charts in ${\mf D}^\hmtp$ extending the FOP perturbations on ${\mf D}^\floer$, ${\mf D}^\pss$, ${\mf D}^\ssp$, ${\mf D}^\morse$. Using the sign read off from \eqref{eqn:hmtp-sign}, the map $\Psi^\hmtp$ is defined by the algebraic count of zeroes of the FOP sections over moduli spaces of virtual dimension $0$. The algebraic relation \eqref{eqn:hmtp-equiv} is again read off from the boundary of moduli spaces of virtural dimension $1$, using the coherence of orientations \eqref{eqn:coh-hmtp}. The $\Lambda$-linearity is a result of the strict $\Pi$-action as before.
\end{proof}

\begin{proof}[Proof of Theorem \ref{thm-intro-pss}]
This is a direct consequence of Corollary \ref{cor:unitriangular} and Equation \eqref{eqn:homology-equal}.
\end{proof}

\subsection{Proof of the integral Arnold conjecture}\label{subsec-algebra}

\begin{lemma}(cf. \cite[Appendix]{Harvey_Minervini_2006})
The Novikov ring $\Lambda$ is a principal ideal domain (PID).
\end{lemma}

\begin{proof}
By definition, we need to show that each ideal is a principal ideal, i.e., generated by a single element. Let $I \subseteq \Lambda$ be an ideal. Denote by $a_0 \in {\mb Z}_+$ the greatest common divisor of the leading coefficients of all elements of $I$. Then there exists ${\mf x}_0 \in I$ of the form ${\mf x}_0 = a_0 + {\mf y}_0$ with ${\mf y}_0 \in \Lambda_+$. We claim that 
\beqn
I = ({\mf x}_0).
\eeqn
Indeed, given any ${\mf x} \in I$, $a_0$ divides the leading order term of ${\mf x}$. We can inductively find an element ${\mf q}\in \Lambda$ such that ${\mf x} = {\mf q} {\mf x}_0$. Hence $I = ({\mf x}_0)$ and $\Lambda$ is a PID.
\end{proof}

Recall the standard structure results about finitely-generated modules over principal ideal domains. If $R$ is a principal ideal domain and $M$ is a finitely-generated $R$-module, then 
\beqn
M \cong F \oplus T
\eeqn
where $F$ is a free $R$-module of finite rank and $T$ is a torsion $R$-module. Moreover, 
\begin{enumerate}
    \item There exist primes $p_1, \ldots, p_k \in R$ and positive integers $m_1, \ldots, m_k$ such that
    \beqn
    T \cong R/ (p_1^{m_1}) \oplus \cdots \oplus R/ (p_k^{m_k} ).
    \eeqn
    The prime powers $p_i^{m_i}$ are called the {\it elementary divisors}.

    \item There exist nonzero nonunit elements $r_1, \ldots, r_l$ of $R$ such that $r_1 | r_2 |\cdots |r_l$ and
    \beqn
    T \cong R/ ( r_1 ) \oplus \cdots \oplus R/ (r_l ).
    \eeqn
    The elements $r_i$ are called the {\it invariant factors}.
\end{enumerate}

\begin{lemma}\label{lemma413}
Let $R$ be a PID. Let $Z$ be a free $R$-module of finite rank and $S \subseteq Z$ be a submodule. Suppose 
\beqn
Z/ S \cong F \oplus R/ (a_1) \oplus \cdots \oplus R/ (a_k)
\eeqn
where $F$ is a free $R$-module and $a_1|a_2|\cdots |a_k$ are invariant factors. Then 
\beqn
{\rm rank} Z \geq {\rm rank} F + k
\eeqn
and 
\beqn
{\rm rank} S  \geq k.
\eeqn
\end{lemma}

\begin{proof}
By tensoring with the quotient field of $R$, we can kill all torsions. Then it follows that ${\rm rank} Z - {\rm rank} S = {\rm rank} F$. Hence it suffices to prove the first lower bound. Choose a prime divisor $p$ of $a_1$. By tensoring with $R_p:= R/ (p)$, we can see that 
\beqn
Z\otimes_R R_p / S \otimes_R R_p \cong R_p^{{\rm rank} F + k}.
\eeqn
Therefore,
\beqn
{\rm rank} Z = {\rm dim}_{R_p} Z\otimes_R R_p \geq {\rm rank} F + k.\qedhere
\eeqn
\end{proof}

\begin{lemma}\label{lemma_torsion}
Let $M$ be a finitely-generated $R$-module and $N\subseteq M$ be a submodule. Then
\begin{enumerate}
    \item ${\rm rank} N \leq {\rm rank} M$.
    
    \item For each prime $p\in R$, the number of elementary divisors of $N$ which are powers of $p$ is less than or equal to the number of elementary divisors of $M$ which are powers of $p$.
    
    \item The number of invariant factors of $N$ is not greater than the number of invariant factors of $M$.
\end{enumerate}
\end{lemma}

\begin{proof}
By the structure theorem of finitely-generated modules over PID, we can write 
\begin{align*}
&\ M \cong F \oplus \bigoplus_{i=1}^k R / (p_i^{m_i}),\ &\ N \cong F' \oplus \bigoplus_{j=1}^l R/ (q_j^{n_j})
\end{align*}
where $p_i$, $q_j$ are primes of $R$ and $m_i, n_j \geq 1$. By tensoring the quotient field of $R$, we can see that ${\rm rank} F' \leq {\rm rank} F$. On the other hand, let $M(p)\subseteq M$ be the direct sum of $p$-torsions. Then the inclusion $N \hookrightarrow M$ induces injections $N(p) \hookrightarrow M(p)$ for each prime $p$. Therefore, the collection of $q_j$'s is a subset of the collection of $p_i$'s. Then for each prime $p$, denote 
\beqn
M[p]:= \{ x \in M\ |\ px = 0\}
\eeqn
which is a submodule. Then the inclusion $N\hookrightarrow M$ induces an injection $N[p]\hookrightarrow M[p]$. Notice that $R/(p)$ is a field and $M[p]$ is a $R/(p)$-vector space. Clearly $N[p]$ is a subspace. Moreover, recall that 
\beqn
(R/(p^m))[p] \cong R/(p).
\eeqn
Hence it follows that the number of $p$-elementary divisors of $N$ is no greater than the number of $p$-elementary divisors of $M$. The relation between their numbers of invariant factors follows immediately.
\end{proof}

\begin{proof}[Proof of Theorem \ref{thm_intro_main}]
Now consider the two ${\mb Z} / 2N$-graded $\Lambda$-modules, $HF_*(H;\Lambda)$ and $H_{*}(M;\Lambda)$. Because $\Psi^\ssp \circ \Psi^\pss = \Psi^\pearl = Id + O(T)$, so for each $i\in {\mb Z}/{2N}$, the PSS map embeds $H_{i}(M;\Lambda)$ as a submodule of $HF_i(H;\Lambda)$. Therefore, for the free part, one has the rank inequality
\beqn
{\rm rank} HF_i(H; \Lambda) \geq  {\rm rank} H_*(M;\Lambda),\ i \in {\mb Z}/2N.
\eeqn
As of torsion, Lemma \ref{lemma_torsion} implies the inequality between number of invariant factors
\beqn
\tau_i^\floer \geq \tau_i^\morse,\ i \in {\mb Z}/2N
\eeqn
On the other hand, we know that
\beqn
HF_i (H;\Lambda)= \frac{ {\rm ker} d_i^\floer}{{\rm image} d_{i+1}^\floer}
\eeqn
where both the numerator and denominator are free $\Lambda$-modules. Hence by Lemma \ref{lemma413}, one has 
\beqn
{\rm rank} ({\rm ker} d_i^\floer) \geq {\rm rank} H_i(M;\Lambda) +  \tau_i^\morse
\eeqn
and
\beqn
{\rm rank}({\rm image} d_{i+1}^\floer) \geq \tau_i^\morse.
\eeqn
Therefore, let $\# {\rm Per} H$ be the number of $1$-periodic orbits of $H$, we have
\begin{multline*}
\# {\rm Per} H = {\rm rank} CF_*(H;\Lambda) = \sum_{i\in {\mb Z}/{2N}} {\rm rank} CF_i (H;\Lambda)\\
= \sum_{i\in {\mb Z}/{2N}} \big( {\rm rank} ({\rm ker} d_i^\floer) + {\rm rank} ({\rm image} d_i^\floer)\big)\\
\geq \sum_{i\in {\mb Z}/{2N}} {\rm rank} H_i(M; \Lambda) + 2 \sum_{i\in {\mb Z}/{2N}} \tau_i^\morse\\
= \text{rank} H_{*}(M;{\mb Q}) + 2\sum_{i \in {\mb Z}/2N} \tau_{i}^{(2N)}(M).
\end{multline*}

\end{proof}

\section{Global Kuranishi charts on Floer moduli spaces}\label{sec-5}

In this section we provide the details for the global Kuranishi chart construction of moduli spaces of Floer trajectories. The construction is primarily inspired by the construction by Abouzaid--McLean--Smith \cite{AMS} of the global chart on a single moduli space of genus zero stable maps. The main effort, however, is to generalize their construction to the case of the Hamiltonian Floer flow categories, where global charts must be constructed consistently for infinitely many moduli spaces. Certain care must be taken in order to have the expected properties. This and the next section serve as the construction of a derived orbifold lift of the Hamiltonian Floer flow category: this section deals with the topological construction while Section \ref{sec-6} deals with the smoothing.

\subsection{Basic notions and the main theorem about global charts}

To state the main theorem of this section, we introduce an alternative version of charts. Because our construction of FOP perturbations is carried over orbifolds instead of smooth manifolds with an almost-free action of a compact Lie group, we refrain to give a general treatment of ``global charts" beyond the concrete geometric setting in this paper.

%\subsubsection{Pseudovector bundles}

%Our consideration is within the category of prestratified topological spaces (see XXX). We will not only consider vector bundles but also more general bundles. A {\bf pseudovector bundle} over a space $M$ is a fiberbundle $E \to M$ whose fiber is (diffeomorphic to) a vector space and whose structure group is the group of self-diffeomorphisms fixing the origin. Then a pseudovector bundle has a canonical {\bf zero section}. We can still talk about the notion of direct sums of pseudovector bundles. If $E_1, E_2 \to M$ are two pseudovector bundles, then their direct sum is denoted by 
%\beq
%E_1 \psum E_2 \to M
%\eeq
%whose fiber is the product of the fibers of $E_1$ and $E_2$. It still has a canonical zero section.

%The generalization to the equivariant situation is also straightforward. Let $M$ be a $G$-space. A $G$-equivariant pseudovector bundle is a pseudovector bundle $E \to M$ together with a $G$-action on the total space such that the projection $\pi:E \to M$ is equivariant and such that the induced maps between fibers fix the origin. Then the zero section of a $G$-equivariant pseudovector bundle is an equivariant section. 

\subsubsection{Kuranishi charts}

We slightly generalize the usual notion of Kuranishi charts. 

\begin{defn}
Let $\bA$ be a countable homogeneous poset and $X$ be an $\bA$-space.

\begin{enumerate}
\item An $\bA$-stratified {\bf (topological) Kuranishi chart} (K-chart for short) on $X$ is a quintuple $(G, V, E, S, \psi)$ where $G$ is a compact Lie group, $V$ is an $\bA$-manifold with a continuous $G$-action, $E \to V$ is a $G$-equivariant vector bundle, $S: V \to E$ is a $G$-equivariant section, and $\psi: S^{-1}(0)/G \to X$ is a homeomorphism. We require the following condition: the stabilizer of each point $x \in V$ is finite, i.e., the $G$-action on $V$ is almost free.

%\item A {\bf (topological) Kuranishi presentation} (K-presentation for short) of an $\bA$-stratified space $X$ is $(G, V, E, S, \psi)$ where $(G, V, E, S)$ is a K-chart and $\psi: S^{-1}(0)/G \to X$ is a homeomorphism.

\item A K-chart $(G, V, E, S, \psi)$ is said to be {\bf smooth} if $V$ is a smooth $\bA$-manifold, the $G$-action is smooth, and $E \to V$ is a smooth equivariant vector bundle (we do not impose any smoothness condition on $S$).
\end{enumerate}
We often omit the map $\psi$ in the notation because in the context its meaning will always be clear.
\end{defn}

%\begin{defn}
%Given a nonempty moduli space $\ov{\mc M}_{pq}$ of Floer trajectories, a {\bf topological Kuranishi chart} (K-chart for short) of $\ov{\mc M}_{pq}$ is a quintuple $(G, V, E, S, \psi)$ where $G$ is a compact Lie group, $V$ is an $\bA^\floer_{pq}$-manifold with a continuous, stratum-preserving $G$-action, $E \to V$ is a $G$-equivariant vector bundle, and $S: U \to E$ is a $G$-equivariant section, and $\psi: S^{-1}(0)/G \to \ov{\mc M}_{pq}$ is a homeomorphism (of $\bA^\floer_{pq}$-spaces). The K-chart is said to be {\bf smooth} if $V$ is a smooth $\bA^\floer_{pq}$-manifold, the $G$-action is smooth, and $E$ is a smooth equivariant vector bundle (however we do not require that $S$ is smooth).
%\end{defn}

%\begin{rem}
%To simplify the notations, we usually omit the homeomorphism $\psi$ when writing the K-charts and use the quadruple $(G, V, E, S)$ to denote a K-chart.
%\end{rem}

\begin{rem}
Historically, there are different notions of Kuranishi charts which could be defined either via orbifolds or via equivariant objects. In this paper, the orbifold version will be labelled as ``derived'' and the name ``Kuranishi'' is reserved for the equivariant version while allowing actions by general compact Lie groups. We also use the prefixes ``D-'' and ``K-'' to denote these two versions.
\end{rem}

\begin{defn}
Let $K = (G, V, E, S)$ be a K-chart and $\pi_F: F \to V$ be a $G$-equivariant vector bundle. The {\bf stabilization} of $K$ by $F$, denoted by ${\rm Stab}_F(K)$, is the K-chart 
\beqn
{\rm Stab}_F(K) = (G, F, \pi_F^* E \oplus \pi_F^* F, \pi_F^* S \oplus \tau_F)
\eeqn
where $\tau_F: F \to \pi_F^* F$ is the tautological section.
\end{defn}

\subsubsection{Change of groups}

For K-charts, the notions of open embedding, germ equivalence, and product are almost identical to the case of derived orbifold charts after imposing the equivariance condition with respect to the Lie group action. There are some care to be taken of when the groups of symmetry change. To this end, we introduce the operation of enlarging the symmetry group.

\begin{defn}
Let $K = (G, V, E, S)$ be a K-chart and let $G \hookrightarrow G'$ be a Lie group embedding. Define the \textbf{change of group}, or the {\bf $G'$-equivariantization}, of $K$, to be
\beqn
G'\times_G K:= (G', G'\times_G V, G'\times_G E, S')
\eeqn
where $G'$ acts on the bundle $G'\times_G E \to G'\times_G V$ in the obvious way and 
\beqn
S'([g', x]) = [g', S(x)].
\eeqn
In the rest of the paper, we also use the notations
\beqn
G' (K):= G'\times_G K:= (G', G'(V), G'(E), G'(S)) := (G', G'\times_G V, G'\times_G E, S').
\eeqn
\end{defn}

\begin{lemma}
Suppose $G \hookrightarrow G'$ is a Lie group embedding and let $V$ be a $G$-space. If ${\bm W}$ is a $G'$-representation and $E \to V$ is the trivial $G$-bundle $V \times {\bm W}$ where we view ${\bm W}$ as a $G$-representation, then $G'(E)$ is isomorphic to the trivial $G'$-bundle over $G'(V) = G'\times_G V$ with fiber ${\bm W}$.
\end{lemma}

\begin{proof}
Define the map $\zeta: G'(E) = G'\times_G (V \times {\bm W}) \to (G'\times_G V)\times {\bm W}$ by
\beqn
\zeta ([g', (v, w)]) = ([g', v], g' w).
\eeqn
Notice that 
\beqn
\zeta( [g' g, (v, w)]) = ( [g'g, v], g' g w) = ([ g', gv], g' (gw)) = \zeta ([g', (gv, gw)]) = \zeta([g', g(v, w)]).
\eeqn
Hence $\zeta$ is well-defined. It is also straightforward to check that $\zeta$ is a map of $G'$-equivariant vector bundles over $G'\times_G V$ and is an isomorphism.
\end{proof}

\subsubsection{Chart embeddings}
Because of the presence of compact Lie group, the definition of chart embeddings of K-charts also differs slightly from the notion of chart embeddings of derived orbifold charts.

\begin{defn}\label{weak_chart_embedding}
Let $K_1 = (G_1, V_1, E_1, S_1, \psi_1)$ and $K_2 = (G_2, V_2, E_2, S_2, \psi_2)$ be two topological K-charts. 
\begin{enumerate}

\item A {\bf weak K-chart embedding}  from $K_1$ to $K_2$, denoted by 
\beqn
{\bm \iota}_{21}: K_1 \wembed K_2,
\eeqn
consists of a group embedding $G_1 \hookrightarrow G_2$,\footnote{In the concrete situations of this paper, the group embeddings are always fixed by the geometric data, and the embeddings are induced from inclusions} an equivairant topological embedding $\iota_{21}: G_2(V_1) \hookrightarrow V_2$ which has a $G_2$-invariant neighborhood equivariantly homeomorphic to a $G_2$-equivariant vector bundle\footnote{This seemingly redundant requirement is necessary because in the topological category submanifolds do not necessarily have vector bundle neighborhoods.}, and an equivariant vector bundle embedding $\widehat \iota_{21}: G_2(E_1) \hookrightarrow E_2$ covering $\iota_{21}$. (Due to a small defect of the global chart construction we will use, a weak K-chart embedding does not necessarily intertwine with the Kuranishi sections. This explains the adjective ``weak" in the terminology.) 

\item A weak K-chart embedding is called a {\bf K-chart embedding} if the following diagrams commute.
\begin{align*}
&\ \vcenter{ \xymatrix{  G_2(E_1) \ar[r]^{\widehat\iota_{21}} & E_2 \\
            G_2(V_1) \ar[r]_{\iota_{21}} \ar[u]^{G_2(S_1)} & V_2 \ar[u]_{S_2}  }}, &\ \vcenter{\xymatrix{ S_1^{-1}(0)/G_1 \ar[r]^{\iota_{21}} \ar[d]_{\psi_1} & S_2^{-1}(0)/G_2 \ar[d]^{\psi_2} \\
              X \ar[r]_{{\rm id}_X} & X }   }.
              \end{align*}

\item A (weak) K-chart embedding is called a (weak) open embedding if $\iota_{21}$ is a homeomorphism onto an open subset and $\widehat \iota_{21}$ is a bundle isomorphism.
\end{enumerate}
\end{defn}

Notice that (weak) K-chart embeddings can be composed in an obvious way.

\subsubsection{Notations}

\begin{notation}\label{notation:G}
\begin{enumerate}

\item For each positive integer $d$, let $\bA_d$ be the poset of all ordered partitions of $d$, i.e., 
\beqn
\bA_d:= \big\{ (d_0, \ldots, d_l)\ |\ d_0 + \cdots + d_l = d,\ d_i \in {\mb Z}_{>0} \big\}.
\eeqn
The partial order is induced from refinements of partitions. It is a homogeneous poset with a unique maximal element $(d)$ and depth function $\dep(d_0, \ldots, d_l) = l$. Moreover, there are natural inclusions
\beqn
\bA_{d_0} \times \cdots \times \bA_{d_l} \to \partial^{(d_0, \ldots, d_l)} \bA_{d_0 + \cdots + d_l}
\eeqn
which satisfy the obvious associativity relation.

\item For each pair $p, q\in {\mc P}^\floer$, define
\beq
d_{pq}:= {\mc A}^\floer (q) - {\mc A}^\floer  (p).
\eeq
By the integrality assumption of the symplectic action (see Hypothesis \ref{hyp51}), $d_{pq}$ is an integer. In practice we only consider the situation when $d_{pq} \geq 0$. Moreover, there is a natural poset map 
\beq\label{eqn52xxx}
\begin{array}{rcl}
\delta: \bA_{pq}^\floer & \to & \bA_{d_{pq}}\\
\alpha = pr_1 \cdots r_l q & \mapsto & \delta(\alpha) =  (d_{pr_1}, \ldots, d_{r_l q})
\end{array}
\eeq
such that the following diagram commutes.
\beqn
\xymatrix{ \bA_{pr}^\floer \times \bA_{rq}^\floer \ar[r] \ar[d] & \partial^{prq} \bA_{pq}^\floer \ar[d] \\
            \bA_{d_{pr}} \times \bA_{d_{rq}} \ar[r] & \partial^{(d_{pr}, d_{rq})} \bA_{d_{pq}}}
\eeqn

\item The {\bf system of extra symmetries} is the collection of compact Lie groups
    \beqn
    G_d:= \{ g \in PU(d +1)\ |\ g([1, 0, \ldots, 0]) = [1, 0, \ldots, 0]\in \mb{CP}^d  \} \cong U(d).
    \eeqn
    for all $d \geq 1$. Here $PU(d+1)$ acts on $\mb{CP}^d$ in the standard way. The identification with $U(d)$ is given by 
    \beq\label{eqn52}
    U(d) \ni g' \mapsto \left[ \begin{array}{cc} 1 & 0 \\ 0 & g' \end{array} \right]\in PU(d + 1).
    \eeq
    Denote 
    \beqn
    G_{pq}:= G_{d_{pq}}\ \forall p, q\in {\mc P}^\floer, {\mc A}^\floer (p) < {\mc A}^\floer (q).
    \eeqn
    
\item For each $\delta = (d_0, \ldots, d_l) \in \bA_d$, there is a group embedding
\begin{equation}\label{eqn:G-embedding}
G_\delta:= G_{d_0} \times \cdots \times G_{d_l} \hookrightarrow G_d
\end{equation}
defined by
\beqn\label{eqn:group_embedding}
 \left( \left[\begin{array}{cc} 1 & 0 \\ 0 & g_0' \end{array}\right], \dots,
 \left[\begin{array}{cc} 1 & 0 \\ 0 & g_l' \end{array}\right] \right) 
 \mapsto
 \left[\begin{array}{cccc} 1 & 0 & \cdots & 0 \\ 0 & g_0' & \cdots & 0
 \\ 0 & \cdots & \ddots & 0
 \\ 0 & 0 & \cdots & g_l'
 \end{array}\right].
\eeqn

\item For each $\alpha = pr_1 \cdots r_l  \in \bA_{pq}^\floer$, we also denote 
\beqn
G_\alpha:= G_{pr_1 \cdots r_l q}:= G_{pr_1}\times \cdots \times G_{r_l q} = G_{\delta(\alpha)}
\eeqn
and identify it with the embedding image in $G_{pq}$. Then whenever $\alpha \leq \beta$ there is a group embedding $G_\alpha \hookrightarrow G_\beta$.
\end{enumerate}

\end{notation}

\subsubsection{The main statement}
To simplify the notations, through out this section, we use $\ov{\mc M}_{pq}$ to denote the moduli space $\ov{\mc M}{}_{pq}^\floer$.

The following definition is the counterpart of Definition \ref{defn:d-chart-pre}, \ref{defn:flow-lift} for K-charts.

\begin{defn}
A \textbf{weak K-chart presentation (with the system of groups $\{ G_\alpha \}_{\alpha \in \bA^\floer_{pq}}$)} of the $\bA^\floer_{pq}$-space $\ov{\mc M}_{pq}$ consists of the following objects. 
\begin{enumerate}
    \item A collection of K-charts
\beqn
\Big( K_\alpha = (G_\alpha, V_\alpha, E_\alpha, S_\alpha, \psi_\alpha) \Big)_{\alpha\in \bA^\floer_{pq}}
\eeqn
of $( \partial^\alpha \ov{\mc M}_{pq} )_{\alpha \in \bA^\floer_{pq}}$.

\item A collection of weak K-chart embeddings
\beqn
\big\{ {\bm \iota}_{\beta\alpha}: K_\alpha \wembed \partial^\alpha K_\beta      \big\}_{\alpha \leq \beta}
\eeqn
\end{enumerate}
They satisfy the following condition.  
\begin{enumerate}
    \item The weak K-chart embeddings satisfy the cocycle condition. Namely, for any triple of strata $\alpha \leq \beta\leq \gamma$, one has
    \beqn
    {\bm \iota}_{\gamma\beta}\circ {\bm \iota}_{\beta\alpha} = {\bm  \iota}_{\gamma\alpha}.
    \eeqn

    \item For each pair of strata $\alpha \leq \beta$, there are a $G_\alpha$-equivariant vector bundle $F_{\beta\alpha} \to V_\alpha$ and a germ of weak open K-chart embedding
    \beqn
    {\bm \theta}_{\beta\alpha}: {\rm Stab}_{F_{\beta\alpha}}( K_\alpha) \wembed \partial^\alpha K_\beta
    \eeqn
    whose restriction to the zero section coincides with ${\bm \iota}_{\beta\alpha}$. We call ${\bm \theta}_{\beta\alpha}$ a {\bf stabilization map}.
    %any pair $P < P'$, there is a germ of open embedding $\phi_{PP'}$ from the $F_{PP'}$-stabilization of the K-chart $G_{P'}(K_P)$ to the restriction $$K_{P'}|_{\partial^{P'} V_P} := (G_{P'}, \partial^{P'} V_P, E_P |_{\partial^{P'} V_P}, S_P |_{\partial^{P'} V_P}).$$
    %\item For any triple of elements $P < P'< P''$, note that $G_{P''}(G_{P'}(\partial^P V)) = G_{P''}(\partial^{P} V)$. Then we require that over the space $G_{P''}(\partial^{P} V)$ there is an isomorphism between $G_{P''}$-equivariant vector bundles
    %$$G_{P''}(F_{PP'})  \oplus F_{P' P''}|_{G_{P''}(\partial^{P} V)} \rightarrow F_{PP''} $$
    %where $F_{P' P''}|_{G_{P''}(\partial^{P} V)}$ is induced from the open embedding $\phi_{PP'}$.
    %\item The cocycle condition $\phi_{PP''} = \phi_{P' P''} \circ \phi_{PP'}$ of the open embedding after keeping track of the change of the groups.
\end{enumerate}
The weak K-chart presentation is called a {\bf K-chart presentation} if all weak K-chart embeddings are K-chart embeddings.
\end{defn}

\begin{defn}\label{Klift_defn}
A \textbf{weak K-chart lift} of the Hamiltonian Floer flow category $T^\floer$ consists of the following objects.
\begin{enumerate}

    \item A collection of topological K-charts
    \beqn
    \big\{ K_{pq} = (G_{pq}, V_{pq}, E_{pq}, S_{pq}, \psi_{pq} ) \big\}_{p< q}
    \eeqn
    of $\ov{\mc M}_{pq}$.
    
    \item A collection of weak K-chart embeddings 
    \begin{equation}\label{eqn:k-chart-embedd}
    \big\{ {\bm \iota}_{prq}: K_{pr}\times K_{rq}  \wembed  \partial^{prq} K_{pq}\big\}_{p< r < q}.
    \end{equation}
    \end{enumerate}
    These objects need to satisfy the following conditions.
    \begin{enumerate}
        \item When $p = q$, $V_{pp}$ is a singleton and $E_{pp} = \{0\}$.
        
        \item The weak chart embeddings satisfy the associativity. More precisely, the following diagram commutes,
        \beqn
        \xymatrix{  &    K_{pr}\times K_{rs}\times K_{sq} \ar[ld] \ar[rd] & \\
                      \partial^{prs} K_{ps} \times K_{sq} \ar[rd] & & K_{pr}\times \partial^{rsq} K_{rq} \ar[ld] \\
                      & \partial^{prsq} K_{pq} & }
        \eeqn
        where the arrows are defined by \eqref{eqn:k-chart-embedd}.
        \item For each stratum $\alpha  =  pr_1 \cdots r_l q \in \bA^\floer_{pq}$, consider the product topological K-chart 
    \beqn
    K_\alpha = (G_\alpha, V_\alpha, E_\alpha, S_\alpha, \psi_\alpha) = K_{pr_1}\times \cdots K_{r_l q}.
    \eeqn
    The above property implies that for each pair of strata $\alpha \leq \beta$ of $pq$, there is a well-defined weak K-chart embedding 
        \beqn
        {\bm \iota}_{\beta\alpha}: K_\alpha \wembed \partial^\alpha K_\beta.
        \eeqn
    Then the collection $((K_\alpha)_{\alpha \in \bA_{pq}}, ({\bm \iota}_{\beta\alpha})_{\alpha \leq \beta})$ form a weak K-chart presentation of $\ov{\mc M}_{pq}$.
    \end{enumerate}
The weak K-chart lift is called a {\bf K-chart lift} if all the weak chart embeddings are chart embeddings. 
\end{defn}

\begin{notation}
We introduce the following notation and convention which are frequently used in this paper. The $G_{pq}$-equivariantization of the product chart $K_\alpha$ is denoted by 
\beqn
K_\alpha^\sim:= G_{pq}\times_{G_\alpha}(K_\alpha) = (G_{pq}, V_\alpha^\sim, E_\alpha^\sim, S_\alpha^\sim).
\eeqn
Then usually the version of notations $K_\alpha$ (or $K_{pr_1\cdots r_l q}$) indicates the objects are constructed from taking products and the version of notations $K_\alpha^\sim$ (or $K_{pr_1 \cdots r_l q}^\sim$) indicates the objects are constructed from taking products and an equivariantization (by a certain group which is clear from the context).
\end{notation}

Now we are ready to state the first main theorem of this section.

\begin{thm}\label{thm:flow-chart}
The Hamiltonian Floer flow category $T^\floer$ admits a weak K-chart lift. %with K-charts $\{ K_{pq} = (G_{pq}, V_{pq}, E_{pq}, S_{pq}) \}_{p<q}$. Moreover, there exist a smooth $\bA^\floer_{pq}$-manifold $B_{pq}$ with a smooth $G_{pq}$-action and a $G_{pq}$-equivariant topological submersion 
%$$ \pi_{pq}: V_{pq} \rightarrow B_{pq} $$
%such that the following is true.
%\begin{enumerate}
%    \item $\pi_{pq}$ defines a relative smooth structure which is continuous in $C^1_{loc}$-topology.
%    \item Product behavior along the boundary strata.
%\end{enumerate}
\end{thm}

We also refer to the weak K-chart lift as the ``global chart.'' This theorem certainly does not contain all the properties we will need in the final construction. Eventually, we need to upgrade the weak lift to a lift and equip it with a smooth structure.

\subsection{Global chart construction I}

Following \cite[Section 6]{AMS} (also in the same spirit as \cite[Section 3]{Siebert_1999}), the goal of this subsection is to introduce a system of auxiliary moduli spaces $(B_d)_{d\geq 1}$ parametrizing stable holomorphic cylinders in $\mb {CP}^{d}$, which serve as Deligne--Mumford type moduli spaces for stable Floer trajectories. Moreover, we will consider the gluing of different (stable) holomorphic cylinders from different complex projective spaces. Such consideration will play an important role in the study of global charts of the boundary strata of moduli spaces of (stable) Floer trajectories.

\subsubsection{Moduli spaces of stable cylinders}

Following \cite{AMS} and \cite{Abouzaid_Blumberg}, we introduce some auxiliary moduli spaces. Consider a genus zero prestable curve $\Sigma$ with two marked points $z_-$, $z_+$. The two marked points induce a decomposition 
\beqn
\Sigma = \Sigma_{\rm cyl}\cup \Sigma_{\rm sph}
\eeqn
of $\Sigma$ into the cylindrical components and spherical components, where the cylindrical components are determined by the vertices on the line connecting $z_-$ and $z_+$ in the dual graph. Each cylindrical component $\Sigma_i \subset \Sigma$ has two special points corresponding to the negative and positive infinities $z_{i, \pm}$. There is a ${\mb C}^*$-action on $\Sigma_\alpha$, given by biholomorphisms fixing $z_{i, \pm}$. 

To fix the rotational gauge on the cylinder, we introduce the following concept. Identifying ${\mb C}^*$ with $S^1 \times {\mb R}$, a {\bf lateral line} on a cylindrical component $\Sigma_i$ is an ${\mb R}$-orbit $\L_i \subset \Sigma_i$. There are other equivalent notions such as asymptotic markers which can also fix the rotational gauge.

\begin{defn}\hfill\label{stablecylinder}
\begin{enumerate}
\item A {\bf prestable cylinder} is a pair $(\Sigma, {\bf L})$ where $\Sigma$ is a genus zero pretable curve with two marked points $z_-, z_+$ and ${\bf L} = (\L_i)$ where each $\L_i \subset \Sigma_i$ is a lateral line on the cylindrical component $\Sigma_i$. The cylindrical irreducible components are also referred to as the {\bf horizontal levels} of $(\Sigma, {\bf L} )$.

\item A {\bf marked stable cylinder} is a triple $(\Sigma, {\bf z}, {\bf L})$ where $(\Sigma, {\bf L})$ is a prestable cylinder and ${\bf z}$ is a list of extra marked points which are different from $z_\pm$ and the nodal points $z_{i,\pm}$, such that each irreducible component is stable (i.e., with at least three special points). 
\end{enumerate}
\end{defn}

The notion of isomorphisms of marked stable cylinders can be defined in the obvious way. Let $\ov{\mc M}{}_{0, 2, d'}^{\mb R}$ be the moduli space of marked stable cylinders with $d'$ marked points. It is a compactification of the moduli space ${\mc M}_{0,2,d'}^{\mb R}$, the moduli space of marked stable cylinders with $d'$ marked points with only one cylindrical component and zero spherical components. Forgetting the lateral line defines a smooth $S^1$-fibration ${\mc M}_{0,2,d'}^{\mb R} \to {\mc M}_{2+d'}$ to the moduli space of genus $0$ curves with $2+d'$ marked points with smooth domains. Following \cite[Section 3.2]{KSV-95}, $\ov{\mc M}{}_{0, 2, d'}^{\mb R}$ could be obtained from the Deligne--Mumford space $\ov{{\mc M}}_{0, 2+d'}$ by performing real blowups along the irreducible components of the normal crossings divisor defined by stable curves with at least two cylindrical components. $\ov{\mc M}{}_{0, 2, d'}^{\mb R}$ is a smooth manifold with corners: for a more detailed discussion, the reader could refer to \cite[Section 2.1]{Liu_Tian_Floer} or \cite[Section 9.7]{Abouzaid_Blumberg}. To simplify the notation, we usually abbreviate $(\Sigma, {\bf L})$ as $\Sigma$ when there is no ambiguity caused by the context.

\subsubsection{Stable maps to projective spaces}

We consider the moduli space of genus zero stable holomorphic maps into $\mb{CP}^d$ with two marked points. Denote by
\beqn
{\mc F}:= {\mc F}_{0, 2}(d)\subset \ov{\mc M}_{0,2}(\mb{CP}^d, d)
\eeqn
the subset of maps whose image is not entirely contained in any hyperplane. Any two smooth curves in  ${\mc F}$ can be mapped to another by an element of the symmetry group $PGL(d+1)$ of $\mb{CP}^d$. Let ${\mc C} = {\mc C}_{0,2}(d) \to {\mc F}_{0, 2}(d)$ be the universal curve which has an induced $PGL(d+1)$-action.

\begin{lemma}\label{lem:ev-submersion}
Both ${\mc F}_{0,2}(d)$ and ${\mc C}_{0,2}(d)$ are smooth quasi-projective varieties. Moreover, the evaluation maps at two marked points
\beqn
\ev_+, \ev_-: {\mc F}_{0,2}(d) \to \mb{CP}^d
\eeqn
are smooth submersions.
\end{lemma}

\begin{proof}
The first statement follows from \cite[Lemma 6.4]{AMS}. For the submersive property, note that the natural action of $PGL(d+1)$ on $\ov{\mc M}_{0,2}(\mb{CP}^d, d)$ preserves the subspace ${\mc F}_{0, 2}(d)$. Because the linearization of $PGL(d+1)$-action on $\mb{CP}^d$ at any point defines a surjection from its Lie algebra to the tangent space of $\mb{CP}^d$ at this point, the linearization of $\ev_{\pm}$ is surjective as well.
\end{proof}

%We need to introduce a specific line bundle over the universal curve. Consider the evaluation map $\ev: {\mc C}_{0,2}(d) \to \mb{CP}^d$ and the pullback
%%\beqn
%{\mc L}:= \ev^* {\mc O}(1).
%\eeqn
%The $PGL(d+1)$-action on $\mb{CP}^d$ induces an action on ${\mc L}$ making it an equivariant line bundle. The Fubini--Study metric endows ${\mc L}$ with a $PU(d+1)$-invariant Hermitian metric. For each fiber $\Sigma \subset {\mc C}$, the restriction of ${\mc L}$ to $\Sigma$ is ample over each nonconstant component. 

\subsubsection{Stable cylinders in projective spaces}

For the purpose of studying Floer theory, one needs to consider the ``real'' version of the above moduli spaces.

The moduli spaces of stable cylinders in $\mb{CP}^d$ is the set of equivalence classes of objects
\beqn
(\Sigma, {\bf L}, u)
\eeqn
such that $(\Sigma, {\bf L})$ is a prestable cylinder and $u: \Sigma \to \mb{CP}^d$ is a stable map, i.e., every constant component has at least three special points. The equivalence relation is defined as follows: $(\Sigma, {\bf L}, u) \sim (\Sigma', {\bf L}', u')$ if there is an isomorphism $\varphi: (\Sigma, {\bf L}) \cong (\Sigma', {\bf L}')$ such that $u= u' \circ \varphi$. 

\begin{defn}
Let $\ov{\mc M}{}_{0,2}^{\mb R}(\mb{P}^d, k)$ denote the moduli space of stable cylinders in $\mb{CP}^d$ with degree $k$ times the generator of $H_2(\mb{CP}^d; {\mb Z})$. Define
\beqn
{\mc F}{}_{0,2}^{\mb R}(d) = \{ [\Sigma, {\bf L}, u] \in \ov{\mc M}{}_{0,2}^{\mb R}(\mb{CP}^d, d)\ |\ [\Sigma, u]\in {\mc F}_{0,2}(d) \}.
\eeqn
\end{defn}

Note that the evaluation maps $\ev_{\pm}$ at the two marked points on ${\mc F}_{0,2}(d)$ naturally lift to evaluation maps on ${\mc F}{}_{0,2}^{\mb R}(d)$.

\begin{lemma}\label{lem:r-blowup}
There exists a natural structure of smooth $\bA_d$-manifold on ${\mc F}{}_{0, 2}^{\mb R}(d)$ %(\textcolor{blue}{$\bA^\floer_{pq}$-manifold structure instead?}) 
satisfying the following conditions.
\begin{enumerate}
    \item The forgetful map ${\mc F}{}_{0,2}^{\mb R}(d) \to {\mc F}_{0,2}(d)$ is smooth. %({\bf need to discuss the notion of smooth maps on manifolds with corners}).
    
    \item The evaluation maps $\ev_{\pm}: [\Sigma, {\bf L}, u] \mapsto u(z_{\pm})$ are smooth and they are stratified submersive (see Definition \ref{defn:strong-sub}).
    
    \item For any element $a \in {\mc F}{}_{0,2}^{\mb R}(d)$ represented by $(u, \Sigma, {\bf L})$, choose a generic collection of $d' = d(d+2)$ hyperplanes and $w_1, \ldots, w_{d'} \in \Sigma$ such that $u$ intersects with $H_i$ transversely at $w_i$. Then the locally defined map given by taking a stable cylinder to the marked stable cylinder where the marked points are specified by the intersections with these hyperplanes is a diffeomorphism to an open subset of $\ov{\mc M}{}_{0, 2+d'}^{\mb R}$.
\end{enumerate}
\end{lemma}

\begin{proof}
Consider the subspace of ${\mc F}_{0,2}(d)$ which consists of stable holomorphic maps with exactly two cylindrical components of degrees $d_1$ and $d_2$. Denote by the closure of this space in ${\mc F}_{0,2}(d)$ by ${\mc F}_{0,2}(d_1, d_2)$. Then ${\mc F}_{0,2}(d_1, d_2)$ is a smooth divisor in the quasi-projective variety ${\mc F}_{0,2}(d)$. We can consider the real blowup of ${\mc F}_{0,2}(d)$ along ${\mc F}_{0,2}(d_1, d_2)$ ranging over all $d_1 + d_2 = d$ and denote it by ${\mb R}{\mc F}_{0,2}(d)$. 

We claim that ${\mc F}_{0,2}^{\mb R}(d)$ is an $S^1$-bundle over ${\mb R}{\mc F}_{0,2}(d)$. Indeed, given any stable cylinder $(\Sigma, {\bf L}, u) \in {\mc F}{}_{0,2}^{\mb R}(d)$ with $k$ horizontal levels, denote by $\theta_i$ to be the $S^1$-parameter of the lateral line on the $i$-th cylindrical component of $\Sigma$, with the convention that the first component contains $z_-$ while the $k$-th component contains $z_+$. Then the datum
$$ (\Sigma, u, [\theta_2 - \theta_1], \dots, [\theta_k - \theta_{k-1}]) $$
defines a point in ${\mb R}{\mc F}_{0,2}(d)$ because $\theta_i - \theta_{i-1}$ can be identified with the $S^1$-parameter in the exceptional divisor associated to some ${\mc F}_{0,2}(d_1, d_2)$. Conversely, given $\theta \in S^1$ and an element in ${\mb R}{\mc F}_{0,2}(d)$, we can define an element $(\Sigma, {\bf L}, u)$ in ${\mc F}_{0,2}^{\mb R}(d)$ by requiring $(\Sigma, u)$ to be the image of the blow-down map ${\mb R}{\mc F}_{0,2}(d) \to {\mc F}_{0,2}(d)$ and the lateral line ${\bf L}$ is defined by the converse process of the previous construction. Note that these constructions are well-defined after choosing local coordinates near each stratum of ${\mc F}_{0,2}(d)$ and they can be patched up together. Using this description, we can view ${\mc F}_{0,2}^{\mb R}(d)$ as manifold with corners because ${\mb R}{\mc F}_{0,2}(d)$ is so. Moreover, ${\mc F}_{0,2}^{\mb R}(d)$ is actually a manifold with faces, with codimension-$1$ faces corresponding to a partition $d= d_1 + d_2$.

The forgetful map ${\mc F}_{0,2}^{\mb R}(d) \to {\mc F}_{0,2}(d)$ is defined by $(u, \Sigma, {\bf L}) \mapsto (u, \Sigma)$, which factors through the projection ${\mc F}_{0,2}^{\mb R}(d) \to {\mb R}{\mc F}_{0,2}(d)$ and the blow-down map ${\mb R}{\mc F}_{0,2}(d) \to {\mc F}_{0,2}(d)$, therefore it is smooth. The evaluation map $\text{ev}_+$ (resp. $\text{ev}_-$) is the composition of the forgetful map ${\mc F}_{0,2}^{\mb R}(d) \to {\mc F}_{0,2}(d)$ and the ordinary evaluation map $\text{ev}_+: {\mc F}_{0,2}(d) \rightarrow {\mb CP}^d$ (resp. $\text{ev}_-: {\mc F}_{0,2}(d) \rightarrow {\mb CP}^d$), so $\text{ev}_{\pm}$ are smooth as well. Accordingly, $\text{ev}_{\pm}$ are stratified submersive following the same proof of Lemma \ref{lem:ev-submersion} because the action of $PGL(d+1)$ preserves the strata of the simple normal crossings divisor defined by $F_{0,2}(d_1 + d_2)$.

For the last statement, it follows from \cite[Proposition 6.5]{AMS} and the descriptions of ${\mc F}_{0,2}^{\mb R}(d)$ and $\ov{\mc M}{}_{0, 2+d'}^{\mb R}$ as $S^1$-bundles over real blowups.
\end{proof}

%\textcolor{red}{
%\begin{rem}\label{rem:blow-less}
%To obtain the correct anchoring data of $\ov{\mc M}_{pq}$, we should not blow up along all the divisors associated to $2$-level buildings. Instead, we should look at nonempty codimension-$1$ faces $\ov{\mc M}_{prq}$ and blow up the divisor associated to the decomposition $d_{pq} = d_{pr} + d_{rq}$. Otherwise, $B_{pq}$ is too large to receive a topological submersion from $V_{pq}$. NEED TO INTRODUCE THE NOTATION ${\mc F}_{pq}^{\mb R}(d)$?
%\end{rem}
%}

%\textcolor{red}{Following the above remark, if we consider the full blowup ${\mc F}_{0,2}^{\mb R}(d)$, the natural map $V_{pq} \to {\mc F}_{0,2}^{\mb R}(d)$ is a map between stratified spaces such that the underlying poset map is indeed an \emph{injection}. We need to modify the definition of topological submersion in Section 6 accordingly.}

\subsubsection{A system of auxiliary moduli spaces}

Stable cylinders in projective spaces serve as models of domains of Floer trajectories. Because of the existence of broken Floer trajectories, we need to construct a system of moduli spaces of stable cylinders to capture the information from these boundary strata.

\begin{notation}\label{notation:B}
\begin{enumerate}
%\item The {\bf system of auxiliary targets} is the collection of projective spaces $\mb{CP}^{d_{pq}}$ for each pair $p, q$ of capped 1-periodic orbits with a nonempty moduli space of Floer trajectories. The compact Lie group $G_{pq}$ and its complexification, consisting of elements in $PGL(d_{pq}+1)$ fixing $[1, 0, \dots, 0] \in \mb{CP}^{d_{pq}}$, act on $\mb{CP}^{d_{pq}}$ in the standard way.

\item Whenever ${\mc A}^\floer (p) \leq {\mc A}^\floer  (r) < {\mc A}^\floer (s) \leq {\mc A}^\floer (q)$, embed $\mb{CP}^{d_{rs}}$ into $\mb{CP}^{d_{pq}}$ using the map
\beqn
[z_0, \ldots, z_{d_{rs}}] \mapsto \big[\ \underbrace{0, \ldots, 0}_{d_{pr}}, z_0, \ldots, z_{d_{rs}}, \underbrace{0, \ldots, 0}_{d_{sq}}\ \big].
\eeqn
%This inclusion is equivariant with respect to the corresponding group embedding $\bG_{rs}\hookrightarrow \bG_{pq}$.\footnote{Notice that there are no natural inclusions $PGL(k) \hookrightarrow PGL(l)$ for $k < l$. This is one reason why our symmetry group is $U(d)$ but not $PU(d)$.}

\item The {\bf system of auxiliary moduli spaces} is the collection of moduli spaces 
\beqn
B_d:= \big\{ x = [\Sigma, {\bf L}, u] \in {\mc F}{}_{0,2}^{\mb R}(d)\ |\ \ev_-(x) = [1, 0, \ldots, 0] \in \mb{CP}^d \big\}.
\eeqn
For a partition of $d$ given by $\delta$ from which $d = d_0 + \cdots + d_l$, we define
\beqn
B_{\delta} := B_{d_0} \times \cdots \times B_{d_l}.
\eeqn
Using this notation, define
\beqn
B_\alpha:= B_{\delta(\alpha)},\ {\rm where}\ \alpha \in \bA_{pq}^\floer
\eeqn
where $\alpha \mapsto \delta(\alpha)$ is the poset map \eqref{eqn52xxx}.\end{enumerate}
\end{notation}

\begin{lemma}\label{lem:b-delta}
$B_d$ is a smooth $\bA_d$-manifold with a smooth $G_d$-action. %such that
%$$ B_{pq} = \bigsqcup_{pr_1 \cdots r_k q} \mathring{B}_{pr_1 \cdots r_k q}, \text{ and } \partial^{pr_1 \cdots r_k q} B_{pq} = B_{p r_1 \cdots r_k q}. $$
\end{lemma}

\begin{proof}
This is a consequence of the blowup description of ${\mc F}{}_{0,2}^{\mb R} (d)$ as from Lemma \ref{lem:r-blowup}, the fact that $\ev_{-}: {\mc F}^{\mb R}_{0,2}(d) \rightarrow \mb{CP}^{d}$ is a stratified submersive, and Lemma \ref{lem:regular-level}. %The desired connectivity assumption from Lemma \ref{lem:regular-level} is true by the transitivity of the action of the subgroup of $PU (d+1)$ fixing the point $[1,0,\dots,0]$.
\end{proof}

For a partition $\delta \in \bA_{d}$, denote by $\partial^{\delta} \mathring{B}$ the locally closed smooth $G_d$-submanifold of $B_d$ given by 
\beqn
\partial^\delta \mathring B:= \partial^\delta B \setminus \bigcup_{\delta' < \delta} \partial^{\delta'} B.
\eeqn

\begin{defn}[Normalized evaluation]\label{defn:norm-ev}
Given $x \in B_d$, its {\bf normalized evaluation} at $z_+$ is a unit vector $(a_0, \ldots, a_d ) \in {\mb C}^d$ specified as follows. Suppose $\delta = (d_0, \ldots, d_l)$ and $x \in \partial^{\delta} \mathring{B}$. Assume that $x$ is represented by a stable cylinder $(\Sigma, {\bf L}, u)$. Let $\Theta_i^0$ be $i$-th cylindrical component of the domain $\Sigma$ and let $u_i^0$ be the restriction of $u$ to this cylindrical component. Using the lateral lines, one can identify $\Theta_i$ with ${\mb C}^*$ with the lateral line identified with the positive real axis, and with $z_{i,-}$ resp. $z_{i,+}$ identified with $0$ resp. $\infty$. Moreover, $u_i^0$ can be written as 
\beqn
u_i^0 (z) = [ f_{i, 0}(z), \ldots, f_{i, d}(z)]
\eeqn
where $f_{i, 0}, \ldots, f_{i, d}$ are complex polynomials of degrees at most $d_i$. Let $d_i^0 \leq d_i$ be the maximal degree of these polynomials. We call a list of complex polynomials $(f_{i, 0}, \ldots, f_{i, d})$ a polynomial representative of $u_i$. For each $i$, the polynomial representative is only unique up to rescaling by a common factor in ${\mb C}^*$ and a common reparametrization $z \mapsto \lambda_i z$ by a positive real number $\lambda_i$. However, there exists a unique set of polynomial representatives for all $u_i$ such that
\begin{enumerate}
    \item $(f_{0, 0}(0), \ldots, f_{0, d}(0)) = (1, 0, \ldots, 0)\in {\mb C}^{d+1}$.
    
    \item For each $i$, the ``evaluation" of $(f_{i, 0}, \ldots, f_{i, d})$ at the marked point $z_{i,+} = \infty$ 
    $$ (a_{i, 0}, \ldots, a_{i, d}) := \lim_{z \to \infty} \frac{1}{z^{d_i^0}}  (f_{i, 0}(z), \ldots, f_{i, d}(z)) $$
    is a unit vector of ${\mb C}^{d+1}$.
    
    \item For each $i\geq 2$, $(f_{i, 0}(0), \ldots, f_{i, d}(0)) = (a_{i-1, 0}, \ldots, a_{i-1, d})$.
\end{enumerate}
Then the unit vector $(a_0, \ldots, a_d):= (a_{l, 0}, \ldots, a_{l, d})$ is called the normalized evaluation of $x\in B_d$ at $z_+$. In particular,
\beqn
\ev_+(x) = [a_0, \ldots, a_d] \in \mb{CP}^d.
\eeqn
\end{defn}

We denote the normalized evaluation map defined as above by
\beqn
\widetilde\ev_+: B_d \to {\mb C}^{d + 1 }.
\eeqn
It is easy to see that $\widetilde\ev_+$ is smooth from its construction.

%\begin{lemma}
%$\widetilde\ev_+: B_{pq} \to {\mb C}^{d_{pq}+1}$ is a smooth map.
%\end{lemma}

%\begin{proof}
%Need a proof?
%\end{proof}

%\begin{rem}
%The polynomial parametrizations in the construction of $\widetilde\ev_+$ should be understood as follows. Consider the fiber of the line bundle ${\mc O}(1) \to {\mb CP}^{d_{pq}}$ over $[1,0,\dots,0]$ and fix an identification of it with ${\mb C}$. Observe that the space of such an identification is a torsor over ${\mb C}^*$. Given $[\Sigma, {\bf L}, u] \in B_{pq}$, choose a representative $u: \Sigma \to {\mb CP}^{d_{pq}}$. Then the standard hyperplane basis of ${\mc O}(1)$ is pulled back to a basis $(f_0, \dots, f_{d_{pq}})$ of the line bundle $u^* {\mc O}(1)$ over the possibly nodal curve $\Sigma$. We fix the isomorphism between ${\mc O}(1)_{[1,0,\dots,0]}$ and ${\mb C}$ in the ${\mb C}^*$-family by asking
%$$ (f_0(z_-), \dots, f_{d_{pq}}(z_-)) = (1,0,\dots,0) %$$
%under the isomorphism. Note that we can identify $u^* {\mc O}(1) |_{\Theta_i}$ with ${\mc O}_{{\mb P}^1}(d_{r_i r_{i+1}})$.
%\end{rem}

\subsubsection{Boundary strata of the auxiliary moduli spaces} 

We introduce the following definitions in order to compare the boundary strata of the auxiliary moduli spaces and the products of moduli spaces with lower degrees. This should be thought of as a toy model for the comparison between the restriction of the geometrically-constructed K-charts along the boundary strata and the product of the K-charts from the factorization description of the boundary strata. 

\begin{defn}[Fans and Flags] Fix $d\geq 1$ and $\delta = (d_0, \ldots, d_l) \in \bA_d$.
\begin{enumerate}
    \item A {\bf fan} of type $\delta$ is a list of  linear subspaces of ${\mb C}^{d+1}$
    \beqn
    (W_{d_0}, \ldots, W_{d_l})
    \eeqn
    satisfying 
    \begin{enumerate}
        \item ${\rm dim}_{\mb C} W_{d_i} = d_i + 1$.
        
        \item $W_{d_0} + \cdots + W_{d_l} = {\mb C}^{d +1}$.
        
        \item ${\rm dim}_{\mb C} (W_{d_{i-1}}\cap W_{d_i}) = 1$.
    \end{enumerate}
    Let ${\rm Fan}_\delta$ be the set of all fans of type $\delta$. Then there is a canonical map 
    \beqn
    \partial^\delta B_d \to {\rm Fan}_\delta
    \eeqn
    which sends any equivalence class of stable cylinders to the fan for which $W_{d_i}$ is the subspace such that the image of the $(i+1)$-th level is contained in the projectivization of $W_{d_i}$. Note that such a map is well-defined because the stable cylinder at the $(i+1)$-th level has degree $d_i$.
    
    \item A {\bf flag} of type $\delta$ is a list of subspaces of ${\mb C}^{d +1}$
    \beqn
    V_0 \subset \cdots \subset V_l = {\mb C}^{d+ 1}
    \eeqn
    satisfying 
    \beqn
    {\rm dim}_{\mb C} V_i = d_0 + \cdots + d_i + 1.
    \eeqn
    Let ${\rm Flag}_\delta$ be the set of all flags of type $\delta$. Then there is a standard flag of each type $\delta$ where
    \beqn
    V_i = {\mb C}^{d_0 + \cdots + d_i +1} \times \{0\}^{d_{i+1} +\cdots+d_l}, i = 0, \ldots, l.
    \eeqn
    
    \item There is a canonical map 
    \beqn
    {\rm Fan}_\delta \to {\rm Flag}_\delta
    \eeqn
    which maps $(W_{d_0}, \ldots, W_{d_l})$ to the flag where 
    \beqn
    V_i = W_{d_0} + \cdots + W_{d_i}.
    \eeqn
    
    \item A fan of type $\delta$ is said to be {\bf in the normal position} if the following is true. For each $i$, let 
    \begin{align*}
        &\ W_i^-:= W_{d_0} + \cdots + W_{d_i},\ &\ W_{i}^+:= W_{d_{i+1}} + \cdots + W_{d_l}.
    \end{align*}
    The condition for being a fan implies that $W_i^- \cap W_i^+$ is a line $L_i$. Let $\mathring W_i^\pm$ be the orthogonal complement of $L_i$ in $W_i^\pm$ (with respect to the standard Hermitian inner product of ${\mb C}^{d + 1}$). Then $\mathring W_i^-$ and $\mathring W_i^+$ are orthogonal for all $i = 0, \ldots, l-1$.
    
    \item A point $x \in B_d$ is said to be in the $\delta$-normal position if $x \in \partial^\delta B_d$ and its associated fan of type $\delta$ is in the normal position. Let $(\partial^\delta B_d)^{\rm normal} \subset \partial^\delta B_d$ be the subset of stable cylinders which are in $\delta$-normal position.
    \end{enumerate}
%\textcolor{red}{We know that decompositions of the form $\ov{\mc M}_{pr} \times \ov{\mc M}_{rq} \hookrightarrow \partial \ov{\mc M}_{p,q}$ can cover the boundary of the moduli spaces, but the inclusion of the product induced by $\tilde{\zeta}_{prq}$ is only mapped to an open subset of the boundary of $B_{pq}$ (after correcting the issue of using $GL(n, {\mb C})$ to change to groups). Should be modify the base of the topological submersions accordingly by removing the boundary strata not coming from the breaking of Floer trajectories? See the discussion in Remark \ref{rem:blow-less}.}
\end{defn}

%\begin{lemma}
%The canonical map $\partial^\delta B_d \to {\rm Fan}_\delta$ extends to a map
%$$ B_{pr_1 \cdots r_k q} \to {\rm Fan}_{pr_1 \cdots r_k q},$$
%and $B_{pr_1 \cdots r_k q}^{\rm normal} \subset B_{pr_1 \cdots r_k q}$ is defined %accordingly.
%\end{lemma}
%\begin{proof}
%For $p r^{\prime}_{1} \cdots r^{\prime}_{k'} q < p r_1 \cdots r_k q$ in $\bA^\floer_{pq}$ such that $\ov{\mc M}_{p r^{\prime}_{1} \cdots r^{\prime}_{k'} q} \neq \emptyset$, any stable cylinder in $B_{p r^{\prime}_{1} \cdots r^{\prime}_{k'} q}$ can be viewed as a stable cylinder with $k+1$ horizontal levels such that the $i+1$-th level is the stable cylinder connecting $z_{i}^-$ and $z_i^+$. The underlying map of the $i+1$-th level has degree $d_{r_i r_{i+1}}$ and its associated fan and flag can be defined as above.
%\end{proof}

We first describe the case of a codimension-one stratum of the auxiliary moduli space as a warm-up. Fix $d \geq 2$ and a partition $(d_0, d_1)\in \bA_d$. There is a corresponding stratum $\partial^{(d_0, d_1)} B_d \subset B_d$ from Notation \ref{notation:B}. We would like to define a map 
\beq\label{eqn42}
\zeta_{(d_0, d_1)}: B_{d_0} \times B_{d_1} \hookrightarrow \partial^{(d_0, d_1)} B_d.
\eeq
Given two arbitrary points $x_0\in B_{d_0}$, $x_1 \in B_{d_1}$, $\zeta_{(d_0, d_1)}(x_0, x_1)$ can be defined as follows. Let $u_0: \Sigma_0 \to \mb{CP}^{d_0}$, $u_1: \Sigma_1 \to \mb{CP}^{d_1}$ be representatives. Let $(a_0, \ldots, a_{d_0})$ be the normalized evaluation of $x_0$ (see Definition \ref{defn:norm-ev}). One can represent $u_0$ resp. $u_1$ as a list of holomorphic sections
\beqn
(f_0, \ldots, f_{d_0})\ {\rm resp}.\ (g_0, \ldots, g_{d_1})
\eeqn
of the line bundle $L_0 = u_0^* {\mc O}(1) \to \Sigma_0$ resp. $L_1 = u_1^* {\mc O}(1)  \to \Sigma_1$. Then define $\zeta_{(d_0, d_1)}(x_0, x_1) \in \partial^{(d_0, d_1)} B_d$ to be the point represented by the map
\begin{align}\label{base_product}
u_{(d_0, d_1)}&: \Sigma_0 \vee \Sigma_1 \to \mb{CP}^d ,\\ 
z &\mapsto \left\{ \begin{array}{cc} {\rm [} f_0(z), \ldots, f_{d_0}(z),0, \ldots, 0 {\rm ]},\ &\ z \in \Sigma_0,\\
{\rm [} a_0 g_0(z), \ldots, a_{d_0} g_0(z), g_1(z), \ldots, g_{d_1}(z) {\rm ]},\ &\ z \in \Sigma_1.
\end{array} \right.\label{eqn54}
\end{align}

%Notice that the product $B_{pr} \times B_{rq}$ has an action by $G_{pr}\times G_{rq}$ and the product group $G_{pr}\times G_{rq}$ embeds into $G_{pq}$ using Equation \eqref{eqn:group_embedding}.

\begin{lemma}
The map $\zeta_{(d_0, d_1)}$ is equivariant with respect to the group embedding $G_{d_0} \times G_{d_1} \hookrightarrow G_d$ from Notation \ref{notation:G}.
\end{lemma}

\begin{proof}
Straightforward from the definitions.
\end{proof}

Hence $\zeta_{(d_0, d_1)}$ can be extended to a $G_d$-equivariant map 
\beq\label{eqn43}
\zeta_{(d_0, d_1)}^\sim: G_d \times_{G_{(d_0, d_1)}} (B_{d_0} \times B_{d_1}) \to \partial^{(d_0, d_1)} B_d. 
\eeq

However, the above map is not surjective as configurations in the image are those stable cylinders whose two levels are in certain ``normal'' positions. 

\begin{lemma}\label{lem:normal-position}
$\zeta_{(d_0, d_1)}^\sim$ is bijective onto $(\partial^{(d_0, d_1)} B_d)^{\rm normal}$.
\end{lemma}

\begin{proof}
From the definition we know that the image of $\zeta_{(d_0, d_1)}^\sim$ is contained in $(\partial^{(d_0, d_1)} B_d)^{\rm normal}$. We first prove the surjectivity. For any $x \in (\partial^{(d_0, d_1)} B_d)^{\rm normal}$ with associated fan $(W_{d_0}, W_{d_1})$, denote $L_x:= W_{d_0} \cap W_{d_1}$. Let the domain of $x$ be $\Sigma = \Sigma_0 \cup \Sigma_1$. By using a unitary transformation on $\mb{CP}^{d}$ which fixes the point $[1, 0, \ldots, 0]$, we may assume that $W_{d_0}$ is spanned by the first $d_0 +1$ coordinates. Then the first level of $x$ can be viewed as a stable cylinder in $\mb{CP}^{d_0}$. Let the normalized evaluation be $(a_0, \ldots, a_{d_0})$. Then the nodal point of $x$ is mapped to $[a_0, \ldots, a_{d_0}, 0, \ldots, 0]$. Then by the definition of being in the normal position, $\mathring W_0^+$ is the subspace spanned by the last $d_1$ coordinates and $W_{d_1} = \mathring W_0^+ + L_x$. Then the second level of $x$ is represented by the map 
\beqn
u_{rq}(z) = [ a_0 f_0 (z), \ldots, a_{d_0} f_0(z), f_1(z), \ldots, f_{d_1} (z)]
\eeqn
where $(f_0, \ldots, f_{d_1})$ is a list of holomorphic sections of a degree $d_1$ line bundle over $\Sigma_1$. As $(a_0, \ldots, a_{d_0})$ is the normalized evaluation of the first level of $x$, one can see that $x = \zeta_{(d_0, d_1)} (x_0, x_1)$ where $x_0$ is the first level of $x$ and $x_1$ is represented by the map $z\mapsto [f_0(z), \ldots, f_{d_1}(z)]$. Therefore $\zeta_{(d_0, d_1)}^\sim$ is surjective.

To prove that $\zeta_{(d_0, d_1)}^\sim$ is injective, suppose 
\beqn
x = \zeta_{(d_0, d_1)}^\sim ( [g, x_0, x_1]) = \zeta_{(d_0, d_1)}^\sim ( [g', x_0', x_1']).
\eeqn
We may assume $g' = 1$. Then the flag associated to $x$ is the standard one and $g \in G_d$ preserves this flag. Hence $g \in G_{d_0}\times G_{d_1}$. Hence we may also assume $g = 1$. Then $\zeta_{(d_0, d_1)}(x_0, x_1) = \zeta_{(d_0, d_1)}(x_0', x_1')$, which implies $x_0 = x_0'$ and $x_1 = x_1'$. Hence $\zeta_{(d_0, d_1)}^\sim$ is injective.
\end{proof}

It is important to prove the associativity of the product construction.

\begin{prop}\label{associativity}
For a partition $(d_0, d_1, d_2)$ of $d\geq 1$, $x_0 \in B_{d_0}$, $x_1 \in B_{d_1}$, and $x_2 \in B_{d_2}$, there holds
\beqn
\zeta_{(d_0, d_1 + d_2)} ( x_0,  \zeta_{(d_1, d_2)} (x_1, x_2)) = \zeta_{(d_0 + d_1, d_2)} ( \zeta_{(d_0, d_1)} (x_0, x_1), x_2).
\eeqn
\end{prop}

\begin{proof}
Choose representatives $u_0, u_1, u_2$ of $x_0, x_1, x_2$ respectively. Let
\begin{align*}
&\ (a_0, \ldots, a_{d_0}),\ &\ (b_0, \ldots, b_{d_1})
\end{align*}
be the normalized evaluations of $u_0$ and $u_1$ at $z_+$. Then by definition, $x_{(d_1, d_2)}:= \zeta_{(d_1, d_2)}(x_1, x_2)$ is represented by the map with domain $\Sigma_1 \vee \Sigma_2$
\beqn
u_{(d_1, d_2)} (z) = \left\{ \begin{array}{cc} {\rm [} u_{1, 0} (z), \ldots, u_{1, d_1}(z), 0, \ldots, 0 {\rm ]},\ z\in \Sigma_1,\\
                                       {\rm [} b_0 u_{2, 0}(z), \ldots, b_{d_1} u_{2, 0}(z), u_{2, 1}(z), \ldots, u_{2, d_2}(z){\rm ]},\ z \in \Sigma_2.\end{array}\right.
\eeqn
Then $\zeta_{(d_0, d_1 + d_2)} (x_0, x_{(d_1, d_2)} )$ is represented by the map with domain $\Sigma_0 \vee \Sigma_1 \vee \Sigma_2$
\begin{multline*}
u_{012} (z) \\
= \left\{ \begin{array}{cc} {\rm [} u_{0,0}(z), \ldots, u_{0, d_0}(z), 0, \ldots, 0 {\rm ]},\ & z\in \Sigma_0,\\
                                        {\rm [} a_0 u_{1, 0}(z), \ldots, a_{d_0} u_{1, 0}(z), u_{1, 1}(z), \ldots, u_{1, d_1 }(z), 0, \ldots, 0 {\rm ]},\ &\  z\in \Sigma_1,\\
                                        {\rm [}  b_0 a_0 u_{2, 0}(z), \ldots, b_0 a_{d_0} u_{2, 0}(z) ,  b_1 u_{2, 0}(z), \ldots, b_{d_1} u_{2,0}(z), u_{2, 1}(z), \ldots, u_{2, d_2}(z) {\rm ]},\ &\ z \in \Sigma_2. \end{array} \right.
                                        \end{multline*}
On the other hand, $x_{(d_0, d_1)}:=\zeta_{(d_0, d_1)} (x_0, x_1)$ is represented by the map $u_{(d_0, d_1)}: \Sigma_0 \vee \Sigma_1 \to \mb{CP}^{d_0 + d_1}$ whose representation is the same as $u_{012}|_{\Sigma_0 \vee \Sigma_1}$ above after removing the last $d_2$ zeroes. Its evaluation at $z_+$ is represented by the vector  
\beqn
(b_0 a_0, \ldots, b_0 a_{d_0}, b_1, \ldots, b_{d_1}) \in \mb{C}^{d_0 + d_1 + 1}.
\eeqn
which is a unit vector and is the normalized evaluation of $x_{(d_0, d_1)}$. Hence we can see from the definition of $\zeta_{(d_0  + d_1, d_2)}$ that the point $\zeta_{(d_0 + d_1, d_2)} ( x_{(d_0, d_1)}, x_2)$ is also represented by the map $u^{012}$.
\end{proof}

Now given a partition $d=d_1 + d_2 + d_3$, use \eqref{eqn:G-embedding}, we can define a map 
\beqn
\zeta_{(d_0, d_1, d_2)}^\sim: G_d \times_{G_{d_0}\times G_{d_1}\times G_{d_2}}( B_{d_0} \times B_{d_1}\times B_{d_2}) \to \partial^{(d_0, d_1, d_2)} B_d
\eeqn
as follows. For any $[g, x_0, x_1, x_2 ] \in G_d \times_{G_{d_0}\times G_{d_1}\times G_{d_2}} ( B_{d_0}\times B_{d_1}\times B_{d_2})$, define 
\beqn
\zeta_{(d_0, d_1, d_2)}^\sim ([g, x_0, x_1, x_2]):= g (\zeta_{(d_0, d_1 + d_2)}(x_0, \zeta_{(d_1, d_2)} (x_1, x_2)) ).
\eeqn
It is straightfoward to check that this is a well-defined equivariant map. On the other hand, we can also define a $G_d$-equivariant map 
$$
[g,x_0,x_1,x_2] \mapsto g(\zeta_{(d_0 + d_1, d_2)} ( \zeta_{(d_0, d_1)} (x_0, x_1), x_2)).
$$
The following commutative diagram coming from Proposition \ref{associativity}
\beqn
\xymatrix{ &  G_d (B_{d_0}\times B_{d_1}\times B_{d_2}) \ar[ld] \ar[rd] &\\
  G_d \big( B_{d_0} \times \partial^{(d_1, d_2)} B_{d_1+d_2} \big) \ar[rd] & &  G_d \big( \partial^{(d_0, d_1)} B_{d_0+d_1} \times B_{d_2} \big) \ar[ld]\\
   &  \partial^{(d_0, d_1, d_2)} B_d &  }
\eeqn
shows that these two equivariantization maps agree with each other. Therefore, the map $\zeta_{(d_0, d_1, d_2)}^\sim$ is indeed unambiguously well-defined, independent of the ways of grouping the partitions.

It is straightforward to carry out the above discussion to the case with more factors. Recall for $\delta = (d_0, \ldots, d_l) \in \bA_d$, we defined
\beqn
B_\delta:= B_{d_0}\times \cdots \times B_{d_l}
\eeqn
which has the action of $G_\delta$. Using the group embedding \eqref{eqn:G-embedding}, define 
\beqn
B_\delta^\sim:= G_d \times_{G_\delta} B_\delta.
\eeqn
Then there is a $G_d$-equivariant map 
\begin{equation}\label{eqn:embedding}
\zeta_\delta^\sim: B_\delta^\sim \hookrightarrow \partial^\delta B_d,
\end{equation}
which, for instance, can be constructed by writing $B_\delta^\sim$ as
\beqn
G_d (B_{d_0} \times \cdots \times_{G_{d_{l-2}} \times G_{d_{l-1} + d_l}} (G_{d_{l-1} + d_l}\times_{G_{d_{l-1}}\times G_{d_l}} (B_{d_{l-1}}\times B_{d_l})) ).
\eeqn
Proposition \ref{associativity} indicates that such an inductive construction of the map $\zeta_\delta^\sim$ is independent of the order of the factorization. This embedding is smooth. The following statement follows from the arguments in the proof of Lemma \ref{lem:normal-position} and we record it here.

\begin{lemma}
$\zeta_\delta^\sim$ is bijective onto $(\partial^\delta B_d)^{\rm normal}$. \qed
\end{lemma}

Lastly we need to analyze the difference between $B_\delta^\sim$ and $\partial^\delta B_d$. We will show that in fact, $\partial^\delta B_d$ can be viewed as the total space of a $G_{\delta}$-equivariant vector bundle over $B_\delta^\sim$. 

We introduce certain notations which will also be used in the thickening construction. For all $d\geq 0$, denote 
\beq\label{eqn57}
{\bm Q}_d = \tilde {\bm Q}_d/ {\mb R}_+\ {\rm where}\ \tilde {\bm Q}_d:= \Big\{ \tilde h \in \mb{C}^{(d+1)\times (d+1)}\ |\ \tilde h^\dagger   = \tilde h,\ \tilde h_{00} \neq 0 \Big\}.
\eeq
Here our convention is that the indices of the Hermitian matrix $\tilde h \in \tilde {\bm Q}_d$ range from $0$ to $d$; the multiplicative group ${\mb R}_+$ acts on $\tilde {\bm Q}_d$ by scalar multiplication on each entry. The ${\mb R}_+$-orbit of $\tilde h \in \tilde {\bm Q}_d$ is denoted by $[\tilde h]$. We identify ${\bm Q}_d$ with
\beq\label{eqn:q-star}
{\bm Q}_d^*:= \Big\{ h \in \mb{C}^{(d+1)\times (d+1)}\ |\ h^\dagger = h,\ h_{00} = 0 \Big\} 
\eeq
in the way that a Hermitian matrix $h$ with $h_{00} = 0$ is identified with the ${\mb R}_+$-orbit of $\tilde h = I_{d+1} + h$. Then ${\bm Q}_d$ is a real vector space with dimension equal to $d^2 + 2d$. Moreover, by identifying $G_d$ with $U(d)\subset U(d+1)$ using the correspondence \eqref{eqn52}, the usual conjugation action of $U(d+1)$ on $(d+1)\times (d+1)$ Hermitian matrices restricts to a linear action of $G_d$ on ${\bm Q}_d$. 

To go further, for each partition $\delta = (d_0, \ldots, d_l) \in \bA_d$, define 
\beqn
{\bm Q}_\delta:= \Big\{ [\tilde h] \in {\bm Q}_d\ |\ \tilde h_{ij} \neq 0 \Longrightarrow \exists\ a\geq -1 \ {\rm s.t.}\ d_0 + \cdots + d_a \leq i, j \leq d_0 + \cdots + d_{a+1} \Big\}
\eeqn
where for $a=-1$, the constraints on the entries are given by $0 \leq i,j \leq d_0$,
and its complement 
\beqn
\begin{split}
\check {\bm Q}_\delta:= &\ \Big\{ [\tilde h] \in {\bm Q}_d\ |\ \tilde h_{ij} = 0 \ \forall a = 0, \ldots, l,\  d_0 + \cdots + d_a \leq i, j \leq d_0 + \cdots + d_{a+1} \Big\}\\
\cong &\ \Big\{ h \in {\bm Q}_d^*\ |\ h_{ij} = 0 \ \forall a = 0, \ldots, l,\  d_0 + \cdots + d_a \leq i, j \leq d_0 + \cdots + d_{a+1} \Big\}.
\end{split}
\eeqn
This vector space is invariant under the adjoint action of $G_\delta$. Hence we have a  $G_{\delta}$-equivariant splitting 
\beqn
{\bm Q}_d = {\bm Q}_\delta \oplus {\bm Q}_{d,\delta}.
\eeqn
Moreover, ${\bm Q}_d$ is stratified by partitions. When $\delta\leq \eta$, one has the obvious inclusion 
\beqn
{\bm Q}_\delta \subset {\bm Q}_\eta.
\eeqn
Define
\beqn
{\bm Q}_{\eta\delta}:= {\bm Q}_\eta \cap {\bm Q}_{d, \delta}.
\eeqn
Using the $G_\delta$-action on the space ${\bm Q}_{d,\delta}$, we define a $G_{\delta}$-equivariant vector bundle 
\beqn
Q_{d,\delta}:= B_\delta \times {\bm Q}_{d,\delta}.
\eeqn
The $G_d$-equivariantization of this vector bundle is denoted by 
\begin{equation}\label{eqn:stabilization-base}
 Q_{d,\delta}^\sim:= G_d\times_{G_\delta} Q_{d,\delta} \to B_{\delta}^\sim.
\end{equation}

\begin{prop}\label{prop523}
There is a $G_d$-equivariant diffeomorphism 
\beqn
\rho_{\delta}^\sim: Q_{d,\delta}^\sim \cong \partial^\delta B_d
\eeqn
which extends the embedding $\zeta_\delta^\sim: B_\delta^\sim \hookrightarrow \partial^\delta B_d$.
\end{prop}

\begin{proof}
We first define a map $\rho_\delta: Q_{d,\delta} \to \partial^\delta B_d$ which extends $\zeta_\delta: B_\delta \to \partial^\delta B_d$ as follows. Given a Hermitian matrix $h \in \check {\bm Q}_\delta$, we identify it with a complex matrix $\rho_h$ obtained by changing all entries in $h$ that are below the diagonal to zero. Then define
\beqn
\rho_\delta (h, x_0, \ldots, x_l):= (I_{d + 1} + \rho_h ) ( \zeta_\delta ( x_0, \ldots, x_d)).
\eeqn
Note that $I_{d + 1} + \rho_h$ is a nonsingular matrix because $\rho_h$ is nilpotent. It is straightforward to check that this map is equivariant with respect to the group embedding $G_\delta \to G_d$. Hence it extends to a $G_d$-equivariant map 
\beqn
\rho_\delta^\sim: Q_{d,\delta}^\sim \to \partial^\delta B_d.
\eeqn
We need to show that this map is a diffeomorphism.

\begin{enumerate}
    \item (Injectivity) Suppose 
    \beqn
  x =   \rho_\delta^\sim ( [ g, h, x_0, \ldots, x_l]) = \rho_\delta^\sim ([g', h', x_0', \ldots, x_l']).
    \eeqn
    Then by the definition of $\rho_\delta^\sim$, we may assume that $g' = 1\in G_d$. Then, notice that the flag induced by $\rho_\delta (h', x_0', \ldots, x_l')$ is the standard one and $g$ preserves the flag. Then $g\in G_\delta$. Hence we may also assume $g = 1 \in G_{pq}$. By considering the fans associated to stable maps, one can also conclude that $h = h'$. As $I_{d+1} + \rho_h$ is invertible, one obtains that $x_i = x_i'$ for all $i = 0, \ldots, l$. Hence $\rho_\delta^\sim$ is injective.
    
    \item (Surjectivity) Let $x\in \partial^\delta B_d$ be represented by a map 
    \beqn
    u = (u_0, \ldots, u_l): \Sigma_0 \vee \cdots \vee \Sigma_l \to \mb{CP}^d.
    \eeqn
    Let $(V_0, \ldots, V_l)$ be the induced flag. Then there exists a unitary matrix $g$ which transforms this flag to the standard one. Hence we may assume the flag associated to $x$ is standard. Let $(W_{d_0}, \ldots, W_{d_l})$ be the fan associated to $x$. We construct an element $h \in {\bm Q}_{d, \delta}$ such that
    \beqn
    (I_{d_{pq}+1} + \rho_h) (W_{d_0}, \ldots, W_{d_l})
    \eeqn
    is in the normal position. $h$ can be constructed inductively. First, let $y_0 \in \mb{C}^{d_0 + 1}$ be the noramlized evaluation of $u_0$ at $z_+$. Regard $y_0$ as a vector in $\mb{C}^{d_0 + d_1 + 1}$. Then there exist vectors $w_1, \ldots, w_{d_1} \in \mb{C}^{d_0 + d_1 +1}$ such that
    \beqn
    W_{d_1} = {\rm span}( y_0, w_1, \ldots, w_{d_1}).
    \eeqn
    Then as $W_{d_0} = {\mb C}^{d_0 + 1}$ and $W_{d_0} + W_{d_1} = {\mb C}^{d_0 + d_1 + 1}$, there exists a complex matrix $A_{d_1}^{d_0}$ whose nonzero entries are contained in the upper-right $d_0 \times d_1$-block such that
    \beqn
    (I_{d_{pr_2}+1} - A_{d_1}^{d_0}) (w_1, \ldots, w_{d_1}) = (e_{d_0 +1}, \ldots, e_{d_0 + d_1} )
    \eeqn
    (where $e_0, \ldots, e_{d_0 + d_1}$ form the standard basis of ${\mb C}^{d_0 + d_1 +1}$). Inductively, one can construct a matrix 
    \beqn
    A:= \left[ \begin{array}{ccccc} 0 & A_{d_1}^{d_0} & A_{d_2}^{d_0} & \cdots & A_{d_l}^{d_0} \\
    0 & 0 & A_{d_2}^{d_1} & \cdots & A_{d_l}^{d_1} \\
    \vdots & \vdots & \vdots & \ddots & \vdots \\
    0 & 0 & 0 & 0 & 0 
    \end{array}\right]
    \eeqn
    such that $(I_{d+1}  - A) ( W_{d_0}, \ldots, W_{d_l})$ is a fan in the normal position. Apply $I_{d +1} - A$ to the stable map $u$, one obtains a stable map $u'$ which is in the normal position in $\partial^\delta B_d$. The nilpotence of $A$ and its block form imply that $-A = \rho_h$ for some $h \in {\bm Q}_{d, \delta}$. As stable maps in normal positions are all in the image of $\zeta_\delta^\sim$, the above argument implies that $\rho$ is surjective.
    \end{enumerate}
The smoothness of $\rho_\delta^\sim$ and its inverse follows from the definition and the smoothness of $\zeta_\delta^\sim$. 
\end{proof}

\subsubsection{Product of the bundle of quadratic forms}

The bundles of quadratic forms play an crucial role in the construction of global charts, as they will be part of the obstruction bundle. We describe the multiplicative structures of these bundles. Define for each $d\geq 1$ the $G_d$-equivariant bundle 
\beqn
Q_d:= B_d \times {\bm Q}_d.
\eeqn
We would like to define a collection of equivariant bundle embeddings
\beqn
\xymatrix{   Q_{d_0}\boxplus Q_{d_1} \ar[rr] \ar[d] & & Q_{d_0 + d_1} \ar[d] \\
             B_{d_0}\times B_{d_1}  \ar[rr]_{\zeta_{(d_0, d_1)}} & & B_{d_0 + d_1} }
             \eeqn
which are associative. Let us look at the partition $d = d_0 + d_1$ and let 
\beqn
Q_{d_0 \to d}, Q_{d_1 \to d} \to B_{(d_0, d_1)} \cong B_{d_0} \times B_{d_1} 
\eeqn
be the bundles $B_{(d_0, d_1)} \times {\bm Q}_{d_0}$ and $B_{(d_0, d_1)}\times {\bm Q}_{d_1}$. Define embeddings 
\begin{align*}
&\ \widehat \zeta_{d_0 \to d}^Q: Q_{d_0 \to d} \to Q_d,\ &\ \widehat \zeta_{d_1 \to d}^Q:  Q_{d_1 \to d} \to Q_d
\end{align*}
as follows. 
\begin{enumerate}
    \item For $\widehat \zeta_{d_0 \to d}^Q$: for each element $h_{d_0} \in {\bm Q}_{d_0}$ (which is regarded as a Hermitian matrix whose $(0, 0)$-entry vanishes), define
\beqn
\Big( \widehat\zeta_{d_0 \to d}^Q (x_{d_0}, x_{d_1}; h_{d_0} ) \Big)_{ij} = \left\{ \begin{array}{cc} (h_{d_0})_{ij},\ &\ 0 \leq i, j \leq d_0,\\
                                             0,\ &\ {\rm otherwise,}
                                             \end{array}\right.
\eeqn
which gives a Hermtian matrix in ${\bm Q}_{d}$ (regarded as a matrix whose $(0, 0)$-entry vanishes). 

\item For $\widehat \theta_{d_1 \to d}^Q$: for $x = (x_{d_0}, x_{d_1})$, let the normalized evaluation of $x_{d_0}$ at $z_+$ to be $(a_0, \ldots, a_{d_0})$. Then there exists an element $g_x \in U (d_0 +1)$ such that 
\beq\label{group2}
g_x (0, \ldots, 1) = (a_0, \ldots, a_{d_0}).
\eeq
Embed $g_x$ into $PU(d_0 +1)$ with image
\beqn
\left[ 
\begin{array}{cc}
(g_x)_{(d_0+1) \times (d_0 + 1)} & 0 \\
0 & I_{d_1}
\end{array}
\right]
\eeqn
and denote it by $g_x^+$. Then for $h_{d_1}\in {\bm Q}_{d_1}$, define the Hermitian matrix with vanishing $(0,0)$-entry
\beqn
h_{d_1}^+\in {\bm Q}_{d},\ (h_{d_1}^+)_{ij} = \left\{  \begin{array}{cc} (h_{d_1})_{i-d_0 \ j-d_0},\ &\ d_0 \leq i, j \leq d,\\
                                             0,\ &\ {\rm otherwise}
                                             \end{array}\right.
                                             \eeqn
                                             and define
\beqn
\widehat\zeta_{d_1 \to d }^Q(x_{d_0 }, x_{d_1}; h_{d_1} ) = g_x^+  h_{d_1}^+ (g_x^+)^{-1}.
\eeqn
We can easily check that this is independent of the choice of $g_x$ satisfying \eqref{group2}.
\end{enumerate}
Then $\widehat\zeta_{d_0 \to d}^Q$ and $\widehat \zeta_{d_1 \to d}^Q$ canonically correspond to bundle maps, which are denoted by the same symbols respectively. Then define a map
\beqn
\widehat\zeta_{(d_0,d_1)}^Q : Q_{d_0 \to d} \oplus Q_{d_1 \to d} \to B_{(d_0, d_1)} \times {\bm Q}_{d}
\eeqn
%\beqn
%\widehat \theta_{prq}^Q: Q_{prq} = Q_{pr \to pq}\oplus Q_{rq \to pq} \to Q_{pq}|_{V_{prq}} = V_{prq}\times {\bm Q}_{d_{pq}}
%\eeqn
to be 
\beqn
\widehat\zeta_{(d_0,d_1)}^Q = \widehat \zeta_{d_0 \to d}^Q + \widehat \zeta_{d_1 \to d}^Q.
\eeqn

\begin{lemma}
The bundle map $\widehat\zeta_{(d_0, d_1)}^Q$ is a $G_{d_0}\times G_{d_1}$-equivariant linear bundle embedding with image being the bundle $B_{(d_0, d_1)}\times {\bm Q}_{(d_0, d_1)}$.
\end{lemma}

\begin{proof}
The linearity follows from the definition of ${\bm Q}_d$ and the definition of the maps. The $G_{d_0}\times G_{d_1}$-equivariance of $\widehat \zeta_{d_0 \to d}^Q$ is obvious from the definition; the $G_{d_1}$-equivariance of $\widehat \zeta_{d_1\to d}^Q$ is also straightforward. However, the $G_{d_0}$-equivariance of $\widehat \zeta_{d_1 \to d}^Q$ requires a verification. Indeed, fix $ x = (x_{d_0}, x_{d_1})$; suppose the normalized evaluation of $x_{d_0}$ is $z = ( a_0, \ldots, a_{d_0} )$. Then choose $g \in G_{d_0}$. Then $g x = (g x_{d_0}, x_{d_1} )$ and the normalized evaluation of $g x_{d_0}$ is $g z$. So if $[z] = [a_0, \ldots, a_{d_0}] = g_x [0, \ldots, 1]$, then $[gz] = g g_x [0, \ldots, 1]$. Then one can check the $G_{d_0}$-equivariance. The injectivity and the fact that the image is exactly $B_{(d_0, d_1)}\times {\bm Q}_{(d_0, d_1)}$ is also straightforward to verify following the proof of Lemma \ref{lem:normal-position}.
\end{proof}

The following proposition is the counterpart of Proposition \ref{associativity} for quadratic bundles, which can be proved based on keeping track of the action of $g_x$ from \eqref{group2}.

\begin{prop}\label{prop:Q-gluing}
The bundle embeddings defined above are associative. Namely, for $d_0, d_1, d_2$ such that $d = d_0 + d_1 + d_2$, recall
$$
B_{(d_0, d_1, d_2)} = B_{d_0} \times B_{d_1} \times B_{d_2}.
$$
Then given $x \in B_{(d_0, d_1, d_2)}$ and 
$$
(h_{d_0}, h_{d_1}, h_{d_2}) \in {\bm Q}_{d_0} \times {\bm Q}_{d_1} \times {\bm Q}_{d_2},
$$
one has
\beqn
\widehat \zeta_{(d_0, d_1 + d_2)}^Q (x, h_{d_0}, \widehat \zeta_{(d_1, d_2)}^Q(x, h_{d_1}, h_{d_2})) = \widehat \zeta_{(d_0 + d_1, d_2)}^Q (x, \widehat \zeta_{(d_0, d_1)}^Q( x, h_{d_0}, h_{d_1}), h_{d_2}). \qed
\eeqn
\end{prop}

As a consequence, there is a well-defined $G_{d_0} \times G_{d_1} \times G_{d_2}$-equivariant linear bundle embedding
$$
\widehat\zeta_{(d_0,d_1,d_2)}^Q: B_{(d_0, d_1, d_2)} \times {\bm Q}_{d_0} \times {\bm Q}_{d_1} \times {\bm Q}_{d_2} \to B_{(d_0, d_1, d_2)} \times {\bm Q}_d,
$$
whose image is $B_{(d_0, d_1, d_2)} \times {\bm Q}_{(d_0, d_1, d_2)}$. The associativity further implies that for any $\delta = (d_0, \dots, d_l) \in \bA_{d}$, there exists a $G_{\delta}$-equivariant linear bundle embedding 
\begin{equation}\label{eqn:bundle-embed}
\widehat \zeta_{\delta}^Q: Q_{\delta} := B_{\delta} \times {\bm Q}_{d_0} \times \cdots \times {\bm Q}_{d_l} \xrightarrow{\sim} B_{\delta} \times {\bm Q}_{\delta} \subset B_{\delta} \times {\bm Q}_{d}.
\end{equation}
Take the $G_d$-equivariantization using the group embedding $G_\delta \hookrightarrow G_d$, we obtain a $G_d$-equivariant vector bundle
$$
Q_{\delta}^{\sim} = G_d \times_{G_\delta} Q_\delta \to B_\delta^\sim
$$
which is $G_d$-equivariantly embedded in $B_\delta^\sim \times {\bm Q}_{d}$.
Combined with the gluing map \eqref{eqn:embedding}, we obtain a $G_d$-equivariant bundle embedding
$$
(\widehat \zeta_{\delta}^Q)^\sim: Q_{\delta}^{\sim} \to Q_d |_{\partial^\delta B_d}
$$
covering the embedding $\zeta_\delta^\sim: B_\delta^\sim \hookrightarrow \partial^\delta B_d$. Using the bundle $\check{Q}_{\delta}^{\sim}$ from \eqref{eqn:stabilization-base}, we can stabilize $\zeta_\delta^\sim$ by $\check{Q}_{\delta}^{\sim}$ to get a diffeomorphism onto its image.

%\textcolor{red}{The expositions here might need to be expanded after checking later sections.}

\subsubsection{Normal complex structure}

Here we prove the following result.

\begin{prop}\label{base_normal_complex}
For each $d\geq 1$, $(B_d \times {\bm Q}_d)/G_d$ is a normally complex orbifold with corners.
\end{prop}

\begin{proof}
We first re-examine the relation between $B_d$ and ${\bm Q}_d$. Define 
\beqn
{\bm Q}_d':= \left\{ \left. A  = \left[ \begin{array}{cc} 0 & * \\ 0 &  a \end{array} \right] \ \right| \ a^\dagger = a \right\}
\eeqn
Then one can check that ${\bm Q}_d'\cong {\bm Q}_d$ as $G_d$-spaces where on both the $G_d$ action is given by the conjugation. Then we define an ``action'' of ${\bm Q}_d'$ on $B_d$ by 
\beqn
A\cdot x \mapsto (I_{d+1} + A) (x).
\eeqn
It is easy to see that the ${\bm Q}_d'$-orbit through $x$ is a local slice of the $G_d$-action. Suppose $\Gamma \subset G_d$ is a finite subgroup and $x \in (B_d)^\Gamma$. Let ${\bm Q}_d^\Gamma \subset {\bm Q}_d$ be the $\Gamma$-invariant subspace. Then we see that the ${\bm Q}_d^\Gamma$-orbit through $x$ is contained in the $\Gamma$-fixed locus. We can check that locally it coincides with the fixed locus. Hence the normal direction to to the $\Gamma$-fixed locus of $B_d$ and the orthogonal complement of ${\bm Q}_d^\Gamma$ are isomorphic representations of $\Gamma$. Hence one can define a natural normal complex structure on $(B_d\times {\bm Q}_d)/G_d$. 
\end{proof}

\subsubsection{Abelian gauge theory on punctured spheres}

We include certain simple facts about abelian gauge theory over the infinite cylinder.%, or more generally, spheres with $\leq 2$ punctures. We provide the details for the cylindrical case and it is straightforward to have parallel discussions for once-punctured spheres. 
Let $\Theta = {\mb R} \times S^1$ be the infinite cylinder, equipped with the standard flat Riemannian metric. Let $(s, t)$ be the standard coordinates where $s \in {\mb R}$ and $t \in S^1$. Consider an everywhere non-negative 2-form 
\beqn
\Omega = \sigma(s, t) ds \wedge dt, \sigma(s, t) \geq 0
\eeqn
which decays exponentially on the ends: there exist a positive real number $\delta>0$ and a sequence of positive real numbers $C_0, C_1, \ldots$ such that 
\beq\label{eqn41}
|\nabla^l \sigma(s, t)| \leq C_l e^{-\delta|s|},\ l = 0, 1, \ldots.
\eeq
We assume that 
\beqn
\int_\Theta \Omega = d \in {\mb Z}_{\geq 0}.
\eeqn

\begin{lemma}\label{lemma:connection-exist}
There exists a unitary connection $A = A_\Omega$ on the trivial line bundle $L = \Theta \times {\mb C}$, unique up to gauge transformation, which satisfies the following conditions:
\begin{enumerate}
    \item The curvature form of $A_\Omega$ is equal to $- 2\pi {\bm i} \Omega$.
    
    \item The holonomy of $A_\Omega$ along any circle $\{s\}\times S^1\subset \Theta$, which is a well-defined element in $U(1)$, converges to the identity as $s \to \pm \infty$.
    \end{enumerate}
\end{lemma}

\begin{proof}
Choose two integers $m_-$ and $m_+$ such that $d = m_- - m_+$. Choose a reference smooth connection $A_0$ which is equal to $d - {\bm i} m_\pm dt$ near $\pm \infty$. Consider an arbitrary connection of the form 
\beqn
A = A_0 + \phi ds + \psi dt.
\eeqn
Then the curvature condition $F_A = -2\pi{\bm i} \Omega$ and a gauge fixing condition induced by the trivial product connection give the equations 
\beq\label{eqnCR}
\left\{ \begin{array}{rl} \displaystyle \frac{\partial \psi}{\partial s} - \frac{\partial \phi}{\partial t} & = *( - 2\pi {\bm i} \Omega  - F_{A_0}),\\ [0.2cm]
\displaystyle \frac{\partial \phi}{\partial s} + \frac{\partial \psi}{\partial t} & = 0.\end{array} \right.
\eeq
The left-hand-side is indeed the standard Cauchy--Riemann operator on $f = \phi + {\bm i} \psi$. Introduce a small $\epsilon>0$. Then the operator 
\beqn
\frac{\partial}{\partial \ov{z}}: W^{1,p, \epsilon}(\Theta) \otimes {\mb C} \to L^{p, \epsilon}(\Theta) \otimes {\mb C}
\eeqn
is a Fredholm operator with index $-1$ with trivial kernel and 1-dimensional cokernel spanned by a function with nonzero total integral. By the exponential decay property of $\Omega$ (see \eqref{eqn41}), when $\epsilon < \delta$, the right hand side of \eqref{eqnCR} belongs to $L^{p,\epsilon}$. The choice of the reference connection $A_0$ implies that the right hand side of \eqref{eqnCR} is in the image of the Cauchy--Riemann operator. Therefore there exists a unique solution $f = \phi + {\bm i} \psi$ of class $W^{1, p, \epsilon}$ to \eqref{eqnCR}. Hence the existence is proved. 

For uniqueness, one can see that any other solution $A'$ differs from $A$ by a closed 1-form. Up to gauge transformation, we can assume that the 1-form is $a dt$ for a constant $a$. Then the holonomy condition forces that $a \in {\mb Z}$, which is given by the effect of a gauge transformation. 
\end{proof}

Now we would like to realize $A$ as certain singular connection on a degree $d$ holomorphic line bundle over $\mb{CP}^1$. To this end, let $L_d \to \mb{CP}^1$ be a degree $d$ holomorphic line bundle equipped with a Hermitian metric and let $A_d$ be the Chern connection. Choosing $A_0$ to be the same reference connection as in the proof of Lemma \ref{lemma:connection-exist}, then over $\Theta = \mb{CP}^1\setminus \{\pm\infty\}$ there exists a trivialization of $L_d$ such that 
\beqn
A_d = A_0 + \phi_d ds + \psi_d dt
\eeqn
such that $\phi_d$, $\psi_d$, and all of their derivatives converge exponentially to $0$ as $s \to \pm\infty$. Moreover, the difference of curvature $\Omega_{A_d} - \Omega_A$, measured in cylindrical metric, decays exponentially like $e^{-\delta|s|}$ with all derivatives. Then consider a general complex gauge transformation $g = e^h = e^{h' + {\bm i} h''}$ where $h', h'': \Theta \to {\mb R}$ are functions. Consider the equation $(e^h)^* A_d = A$ which is equivalent to 
\beqn
{\bm i} dh'' + {\bm i} * dh' = A_d - A
\eeqn
which is equivalent to 
\beq\label{eqn44}
\left\{ \begin{array}{cc}  \displaystyle {\bm i} \left( \frac{ \partial h''}{\partial s} - \frac{\partial h'}{\partial t} \right) & = \phi_d - \phi,\\
                           \displaystyle {\bm i} \left( \frac{\partial h'}{\partial s} + \frac{ \partial h''}{\partial t} \right) & = \psi_d - \psi. \end{array} \right.
\eeq
This is again the standard Cauchy--Riemann operator on cylinder. If we regard the left-hand-side as a Fredholm operator from $W^{1,p,\epsilon}$ to $L^{p,\epsilon}$, then it has (complex) index $-1$ with the cokernel generated by a function whose total integral is nonzero. Now we allow $h$ to have nonzero limits at $\pm\infty$ so that $h-h(\pm\infty)$ is of class $W^{1, p, \epsilon}$ near the infinities. Then the Fredholm index becomes $1$ with kernel being the subspace of constant functions. Therefore, there exists a solution to \eqref{eqn44} unique up to adding a constant. We may then view $A$ as the Chern connection on $L_d$ with respect to the Hermitian metric rescaled by $e^{2h'}$, which may be a singular metric on $\mb{CP}^1$ but is smooth over the cylinder and continuous over the poles. We summarize these elementary discussions as the following lemma.

\begin{lemma}\label{HK}
Given a 2-form $\Omega$ as in Lemma \ref{lemma:connection-exist}, there exists a continuous Hermitian metric on $L_d$ whose Chern connection has curvature form equal to $-2\pi {\bm i} \Omega$. Moreover two such Hermitian metrics differ by a constant.
\end{lemma}

We generalize the above lemma to the case of prestable cylinders. On a nodal curve, a smooth/holomorphic function or map means a continuous function or map whose pullback to the normalization is smooth/holomorphic. A smooth/holomorphic $k$-form (with $k\geq 1$) means a collection of smooth/holomorphic $k$-forms on its normalization. Then given a genus zero nodal curve $\Sigma$ with irreducible components $\Sigma_1, \ldots, \Sigma_s$ and integers $d_1, \ldots, d_s$, there exists a unique up to isomorphism smooth/holomorphic line bundle over $\Sigma$ whose restriction to $\Sigma_i$ has degree $d_i$. Moreover, the automorphism group of such a holomorphic line bundle is isomorphic to ${\mb C}^*$. Then Lemma \ref{HK} has the following corollary.

\begin{cor}(cf. \cite[Lemma 6.8]{AMS})\label{HK2}
Let $(\Sigma, {\bf L})$ be a prestable cylinder and $\Omega$ be a smooth 2-form on $\Sigma$. Suppose
\begin{enumerate}
    \item the integration of $\Omega$ over each irreducible component is integral, and
    
    \item over cylindrical components, $\Omega$ satisfies \eqref{eqn41} for some common $\delta>0$.
    \end{enumerate}
Then there exists a holomorphic Hermitian line bundle on $\Sigma$ whose curvature form is $-2\pi {\bm i} \Omega$. Moreover, this line bundle is unique up to isomorphism.
\end{cor}

Under Hypothesis \ref{hyp51}, suppose $u: \Sigma \to X$ is a smooth map whose restriction to each cylindrical component converges to a periodic orbit of $H$ at $\pm\infty$ in an exponential rate. Then consider the 2-form $\Omega_{u, H} \in \Omega^2(\Sigma)$ defined as follows:

\beq\label{2form}
\Omega_{u, H} := \left\{ \begin{array}{lr} 
\displaystyle u^* \omega & \text{ on spherical components} \\
\displaystyle  u^* \omega - d( H_t (u) dt) & \text{ on cylindrical components.}
\end{array} \right.
\eeq
Then by Corollary \ref{HK2} there is a unique Hermitian line bundle $L_u \to \Sigma$ if the integration of the $2$-form $\Omega_{u, H}$ is integral over each cylindrical component. In particular, each stable Floer trajectory induces a line bundle over its domain, thanks to Hypothesis \ref{hyp51}. The desired integrality property of $\Omega_{u, H}$ for stable Floer trajectories follows from the Stokes' formula and the integrality of symplectic actions. Moreover, the Hermitian line bundle $L_u$ has strictly positive curvature $2$-forms over the nontrivial components of $u$.

%\textcolor{purple}{Construct $2$-forms on once-punctured spheres.}

%Now we briefly describe the situation for once-punctured spheres. Let $\Xi = S^2 \setminus \{\infty\}\cong {\mb C}$ equipped with a cylindrical metric, say
%\beqn
%\frac{1}{1 + x^2 +y^2} dx dy
%\eeqn
%where $z = x + {\bm i} y$ is the Euclidean coordinate on $\mb{C}$. Then $s + {\bm i} t:= \log z$ is the standard cylindrical %coordinate near infinity.

%\begin{lemma}
%\begin{enumerate}
%    \item Let $\Omega \in \Omega^2( \Xi)$ be an everywhere non-negative 2-form which decays exponentially near infinity, i.e., satisfying \eqref{eqn41} with respect to the cylindrical metric. Moreover, the integral of $\Omega$ is a nonnegative integer $d \geq 0$. Then there exists a unitary connection $A_\Omega$ on the trivial line bundle $\Xi \times \mb{C}$, unique up to gauge transformation, such that the curvature form of $A_\Omega$ is $-2\pi {\bm i} \Omega$ and the holonomy of $A_\Omega$ along loops around the infinity converges to the identity.
    
%    \item Under the same assumption, let $L_d \to S^2$ be the degree $d$ smooth complex line bundle. Then there exists a continuous Hermitian metric on $L_d$ which is smooth away from infinity such that its Chern connection has curvature form equal to $-2\pi{\bm i} \Omega$. Moreover any two such Hermitian metrics differ by a constant factor.

%\end{enumerate}
%\end{lemma}

\subsection{Global chart construction II}

In this subsection, we provide a construction of a global Kuranishi chart for a single moduli space of Floer trajectories. The purpose is to showcase how to use the auxiliary moduli space $B_{pq}$ as a model for the Deligne--Mumford space in the Hamiltonian Floer theory which allows one to present the moduli space $\ov{\mc M}_{pq}$ as a global quotient, and to present how to adapt the geometric regularization method from \cite{AMS} to this setting. 

\subsubsection{Description of the K-chart.}

We first define the notion of framed curves in the case of Floer trajectories. 

\begin{defn}(cf. \cite[Definition 6.10]{AMS})\label{defn:framed_cylinder}
Given a moduli space $\ov{\mc M}_{pq}$ of Floer trajectories and a stratum $\alpha = pr_1 \cdots r_l q$, denoting $d=d_{pq}$, a {\bf framed cylinder} (of type $\alpha$) is a tuple $(u, \Sigma, F)$ where 
\begin{enumerate}
    \item $\Sigma$ is a prestable cylinder with $l+1$ horizontal levels (see Definition \ref{stablecylinder}). 
    
    \item $u: \Sigma \to M$ is a smooth map whose restriction to each cylindrical component converges to periodic orbits prescribed by the capped orbits $p, r_1,\dots, r_l, q$ at $\pm\infty$ in an exponential rate, and the topological energy of each horizontal level is prescribed by \eqref{eqn:floer-energy}. Moreover, the $2$-form $\Omega_{u, H}$ defined in Equation \eqref{2form} is non-negative and is strictly positive on each unstable component of $u$.
    
    %\item A Hermitian metric on $L_u$ whose curvature form is equal to $-2\pi {\bm i} \Omega_{u, H}$ (see \eqref{2form}) (the choice is unique up to a constant factor). We suppress this datum from the notation.
    
    \item $F = (f_0, f_1, \ldots, f_d)$ is basis %\footnote{We do not require that the list is a basis at this moment.}
    of global sections of the line bundle $L_u$ constructed using the $2$-form \eqref{2form}.     Moreover, the induced holomorphic map 
    \beq\label{framemap}
    \iota_F: \Sigma \to \mb{CP}^d,\ w \mapsto [f_0(w), \ldots, f_d(w)]
    \eeq
    is a stable map, which represents a point $[\iota_F] \in {\mc F}_{0,2}(d)$. Denote by 
    \beqn
    \tilde \iota_F: \Sigma \to {\mc C} = {\mc C}_{0,2}(d)
    \eeqn
    to be the identification between $\Sigma$ and the fiber of the universal curve ${\mc C}_{0,2}(d) \to {\mc F}_{0,2}(d)$ over $[\iota_F]$.\footnote{$\tilde \iota_F$ is called a ``domain map'' in \cite{AMS}.}
\end{enumerate}
\end{defn}

Note that Condition (2) above guarantees the existence of the Hermitian line bundle $L_u \to \Sigma$ by Corollary \ref{HK2} whose restriction to each unstable component has strictly positive degree. We can define the notion of isomorphisms of framed curves in an obvious way. 

\begin{defn}\label{defn:isomorphism_framed_cylinder}
An isomorphism from a framed cylinder $(u, \Sigma, F)$ to another framed cylinder $(u', \Sigma', F')$ consists of an isomorphism $\varphi: \Sigma \to\Sigma'$ of prestable cylinders such that $u' \circ \varphi = u$ and an isomorphism of holomorphic line bundles $\widehat \varphi: L_u \to L_{u'}$ which covers $\varphi$ and which is an {\it isometry up to a constant factor}, such that $f_i' \circ \varphi =  \widehat \varphi \circ f_i$ for all $i = 0, \ldots, d$. Two framed cylinders are isomorphic if there exists an isomorphism between them. Notice that scaling the frame $F = (f_0, \ldots, f_d)$ by a nonzero complex number produces an isomorphic framed curve.
\end{defn}

Given a framed cylinder $(u, \Sigma, F)$ and an element $g \in PGL(d+1)$, we can construct another framed cylinder in the following way. Observe that the sections $f_0, \dots, f_d$ define an embedding $\Sigma \to {\mb CP}^d$, under which the line bundle $L_u$ is the pullback of ${\mc O}(1)$ and the basis $(f_0, \dots, f_d)$ is obtained by pulling back the standard hyperplane sections of ${\mc O}(1)$. As an element of the automorphism group of ${\mc CP}^d$, the action of $g$ can be lifted to an automorphism of the line bundle ${\mc O}(1)$. Therefore, $g$ takes the standard hyperplane sections of ${\mc O}(1)$ to another basis of ${\mc O}(1)$. Composing such an action with the embedding induced by $(f_0, \dots, f_d)$ defines another framed cylinder and we denote it by
\beq\label{eqn:change-framing}
(u, \Sigma, g_* F).
\eeq

We construct a geometric thickening of a moduli space of Floer trajectories using framed cylinders. We choose the following data.

\begin{enumerate}
    %\item A $U(d+1)$-invariant Hermitian metric on ${\mc C} = {\mc C}_{0,2}(d)$, which induces a $U(d+1)$-invariant Chern connection on the tangent bundle $T{\mc C}$.

%    \item A collection of holomorphic Hermitian line bundles ${\mc L}_1, \ldots, {\mc L}_s \to {\mc C}$ such that the restriction of ${\mc L}_i$ to each fiber of ${\mc C}$ is ample and such that the $U(d-1)$-action on ${\mc C}$ is lifted to ${\mc L}_i$ which preserves the Hermitian structure.

    \item A Hermitian connection $\nabla^{TM}$ on $TM$ with respect to the Hermitian metric induced from $\omega$ and $J$ satisfying the following condition: $\nabla^{TM}$ is flat near all 1-periodic orbits (which are all embedded and disjoint by assumption) of $H$ and its holonomy along each such orbit is nondegenerate, i.e., does not have $1$ as an eigenvalue.
    
    \item For each moduli space $\ov{\mc M}_{pq}$, a positive integer $k$.\footnote{Later we will need $k$ to be sufficiently large to achieve transversality. We will also need to choose a list of integers instead of a single one to obtain a Kuranishi flow category.}
\end{enumerate}

%To describe the perturbed Floer equation over the cylinder defining the thickened moduli spaces we need to choose a certain Hermitian connection on the tangent bundle. Let $\nabla^{TM}$ be a Hermitian connection whose curvature 2-form vanishes near each 1-periodic orbit of $H$ and whose holonomy along each 1-periodic orbit has no eigenvalue being $1$.  Recall that ${\mc L} \rightarrow {\mc C}$ is the pullback of ${\mc O}(1)$ under the evaluation map.

The following definition introduces the thickening induced by one single line bundle. The discussion here is of expository nature, and the actual thickening we need involves multiple line bundles, as presented in detail in Section \ref{subsubsec:multi}. 

\begin{defn}[Thickened moduli]\label{def:thick-mod}
Fix a nonempty moduli space $\ov{\mc M}_{pq}$. Abbreviate $d = d_{pq}$. We define the following objects.
\begin{enumerate}
    \item The \emph{symmetry group} $G_{pq}$ is $G_{d_{pq}}$.

    \item The \emph{thickened moduli space} $V_{pq}$ parametrizes isomorphism classes of quadruples of the form
\beqn
(u, \Sigma, F, \eta)
\eeqn
where 
\begin{enumerate}
    \item $(u, \Sigma, F)$ is a framed cylinder. Suppose it is of type $\alpha$ for some stratum $\alpha = pr_1 \cdots r_l q$.
    
    \item The framing $F=(f_0, \dots, f_d)$ satisfies 
    $$[f_0(z_-): \cdots : f_d(z_-)] = [1:0:\cdots :0].\footnote{This condition will be used to construct a map from the thickened moduli space to $B_{pq}$.}$$
    
    \item $\eta$ is an element of the vector space 
    \beq\label{obstruction1}
    H^0 ( \ov{\rm Hom}( \iota_F^* T \mb{CP}^d, u^* TM) \otimes \iota_F^* {\mc O}(k) ) \otimes_{\mb C} \ov{H^0( \iota_F^* {\mc O}(k))_0} .
    \eeq
    Here the first $H^0$ is the kernel of the Cauchy--Riemann operator induced from the Hermitian connection $\nabla^{TM}$ on $TM$ and the standard complex structure of $\mb{CP}^d$; on the other hand, $H^0(\iota_F^* {\mc O}(k))_0$ is the space of holomorphic sections of $\iota_F^* {\mc O}(k) \to \Sigma$ which vanish at the two marked points $z_-$ and $z_+$.
    \end{enumerate}
The quadruple $(u, \Sigma, F, \eta)$ needs to satisfy the following perturbation of the Floer's equation (cf. \cite[Equation (6.7)]{AMS}): on each component of $\Sigma$ one has
\beq\label{thickening_equation}
\ov\partial_{J, H} u + \langle \eta \rangle \circ d \iota_F = 0,
\eeq
where $\ov\partial_{J, H} u = (du - X_{H} (u))_{J}^{0,1}$ in which $X_{H} (u) \in \Omega^1(\Sigma, u^* TM)$ is defined by
\beqn
X_{H} (u) := \left\{ \begin{array}{lr} 
\displaystyle 0 & \text{ on spherical components} \\
\displaystyle  dt \otimes X_{H_t} & \text{ on cylindrical components,}
\end{array} \right.
\eeqn
and the map $\eta\mapsto \langle \eta\rangle$ is induced from the Hermitian pairing on ${\mc O}(k)$.

\item The {\it obstruction bundle} $E_{pq} \to V_{pq}$ is the direct sum
\beqn
E_{pq} = O_{pq} \oplus Q_{pq}.
\eeqn
Here 
\beqn
Q_{pq}:= V_{pq} \times {\bm Q}_{d_{pq}}
\eeqn
and ${\bm Q}_d$ for all $d \geq 1$ was defined in \eqref{eqn57}; the fiber of $O_{pq}$ at a point $x \in V_{pq}$ represented by $(u, \Sigma, F, \eta)$ is 
\beqn
H^0 ( \ov{\rm Hom}( \iota_F^* T\mb{CP}^d, u^* TM) \otimes \iota_F^* {\mc O}(k)) \otimes_{\mb C} \ov{H^0( \iota_F^* {\mc O}(k) )_0}\footnote{These vector spaces indeed assemble to a vector bundle over $V_{pq}$ if $k$ is suffciently positive.} .
\eeqn

\item The group $G_{pq}$ acts on $E_{pq} \to V_{pq}$ (on the left) as follows. Given a quadruple $(u, \Sigma, F, \eta)$ where $F = (f_0, f_1, \ldots, f_d)$ and $g \in G_{pq} \subset PGL(d+1)$, define the the framed curve 
\beqn
(u, \Sigma, F') := (u, \Sigma, g_* F)
\eeqn
as from Equation \eqref{eqn:change-framing}. Moreover, $g$ induces linear isomorphisms 
\beqn
g: H^0 ( \ov{\rm Hom}( \iota_F^* T\mb{CP}^d, u^* TM) \otimes \iota_F^* {\mc O}(k)) \to H^0 ( \ov{\rm Hom}( \iota_{F'}^* T\mb{CP}^d, u^* TX) \otimes \iota_{F'}^* {\mc O}(k) ) 
\eeqn
and 
\beqn
g: H^0( \iota_F^* {\mc O}(k))_0 \to H^0( \iota_{F'}^* {\mc O}(k) )_0.
\eeqn
The action on the $Q_{pq}$-component of the obstruction bundle is defined as
\beqn
g \cdot Q \mapsto  g Q g^{-1},\ g \in G_{pq},\ Q \in {\bm Q}_d.
\eeqn

\item The \emph{Kuranishi map} is 
\beqn
\begin{split}
S_{pq}: V_{pq} &\ \to E_{pq}\\
        [u, \Sigma, F, \eta] &\ \mapsto (\eta, Q(u, \Sigma, F))
\end{split}
\eeqn
where $Q(u, \Sigma, F)$ the image of $(d+1)\times (d+1)$ Hermitian matrix whose $(i, j)$-entry is  
\beqn
\left[ \int_\Sigma \langle f_i, f_j \rangle \Omega_{u, H}\right]\in {\bm Q}_d.
\eeqn
It is easy to see that $Q (u, \Sigma, F)$ only depends on the isomorphism class of the framed curve.

\item If $[u, \Sigma, F, \eta] \in S_{pq}^{-1}(0)$, one can see that $u: \Sigma \to M$ represents an element of $\ov{\mc M}_{pq}$. Define the \emph{footprint map}
\beqn
\tilde \psi_{pq}: S_{pq}^{-1}(0) \to \ov{\mc M}_{pq},\ [u, \Sigma, F, \eta]\mapsto [u]
\eeqn
which induces a continuous map
\beqn
\psi_{pq}: S_{pq}^{-1}(0)/G_{pq} \to \ov{\mc M}{}_{pq}.
\eeqn
\end{enumerate}
\end{defn}

\begin{rem}
Our definition closely follows \cite{AMS}. However one difference is that in \cite{AMS}, they used the obstruction space 
\beqn
H^0( \ov{\rm Hom}( \tilde \iota_F^* T{\mc C}, u^* TM) \otimes \tilde \iota_F^*{\mc L}^k) \otimes \ov{H^0( \tilde \iota_F^* {\mc L}^k)}
\eeqn
where $T{\mc C}$ is the tangent bundle of the universal curve ${\mc C}_{0,2}(d)$ and ${\mc L}$ denotes a relatively ample line bundle over the universal family ${\mc C}_{0,2}(d) \to {\mc F}_{0,2}(d)$, and $\tilde \iota_F$ is the domain map,
while we replace this space by ours \eqref{obstruction1}. 
\end{rem}

\begin{lemma}
$\psi_{pq}$ is a homeomorphism.
\end{lemma}

\begin{proof}
See \cite[Lemma 6.14]{AMS}.
\end{proof}

\begin{lemma}
For each $x \in S_{pq}^{-1}(0)$, the stabilizer $G_x \subset G_{pq}$ is isomorphic to the stabilizer of $\psi_{pq}(x) \in \ov{\mc M}_{pq}$. Moreover, in a $G_{pq}$-invariant open neighborhood of $S_{pq}^{-1}(0)$ every point has finite isotropy group.
\end{lemma}

\begin{proof}
See the proof of \cite[Lemma 6.4]{AMS}.
%Suppose $x$ is represented by $(u, \Sigma, F, 0)$ and is fixed by $g \in G_{pq} = PU(d_{pq}+1)$. Choose a lift $\tilde g \in U(d_{pq}+1)$. Then $(u, \Sigma, F)$ is equivalent to $(u, \Sigma, F'): = (u, \Sigma, \tilde g_* F)$. Therefore, there exists a domain isomorphism $\varphi: \Sigma \to \Sigma$ such that $u \circ \varphi = u$ and an isomorphism of line bundles 
%\beqn
%\widehat\varphi: L_u \to L_u
%\eeqn
%such that 
%\beqn
%f_i' \circ \varphi = \widehat \varphi \circ f_i,\ i = 0, \ldots, d_{pq}.
%\eeqn
%{\bf not finished},
%\textcolor{red}{a proof is sketched in the comment section, and the proof of \cite[Lemma 6.4]{AMS} should essentially do the job.}
%The finiteness of stabilizer of points near $S_{pq}^{-1}(0)$ follows from a compactness argument.
\end{proof}

Therefore, we see that Definition \ref{def:thick-mod} presents $\ov{\mc M}_{pq}$ as a global quotient preserving the stabilizers. Now we can discuss about the regularity properties.

\subsubsection{Transversality of the perturbed Floer equation}\label{subsubsec:shear}
We review Gromov's graph trick used in the specific setting of global Kuranishi charts in \cite{AMS}. Let $(N, J_N)$ be an almost complex manifold and $E \to N$ be a Hermitian vector bundle with a Hermitian connection $\nabla^E$. Then using the decomposition $TE = T^h E \oplus T^v E$ of the tangent bundle of the total space of $E$ into horizontal and vertical tangent bundles induced from $\nabla$, there is an induced ``product type'' almost complex structure $J^E$ on $E$ by combining $J_N$ with the fiberwise complex structure on $E$. Now consider a ${\mb R}$-linear bundle map 
\beqn
\Psi: E \oplus T N \to T N
\eeqn
satisfying 
\beq
J_N ( \Psi (e, h)) = - \Psi (e, J_N (h))
\eeq
and
\beq
\Psi(e, \Psi(e, h)) = 0.
\eeq
We define 
\beqn
\Phi: T^v E \oplus T^h E \to T^v E\oplus T^h E
\eeqn
whose restriction at each $e \in E$ is 
\beqn
\Phi(v, h) = (v, h + \Psi(e, h)).
\eeqn
It is easy to check that 
\beqn
\Phi^{-1}(v, h) = (v, h - \Psi(e, h)).
\eeqn
Define an almost complex structure on $E$ by
\beqn
J_\Psi^E = \Phi \circ J^E \circ \Phi^{-1}: TE \to T E.
\eeqn
More explicitly,
\beqn
J_{\Psi}^E (v, h) = ( J^E v, J_N h + 2 \Psi( e, J_N h)).
\eeqn

Now we consider the Cauchy--Riemann equation with respect to the sheared almost complex structure on the total space, with an additional Hamiltonian perturbation. Let $\Sigma$ be a Riemann surface with complex structure $j$. Let $X \in \Omega^{0,1}(\Sigma, {\rm Vect}(N))$ be a perturbation. The Hermitian connection on $E$ induces a horizontal lift of $X$ on the total space, denoted by $X^E$. Then consider the Cauchy--Riemann equation for $\tilde{u}: \Sigma \to E$
\beqn
\ov\partial_{J^E_\Psi} \tilde u + X^E( \tilde u) = 0.
\eeqn
If we write $\tilde u = (u, s)$ where $u: \Sigma \to N$ and $s \in \Gamma (u^* E)$, then with respect to the horizontal-vertical decomposition of $TE$, one has 
\beqn
d\tilde u = du \oplus \nabla^E s
\eeqn
and 
\beqn
\begin{split}
(d\tilde u)^{0,1}_{J^E_\Psi} = &\ \frac{1}{2} \left( du + J^E_\Psi \circ du \circ j \right) \oplus (\nabla^E s)^{0,1} \\
                  = &\ \left( \ov\partial_{J_N} u + X(u) + \Psi (s, J_N \circ du \circ j ) \right) \oplus (\nabla^E s)^{0,1}.
                  \end{split}
\eeqn

\begin{lemma}\cite[Lemma 6.18]{AMS} Let $\Sigma$ be a Riemann surface with complex structure $j$ and $\tilde u = (u, s):\Sigma \to E$ be a smooth map. Then $\tilde u$ corresonds to a smooth map $u:= \pi_E \circ \tilde u: \Sigma \to N$ together with a section $s \in \Gamma ( u^* E)$. Under this correspondence, $\tilde u$ is $\tilde J_{\Psi}^E$-holomorphic if and only if $s$ is a holomorphic (with respect to the $(0, 1)$-part of $\nabla$) and $u$ satisfies 
\beqn
\ov\partial_{J_N} u + X(u) +  \Psi( s,  J_N \circ du \circ j ) = 0. \qed
\eeqn
\end{lemma}

\vspace{0.2cm}

\noindent {\it The geometric thickening revisited.}
%Now we consider a geometrically engineered thickening of a moduli space of stable Floer trajectories. The construction is inspired by and very similar to the one given by Abouzaid--McLean--Smith \cite{AMS}. Fix two different capped 1-periodic orbits $p, q$ with $\ov{\mc M}_{pq} \neq \emptyset$. Define and denote 
%%\beqn
%d:= d_{pq}:= \left\lfloor {\mc A}_H(p) \right\rfloor - \left\lfloor {\mc A}_H(q) \right\rfloor\in {\mb Z}_+.
%\eeqn
%\textcolor{red}{Maybe we should remind the reader that the base of the topological submersion is taken to be ${\mc F}_{0,2}^{\mb R}(d)$, while to construct geometric perturbations from ample line bundles, we actually take the domain map with image in ${\mc F}_{0,2}(d)$.}
Recall that one has the smooth quasiprojective variety ${\mc C} = {\mc C}_{0,2}(d)$, which is the universal curve of a submanifold ${\mc F} = {\mc F}_{0,2}(d) \subset \ov{\mc M}_{0,2}(\mb{CP}^d, d)$ (after imposing the constraint at $z_-$ by the point $[1,0, \ldots, 0]$0. Denote by 
\beqn
{\rm univ}: {\mc C} \to {\mc F}
\eeqn
the canonical holomorphic projection map. Also denote by 
\beqn
{\rm ev}: {\mc C} \to \mb{CP}^d
\eeqn
the evaluation map. Consider the almost complex manifold 
\beqn
N = M \times {\mc C}
\eeqn
equipped with the product almost complex structure $J_N = J_M \times J_{\mc C}$. Using the projection $p_M: M \times {\mc C} \to M$, $p_{\mc C}: X \times {\mc C} \to {\mc C}$ and the evaluation map $\ev: {\mc C} \to \mb{CP}^d$, define the vector bundles\footnote{Again, $k$ is chosen to be sufficiently large so that they are vector bundles}
\beqn
\begin{split}
E^0(k):= &\ \ov{\rm Hom} (   (\ev\circ p_{\mc C})^* T \mb{CP}^d, p_M^* TM ) \otimes (\ev\circ p_{\mc C})^* {\mc O}(k),\\
E^1(k):= &\ p_{\mc C}^* {\rm univ}^* ({\rm univ}_* (\ev^* {\mc O}(k)))_0.
\end{split}
\eeqn
Here $({\rm univ}_* (\ev^* {\mc O}(k)))_0$ is the bundle whose fibers are fiberwise global sections of $\ev^* {\mc O}(k)$ which vanish at the two marked points. Define 
\beq\label{thickenbundle}
E(k):= E^0(k) \otimes_{\mb C} \ov{E^1(k)}.
\eeq 
Before talking about the shearing map we also need to specify a Hermitian metric and Hermitian connection on $E$. From the definition we see that all factors of $E$ has a natural Hermitian metric and connection except the bundle $({\rm univ}_*(\ev^* {\mc O}(k)))_0$ which only has a complex structure but no canonical Hermitian structure. We choose an arbitrary one which then induces a Hermitian metric and a Hermitian connection on $E$. This choice does not affect the structures we are going to construct. It is only an auxiliary object to be used to show regularity. The regularity is independent of the choices.

The shearing map $\Psi: E \oplus TN \to TN \cong p_M^* TM \oplus p_{\mc C}^* T{\mc C}$ is defined to be 
\beq\label{shearing}
\Psi( \eta, v_1, v_2) =  \langle \eta \rangle(v_2) \oplus 0.
\eeq
Here the map $\langle \eta \rangle$ is defined at each point $y \in {\mc C}$ (whose evaluation is $z \in \mb{CP}^d$) with a fiber ${\mc C}_y \subset {\mc C}$ the pairing between $e \in {\mc O}(k)|_z$ and the value of $\zeta \in \ov{H^0( {\mc C}_y, \ev^* {\mc O}(k)|_{{\mc C}_y})_0}$ using the Hermitian metric on ${\mc O}(k)$, where we identify $H^0({\mc C}_y,  \ev^* {\mc O}(k) |_{{\mc C}_y})$ with the space of global holomorphic sections on the curve ${\mc C}_y$.

\vspace{0.2cm}

\noindent {\it Lift the Hamiltonian.} 
We need to lift the Hamiltonian vector field to the total space of $E = E(k)$. Let $X_{H_t} \in \Gamma(TM)$ be the time $t$ Hamiltonian vector field on $M$. Then it pulls back to a vector field on $M \times {\mc C}$ which is zero in the $T{\mc C}$-direction. On the other hand, the Hermitian connection on $E$ induces a horizontal-vertical decomposition of $TE$. By identifying the horizontal distribution with the pullback of the tangent bundle of the base, one can lift the vector field $X_{H_t}$ to the total space $E$, which is denoted by $\tilde X_{H_t}$. Notice that by the definition of $E$ and the choice of the Hermitian connection $\nabla^{TM}$, all 1-periodic orbits in the total space of $E$ are contained in the zero section and are exactly the same as the 1-periodic orbits of the original Hamiltonian on $M$ multiplied by the additional factor ${\mc C}$. 

\vspace{0.2cm}

\noindent {\it An alternative description of the thickened moduli space.} 
\begin{defn}
Let $(\Sigma, {\bf L})$ be a prestable cylinder. A smooth map $\tilde u: \Sigma \to E$ is called a $(\tilde J_{\Psi}, H)$-holomorphic map if it is $\tilde J_\Psi$-holomorphic on all sphere components and on each cylindrical components it satisfies the equation 
\beq\label{eqn:Floer-lift}
\frac{\partial \tilde u}{\partial s} + \tilde J_\Psi \left( \frac{\partial \tilde u}{\partial t} - \tilde X_{H_t} (\tilde u) \right) = 0.
\eeq
(This is the same as \eqref{thickening_equation}.)
\end{defn}

\vspace{0.2cm}

Now we give another description of the thickened moduli space.

\begin{defn}
Let ${\mc M}_{\tilde J_\Psi, H}$ be the moduli space of $(\tilde J_\Psi, H)$-holomorphic maps $\tilde u$ from a prestable cylinder $(\Sigma, {\bf L})$ to the total space of $E$ satisfying the following conditions. Write the projection of $\tilde u$ to the base of $E$, which is $N = M \times {\mc C}$, by $(u, \mu)$. 
\begin{enumerate}
    \item $\mu: \Sigma \to {\mc C}$ is a domain map, i.e., it is an isomorphism onto a fiber of ${\mc C}$ which sends the marked points to the marked points.
    
    \item $\tilde u$ converges to periodic orbits of $\tilde X_{H_t}$ at cylindrical nodes. Moreover, its projection onto $M$ has the same homotopy type as elements in $\ov{\mc M}_{pq}$.
    
    \item For each component $\Sigma_\alpha \subset \Sigma$, the degree of the restriction of the underlying map $\uds \mu: \Sigma \to \mb{CP}^{d}$ is equal to the degree determined by the restriction $u: \Sigma_\alpha \to M$.
\end{enumerate}
\end{defn}

Notice that there is also a $G_{pq}$-action on ${\mc M}_{\tilde J_\Psi, H}$.

Now we compare the thickened moduli space $V_{pq}$ with ${\mc M}_{\tilde J_\Psi, H}$. Indeed, given any point of $V_{pq}$ represented by $(\Sigma, u, F, \eta)$, we can identify it with a stable $(\tilde J_\Psi, H)$-holomorphic map $\tilde u = (u, \mu, \eta)$ where we just replace the frame $F$ by the induced map $\uds\mu = \iota_F$. This map is clearly well-defined and $G_d$-equivariant. 

\begin{lemma}(cf. \cite[Lemma 6.25]{AMS} The natural map 
\beqn
V_{pq} \to {\mc M}_{\tilde J_\Psi, H}
\eeqn
is an equivariant open embedding (of topological spaces).
\end{lemma} \qed

Therefore one can identify the obstruction bundle and the Kuranishi map as defined over the space ${\mc M}_{\tilde J_\Psi, H}$ after restricting to a $G_{pq}$-invariant open subset.

\vspace{0.2cm}

\noindent {\it Regularity.} 
Remember that the thickened moduli space depends on a choice of an integer $k$. To emphasize the role of these integers, we temporarily denote 
\beqn
{\mc M}_{\tilde J_\Psi, H} = {\mc M}_{\tilde J_{\Psi}, H}(k).
\eeqn

\begin{thm}(cf. \cite[Corollary 6.27]{AMS})\label{regularity1}
For a fixed $pq$, there exists a positive integer $k_{pq}>0$ such that if $k \geq k_{pq}$, then each element of $S_{pq}^{-1}(0)$ is a regular element of ${\mc M}_{\tilde J_{\Psi}, H}(k)$, i.e., the linearization of \eqref{eqn:Floer-lift} is surjective. %In particular, ${\mc M}_{\tilde J_{\Psi}, H}(k)$ is a topological manifold admitting a locally linear $G_{pq}$-action.
\end{thm}

\begin{proof}
Choose $x\in S_{pq}^{-1}(0)$. Notice that this element can be viewed as a representative of ${\mc M}_{\tilde J_\Psi, H}(k)$ for all $k\geq 1$. We will prove that for sufficiently large $k$, $x$ is regular in ${\mc M}_{\tilde J_\Psi, H}(k)$. By compactness of the moduli space $\ov{\mc M}_{pq}$, one can find a common large $k$ which regularizes all elements of $S_{pq}^{-1}(0)$.

Suppose $x$ is represented by the framed cylinder $(\Sigma, u, F, 0)$ where $F$ is a unitary frame of the line bundle $L_u \to \Sigma$ which is induced from the 2-form $\Omega_{u, H}$ whose curvature form is equal to $-2\pi {\bm i} \Omega_{u, H}$. The existence of $L_u$ is guaranteed by Corollary \ref{HK2}. Abbreviate $d = d_{pq}$. Denote the frame $F$ by
\beqn
F:=(f_0, \ldots, f_d).
\eeqn
Then the map 
\beqn
\Sigma \to \mb{CP}^d,\ z \mapsto [ f_0(z), \ldots, f_d(z)]
\eeqn
represents an element in ${\mc F}_{0,2}(d)$. Therefore, the unitary frame $F$ induces a domain map $\mu: \Sigma \to {\mc C}$. Notice that this construction is independent of $k$. Now for any $k$, denote $(u, \mu): \Sigma \to M \times {\mc C}$ which together with the inclusion map of the zero section of $u^* E(k)$ gives an element in ${\mc M}_{\tilde J_\Psi, H}(k)$. 

We would like to show that the linearization of Equation \eqref{eqn:Floer-lift} at $x$ is surjective when $k$ is sufficiently large. We only show its surjectivity over each component, while skipping the argument showing the surjectivity after imposing the matching conditions at nodes (which is the same as the case of pseudoholomorphic curves).

Over each component $\Sigma_\alpha \subset \Sigma$, the deformation space of the map $(u, \mu, 0)$ splits as 
\beqn
W^{1,p}(\Sigma_\alpha, \mu^* T{\mc C}) \oplus W^{1,p}(\Sigma_\alpha, u^* TM ) \oplus W^{1,p}(\Sigma_\alpha, (u, \mu)^* E(k) )
\eeqn
and the linearization is of the block form 
\beq\label{eqn:block-matrix}
\left[ \begin{array}{ccc}  D_\mu &  0  & 0 \\
                            0  &  D_u &  P \\
                            0  &  0     &   D_E   \end{array}\right]
\eeq
where the diagonal terms are the standard linearization of the Cauchy--Riemann equation (with or without Hamiltonian perturbation term, depending on whether the component is spherical or cylindrical), and the off-diagonal term $P$ is the perturbation of the $TM$-direction (i.e. the inhomogeneous term of \eqref{thickening_equation}). Notice that the operator $D_\mu$ is always surjective, reflecting the fact that domain maps to ${\mc C}$ are always unobstructed. This is a classical fact due to the ``convexity" of the bundle $\mu^* T{\mc C}$. Hence one only needs to consider the lower-right $2\times 2$-block. 

We first consider a spherical component $\Sigma_\beta \subset \Sigma$ where the Hamiltonian perturbation term is zero. If $\Sigma_\beta$ has positive degree, then the line bundle $\uds\mu^* {\mc O}(k)$ is a positive line bundle on this component. When $k$ is sufficiently large, one can argue in the same way as the case of \cite{AMS} to show that the linearized operator is surjective, even after restricting to the subspace where the values at nodes vanish. On the other hand, if $\Sigma_\beta$ is a ghost component. Then $(u, \mu)^* E(k)|_{\Sigma_\beta}$ is a trivial vector bundle equipped with the trivial Cauchy--Riemann operator. This is the same as the linearized operator in the $TM$-direction. Hence the linear operator is surjective with kernels being constant sections. 

Second, we consider a cylindrical component $\Sigma_\alpha\subset \Sigma$. We identify $\Sigma_\alpha \cong (-\infty, +\infty) \times S^1$ with cylindrical coordinates $(s, t)$. Using the peak-section argument as in \cite{AMS}, one can show that the $TM$-direction is surjective. Hence we only need to show that the operator 
\beqn
D_E: W^{1,p} ( (u, \mu)^* E(k)) \to L^p(\Lambda^{0,1}\otimes (u, \mu)^* E(k))
\eeqn
is surjective. When $\Sigma_\alpha$ is a ghost component (i.e. a trivial cylinder), this surjectivity is obvious (its kernel is zero). When $\Sigma_\alpha$ is not a ghost component, we need to use the positivity of ${\mc O}(k)$. One can show that the curvature of $(u, \mu)^* E(k)$ is everywhere positive when $k$ is sufficiently large. Indeed, for $S>0$ sufficiently large, $u((-\infty, -S]\cup [S, +\infty))$ is contained in a neighborhood of the union of periodic orbits where $\nabla^{TM}$ is flat. Hence over the region where $|s|\geq S$, the curvature of $(u, \mu)^* E(k)$ is pulled back from the universal curve ${\mc C}$. Hence when $k$ is sufficiently large, the curvature is everywhere positive. Then by Lemma \ref{surjective} below, $D_E$ is surjective on this component.
\end{proof}

\begin{lemma}\label{surjective}
Let $E \to \Theta$ be a Hermitian vector bundle and $\nabla$ be a Hermitian connection on $E$ whose curvature form is everywhere positive definite and whose limiting holonomy has no eigenvalue $1$. Then the Cauchy--Riemann operator 
\beqn
\nabla^{0,1}: W^{1,p} (\Theta, E) \to L^p(\Theta, \Lambda^{0,1}\otimes E)
\eeqn
is surjective.
\end{lemma}

\begin{proof}
We prove by using the maximal principle. Suppose this is not the case. Then there exists $\xi$ in the kernel of the formal adjoint of $\nabla^{0,1}$. Using the local coordinate $(s, t)$ on $\Theta$ we write 
\beqn
\nabla = d + \phi ds + \psi dt.
\eeqn
Let the curvature form of $\nabla$ be $Q ds dt$. Then one can identify $\nabla^{0,1}$ with the operator 
\beqn
D^{0,1} =\left( \frac{\partial}{\partial s} + \phi \right)+ {\bm i} \left( \frac{\partial}{\partial t} + \psi \right).
\eeqn
Then its formal adjoint is $- D^{1,0}$. As $\xi \in L^q$, the function $|\xi|^2$ has a maximal point $z_0 \in \Theta$. Then 
\beqn
0 \geq \Delta |\xi|^2 = \frac{\partial}{\partial \ov{z}} \frac{\partial}{\partial z} |\xi|^2 = | D^{0,1}\xi|^2 + \langle \xi, D^{1,0} D^{0,1} \xi \rangle = |D^{0,1}\xi|^2 + \langle \xi, Q \xi \rangle.
\eeqn
As $Q$ is everywhere positive, this is a contradiction.
\end{proof}

\begin{lemma}\label{lem:positive-curvature}
When $k$ is sufficiently large, the curvature of the connection on $\ov{\rm Hom}( \iota_F^* T\mb{CP}^d, u^* TM ) \otimes \iota_F^* {\mc O}(k)$ is everywhere positive. 
\end{lemma}

\begin{proof}
The curvature is the pullback of the curvatures on the corresponding bundles on the target $X \times \mb{CP}^d$. Notice that near $\pm\infty$, $u_M^* TM$ is flat. Therefore, the positivity of the curvature of ${\mc O}(k)$ implies that the positivity of the curvature of the induced connection on $\ov{\rm Hom}( \iota_F^* T\mb{CP}^d, u^* TM ) \otimes \iota_F^* {\mc O}(k)$ is everywhere positive for $k$ sufficiently large, because such a claim is true for the corresponding vector bundles on the target $X \times \mb{CP}^d$.
\end{proof}

%\textcolor{purple}{Also conclude transversality for PSS/SSP moduli spaces.}

\subsubsection{Multi-layered thickening}\label{subsubsec:multi}

In Floer theory we need to deal with many moduli spaces simultaneously, and it is important to make sure that various structures on different moduli spaces fit together in a coherent way. The expected relations between different moduli spaces then require us to consider a more complicated thickening procedure which is referred to as the {\bf multi-layered thickening}. Namely, instead of just considering the perturbation induced from one single relatively ample bundle $\text{ev}^* {\mc O}(k) \to {\mc C}$, we need to look at direct sum of such line bundles. This slight generalization of Abouzaid--McLean--Smith's perturbation scheme will be used in the construction of the K-chart lift.

Fix a positive integer $d$ and abbreviate ${\mc C} = {\mc C}_{0,2}(d)$. Suppose we have a sequence of positive integers $k_1< k_2 < \cdots < k_s$. We are going to define complex vector bundles inductively
\beqn
E_1', E_1, \ldots, E_s', E_s \to M \times {\mc C} = M\times {\mc C}_{0,2}(d_{pq}).
\eeqn
%Suppose we are given an increasing sequence of positive integers $k_1 < k_2 < \cdots$. We are going to construct a tower of Hermitian vector bundles 
%\beqn
%\xymatrix{\cdots \ar[r] & E^{(i)} \ar[r] & \cdots \ar[r] & E^{(2)} \ar[r] & E^{(1)} \ar[r] & M \times %{\mc C}}
%\eeqn
%together with Hermitian connections on them. Abbreviate $d = d_{pq}$. 
First, define 
\beqn
E_1':=\ov{\rm Hom} ( (\ev\circ p_{\mc C})^* T\mb{CP}^d, p_M^* TM) \otimes (\ev\circ p_{\mc C})^* {\mc O}(k_1).
\eeqn
and 
\beqn
E_1:= E_1' \otimes_{\mb C} \ov{ {\rm univ}^* ( {\rm univ}_* \ev^* {\mc O}(k_1))_0}.
\eeqn
Inductively, suppose we have defined $E_1', E_1, \ldots, E_{i-1}', E_{i-1}$. Then define 
\beqn
E_i':= \ov{\rm Hom} \Big( (\ev\circ p_{\mc C})^* T\mb{CP}^d, p_M^* TM \oplus \bigoplus_{j=1}^{i-1} E_i \Big)  \otimes (\ev\circ p_{\mc C})^* {\mc O}(k_i)
\eeqn
and 
\beq\label{eqn:e-i}
E_i:= E_i' \otimes_{\mb C} \ov{ {\rm univ}^* ( {\rm univ}_* \ev^* {\mc O}(k_i))_0}.
\eeq
Notice that once $TM$ is equipped with a Hermitian connection (and hence a Cauchy--Riemann operator). Together with the holomorphic structure of the involved bundles over ${\mc C}$, the connection on $TM$ induces a Cauchy--Riemann operator on $E_i$. 

\begin{rem}\label{rem:explain-2-layer}
We explain the above construction for the simplest nontrivial example $E_2$. The total space of $E_1 \to M \times {\mc C}$ is an almost complex manifold which has a projection $p_{\mc C}: E_1 \to {\mc C}$. The goal is to construct certain perturbations of a (perturbed) $J$-holomorphic map
$ u: \Sigma \to E_1 $
such that $p_{\mc C} \circ u: \Sigma \to {\mc C}$ is a domain map, which in particular is holomorphic. Note that 
$$ (du)^{0,1}_J \in \ov{\rm Hom}(T\Sigma, u^* TE_1) = \ov{\rm Hom}(T\Sigma, u^* TM \oplus u^* T {\mc C} \oplus u^* E_1),
$$
where the equality is induced from a Hermitian connection on $E_1$. Therefore, to keep the map $p_{\mc C} \circ u$ being a domain map, the perturbation should take value in $\ov{\rm Hom}(T\Sigma, u^* TM \oplus u^* E_1)$. The vector bundle $E_2 \to M \times {\mc C}$ is constructed just as the construction of $E_1 \to M \times {\mc C}$ in order to host an equation of the form \eqref{thickening_equation}, with the assistance of the relatively ample line bundle ${\rm ev}^* {\mc O}(k_2) \to {\mc C}$.
\end{rem}
%This is a Hermitian vector bundle over $E^{(i)}$ with a naturally induced Hermitian connection $\nabla^{(i+1)}$. 

%Now we would like to define a sequence of shearing maps 
%\beqn
%\Psi^{(i+1)}: TE^{(i+1)} \to TE^{(i)}.
%\eeqn
%Indeed, the connection $\nabla^{(i+1)}$ induces a splitting 
%\beqn
%TE^{(i+1)} \cong \pi_{E^{(i+1)}}^* E^{(i+1)} \oplus \pi_{E^{(i+1)}}^* TE^{(i)}
%\eeqn
%and the bundle
%\beqn
%TM \oplus \bigoplus_{j=1}^i E^{(j)}
%\eeqn
%is naturally contained in $TE^{(i)}$. Then define 
%\beqn
%\Psi^{(i+1)}(v_1, v_2, \eta^{(i+1)}) = 0\oplus \langle \eta^{(i+1)} \rangle (v_1).
%\eeqn
%Let $J^{(i)}$ be the complex structure on the fibers of $E^{(i)}$. It is easy to check from the %definition of $\Psi^{(i)}$ that 
%\begin{align*}
%&\ \Psi^{(i)}(e, J(h)) = - J \Psi^{(i)} (e, h),\ &\ \Psi^{(i)}(e, \Psi^{(i)}(e, h)) = 0.
%\end{align*}
%Hence the sheared almost complex structure $J_{\Psi^{(i)}}: TE^{(i)} \to TE^{(i)}$ is defined. 

%Now we abbreviate $E = E(k_1, \ldots, k_s) = E_1 \oplus \cdots \oplus E_s$.  also view $E^{(s)}$ as the total space of a Hermitian vector bundle $E^{\leq s} \to M\times {\mc C}$. The constructed connections induce a Hermitian connection $\nabla^{\leq s}$ on $E^{\leq s}$. The sheared almost complex structure then provides a sheared almost complex structure $J^{\leq s}$ on $E^{\leq s}$. The Hamiltonian vector field $X_t$ on $M$ can also be lifted using the connection to a Hamiltonian vector field $X_t^{\leq s}$ on $E^{\leq s}$. Then one can consider the perturbed pseudoholomorphic curve equation for maps into $E^{\leq s}$. 

\begin{defn}[Multi-layered thickening]\label{defn:multi-thick}
For any pair $p,q$ of capped periodic orbits, denote $d = d_{pq}$. Let $V_{pq}^d= V_{pq}(k_1, \ldots, k_d)$ be the moduli space of tuples
\beqn
(\Sigma, u, F, \eta_1, \ldots, \eta_d)
\eeqn
where $\Sigma$ is a prestable cylinder, $u: \Sigma \to M$ is a smooth map with topological energy $d$ converging to $p$ and $q$ at the two marked points, $F$ is a frame of the line bundle $L_{u,\Omega, H}\to \Sigma$ inducing a holomorphic map $\iota_F: \Sigma \to \mb{CP}^d$ such that the framing $F=(f_0, \dots, f_d)$ satisfies $[f_0(z_-): \cdots : f_d(z_-)] = [1:0:\cdots :0]$, and the corresponding domain map $\tilde \iota_F: \Sigma \to {\mc C}$, and (denoting by $\tilde u = (u, \tilde \iota_F): \Sigma \to M \times {\mc C}$)
\beqn
\eta_i\in \tilde{u}^* E_i = \Gamma \left( \ov{\rm Hom} \big( \iota_F^* T\mb{CP}^d, u^* TM \oplus \bigoplus_{j=1}^{i-1} \tilde u^* E_j \big) \otimes \iota_F^* {\mc O}(k_i) \right) \otimes \ov{ H^0(\iota_F^* {\mc O}(k_i))_0} 
\eeqn
satisfying the following equations
\begin{equation}\label{eqn:floer-thicken}
\left\{
\begin{aligned}
\ov\partial_{J, H} u +  \pi_{u^* TM} \left( \sum_{i=1}^d \langle \eta_i \rangle \circ d\iota_F \right) &\ = 0,\\
\ov\partial \eta_i + \pi_{E_i} \left( \sum_{j = i+1}^d \langle \eta_j \rangle \circ d\iota_F \right) &\ = 0,\ i = 1, \ldots, d.
\end{aligned}\right.
\end{equation}
\end{defn}

\begin{rem}
Continuing Remark \ref{rem:explain-2-layer}, we explain the origin of the above equations for $d=2$. Note that the pair $(\tilde{u}, \eta)$ defines a smooth map $\tilde{u}_1: \Sigma \to E_1$. The purpose is to write down a perturbation of certain $\ov{\partial}$-equation for $\tilde{u}_1$ using a section of $\eta_2 \in \tilde{u}^* E_2$, which could be schematically written as
$$
\ov{\partial}_{\tilde{J}_\Psi, H}  u_1 + \langle \eta_2 \rangle \circ d \iota_F = 0,
$$
where $\tilde{J}_\Psi$ is the sheared almost complex structure constructed using \eqref{shearing} and $u_1$ is obtained from $\tilde{u}_1$ by forgetting the domain map. Then Equation \eqref{eqn:floer-thicken} is obtained by projecting the above equation using $\pi_{u^* TM}$ and $\pi_{E_1}$. For general $d \geq 2$, one just need to carry out this construction iteratively. 
\end{rem}

Similar to the case with a single thickening bundle, we have a regularity result for the above defined multi-layered thickened moduli space. The details are given in Subsection \ref{multi_regularity}. We remark that the Fredholm model for this thickened moduli space depends on certain choices which may not be canonical. However, being regular or not is a condition independent of these choices. As in Definition \ref{def:thick-mod}, $V_{pq}^d$ is part of the data of a K-chart.

\vspace{0.2cm}

\noindent {\it Obstruction bundle} The first part of the obstruction bundle has its fiber over a point represented by $(\Sigma, u, F, \eta)$ the vector space
\beq\label{eqn:obstruction-O}
\left\{ \zeta = \left( \begin{array}{c} \zeta_1 \\
\vdots \\
\zeta_d \end{array}\right) \left| \begin{array}{c} \zeta_i \in \Gamma \left( \ov{\rm Hom}(\iota_F^* T\mb{CP}^d, u^* TM \oplus \displaystyle \bigoplus_{j=1}^{i-1} \tilde u^* E_j \right) \otimes \iota_F^* {\mc O}(k_i) \otimes \ov{ H^0(\iota_F^* {\mc O}(k_i))_0},  \\
\ov\partial \zeta_i + \pi_{E_i} \left( \displaystyle \sum_{j=i+1}^d \langle \zeta_j \rangle \circ d \iota_F\right) = 0\end{array} \right. \right\}.
\eeq
It is standard that it is finite-dimensional. We denote this bundle by 
\beqn
O_{pq} = O_{pq}^{(d)} \to V_{pq}.
\eeqn

On the other hand, as before, the second factor of the obstruction bundle is the trivial bundle $Q_{pq} := Q_d$ whose fiber ${\bm Q}_d$ is defined by \eqref{eqn57}. Let $Q_{pq} \to V_{pq}$ be the trivial bundle. Note that from its definition, $Q_{pq}$ is only a fiber bundle with a canonical smooth section, where the section corresponds to the element in $Q_d$ represented by the identity matrix. However, using the vector space ${\bm Q}^*_d$ defined by \eqref{eqn:q-star}, the bundle $Q_{pq}$ can be endowed with a vector bundle structure. We switch between these two viewpoints in different contexts. We define 
\beqn
E_{pq} = O_{pq}\oplus Q_{pq} \to V_{pq}.
\eeqn
Moreover, similar to the single-layered case, there is a $G_{pq}$-action on $V_{pq}$ as well as on the bundles $O_{pq}$ and $Q_{pq}$ making $E_{pq}$ an equivariant vector bundle. 

\vspace{0.2cm}

\noindent {\it Properties of obstruction bundle.} There are many structural facts about the obstruction bundles which will play important roles in the construction of the K-chart lift of the Hamiltonian Floer flow category. We summarize them here. First, recall that the thickening depends on the choice of the sequence $k_1 < \cdots < k_d$. We (will) fix an increasing sequence $\{k_i\}_{i = 1}^\infty$ such that for all $pq$, the thickening $V_{pq}(k_1, \ldots, k_{d_{pq}})$ is regular (see Proposition \ref{prop525}). However, we can have higher obstruction spaces. Given $d' > d \geq d_{pq}$, one can construct thickenings $V_{pq}(k_1, \dots, k_d)$ and $V_{pq}(k_1, \dots, k_{d'})$, over which live the two vector bundles
\begin{align*}
    &\ O_{pq}^{(d)} \to V_{pq}(k_1, \dots, k_d),\ &\ O_{pq}^{(d')} \to V_{pq}(k_1, \dots, k_{d'}).
\end{align*}
As we use sections with higher $k_i$ labels to perturb sections with lower $k_i$ labels, there is a natural inclusion 
\beq\label{obstruction_inclusion}
O_{pq}^{(d)} \hookrightarrow O_{pq}^{(d')}
\eeq
covering the natural inclusion $V_{pq}(k_1, \dots, k_d) \hookrightarrow V_{pq}(k_1, \dots, k_{d'})$
such that when $d< d' < d''$, the following diagram commutes
\beqn
\vcenter{ \xymatrix{ O_{pq}^{(d)} \ar[r] \ar@/_1.0pc/@[][rr] & O_{pq}^{(d')} \ar[r] & O_{pq}^{(d'')}}  }.
\eeqn
Then we can define the direct limit of the obstruction bundles as 
\beq
O_{pq}^\infty:= \lim_{d \to \infty} O_{pq}^{(d)}.
\eeq

The obstruction bundle also splits over each boundary stratum. Let $\alpha = pr_1 \cdots r_l q \in \bA_{pq}^\floer$ and let $\partial^\alpha V_{pq} \subset V_{pq}$ be the subspace of the thickened moduli space which consists of elements such that the underlying map $(u, \Sigma)$ is of type $\alpha$ (see Definition \ref{defn:framed_cylinder}). Then one has a direct sum decomposition
\beq\label{obstruction_splitting}
O_{pq}^{(d)}|_{\partial^\alpha V_{pq}} \cong O_{pq; pr_1}^{(d)} \oplus \cdots \oplus O_{pq; r_l q}^{(d)}
\eeq
where the fiber of $O_{pq; r_i r_{i+1}}^{(d)}$ over the point $[\Sigma, u, F, \eta]$ is 
\beqn
\left\{ \zeta = (\zeta_1, \ldots, \zeta_d)^T\ |\ {\rm supp} \zeta \subset \Sigma_{r_i r_{i+1}} \subset \Sigma  \right\},
\eeqn
where we use the convention that $r_0 = p$ and $r_{l+1}=q$.
\eqref{obstruction_splitting} is a direct sum because elements of $O_{pq}^{(d)}$ necessarily vanish at nodes connecting cylindrical components.

\vspace{0.2cm}

\noindent {\it Kuranishi section.} Define 
\beqn
\begin{split}
S^d_{pq}: V_{pq} &\ \to E_{pq}\\
        [\Sigma, u, F, \eta] &\ \mapsto ( \eta, H_F)
\end{split}
\eeqn
where $H_F\in Q_d$ is represented by the Hermitian matrix whose entries are
\beqn
\int_\Sigma \langle f_i, f_j \rangle_{L_{u, \Omega, H}} \Omega_u,
\eeqn
where $f_0, \dots, f_d$ is a basis representing the frame $F$.
It is also the same as the single-layered case that $S^d_{pq}$ is $G_{pq}$-equivariant,

Therefore, one obtains a quadruple
\beq\label{eqn:k-chart-floer}
K_{pq} = (G_{pq}, V_{pq}, E_{pq}, S_{pq}) := (G_{pq}, V^d_{pq}, E^d_{pq}, S^d_{pq})
\eeq
which is a candidate for a global Kuranishi chart on $\ov{\mc M}_{pq}$. Still, there is a canonical map 
\beqn
S_{pq}^{-1}(0)/ G_{pq} \to \ov{\mc M}_{pq}
\eeqn
which defines a homeomorphism onto its image.

\subsubsection{Regularity for the multi-layered thickening}\label{multi_regularity}

Now we prove the regularity of the multi-layered thickening. First we transform the description of the thickened moduli space to a moduli space where one can use Gromov's graph trick. Recall that one has the bundles $E_1, \ldots, E_d \to M \times {\mc C}_d$. Then one can define the tower of bundles 
\beqn
{\mc E}_d \to {\mc E}_{d-1} \to \cdots \to {\mc E}_1 \to M \times {\mc C}_d
\eeqn
where each ${\mc E}_i$ is the total space of the pullback of $E_i \to M \times {\mc C}_d$ onto the previous ${\mc E}_i \to M \times {\mc C}_d$. Then ${\mc E}_i$ is still a complex vector bundle over $M \times {\mc C}_d$. To proceed we would like to equip each ${\mc E}_i$ with a Hermitian metric and a Hermitian connection. From the construction one can see that all the ingredients in constructing these bundles have a natural Hermitian structure except for the bundle $({\rm univ}_* \ev^* {\mc O}(k_i))_0$, which is the bundle of fiberwise global sections of the holomorphic line bundle $\ev^* {\mc O}(k_i)$ that vanish at the markings $z_-$ and $z_+$. We choose an arbitrary smooth Hermitian metric on ${\rm univ}_* \ev^* {\mc O}(k_i)$. Then each ${\mc E}_i$ has an induced Hermitian metric. Moreover, together with the Chern connection on all involved holomorphic vector bundles and the chosen Hermitian metric on $TM$, ${\mc E}_i$ is equipped with a Hermitian connection. 

Now we inductively define a sheared almost complex structure on the total space of ${\mc E}_i$ similar to the constructions in Section \ref{subsubsec:shear}. The details are almost identical and omitted. We denote the resulting almost complex structure on ${\mc E}_d$ by $J_\Psi$.

\begin{thm}\label{regularity2}
For a given list $k_1 < \cdots < k_{d-1}$ of positive integers. There exists $k^* >0$ such that for all $k_d \geq k^*$, for all pairs $pq$ with $d_{pq} \leq d$, all elements of $S_{pq}^{-1}(0)$ are regular elements of ${\mc M}_{J_{\tilde \Psi}, H}(k_1, \ldots, k_d)$.  
\end{thm}

\begin{proof}
The proof is essentially the same as the single-layered case. Indeed, we can replace the almost complex manifold $(M, J)$ by the pair $({\mc E}_{d-1}, J_{d-1})$ where $J_{d-1}$ is the sheared almost complex structure on the total space ${\mc E}_{d-1}$ lifted via the chosen connections. Then we can proceed by writing the linearized Cauchy--Riemann operator in the block form as in \eqref{eqn:block-matrix}. As long as $k_d$ is sufficiently positive, we can make the linearization of the first equation in \eqref{eqn:floer-thicken} surjective by the peak section argument from \cite[Proposition 6.26]{AMS}. For the surjectivity of the linearization of the second part in \eqref{eqn:floer-thicken}, apply Lemma \ref{surjective} and Lemma \ref{lem:positive-curvature}.
\end{proof}

\begin{lemma}\label{lem:fiberwise-smooth}
Each fiber of the natural forgetful map 
\beq\label{eqn:pi-V-B}
\pi_{pq}: V_{pq} \to B_{pq}
\eeq
which takes a representative $(\Sigma, u, F, \eta_1, \ldots, \eta_d)$ to the equivalence class of the holomorphic map
$$
\iota_F : \Sigma \to {\mb CP}^d
$$
has a canonical structure of smooth manifold and the restriction of the obstruction bundle $O_{pq}^{(d)}$ to each fiber is canonically a smooth vector bundle. Moreover, each $g \in G_{pq}$ induces diffeomorphisms between fibers and smooth isomorphisms between the fiberwise restrictions of the obstruction bundle.
\end{lemma}

\begin{proof}
As the domain curve of elements in each fiber of $\pi_{pq}$ is fixed, a fiber of $\pi_{pq}$ is the zero locus of a smooth Fredholm section of a Banach vector bundle. The regularity along the vertical direction implies that the fibers are smooth because the complex structure on the domain is unchanged. For the same reason the restriction of $O_{pq}^{(d)}$ to each fiber is smooth. Lastly, if $g \in G_{pq}$, the $g$ induces a smooth identification between two Banach vector bundles which intertwines with the Fredholm sections. 
\end{proof}

In fact, the forgetful map $\pi_{pq}: V_{pq} \to B_{pq}$ is a $G_{pq}$-equivariant map between stratified spaces, where the underlying map between posets is the same as \eqref{eqn52xxx}. Lemma \ref{lem:fiberwise-smooth} says that this map further has a fiberwise smooth structure. However, Lemma \ref{lem:fiberwise-smooth} \emph{does not} assert that the locally closed stratum $\partial^{\alpha} \mathring{V}_{pq}$ has a smooth structure. Indeed, the $\bA_d$-stratification on $B_d$ ignores the appearance of sphere bubbles. More refined structures of $\pi_{pq}$ will be explored in Section \ref{sec-6}.

\subsection{Global chart construction III}\label{subsection54}

\subsubsection{Inductive construction of the thickened moduli spaces}

\begin{prop}\label{prop525}
There exists an increasing sequence of positive integers 
\beqn
k_1 < k_2 < \cdots < k_d < \cdots
\eeqn
which satisfies the following condition: for each $d \geq 1$, each $pq$ with $d_{pq} \leq d$, the thickened moduli space $V_{pq}(k_1, \ldots, k_d)$ is regular near the zero locus of the Kuranishi map.
%\begin{enumerate}
    %\item $I_{pq} = I_{rs}$ if and only if there exists an element $a \in \Pi$ such that $p = a \cdot s$, $q = a \cdot t$. 
    
    %\item If $[r, s]\prec [p, q]$, then $I_{rs} \subsetneq I_{pq}$.
    
    %\item For each face $\ov{\mc M}_{prq} \subset \ov{\mc M}_{pq}$, one has $I_{pr} \cap I_{rq} = \emptyset$.
    
    %\item Let $k_{pq} \in I_{pq}$ be the maximal element. Then 
    %\beqn
    %I_{pq} \setminus \{k_{pq}\} = \bigcup_{[r, s]\prec [p, q]} I_{rs}.
    %\eeqn
    
    %\item Denote 
    %\beqn
    %A = \bigcup_{p, q} I_{pq} \subset {\mb Z}_+.
    %\eeqn
    %Define a partial order $l' \to l$ in $A$ as follows. $l' \to l$ if there exists $pq$ such that $l' \in I_{pq}$ and $l = k_{pq}$. Let $I_{pq}$ be equipped with the induced partial order. Then the thickened moduli space $\ov{\mc M}{}_{pq}^{(I_{pq})}$ is regular near $\ov{\mc M}_{pq}$.
%\end{enumerate}
\end{prop}

\begin{proof}
This statement can be proved by induction on $d_{pq}$, combining the proof for the case of a single multi-layered thickening in Theorem \ref{regularity2}.  
\end{proof}

\begin{rem}
One can choose the sequence $k_1, k_2, \ldots$ such that the corresponding thickening of $\pss$, $\ssp$, $\pearl$, and the homotopy moduli spaces are all regular (see Section \ref{sec:pss}). 
\end{rem}

From now on we fix the sequence $k_1< k_2< \cdots$ which satisfies the conditions of Proposition \ref{prop525}. By shrinking to an open neighborhood of $S_{pq}^{-1}(0)$ inside $V_{pq}$, one obtains a  global Kuranishi chart. Just as Lemma \ref{lem:fiberwise-smooth}, there are also the natural forgetful maps
\beqn
\pi_{pq}: V_{pq} \to B_{pq}.
\eeqn
These maps should be thought of as the natural forgetful map which takes a $J$-holomorphic curve to its moduli parameter in the Deligne--Mumford space.

\subsubsection{Product construction: the case of two factors}

To obtain a K-chart lifting, one needs to define embeddings of products of charts into corresponding boundary strata of a bigger chart. We call the definition of such embeddings the product construction. This construction needs to be compatible with the product construction of the auxiliary moduli spaces $B_d$.

We first consider the case of two factors. Consider a nonempty codimension one stratum $\ov{\mc M}_{prq} \subset \ov{\mc M}_{pq}$. We would like to define an equivariant map 
\beqn
\vartheta_{prq}: V_{pr}\times V_{rq} \to \partial^{prq} V_{pq}
\eeqn
which lifts the map $\zeta_{prq}$ (see \eqref{eqn42} and \eqref{eqn43}). Suppose a point $\tilde x_{pr}\in V_{pr}$ is represented by a quadruple $(\Sigma_{pr}, u_{pr}, F_{pr}, \eta_{pr})$ and a point $\tilde x_{rq}\in V_{rq}$ is represented by a quadruple $(\Sigma_{rq}, u_{rq}, F_{rq}, \eta_{rq})$. Let $x_{pr}\in B_{pr}$, $x_{rq}\in B_{rq}$ be the underlying stable maps into projective spaces. Then using the map $\zeta_{prq}: B_{pr}\times B_{rq} \to \partial^{prq}B_{pq}$ one can define a point 
\beqn
x_{prq}:= \zeta_{prq}(x_{pr}, x_{rq})
\eeqn
which can be represented by a stable cylinder $\uds \mu_{prq}: \Sigma_{prq} \to \mb{CP}^{d_{pq}}$. Moreover, there is a natural framing 
\beqn
F_{prq} = (f_0, \ldots, f_{d_{pq}})
\eeqn
which lifts $u_{prq}$ defined as follows. Let $L_{prq}\to \Sigma_{prq}$ be the pullback of ${\mc O}(1) \to \mb{CP}^{d_{pq}}$ by $u_{prq}$ which carries the pullback of the Fubini--Study metric. Then the map $u_{prq}$ can be represented by a basis
\beqn
f_0, \ldots, f_{d_{pq}} 
\eeqn
of $H^0(L_{prq})$ which is well-defined up to a ${\mb C}^*$-factor. Moreover, the $M$-factors of $u_{pr}$ and $u_{rq}$ naturally define a map 
\beqn
u_{prq}: \Sigma_{prq} \to M
\eeqn
which converges to the correct periodic orbits at $z_\pm$. Lastly, one needs to define the combined vector $\eta_{prq}$. 

We first describe the single-layered case. The map $u_{prq}: \Sigma_{prq}\to \mb{CP}^{d_{pq}}$ induces linear embeddings
\begin{align}\label{linear_embedding}
&\ \mb{CP}^{d_{pr}} \cong Y_{pr}\subset \mb{CP}^{d_{pq}},\ &\ \mb{CP}^{d_{rq}}\cong Y_{rq} \subset \mb{CP}^{d_{pq}}.
\end{align}
Indeed, using the notations similar to \eqref{eqn54}, the first embedding is 
\beqn
[z_0, \ldots, z_{d_{pr}}] \mapsto [z_0, \ldots, z_{d_{pr}}, \underbrace{0, \ldots, 0}_{d_{rq}} ]
\eeqn
and the second embedding is 
\beqn
[z_0, \ldots, z_{d_{rq}}] \mapsto [a_0 z_0, \ldots, a_{d_{pr}} z_0, z_1, \ldots, z_{d_{rq}}].
\eeqn
Here $(a_0, \ldots, a_{d_{pr}})\in \mb{C}^{d_{pr}+1}$ is the normalized evaluation of the first component at $z_+$. Over these embeddings there are the natural bundle maps
\begin{align*}
&\ \vcenter{  \xymatrix{ {\mc O}(k) \ar[r] \ar[d] & {\mc O}(k) \ar[d] \\
                          \mb{CP}^{d_{pr}} \ar[r] & \mb{CP}^{d_{pq}}  } },\ &\ \vcenter{ \xymatrix{ {\mc O}(k) \ar[r] \ar[d] & {\mc O}(k) \ar[d] \\
                                      \mb{CP}^{d_{rq}} \ar[r] & \mb{CP}^{d_{pq}} } }.
\end{align*}
By definition,
\beqn
\eta_{pr}\in H^0 \left( \ov{\rm Hom} ( \iota_{F_{pr}}^* T\mb{CP}^{d_{pr}}, u^* TM) \otimes \iota_{F_{pr}}^* {\mc O}(k) \right) \otimes \ov{ H^0( \iota_{F_{pr}}^* {\mc O}(k))_0}.
\eeqn
Notice that there are natural linear inclusions
\begin{align*}
&\ H^0(\iota_{F_{pr}}^* {\mc O}(k))_0 \hookrightarrow H^0(\iota_{F_{prq}}^* {\mc O}(k))_0,\ &\ H^0(\iota_{F_{rq}}^* {\mc O}(k))_0 \hookrightarrow H^0( \iota_{F_{prq}}^* {\mc O}(k))_0
\end{align*}
defined by extending by zero. Moreover, the linear embeddings \eqref{linear_embedding} induce bundle maps
\beqn
\ov{\rm Hom}(\iota_{F_{pr}}^* T\mb{CP}^{d_{pr}}, u_{pr}^* TM) \otimes \iota_{F_{pr}}^* {\mc O}(k) \to \ov{\rm Hom}(\iota_{F_{prq}}^* T\mb{CP}^{d_{pq}}, u_{prq}^* TM) \otimes \iota_{F_{prq}}^* {\mc O}(k) 
\eeqn
and
\beqn
\ov{\rm Hom}(\iota_{F_{rq}}^* T\mb{CP}^{d_{rq}}, u_{rq}^* TM) \otimes \iota_{F_{rq}}^* {\mc O}(k) \to \ov{\rm Hom}(\iota_{F_{prq}}^* T\mb{CP}^{d_{pq}}, u_{prq}^* TM) \otimes \iota_{F_{prq}}^* {\mc O}(k)
\eeqn
covering the inclusion maps between domains.
These are defined by extending elements in $Hom$ by zero in the normal directions to $Y_{pr}$ resp. $Y_{rq}$. As elements of $H^0( \ov{\rm Hom} ( \iota_{F_{pr}}^* T\mb{CP}^{d_{pr}}, u^* TM) \otimes \iota_{F_{pr}}^* {\mc O}(k))$ vanish at nodes (because of the holonomy of $\nabla^{TM}$), one naturally has the linear inclusions 
\beqn
H^0 \left( \ov{\rm Hom}(\iota_{F_{pr}}^* T\mb{CP}^{d_{pr}}, u_{pr}^* TM) \otimes \iota_{F_{pr}}^* {\mc O}(k) \right) \to H^0 \left(  \ov{\rm Hom}(\iota_{F_{prq}}^* T\mb{CP}^{d_{pq}}, u_{prq}^* TM) \otimes \iota_{F_{prq}}^* {\mc O}(k) \right)
\eeqn
and
\beqn
H^0 \left( \ov{\rm Hom}(\iota_{F_{rq}}^* T\mb{CP}^{d_{rq}}, u_{rq}^* TM) \otimes \iota_{F_{rq}}^* {\mc O}(k) \right) \to H^0 \left(  \ov{\rm Hom}(\iota_{F_{prq}}^* T\mb{CP}^{d_{pq}}, u_{prq}^* TM) \otimes \iota_{F_{prq}}^* {\mc O}(k)\right).
\eeqn
Therefore, by taking the sum of the images of $\eta_{pr}$ and $\eta_{rq}$, one obtains an element 
\beqn
\eta_{prq}\in  H^0 \left(  \ov{\rm Hom}(\iota_{F_{prq}}^* T\mb{CP}^{d_{pq}}, u_{prq}^* TM) \otimes \iota_{F_{prq}}^* {\mc O}(k)\right)\otimes \ov{H^0 (\iota_{F_{prq}}^* {\mc O}(k))_0}.
\eeqn
Consider the quadruple $(\Sigma_{prq}, u_{prq}, F_{prq}, \eta_{prq})$. One can check that this is a framed cylinder of type $prq$ which represents an element $\tilde x_{prq}$ in the thickened moduli space $V_{prq}$. It is also straightforward to see that we just defined a continuous embedding 
\beqn
\vartheta_{prq}: V_{pr}\times V_{rq} \to \partial^{prq} V_{pq}
\eeqn
which is equivariant with respect to the group inclusion $G_{pr}\times G_{rq} \hookrightarrow G_{pq}$. Denote the image of $\vartheta_{prq}$ by $V_{prq}$. The equivariance implies that one can extend $\vartheta_{prq}$ to 
\beq\label{eqn:theta-V}
\vartheta_{prq}^\sim : G_{pq}\times_{G_{prq}} ( V_{pr}\times V_{rq}) \to \partial^{prq} V_{pq}
\eeq
whose image is denoted (in the same pattern as $B_{(d_0, d_1)}^\sim$) by 
\beqn
V_{prq}^\sim \subset \partial^{prq} V_{pq}  \subset V_{pq}.
\eeqn

\begin{rem}
However, the restriction of the Kuranishi map $S_{pq}: V_{pq} \to E_{pq}$ to the stratum $\partial^{prq} V_{pq}$ does not match with the Kuranishi map of the product of $S_{pr}\times S_{rq}$. For example, if $x_{pr}\in V_{pr}$ and $x_{rq}\in V_{rq}$ have orthonormal framings, the corresponding framing on $\vartheta_{prq}(x_{pr}, x_{rq})$ is not necessarily orthonormal (with respect to the $L^2$-pairing). Nonetheless, the Hermitian matrix part of the Kuranishi maps are given by two different kinds of normalization scheme, which can be interpolated once we have an outer-collaring (see Subsection \ref{subsection57}). 
\end{rem}

Just as the space $B_{d}$, for $p<r<s<q$, the embeddings of the thickened moduli spaces $V_{pq}$ fit in to the following commutative diagram
\beq\label{diag-V}
\xymatrix{ &  G_{pq} (V_{pr}\times V_{rs}\times V_{sq}) \ar[ld] \ar[rd] &\\
  G_{pq} \big( V_{pr} \times \partial^{rsq} V_{rq} \big) \ar[rd] & &  G_{pq} \big( \partial^{prs} V_{ps} \times V_{sq} \big) \ar[ld]\\
   &  \partial^{prsq} V_{pq} &  },
\eeq
where the arrows are constructed from \eqref{eqn:theta-V}. Indeed, the associativity on the gluing of the underlying map $(u, \Sigma)$ and the superposition of perturbation sections $\eta$ is straightforward from definition, while the associativity of the framing component $F$ follows from Proposition \ref{associativity}.

\vspace{0.2cm}

\noindent {\it Embedding obstruction bundles.} We now describe how the obstruction bundles $E_{pr}\to V_{pr}$ and $E_{rq}\to V_{rq}$ can be embedded into $E_{pq}$ restricted to $V_{prq}$. Recall that 
\begin{align*}
&\ E_{pr} = O_{pr}\oplus Q_{pr},\ &\ E_{rq} = O_{rq}\oplus Q_{rq}
\end{align*}
where the $O$-factors are the spaces of the sections $\eta$ and the $Q$-factors are the trivial bundle which are use to normalize the framings. The normalization strategy is to use the $L^2$-metric on the domain curves and to require the framings are orthonormal. 

Indeed, the embeddings $\vartheta_{{prq}}$ and $\vartheta_{{prq}}^\sim$ described above illustrate how to embed the $O$-factors. Hence there are natural bundle maps
\beqn
\xymatrix{ O_{prq} = O_{pr} \boxplus O_{rq} \ar[rr]^-{\widehat\vartheta_{prq}^O} \ar[d] & & O_{pq}|_{\partial^{prq} V_{pq}}  \ar[d] \\ 
           V_{prq} = V_{pr}\times V_{rq} \ar[rr]_-{\vartheta_{prq}}      &     & \partial^{prq} V_{pq} }.
\eeqn
Moreover, as the $Q$-bundles are pulled back from the auxiliary spaces $B_d$ and we have defined the corresponding bundle embeddings, we just pull back these embeddings to the thickened moduli space. Hence we have the bundle maps
\beqn
\xymatrix{ Q_{prq} = Q_{pr}\boxplus Q_{rq} \ar[rr]^-{\widehat \vartheta_{prq}^Q} \ar[d] & &  Q_{pq}|_{\partial^{prq} V_{pq}} \ar[d] \\
              V_{prq} = V_{pr} \times V_{rq} \ar[rr]_-{\vartheta_{prq}} & & \partial^{prq} V_{pq}} .
\eeqn
By taking the direct sum of $\widehat\vartheta_{prq}^O$ and $\widehat \vartheta_{prq}^Q$, we obtain the map
\beqn
\xymatrix{ E_{prq} = E_{pr}\boxplus E_{rq} \ar[rr]^-{\widehat \vartheta_{prq}} \ar[d] & &  E_{pq}|_{\partial^{prq} V_{pq}} \ar[d] \\
              V_{prq} = V_{pr} \times V_{rq} \ar[rr]_-{\vartheta_{prq}} & & \partial^{prq} V_{pq}} .
\eeqn
Take the $G_{pq}$-equivariantization with respect to the embedding $G_{pr} \times G_{rq} \hookrightarrow G_{pq}$, the above diagram lifts to
\beqn
\xymatrix{ G_{pq}(E_{prq}) \ar[rr]^-{\widehat \vartheta_{prq}^{\sim}} \ar[d] & &  E_{pq}|_{\partial^{prq} V_{pq}} \ar[d] \\
              G_{pq}(V_{prq}) \ar[rr]_-{\vartheta_{prq}^{\sim}} & & \partial^{prq} V_{pq}} 
\eeqn
and we denote the image of $G_{pq}(E_{prq})$ under the map $\widehat \vartheta_{prq}^{\sim}$ by $E^\sim_{prq}$. For a quadruple of elements $p<r<s<q$, the following diagram lives over \eqref{diag-V}.
\beq\label{diag-E}
\xymatrix{ &  G_{pq} (E_{pr}\times E_{rs}\times E_{sq}) \ar[ld] \ar[rd] &\\
  G_{pq} \big( E_{pr} \times \partial^{rsq} E_{rq} \big) \ar[rd] & &  G_{pq} \big( \partial^{prs} E_{ps} \times E_{sq} \big) \ar[ld]\\
   &  \partial^{prsq} E_{pq} &  }
\eeq
Beyond the ingredients of showing the commutativity of \eqref{diag-V}, Proposition \ref{prop:Q-gluing} covers the associativity of the superposition of the $Q$-component of the obstruction bundle.

\subsubsection{The general case}

Now we describe the product construction in the general situation, i.e. with multi-layered thickenings and multiple breakings. Consider a stratum $\alpha = pr_1 \cdots r_l q \in \bA^\floer_{pq}$. We can form the product chart 
\beqn
K_\alpha := K_{pr_1}\times \cdots \times K_{r_k q}
\eeqn
whose domain is $V_\alpha:= V_{pr_1}\times \cdots \times V_{r_l q}$ and whose obstruction bundle is 
\beqn
E_\alpha:= E_{pr_1}\boxplus \cdots \boxplus E_{r_l q}.
\eeqn
We would like to construct an equivariant embedding $V_\alpha$ into $\partial^\alpha V_{pq}$ and embed the obstruction bundle equivariantly as well. 

To embed the thickened moduli space, the first step is to unify the number of layers. Recall that in the multi-layered thickening strategy, each moduli space $\ov{\mc M}_{rs}$ is thickened to $V_{rs}^{d_{rs}}$ which depends on the sequence of integers $k_1< \cdots < k_{d_{rs}}$. 
\begin{lemma}
There exists a natural $G_{rs}$-equivariant inclusion
\beqn
V_{rs}^d \hookrightarrow V_{rs}^{d'}\ {\rm if}\ d < d'
\eeqn
defined by 
\beqn
[\Sigma, u, F, \eta_1, \ldots, \eta_d] \mapsto [\Sigma, u, F, \eta_1, \ldots, \eta_d, \underbrace{0, \ldots, 0}_{d'-d}]
\eeqn
where the last $d' - d$ zeroes are the zero vectors of the bundles of $(u, \iota_F)^* E_{k_{d+1}}$, $\ldots$, $(u, \iota_F)^* E_{k_{d'}}$. 
\end{lemma}
\begin{proof}
Indeed, as we use $\eta_j$ to perturb the equation for $\eta_i$ if $j > i$, the degree $d$ version of \eqref{eqn:floer-thicken} implies that the representative $(\Sigma, u, F, \eta_1, \ldots, \eta_d, \underbrace{0, \ldots, 0}_{d'-d})$ satisfies the degree $d'$ version of \eqref{eqn:floer-thicken}. As the $G_{rs}$-action changes the framing $F$, such a map is indeed $G_{rs}$-equivariant.
\end{proof}

Now abbreviate $d = d_{pq}$. For $\alpha = p r_1 \cdots r_l q$, one then obtains a $G_{\alpha}$-equivariant inclusion between the thickened moduli spaces
\beqn
V_\alpha = V^{d_{pr_1}}_{pr_1} \times \cdots \times V^{d_{r_l q}}_{r_l q} \hookrightarrow V^{d}_\alpha :=  V^{d}_{pr_1} \times \cdots \times V^{d}_{r_l q}.
\eeqn
The space $V^{d}_\alpha$ on the right hand side can be embedded into $\partial^\alpha V_{pq}^{d_{pq}}$. Indeed, it is straightforward to generalize the construction in the previous single-layered case with only two horizontal levels and the associativity can be verified as before.

\vspace{0.2cm}

\noindent {\it Embedding obstruction bundles}

\vspace{0.2cm}

Now for each stratum $\alpha = pr_1 \cdots r_l q \in \bA_{pq}^\floer$, define the bundles
\beqn
O_\alpha:= O_{pr_1 \to pq}\oplus \cdots \oplus O_{r_l q \to pq}:= O_{pr_1}\boxplus \cdots \boxplus O_{r_l q} \to V_\alpha = V_{pr_1}\times \cdots \times V_{r_l q}
\eeqn
and 
\beqn
Q_\alpha:= Q_{pr_1 \to pq} \oplus \cdots \oplus Q_{r_l q \to pq}:= Q_{pr_1}\boxplus \cdots \boxplus Q_{r_l q} \to V_\alpha = V_{pr_1}\times \cdots \times V_{r_l q}.
\eeqn
The construction of the two-factor case and the corresponding associativity induce $G_\alpha$-equivariant bundle embeddings 
\beqn
\xymatrix{ E_\alpha:= O_\alpha\oplus Q_\alpha \ar[rr]^-{ \widehat\vartheta_\alpha = \widehat \vartheta_\alpha^O \oplus \widehat \vartheta_\alpha^Q } \ar[d]  & & E_{pq}|_{\partial^\alpha V_{pq}} \ar[d] \\ 
           V_\alpha  \ar[rr]_-{\vartheta_\alpha} & &  \partial^\alpha V_{pq}}.
\eeqn

\begin{notation}
The $G_{pq}$-equivariantization of the embedding $\vartheta_\alpha$ is denoted by 
\beqn
\vartheta_\alpha^\sim: G_{pq}\times_{G_\alpha} V_\alpha \to \partial^\alpha V_{pq}
\eeqn
and its image is denoted by 
\beqn
V_\alpha^\sim \subset \partial^\alpha V_{pq}.
\eeqn
Points in $V_\alpha^\sim$ are those maps whose underlying map into $\mb{CP}^d$ are in $\delta(\alpha)$-normal position and whose obstruction vector are coming from factors of $\alpha$. %There is another larger space
%\beqn
%V_\alpha^\approx \subset \partial^\alpha V_{pq}
%\eeqn
%consisting of equivalence classes of configurations $(\Sigma, u, F, \eta)$ where $\Sigma$ has horizontal components $\Sigma_{pr_1}, \ldots, \Sigma_{r_l q}$ such that $\eta|_{\Sigma_{r_i r_{i+1}}}$ is contained in $E_1 \oplus \cdots \oplus E_{d_{r_i r_{i+1 }}}$. So the corresponding map $\iota_F: \Sigma \to \mb{CP}^d$ may not be in normal position.
\end{notation}

\subsubsection{More on obstruction bundles}

Here we single out a particular pattern of the obstruction bundles which will be used several times in the rest of the paper. We fix $p, q\in {\mc P}^\floer$ and $\alpha = pr_1 \cdots r_l q \in \bA_{pq}^\floer$. For each $d \geq d_{pq}$, one has part of the obstruction bundles
\beqn
O_{pr_1}^{(d)} \to V_{pr_1}, \ldots, O_{r_l q}^{(d)}\to V_{r_l q}.
\eeqn
Via the above product construction, denote their images by 
\beqn
O_{pr_1 \to pq}^{(d)},\ldots, O_{r_l q \to pq}^{(d)}  \to V_\alpha = V_{pr_1}\times \cdots \times V_{r_l q}.
\eeqn
They are embedded into the bundle $O_{pq}^{(d)} |_{V_\alpha}$. Denote their (direct) sum by 
\beqn
O_\alpha^{(d)} \to V_\alpha.
\eeqn
Notice that inside $O_{pq}^{(d)}|_{V_\alpha}$ there are also subbundles which are embedded images of $O_\beta^{(d)}|_{V_\alpha}$ for all intermediate strata $\beta$. 

In many inductive constructions, we wouldl like to choose structures (such as inner products, connections, maps, etc.) on $O_{pq}^{(d)}$ inductively. When we make this choice over a particular stratum $V_\alpha \subset V_{pq}$, we would like to extend the existing structures on those subbundles coming from product construction. These subbundles have overlaps and the existing structures agree on overlaps by the correct inductive hypothesis. However, this does not guarantee the existence of an extension; indeed some extra properties of these subbundles $O_\beta^{(d)}$ is needed. 

\begin{lemma}\label{obundle_property}
Fix $\alpha = pr_1 \cdots r_l q \in \bA_{pq}^\floer$. Consider the two subbundles 
\beqn
O_{pr_l \to pq}^{(d)}\to V_{pr_l q} \subset O_{pq}^{(d)}|_{V_{pr_lq}},  O_{r_1 q \to pq}^{(d)} \to V_{pr_1 q} \subset O_{pq}^{(d)}|_{V_{pr_1 q}}. 
\eeqn
\begin{enumerate}
    \item One has 
    \beq\label{obundle_intersection}
    O_{pr_l \to pq}^{(d)} \cap O_{r_1 q\to pq}^{(d)} = O_{r_1 r_l\to pq}^{(d)}.
    \eeq

    \item For each $\beta \in \bA_{pq}^\floer$ with $\alpha \leq \beta$, there one has
    \beq\label{obundle_inclusion}
    O_\beta^{(d)}|_{V_\alpha} \subset O_{pr_l \to pq}^{(d)}|_{V_\alpha} + O_{r_1 q \to pq}^{(d)}|_{V_\alpha}.
    \eeq
    \end{enumerate}
\end{lemma}

\begin{proof}
We first consider the single-layered case. Choose a point $x \in V_{pr_1 r_l q}$ with underlying curve $\Sigma$. Recall that vectors in $O_{pq}$ are $\eta \in H^0(Y) \otimes \ov{H^0(L)_0}$ where $Y \to \Sigma$ is a Hermitian vector bundle and $L \to \Sigma$ is a holomorphic line bundle. Let $\Sigma_{r_1 r_l}\subset \Sigma$ be the union of components ``between'' $r_1$ and $r_l$. Then a vector in the intersection \eqref{obundle_intersection} is $\eta\in H^0(Y) \otimes \ov{ H^0(L)_{r_1 r_l}}$ where $H^0(L)_{r_1 r_l}\subset H^0(L)_0$ is the subspace of holomorphic sections supported in $\Sigma_{r_1 r_l}$ which vanish at the corresponding two nodes. On the other hand, recall $Y = \ov{\rm Hom} (\iota_F^* T\mb{CP}^d, u^* TM)\otimes \iota_F^* {\mc O}(k)$. As the underlying curve inside $\mb{CP}^d$ is in normal position, there is a corresponding fan $W_{x, pr_1}, W_{x, r_1 r_l}, W_{x, r_l q}$ in $\mb{C}^{d+1}$ in the normal position. The condition $\eta\in O_{pr_l\to pq}$ implies that $\eta$ vanishes on the normal direction of $W_{x, pr_1} + W_{x, r_1 r_l}$; the condition $\eta \in O_{r_1 \to pq}$ implies that $\eta$ vanishes on the normal direction of $W_{x, r_1 r_l} + W_{x, r_l q}$. Therefore, $\eta$ vanshes on the normal direction of $W_{x, r_1 r_l}$. By the way we define the inclusion $O_{r_1 r_l \to pq} \hookrightarrow O_{pq}$, we see that $\eta\in O_{r_1 r_l \to pq}$. This proves \eqref{obundle_intersection} in the single layered case. The multi-layered case, more generally, can be proved by induction on the number of layers.

For \eqref{obundle_inclusion}, suppose $\beta = ps_1, \ldots, s_m q$. Then $s_1 \geq r_1$ and $s_m \leq r_l$. Hence each summand of $O_\beta^{(d)}$ i.e., one of $O_{ps_1 \to pq}^{(d)}$, $\ldots$, $O_{s_m q \to pq}^{(d)}$, is either contained in $O_{pr_l \to pq}^{(d)}$ or in $O_{r_1 q \to pq}^{(d)}$. 
\end{proof}

\subsection{Stabilization maps}\label{subsection55}

Now we state our proposition regarding the stabilization maps.

\begin{prop}
Fix $p< q$ in ${\mc P}^\floer$. There exist the following objects. 
\begin{enumerate}

\item For each $\alpha \in \bA_{pq}^\floer$, a $G_\alpha$-equivariant %stratified smooth
subbundle $F_{pq, \alpha} \subset E_{pq}|_{V_\alpha}$ which is complementary to $E_\alpha$ and carries a stratification by $G_\alpha$-invariant subbundles
\beqn
F_{\beta\alpha}\ \forall \beta\in \bA_{pq}^\floer, \alpha \leq \beta
\eeqn
such that 
\beqn
E_\alpha \oplus F_{\beta\alpha} = E_\beta|_{V_\alpha}.
\eeqn

\item A germ of $G_\alpha$-equivariant %stratified smooth 
embedding
\beqn
\theta_{pq, \alpha}: {\rm Stab}_{F_{pq, \alpha}}(V_\alpha) \to \partial^\alpha V_{pq}
\eeqn
satisfying the following conditions.
\begin{enumerate}
\item Its $G_{pq}$-equivariantization is a germ of open embedding.

\item Its restriction to $V_\alpha$ (the zero section) coincides with the embedding $\iota_{pq,\alpha}: V_\alpha \to \partial^\alpha V_{pq}$.

\item For each stratum $\beta$ between $\alpha$ and $pq$, one has
\beq\label{stabilization_property}
\theta_{pq, \alpha} \left( {\rm Stab}_{F_{\beta\alpha}} (V_\alpha) \right) \subset \partial^\alpha V_\beta.
\eeq
\end{enumerate}
\end{enumerate}
\end{prop}

In fact we would like the maps constructed to have better properties. The discussion here is to showcase the analytical part of the construction. The above proposition is enough to establish the construction at this moment, i.e., Theorem \ref{pre-Kuranishi}.

\subsubsection{The difference bundle}

We first construct the bundle appeared in the stabilization map. Recall that $E_{pq} \to V_{pq}$ is naturally decomposed into $O_{pq}$ and $Q_{pq}$ where the latter is a trivial bundle. The embedding $\iota_{pq, \alpha}: V_\alpha \to \partial^\alpha V_{pq}$ is covered by a bundle embedding $O_\alpha \to O_{pq}|_{\partial^\alpha V_{pq}}$. We can take a $G_{pq}$-invariant inner product on $O_{pq}$ and consider the orthogonal complement $O_{pq, \alpha}:= (O_\alpha)^\bot$. On the other hand, we have already chosen a good complement of $Q_\alpha \subset Q_{pq}$, denoted by $Q_{pq, \alpha}$ (which is trivial). Define 
\beqn
F_{pq, \alpha}:= O_{pq, \alpha}\oplus Q_{pq, \alpha}.
\eeqn

\subsubsection{Analytical setup}

As expected, the stabilization map is constructed using the infinite-dimensional implicit function theorem in a global fashion. We need a global analytical setup for this operation.

Fix $pq$ and abbreviate $d = d_{pq}$. For each point $\uds x \in B_d$, its fiber in the universal curve is a well-defined prestable cylinder $\Sigma_{\uds x}$ over which there is a holomorphic map $\uds \mu: \Sigma_{\uds x} \to \mb{CP}^d$. Then there is a fiber bundle over $\Sigma_{\uds x} \times M$ denoted by
\beqn
E_{\uds x}^{(d)} \to \Sigma_{\uds x} \times M.
\eeqn
More explicitly, in the single-layered case,
\beq\label{fiberwise_bundle}
E_{\uds x}^{(d)} = \ov{\rm Hom} ( \uds \mu^* T\mb{CP}^d, TM ) \otimes \uds \mu^* {\mc O}(k) \otimes \ov{ H^0( \uds \mu^* {\mc O}(k))_0}.
\eeq
If we fix the point $\uds x$, then the fiber $\pi_{pq}^{-1}(\uds x)$ is a subset of the space of sections $(u, \eta): \Sigma \to E_{\uds x}^{(d)}$ where $u: \Sigma \to M$ is a map and $\eta$ is a section of $E_{\uds x}^{(d)}|_{{\rm graph} (u)}$. 

Now fix a positive integer $l$. Let ${\mc C}^l(x):= C^l(\Sigma, E_{\uds x}^{(d)})$ be the space of $C^l$-sections $(u, \eta): \Sigma \to E_{\uds x}^{(d)}$. Here we impose the asymptotic constraints at nodes and marked points to periodic orbits. Take the union over $B_d$, we obtained a well-defined topological space, denoted by ${\mc C}^l_{pq}$. It is also stratified by $\bA_{pq}^\floer$. Now notice that one has a $G_{pq}$-action on ${\mc C}^l_{pq}$ and an invariant subspace 
\beqn
V_{pq} = V_{pq}^{(d)} \subset {\mc C}^l_{pq}.
\eeqn
In fact, the complexification $G_{pq}^{\mb C}$ acts on ${\mc C}_l^{pq}$ by transforming the underlying points $\uds x$ and the maps $(u, \eta)$, although $V_{pq}$ is not $G_{pq}^{\mb C}$-invariant. 

Now we consider the difference between $V_\alpha$ and $\partial^\alpha V_{pq}$. We regard both as subspaces of the infinite-dimensional object, i.e. 
\beqn
V_\alpha \subset \partial^\alpha V_{pq} \subset \partial^\alpha  {\mc C}_{pq}^l.
\eeqn
We define an {\bf approximate solution map}
\beq\label{eqn:app-solution}
\theta_{pq, \alpha}^{\rm app}: {\rm Stab}_{F_{pq, \alpha}} (V_\alpha) \to \partial^\alpha {\mc C}_{pq}^l.
\eeq
by
\beqn
[\Sigma, u, \mu, \eta_\alpha, \eta_{pq, \alpha}, h_{pq, \alpha}] \mapsto ({\rm Id} + \rho_{h_{pq, \alpha}} ) [\Sigma, u, \mu, \eta_\alpha + \eta_{pq, \alpha}].
\eeqn
Here $\rho_{h_{pq, \alpha}}$ is the upper-triangular matrix defined in the proof of Proposition \ref{prop523}. The matrix ${\rm Id} + \rho_{h_{pq, \alpha}}$ acts on the framing map $\uds{\mu}: \Sigma \to {\mb CP}^d$ and the obstruction sections. This is a $G_\alpha$-equivariant map and satisfies
\beqn
\theta_{pq, \alpha}^{\rm app} \left( {\rm Stab}_{F_{\beta\alpha}} (V_\alpha) \right) \subset \partial^\beta {\mc C}_{pq}^l.
\eeqn
However, it is not necessarily an exact solution, because the extra $\eta_{pq, \alpha}$ can perturb the equation and the equation is not invariant under the complexification of $G_{pq}$. We need to use right inverses to correct the solution.

\begin{rem}
When we construct stabilization maps inductively, the approximate solution map may not be just the linear addition in the fiber direction. In fact we may also need to deform the map $u$.
\end{rem}

We first discuss the notion of linearizations of the defining equation of the thickened moduli space more carefully. Fix a real number $a > 2$. Fix an element $x \in V_{pq}$. 

\begin{defn}
Given $\uds x \in B_d$, one has the Banach manifold ${\mc B}_{pq, \uds x}$ of $W^{1,a}$-sections $(u, \eta): \Sigma_{\uds x} \to E_{\uds x}^{(d)}$ which are asymptotic to prescribed periodic orbits and a Banach bundle ${\mc E}_{pq, \uds x} \to {\mc B}_{pq, \uds x}$ of corresponding inhomogeneous terms of class $L^a$. Here the Banach norms are defined with respect to the cylindrical metric on cylindrical components and arbitrary smooth metric on spherical components. 

For each point $x \in \pi_{pq}^{-1}(\uds x)$, the linearization of the defining equation at $x$ is a Fredholm operator
\beq
D_{pq, x}: T_x {\mc B}_{pq, \uds x} \to {\mc E}_{pq, \uds x}|_x.
\eeq
By the regularity statement Proposition \ref{prop525}, the linearizations are always surjective in our setting. Then we can talk about bounded right inverses. A right inverse of $D_{pq, x}$ is denoted by $T_{pq, x}$. When $pq$ is understood, we use $D_x$ and $T_x$ instead.
\end{defn}

To obtain global stabilization maps using the implicit function theorem, we need to choose a family of right inverses to the linearization. Obviously, we need a continuity condition. We define the notion of continuity by comparing with the local families obtained from the standard gluing procedure. Let us briefly recall the gluing construction in the context of pseudoholomorphic curves which should be most familiar to the reader. Given a nodal $J$-holomorphic curve $(\Sigma, u)$ with a stable domain $\Sigma$ therefore finite automorphism group, there exists a local universal unfolding of the domain $\Sigma$. There is also a way to identify any compact subset of the nodal $\Sigma$ away from the nodes with a compact subset of nearby fibers. Choose a right inverse $T$ at the point represented by $(\Sigma, u)$. Then using such identifications and certain cut-and-paste procedure, one can obtain a family of approximate right inverses for nearby solutions. For details, the reader could consult \cite[Equation (C.7.7)]{pardon-VFC} for instance. These approximate right inverses determine exact right inverses which have the same images, as one could read off from \cite[Equation (C.7.30)]{pardon-VFC}. If $\zeta$ is a parameter in a neighborhood of the point represented by $(\Sigma, u)$, then we denote by $T_\zeta^{\rm glue}$. The construction further depends on some auxiliary data including Riemannian metrics which define the parallel transport to compare the deformation space between the central curve and nearby curves, and a cut-off function which is needed for the construction of the pre-gluing map \cite[Equation (C.7.9)]{pardon-VFC}. These data are referred to as a {\bf gluing profile}.

Similar construction applies to our current situation for the thickened moduli spaces $V_{pq}$. Indeed, as the spaces $B_d$ are locally isomorphic to the moduli space $\ov{\mc M}{}_{0, d'+2}^{\mb R}$ for some $d' > 0$,the space of stable marked cylinders we can compare nearby fibers of the fiberwise Banach manifolds and identify compact subsets away from nodes with subsets of nearby fibers. The corresponding gluing construction can be performed in a similar way. The reader could consult Proposition \ref{prop:c1loc} for more arguments along this line.

\begin{defn}[Continuous family of right inverses]
A family of right inverses $\{ T_x \}_{x \in V_{pq}}$ is said to be {\bf continuous} at $x$ if the following is true. On one hand, for each $x \in V_{pq}$ and any representative $\tilde x = (\Sigma, u, \mu, \eta)$, the restriction of $T_x$ to each irreducible component of the domain defines a right inverse for the restriction of the linearized Cauchy--Riemann operator. Then given a gluing profile for $\tilde{x}$, these right inverses induce a right inverse $T_\zeta^{\rm glue}$ for any $\zeta$ near $\tilde{x}$. On the other hand, $T_x$ induces a family of right inverses $T_\zeta$ over this open neighborhood. Then for each compact set $Z \subset \Sigma \setminus \{ {\rm nodes} \}$ there holds
\beqn
\lim_{\zeta \to 0}  \| T_\zeta^{\rm glue} - T_\zeta \|_{Z} = 0.
\eeqn
Here we use the operator norm on the domain $Z$. 
\end{defn}

\begin{rem}
We left to the reader to check that the above notion of continuity is intrinsic, i.e., independent of the choice of the gluing profile. In fact one only needs to check that the family of right inverse obtained from one gluing profile is continuous with respect to another gluing profile. 
\end{rem}

\begin{lemma}
For each pair $p< q$ in ${\mc P}^\floer$, there exists a $G_{pq}$-equivariant continuous family of right inverses $T_x$ for $x \in V_{pq}$.
\end{lemma}

\begin{proof}
It suffices to construct a continuous family of right inverses without the equivariance condition as we can average over the symmetry group $G_{pq}$. Moreover, as any convex linear combination of two right inverses is still a right inverse, we only needs to construct continuous families locally and patch together using a continuous partition of unity. From the definition we know the local family of right inverses given by a gluing construction is locally continuous. 
\end{proof}

Let us fix a choice of a $G_{pq}$-equivariant continuous family of right inverses $T_x$ for $x \in V_{pq}$. Now using the product construction, for $\alpha  = pr_1 \cdots r_l q \in \bA_{pq}^\floer$, one obtains a fiberwise Banach manifold ${\mc B}_{\alpha, \uds x}$ for each $\uds x\in B_{\delta(\alpha)}$ and corresponding linearized Cauchy--Riemann operator $D_{\alpha, x}$ at each $x \in \pi_\alpha^{-1}(\uds x)$. A chosen family of right inverses $T_{pr_1}, \ldots, T_{r_l q}$ provides a family of right inverse $T_{\alpha, x}$ for all $x \in V_\alpha$. Now as we have another layer of thickening, and we would like to solve the corresponding equation for this ``thicker'' problem, we need to stabilize the right inverse. 

First, notice that we have the natural inclusions
\beqn
\xymatrix{
{\mc E}_{\alpha, \uds x} \ar[r] \ar[d] & {\mc E}_{pq, \uds x} \ar[d]\\
{\mc B}_{\alpha, \uds x} \ar[r]        & {\mc B}_{pq, \uds x} }.
\eeqn
The differences are given by deformations/obstructions in the difference between the two bundles over $M$, $E_{\alpha, \uds x}$ and $E_{pq, \uds x}$. By looking at the thickened equation \eqref{eqn:floer-thicken}, we can see that the linearization at $x \in V_\alpha \subset \partial^\alpha V_{pq}$ is block upper-triangular
\beqn
D_{pq, x} = \left[ \begin{array}{cc}  D_{\alpha, x} & * \\   0 &  D_{pq, \alpha, x} \end{array} \right]
\eeqn
where $D_{pq, \alpha, x}$ is a surjective linear Cauchy--Riemann operator. Hence we can extend $T_{\alpha, x}$ to a block-upper triangular right inverse
\beqn
T_{pq, x} = \left[ \begin{array}{cc} T_{\alpha, x} & *  \\ 0 & T_{pq, \alpha, x} \end{array} \right].
\eeqn
In fact this only depends on the choice of the right inverse $T_{pq, \alpha, x}$ and we can make this choice continuously depending on $x$.

The last ingredient we need for applying the implicit function theorem is a choice of local charts in the Banach manifold and local frames of the Banach vector bundle. In fact there is a natural one coming from the Riemannian metric on $M$. For each $\uds x \in B_d$, the fiberwise Banach manifold is a space of maps into the total space of a vector bundle $E_{pq, \uds x} \to M$ which is essentially a tensor bundle over $M$. Hence the parallel transport using the Levi--Civita connection associated to the metric $\omega(\cdot, J)$ induces local charts of the fiberwise Banach manifold. 

We would like to say that so far the right inverse and the local charts/frames are chosen for points which are in normal position. If we go slightly away from the locus of normal position, the parallel transport along shortest geodesics automatically produce a family of local charts/frames as well as right inverses.

\subsubsection{Definition of the stabilization map}

Now we can describe the stabilization map. Let us temporarily denote by ${\mc F}_{pq}$ the nonlinear Cauchy--Riemann operator \eqref{eqn:floer-thicken}. First, apply the approximate solution map \eqref{eqn:app-solution}
\beqn
\theta_{pq, \alpha}^{\rm app}: {\rm Stab}_{F_{pq, \alpha}}(V_\alpha) \to \partial^\alpha {\mc C}_{pq}^l.
\eeqn
(which is in fact only defined in a small disk bundle in $F_{pq, \alpha}$). Then there exists $\epsilon>0$ such that for all $x \in V_\alpha$ and $(\eta_{pq, \alpha}, h_{pq,\alpha}) \in F_{pq, \alpha}|_x$ with $\| \eta_{pq, \alpha} \| + \| h_{pq, \alpha}\| < \epsilon$, ${\mc F}_{pq}(\theta^{\rm app}_{pq, \alpha}(x, \eta_{pq, \alpha}, h_{pq, \alpha}))$ is very small. Then using the local charts on the fiberwise Banach manifold induced from the Levi--Civita connection centered at $x$ and the corresponding right inverse of the linear operator, one obtains an exact solution to ${\mc F}_{pq}(x) = 0$ which is sufficiently close to the approximate solution and its difference with the approximate solution is contained in the image of the right inverse. 

There are several things to check. First, we need to show that the map is continuous. In fact, this is due to the continuity of the family of right inverses and approximate solutions and how Gromov topology on $V_{pq}$ is defined. Second, we show that the map is injective. This is a consequence of the implicit function theorem and the fact that (on each fiber over $B_{pq}$), the approximate solution is transverse to the image of the right inverse we use. Third, we show that the map is onto an open subset. Indeed, from the construction we know that the stabilization map is onto an open subset on each fiber over $B_d$. As the image also covers a neighborhood of $B_{\delta(\alpha)}^\sim$ inside $\partial^{\delta(\alpha)} B_d$, the stabilization map is onto an open subset. In fact, this is exactly the content of the gluing theorem presented in \cite[Appendix C]{pardon-VFC}. Again, the reader could refer to Proposition \ref{prop:c1loc} to see how to implement the aforementioned theorem into our setting. Lastly, it satisfies the stratum-preserving property \eqref{stabilization_property} which can be seen from the fact that the right inverses we used are all block upper-triangular. 

\begin{comment}

We do not have the local surjectivity yet because we haven't use the other part of the obstruction bundle. Nonetheless, the range of this map contains all elements in a neighborhood of 

\subsubsection{Relative case}

\begin{lemma}
Fix a stratum $\alpha = pr_1 \cdots r_l q$. Suppose we are given a stabilization map defined over a $G_{pq}$-invariant open neighborhood $U$ of $\partial V_{\bm \alpha}$, i.e., a weak open embedding 
\beqn
{\bm \theta}_U: {\rm Stab}_{F_{pq, {\bm \alpha}}} (K_{\bm \alpha}|_U) \wembed \partial^\alpha K_{pq}
\eeqn
satisfying 1) combinatorially stratified smoothness... Then there exists a stabilization map ${\bm \theta}_{pq, \alpha}$ which is combinatorailly stratified smooth and coincides with ${\bm \theta}_U$ near $\partial V_{\bm \alpha}$. 
\end{lemma}

\begin{proof}
In the same way as proving ..., we can construct a stabilization map ${\bm \theta}_{\rm new}$ using a continuous family of right inverses and a continuous family of approximate solutions. We would like to interpolate between ${\bm \theta}_U$ and ${\bm \theta}_{\rm new}$. We only describe how two interpolate between the open embeddings of domains of the charts as the bundle part is much easier. Indeed, as over each combinatorial stratum both stabilization maps are smooth, 
\end{proof}

\end{comment}

\begin{rem}
It is natural to require the stabilization map to be compatible with compositions of chart embeddings. In fact to achieve this goal one needs to make a number of compatible choices, including showing that such compatible choices exist. The choices include a system of compatible inner products on the obstruction bundles, a system of stabilization maps, and a system of bundle identifications. The details are given in Appendix \ref{appendixc}. 
\end{rem}

\subsection{Weak K-chart lift}

We summarize the construction we have done. The following theorem is a more comprehensive version of Theorem \ref{thm:flow-chart}.

\begin{thm}\label{pre-Kuranishi}
There exist the following objects.
\begin{enumerate}
    \item A weak K-chart lift of $T^\floer$ (see Definition \ref{Klift_defn}), consisting of a collection of K-charts $K_{pq} = (G_{pq}, V_{pq}, E_{pq}, S_{pq}, \psi_{pq})$ and a collection of weak K-chart embeddings
    \beqn
    {\bm \iota}_{prq}: K_{pr}\times K_{rq} \wembed \partial^{prq} K_{pq}
    \eeqn
    (which satisfy the required conditions).
    
    \item A collection of $G_{pq}$-equivariant stratified maps
    \beqn
    \vcenter{ \xymatrix{ V_{pq} \ar[r]^{\pi_{pq}} \ar[d] & B_{d_{pq}}   \ar[d]\\
                \bA_{pq}^\floer \ar[r]_\delta & \bA_{d_{pq}} } } .
    \eeqn
\end{enumerate}
Moreover, the collection of maps $\pi_{pq}$ satisfy the following conditions. 
\begin{enumerate}
    \item Every fiber of $\pi_{pq}$ has a canonical smooth structure and the restriction of $E_{pq}$ to each fiber of $\pi_{pq}$ has a canonical smooth bundle structure. Moreover, each $g \in G_{pq}$ induces diffeomorphisms between fibers and smooth bundle isomorphisms between the fiberwise restrictions of the obstruction bundles.
    
    \item The collection of maps $\pi_{pq}$ are compatible in the following sense. For each $\alpha = pr_1 \cdots r_l q \in \bA_{pq}^\floer$, the products of the projections induces a map 
    \beqn
    \pi_\alpha: V_\alpha \to B_{\delta(\alpha)}.
    \eeqn
    Then the following diagram commutes.
    \beqn
    \vcenter{ \xymatrix{ V_\alpha \ar[r]^{\vartheta_\alpha} \ar[d]_{\pi_\alpha} &  \partial^\alpha V_{pq}\ar[d]^{\pi_{pq}} \\
                         B_{\delta(\alpha)} \ar[r]_{\zeta_{\delta(\alpha)}} & \partial^{\delta(\alpha)} B_{d_{pq}}}}.
    \eeqn
\end{enumerate}
\end{thm}

\subsection{Outer collaring}\label{subsection57}

We need to apply the outer-collaring construction in order to improve the weak K-chart lift constructed above to a K-chart lift and set space for the inductive smoothing procedure. Recall that the outer-collaring construction has been reviewed in Subsection \ref{subsection:outer_collar} in the topological category, including both the case of a single stratified space or a flow category, bimodule, etc. There is no more complexity to generalize to the case of the K-chart lift ${\mf K}^\floer$. Indeed, for each pair $p<q$, if we apply the outer-collaring construction to the thickened moduli space $V_{pq}$ (with a fixed width, say $1$), then one obtains a new $\bA_{pq}^\floer$-manifold, denoted by $V_{pq}^+$. The $G_{pq}$-action naturally extends to a $G_{pq}$-action on $V_{pq}^+$. The obstruction bundle $E_{pq}$ and the Kuranishi section $S_{pq}$ are also naturally extended to a bundle $E_{pq}^+\to V_{pq}^+$ and a section $S_{pq}^+: V_{pq}^+ \to E_{pq}^+$. The footprint map $\psi_{pq}: S_{pq}^{-1}(0)/G_{pq} \to \ov{\mc M}{}_{pq}^\floer$ is also canonically extended to a map 
\beqn
\psi_{pq}^+: (S_{pq}^+)^{-1}(0)/ G_{pq} \to (\ov{\mc M}{}_{pq}^\floer)^+
\eeqn
where (recall that) $(\ov{\mc M}{}_{pq}^\floer)^+$ is the outer-collaring of $\ov{\mc M}{}_{pq}^\floer$. Notice that the structure maps required for a K-chart lift of a flow category are also obtained among the new charts via outer-collaring and their properties are still satisfied. In summary, one obtains a K-chart lift $({\mf K}^\floer)^+$ of the outer-collared flow category $(T^\floer)^+$. 

In a similar way, one can apply the outer-collaring construction to the system of auxiliary moduli spaces. Each space $B_d$ becomes a new space $B_d^+$ and the projection maps $\pi_{pq}: V_{pq} \to B_{d_{pq}}$ are extended to $\pi_{pq}^+: V_{pq}^+ \to B_{d_{pq}}^+$ with listed properties in Theorem \ref{pre-Kuranishi} remain valid. 

What we need to take care of is the smoothness of the outer-collared auxiliary moduli spaces. Though $B_d$ is a smooth manifold with faces, it is not {\it a priori} clear whether $B_d^+$ has an automatically induced smooth structure or if the outer-collaring of the maps $\zeta_{(d_0, d_1)}: B_{d_0} \times B_{d_1} \to \partial^{(d_0, d_1)} B_{d_0+d_1}$ are still smooth. What we learned from the work of Fukaya--Oh--Ohta--Ono \cite{FOOO_smooth} is that moduli spaces of pseudoholomorphic curves have a stronger structure near corners which comes from their ``exponential decay estimates'' associated to the gluing construction. This kind of structure is formalized in \cite[Chapter 25]{FOOO_Kuranishi} under the name ``admissible smooth structure.'' In short, an admissible structure on a smooth manifold (or orbifold) with faces is an atlas of smooth charts such that the coordinate changes among them satisfy certain exponential decay estimates near the boundary and corner; moreover, the outer-collaring of an admissible smooth manifold or orbifolds automatically carries an admissible smooth structure. This argument applies to our situation of $B_d$. Indeed, $B_d$ comes from the moduli space of stable maps into $\mb{CP}^d$ which is an algebraic object. The exponential decay estimate holds automatically, for both the original moduli and the real blowup $B_d$. Therefore, the outer-collaring $B_d^+$ is still smooth and the embeddings
\beqn
\zeta_{(d_0, d_1)}^+: B_{d_0}^+\times B_{d_1}^+ \to B_{d_0+d_1}^+
\eeqn
are still smooth.

\subsubsection{Matching the Kuranishi sections}

The collection of global Kuranishi charts constructed in the previous section does not have matching Kuranishi maps over boundary or corners. Using the room constructed by the outer-collaring operation, we can connect these Kuranishi maps.

\begin{prop}
There exist a collection of $G_{pq}$-equivariant sections
\beqn
\check S_{pq}^+: V_{pq}^+ \to E_{pq}^+
\eeqn
satisfying the following conditions.
\begin{enumerate}
    \item The restriction of $\check S_{pq}^+$ to $V_{pq} \subset V_{pq}^+$ coincides with $S_{pq}: V_{pq} \to E_{pq}$.
    
    \item With respect to the direct sum decomposition $E_{pq}^+ = O_{pq}^+ \oplus Q_{pq}^+$, if we write $\check S_{pq}^+$ as $\check S_{pq}^{O, +}\oplus \check S_{pq}^{Q, +}$, then $\check S_{pq}^{O,+} = S_{pq}^{O, +}$. 
    
    \item There is a homeomorphism
    \beqn
    \psi_{pq}^+: (\check S_{pq}^+)^{-1}(0)/ G_{pq} \cong \ov{\mc M}{}_{pq}^\floer.
    \eeqn
    
    \item Let $\check \psi_{pq}$ be the corresponding extension of the footprint map. If we replace $S_{pq}^+$ by $\check S_{pq}^+$ and replace $\psi_{pq}^+$ by $\check \psi_{pq}^+$ for all $pq$, then the collection 
    \beqn
    ({\mf K}^\floer)^+:= \Big( (K_{pq}^+)_{p<q}, ({\bm \iota}_{prq}^+)_{p< r< q} \Big)
    \eeqn
    is a K-chart lift of the flow category $(T^\floer)^+$.
\end{enumerate}
\end{prop}

\begin{proof}
The failure of the naively obtained outer-collaring to be a K-chart lift is that the normalization condition (i.e. the section $S_{pq}^{Q, +}: V_{pq}^+ \to Q_{pq}^+$) does not match with the products on the boundary. We need to interpolate between $S_{pq}^{Q, +}$ and the products on the collar parts. Here we only describe the situation near a codimension one stratum $prq$. Let $(x, t)\in V_{pq}^+$ be a point in the collar region $\partial^{prq} V_{pq}\times [-1, 0]$. Geometrically, $x$ is represented by a framed cylinder $(\Sigma, u, \mu, \eta)$ where $\Sigma = \Sigma_{pr}\cup \Sigma_{rq}$ and the map $\mu: \Sigma \to {\mc C}_{0, 2}(d_{pq})$ is represented by a frame 
\beqn
F = (f_0, \ldots, f_{d_{pq}}) \in (H^0( L_u, \Sigma))^{d_{pq}+1}.
\eeqn
Proposition \ref{prop523} tells us that the underlying map $\uds\mu: \Sigma \to \mb{CP}^{d_{pq}}$ can be transformed canonically to a map $\uds\mu^{\rm normal}: \Sigma \to \mb{CP}^{d_{pq}}$ via an invertible complex matirx of the form $I_{d_{pq}+1} + \rho_Q$. Choose a smooth cut-off function $\chi(t)$ which is equal to $1$ for $t$ near $-1$ and euqal to $0$ for $t$ near $0$. Then define 
\beqn
\uds\mu^t = (I_{d_{pq}+1} + \chi(t) \rho_Q ) \uds \mu: \Sigma \to \mb{CP}^{d_{pq}}.
\eeqn
Then $\uds \mu^t$ is in normal position for $t$ near $-1$ and corresponding to a frame
\beqn
F^t = (f_0^t, \ldots, f_{d_{pq}}^t)\in (H^0(L_u, \Sigma))^{d_{pq}+1}
\eeqn
(which is well-defined up to ${\mb C}^*$). For those $t$ near $-1$, the frame $F^t$ can be reconstructed from a frame on the component $\Sigma_{pr}$ and a frame on $\Sigma_{rq}$ and one can interpolate between the normalization using from the $L^2$-inner product on $\Sigma$ and the normalization using the product construction. The situation of the interpolation near deeper strata can be iterated by a routine induction argument.
\end{proof}

\subsubsection{Collar structure}

The outer-collaring provides an analogous collar structure as in the case of derived orbifold liftings. To save notations we would rather view the collar structure as being given by ``interior'' collars than ``exterior'' ones. Therefore, we have obtained a K-chart lift of a flow category over ${\mc P}^\floer$ such that each chart $K_{pq} = (G_{pq}, V_{pq}, E_{pq}, S_{pq})$ has a system of collars
\beq\label{eqn:collar-interior}
\vcenter{ \xymatrix{    E_{pq}|_{\partial^\alpha V_{pq}} \times [0, \epsilon)^{\bF_\alpha} \ar[rr]^-{\widehat \theta_{pq, \alpha}^{\rm collar}} \ar[d] &  &     E_{pq}\ar[d]\\
              \partial^\alpha V_{pq}\times [0, \epsilon)^{\bF_\alpha} \ar[rr]_-{\theta_{pq, \alpha}^{\rm collar}} & &   V_{pq} }},\ \forall \alpha \in \bA_{pq}^\floer
              \eeq
              which are compatible in a sense similar to the requirements of collar structures of a derived orbifold lift stated in Definition \ref{defn:collar}, where one needs to replace the notions for derived orbifold charts by the ones for K-charts, with extra care taken on the equivariance with respect to the Lie group actions. The readers are invited to fill in the details.

\subsection{Constructing scaffolding}\label{subsection59}

The stabilization map constructed previously describes the relation between a product chart and the corresponding stratum of a larger thickened moduli space. However, to inductively construct the FOP perturbations, we need to specify such stabilization maps in all such situations and we require that the stabilization maps are compatible in some natural way. Such a requirement was packaged in terms of the notion ``scaffolding'' for derived orbifold lifts of a flow category and in particular, in the smooth category (see Definition \ref{defn_scaffolding} and Definition \ref{scaffolding_2}). In our construction of the derived orbifold lift of $(T^\floer)^+$, we need a scaffolding-like structure to do smoothing. Hence, we first need such structures in the topological category, also in the equivariant setting but not in the orbifold setting. Below we state what we can construct in this setting; its proof is given in Appendix \ref{appendixc}. Just as \eqref{eqn:collar-interior}, we simplify the notations by pretending that the K-charts are equipped with interior collars. The reader should keep in mind that the following proposition holds only after taking the outer-collaring. In the sequel, the resulting structure from the following proposition will be referred to as a {\bf collared scaffolding}.

\begin{prop}\label{bigprop}
There exist the following objects. 
\begin{enumerate}
    \item For each $pq$ and each integer $d \geq d_{pq}$, a $G_{pq}$-invariant inner product on the bundle $O_{pq}^{(d)} \to V_{pq}$. To state the next class of objects, introduce the following notions. Let $O_{pq,\alpha}\subset O_{pq}|_{V_\alpha}$ be the orthogonal complement of $O_\alpha \to V_\alpha$ and let $Q_{pq, \alpha}\subset Q_{pq}|_{V_\alpha}$ be the orthogonal complement of $Q_\alpha\to V_\alpha$ (with resepct to the trivial inner product on $Q_{pq}$). Denote
    \beqn
    F_{pq, \alpha}:= O_{pq, \alpha}  \oplus  Q_{pq, \alpha}  \subset E_{pq}|_{V_\alpha}
    \eeqn
    and its equivariantization 
    \beqn
    F_{pq,\alpha}^\sim:= G_{pq}\times_{G_\alpha} F_{pq, \alpha} \subset E_{pq}|_{V_\alpha^\sim}.
    \eeqn
    
    \item For each stratum $\alpha$, a stratified topological embedding over the poset $\partial^\alpha \bA^\floer_{pq}$
    \beq\label{eqn:sca-K}
    \theta_{pq,\alpha}: {\rm Stab}_{F_{pq,\alpha}}( V_\alpha) \hookrightarrow \partial^\alpha V_{pq}
    \eeq
    whose restriction to $V_\alpha$ coincides with the embedding $\iota_{pq,\alpha}: V_\alpha \hookrightarrow \partial^\alpha V_{pq}$ such that its equivariantization
    \beqn
    \theta_{pq,\alpha}^\sim: {\rm Stab}_{F_{pq,\alpha}^\sim} (V_\alpha^\sim) \to \partial^\alpha V_{pq}
    \eeqn
    is an open embedding. This embedding induces a projection map
    \beq
    \pi_{pq, \alpha}: \partial^\alpha V_{pq} \to V_\alpha^\sim.
    \eeq

    \item For each $d \geq d_{pq}$, an equivariant bundle isomorphism
    \beq
    \widehat \theta_{pq,\alpha}^{(d),\sim}: \pi_{pq,\alpha}^* \Big( O_{pq}^{(d)}|_{V_\alpha^\sim} \Big) \cong O_{pq}^{(d)}|_{\partial^\alpha V_{pq}}.
    \eeq
    \end{enumerate}
    These objects need to satisfy the following list of conditions. 

\begin{enumerate}

\item[(A1)] For all $pq$ and $d' > d \geq d_{pq}$, the inclusion $O_{pq}^{(d)}\hookrightarrow O_{pq}^{(d')}$ \eqref{obstruction_inclusion} is isometric.

%\item[(A2)] For each stratum $\alpha = pr_1 \cdots r_l q$ of $pq$, the natural decomposition \eqref{obstruction_splitting}
%\beqn
%O_{pq}^{(d)}|_{\partial^\alpha V_{pq}} = O_{pq; pr_1}^{(d)} \oplus \cdots \oplus O_{pq; %r_l q}^{(d)}
%\eeqn
%is orthogonal.

\item[(A2)] For each stratum $\alpha = pr_1 \cdots r_l q$ of $pq$, the bundle embedding
\beq\label{eqn:o-embedding}
\widehat \iota_{pq, \alpha}^{(d),\sim}: O_{pq,\alpha}^{(d)} \hookrightarrow O_{pq}^{(d)}|_{\partial^\alpha V_{pq}}
\eeq
is isometric. By Lemma \ref{lemma62} below, it follows that, for any intermediate stratum $\beta = ps_1 \cdots s_m q$ between $\alpha$ and $pq$, the following is true. Abbreviate $\alpha_0 = pr_1 \cdots s_1$, $\ldots$, $\alpha_m = s_m \cdots r_l q$ as strata of $ps_1$, $\cdots$, $s_m q$ respectively. Then
\beqn
O_{pq, \alpha} \cong O_{pq, \beta} |_{V_\alpha} \oplus \left( O_{ps_1, \alpha_0} \boxplus \cdots \boxplus O_{s_m q, \alpha_m} \right).
\eeqn
(We explain the notations here. The bundles $O_{ps_1, \alpha_0} $ etc. in the second summand above are subbundles of $O_{ps_1}|_{V_{\alpha_0}}$ etc. Their product is a subbundle of $O_{pq}|_{V_\alpha}$.) It implies that the bundle $O_{pq,\alpha}$ is stratified (linearly) by all strata between $\alpha$ and the top stratum, where the stratum corresponding to the intermediate stratum $\beta$ is the subbundle
\beqn
O_{\beta\alpha}:= O_{ps_1, \alpha_0} \boxplus \cdots \boxplus O_{s_m q, \alpha_m}.
\eeqn

Similarly, the bundle $Q_{pq, \alpha}$ is also stratified by subbundles $Q_{\beta\alpha}$ where 
\beqn
Q_{\beta\alpha} = Q_{ps_1, \alpha_0} \boxplus \cdots \boxplus Q_{s_m q, \alpha_m}.
\eeqn
Denote 
\beqn
F_{\beta\alpha}: = O_{\beta\alpha} \oplus Q_{\beta\alpha} \subset F_{pq, \alpha}.
\eeqn
Then there is a splitting
\beq\label{ssss}
F_{pq, \alpha} = F_{pq, \beta}|_{V_\alpha} \oplus F_{\beta\alpha}.
\eeq

\item[(B1)] The stabilization maps are stratum-preserving, i.e., for any intermediate stratum $\beta$ between $\alpha$ and the top stratum, there holds
\beqn
\theta_{pq,\alpha} \Big( {\rm Stab}_{F_{\beta\alpha}} ( V_\alpha) \Big) \subset \partial^\alpha V_\beta,
\eeqn
where $F_{\beta \alpha}$ is viewed as a subbundle of $F_{pq, \alpha}$ under the splitting \eqref{ssss}, and the space $V_{\beta}$ is identified with its image in $\partial^{\beta} V_{pq}$ under $\iota_{pq, \beta}$.

\item[(B2)] The above restriction of the stabilization map coincides with the product of stabilization maps. More precisely this means the following. Abbreviate $\alpha_0 = pr_1 \cdots s_1$, $\ldots$, $\alpha_m = s_m \cdots r_l q$ as strata of $ps_1$, $\cdots$, $s_m q$ respectively. The stabilization maps $\theta_{ps_1, \alpha_0}$, $\cdots$, $\theta_{s_m q, \alpha_m}$ are germs of homeomorphisms 
\beqn
\theta_{ps_1, \alpha_0}: {\rm Stab}_{F_{ps_1, \alpha_0}}  ( V_{\alpha_0}) \to\partial^{\alpha_0} V_{ps_1},\ \cdots,\ \theta_{s_m q, \alpha_m}: {\rm Stab}_{F_{s_m q, \alpha_m}}  ( V_{\alpha_m}) \to \partial^{\alpha_m} V_{s_m q}.
\eeqn
Their product gives a germ of homeomorphisms
\beqn
{\rm Stab}_{F_{ps_1, \alpha_0} \boxplus \cdots \boxplus F_{s_m q, \alpha_m}} (V_\alpha)  \to \partial^{\alpha_0} V_{ps_1}\times \cdots \times \partial^{\alpha_m} V_{s_m q} \hookrightarrow \partial^\alpha V_{pq}.
\eeqn
Notice that the left hand side is just ${\rm Stab}_{F_{\beta\alpha}}(V_\alpha)$. We require that 
\beqn
\theta_{pr_1, \alpha_0}\times \cdots \times \theta_{s_m q, \alpha_m} = \theta_{pq, \alpha}|_{{\rm Stab}_{F_{\beta\alpha}}(V_\alpha)}.
\eeqn
(Compare this conditions with the conditions in Definition \ref{scaffolding_2}).

\item[(C1)] The bundle isomorphisms satisfy the following commutative diagram for all $d < d'$.
\beqn
\xymatrix{   \pi_{pq,\alpha}^* \left( O_{pq}^{(d)}|_{V_\alpha^\sim}\right) \ar[r]^{\widehat \theta_{pq,\alpha}^{(d),\sim}} \ar[d]_{\eqref{obstruction_inclusion}}  &  O_{pq}^{(d)}|_{\partial^\alpha V_{pq}} \ar[d]^{\eqref{obstruction_inclusion}} \\
              \pi_{pq,\alpha}^* \left( O_{pq}^{(d')}|_{V_\alpha^\sim} \right)  \ar[r]_{\widehat \theta_{pq,\alpha}^{(d'),\sim}}   &  O_{pq}^{(d')}|_{\partial^\alpha V_{pq}} }.
\eeqn
Hence they induce an isomorphism 
\beqn
\widehat \theta_{pq, \alpha}^O: \pi_{pq, \alpha}^* \left( O_{pq}^{\infty}|_{V_\alpha^\sim}\right) \cong O_{pq}^\infty|_{\partial^\alpha V_{pq}}. 
\eeqn

%\item[(C2)] $\widehat \theta_{pq, \alpha}^O$ is isometric and for each $i = 0, 1, \ldots, l$, 
%\beqn
%\widehat\theta_{pq, \alpha}^O \left( \pi_{pq, \alpha}^* ( O_{pq; r_i r_{i+1}}^\infty %|_{V_\alpha^\sim}  ) \right) = O_{pq; r_i r_{i+1}}^\infty|_{\partial^\alpha V_{pq}}.
%\eeqn

\item[(C2)] The bundle isomorphism $\widehat \theta_{pq, \alpha}^O$ preserves the stratification. Namely, for any intermediate stratum $\beta$ between $\alpha$ and $pq$, one has 
\beqn
\widehat \theta_{pq, \alpha}^O \left( \pi_{pq,\alpha}^* ( O_\beta|_{V_\alpha^\sim})\right) = O_\beta |_{\partial^\alpha V_{pq}}.
\eeqn
%and for each intermediate stratum $\beta$, 
%\beq\label{nnn}
%\widehat\psi_\alpha \Big( \pi_\alpha^* (O_{pq, \uds{\bm \beta}}^\bot|_{V_{pq, \uds{\bm\alpha}\uds{\bm \beta}}}) \Big) = O_{pq, \uds{\bm\beta}}^\bot|_{V_{pq, \alpha\uds{\bm\beta}}}.
%\eeq

\item[(D)] Let $\zeta_{pq, \alpha} = \zeta_{pq, \beta} + \zeta_{\beta\alpha}$ be a point of $F_{pq, \alpha}$ at a point $x_\alpha \in V_\alpha$ with respect to the splitting \eqref{ssss}. Then the above condition (C3) implies that the bundle isomorphism $\widehat \theta_{pq, \alpha}^O$ identifies $\zeta_{pq, \beta}$ with a vector of $F_{pq,\beta}$ at $x_\beta = \theta_{pq, \alpha}(x_\alpha, \zeta_{\beta\alpha}) \in \partial^\alpha V_\beta$. Then we require that 
\beqn
\theta_{pq, \beta}(x_\beta, \zeta_{pq, \beta} ) = \theta_{pq, \alpha} (x_\alpha, \zeta_{pq, \alpha} ).
\eeqn
This implies that 
\beq
\pi_{pq, \alpha} \circ \pi_{pq, \beta} = \pi_{pq, \alpha}.
\eeq

\item[(E)] The above implies that 
\beqn
\pi_{pq, \alpha}^*( O_{pq}^\infty|_{V_\alpha^\sim} ) = \pi_{pq, \beta}^* \left( \pi_{pq, \alpha}^* (O_{pq}^\infty|_{V_\alpha^\sim} ) |_{\partial^\alpha V_\beta^\sim}\right).
\eeqn
We require that the bundle isomorphism satisfies 
\beq
\widehat \theta_{pq, \alpha}^O = \widehat \theta_{pq, \beta}^O \circ \widehat \theta_{pq, \alpha}^O.
\eeq
More precisely, using notations in condition (D) above, the following diagram commutes.
\beqn
\xymatrix{  O_{pq}^\infty|_{x_\alpha}  \ar[r]^{\widehat \theta_{pq, \alpha}^O} \ar[rd]_{\widehat\theta_{pq, \alpha}^O }     &  O_{pq}^\infty|_{ x_\beta } \ar[d]^{\widehat\theta_{pq, \beta}^O } \\
 & O_{pq}^\infty|_{\theta_{pq, \alpha}(x_\alpha, \zeta_{pq, \alpha})} }
\eeqn

\item[(F)] Given any intermediate stratum $\beta = ps_1 \cdots s_m q$, using the notations used in (B2), one has 
\beqn
\widehat \theta_{pq, \alpha}^O|_{\partial^\alpha V_\beta} = \widehat \theta_{ps_1, \alpha_0}^O \times \cdots \times \widehat \theta_{s_m q, \alpha_m}^O.
\eeqn

\item[(G)] The metric on $O_{pq}^{(d)}$, the stabilization map $\theta_{pq, \alpha}$, and the bundle isomorphism $\widehat \theta_{pq, \alpha}^O$ respect collars. More precisely, 
\begin{enumerate}
    \item The bundle isomorphism associated to the collar structure $\widehat \theta_{pq, \alpha}^{\rm collar}$ is isometric. It follows that following diagram commutes
    \beqn
    \xymatrix{ F_{pq, \beta}|_{\partial^\alpha V_\beta} \times [0, \epsilon)^{\bF_\alpha \setminus \bF_\beta} \ar[rr]^-{\widehat \theta_{pq, \alpha}^{\rm collar}} \ar[d]  & & F_{pq, \beta} \ar[d] \\
               E_{pq}|_{\partial^\alpha V_\beta} \times [0, \epsilon)^{\bF_\alpha \setminus \bF_\beta}  \ar[rr]_-{\widehat \theta_{pq, \alpha}^{\rm collar}}  &    & E_{pq}|_{V_\beta} }
    \eeqn
    
    \item The following diagram commutes.
    \beqn
    \xymatrix{   \partial^\alpha \left(  {\rm Stab}_{F_{pq, \beta}}(V_\beta) \right) \times [0, \epsilon)^{\bF_\alpha \setminus \bF_\beta} \ar[rr] \ar[d]   & &   {\rm Stab}_{F_{pq, \beta}} (V_\beta) \ar[d] \\
                  \partial^\alpha  V_{pq}  \times [0, \epsilon)^{\bF_\alpha \setminus \bF_\beta} \ar[rr] & &   \partial^\beta V_{pq} }.
    \eeqn
    
    \item The following diagram commutes.
    \beqn
    \xymatrix{  ( \pi_{pq, \beta}^* O_{pq}^\infty|_{V_\beta^\sim} )|_{\partial^\alpha V_{pq}} \times [0, \epsilon)^{\bF_\alpha \setminus \bF_\beta} \ar[rr] \ar[d] & & \pi_{pq, \beta}^* O_{pq}^\infty|_{V_\beta^\sim} \ar[d]  \\
          O_{pq}^\infty|_{\partial^\alpha V_{pq}} \times [0, \epsilon)^{\bF_\alpha \setminus \bF_\beta}    \ar[rr]     & &  O_{pq}^\infty|_{\partial^\beta V_{pq}} } 
    \eeqn
\end{enumerate}

\end{enumerate}
\end{prop}

\begin{proof}[Sketch of proof]
See Appendix \ref{appendixc}.
\end{proof}

\begin{lemma}\label{lemma62}
Suppose the bundles $O_{pq}^{(d)}\to V_{pq}$ are equipped with $G_{pq}$-invariant inner products which satisfy conditions (A1) and (A3) of Proposition \ref{bigprop}. Then the following is true. Consider a pair of strata $\alpha, \beta \in \bA_{pq}^\floer$ where $\alpha = pr_1 \cdots r_l q$ and $\beta = ps_1 \cdots s_m q$. Denote $\alpha_0 = pr_1 \cdots s_1$, $\ldots$, $\alpha_m = s_m \cdots r_l q$. Consider the bundles 
\beqn
O_{ps_1, \alpha_0} \to V_{\alpha_0},\ \cdots,\ O_{s_m q, \alpha_m} \to V_{\alpha_m}
\eeqn
whose product is a subbundle 
\beqn
O_{\beta\alpha} \subset O_{pq}|_{V_\alpha}.
\eeqn
Then as subbundles of $O_{pq}|_{V_\alpha}$, one has
\beqn
O_{pq, \alpha} = O_{pq, \beta} |_{V_\alpha} \oplus O_{\beta\alpha}.
\eeqn
\end{lemma}

\begin{proof}
Linear algebra.
\end{proof}

\section{Smoothing}\label{sec-6}

In this section, we use equivariant (relative) smoothing theory to equip the K-chart presentation of various moduli spaces constructed in Section \ref{sec-5} with smooth structures after stabilization. The proof is based on induction, which is quite similar to the construction of straightening and coherent FOP perturbations discussed in Section \ref{sec-3} and Section \ref{sec-4}. After taking the associated quotient orbifold of the smooth K-chart presentation, we obtain a derived orbifold lift of the Hamiltonian Floer flow category $T^\floer$. Moreover, we explain how to construct a normal complex structure on the resulting derived orbifold lift.

\subsection{Preliminaries}
We list out some basic notions relevant for the smoothing theory in this subsection. Then we state the relative smoothing statement Theorem \ref{relative_smoothing} which will proved in Appendix \ref{appendixb}. Following that, some elementary technical results in the smoothing construction will be proved.

\subsubsection{Microbundles}
We start by discussing some notions and facts about microbundles and their roles in smoothing theory.

\begin{defn}
Let $X$ be a topological space. 
\begin{enumerate}

\item A {\bf microbundle} (of rank $n$) over $X$ is a diagram 
\beqn
\xymatrix{ X \ar[r]^s & {\ms E} \ar[r]^p & X }
\eeqn
where ${\ms E}$ is a topological space, $p$ and $s$ are continuous maps satisfying
\begin{enumerate}
    \item $p \circ s = {\rm Id}_X$.
    
    \item for each $x \in X$ there exist a neighborhood $U_x \subset X$ of $x$, a neighborhood ${\ms V}_x\subset {\ms E}$ of $s(x)$, and a homeomorphism
    \beqn
    h_x: U_x \times {\mb R}^n \to {\ms V}_x
    \eeqn
    for which $p \circ h_x = {\rm pr}_{U_x}$ and $h_x|_{U_x \times \{0\}} = s$. Here ${\rm pr}_{U_x}$ is the natural projection onto the factor $U_x$. 
\end{enumerate}
Often we use a single symbol ${\ms E}$ to denote a microbundle.

\item Given two microbundles over $X$ (not necessarily of the same rank)
\begin{align*}
&\ \xymatrix{ X \ar[r]^{s_1} & {\ms E}_1 \ar[r]^{p_1} & X},\ &\ \xymatrix{ X \ar[r]^{s_2} & {\ms E}_2 \ar[r]^{p_2} & X}
\end{align*}
a {\bf morphism} from ${\ms E}_1$ to ${\ms E}_2$ is a germ of continuous maps $\phi: U_1 \to E_2$ defined over a neighborhood ${\ms U}_1 \subset {\ms E}_1$ of $s_1(X)$ which commutes with the maps $s_1, s_2$ and $p_1, p_2$. Two morphisms $\phi_1, \phi_2$ from ${\ms E}_1$ to ${\ms E}_2$, defined over ${\ms U}_1\subset {\ms E}_1$ and ${\ms U}_2 \subset {\ms E}_2$ respectively, are called {\bf equivalent} if there exists a neighborhood of $s_1(X)$ in ${\ms E}_1$ contained in ${\ms U}_1 \cap {\ms U}_2$ on which $\phi_1 = \phi_2$.

\item Using the notations as above, a morphism between microbundles
\begin{equation*}
\begin{tikzcd}
                                       & {\ms E}_1 \arrow[rd, "p_1"] \arrow[dd, "\phi"'] &   \\
X \arrow[ru, "s_1"] \arrow[rd, "s_2"'] &                                           & X \\
                                       & {\ms E}_2 \arrow[ru, "p_2"']                    &  
\end{tikzcd}
\end{equation*}
is said to be an {\bf isomorphism} if $\phi$ maps ${\ms U}_1$ homeomorphically onto an open subset ${\ms U}_2 \subset {\ms E}_2$ containing $s_2(X)$.

\item An {\bf isotopy} between two isomorphisms $\phi_0, \phi_1: {\ms E} \to {\ms E}'$ of microbundles over $X$ is a family of microbundle isomorphisms $\tilde \phi_t: {\ms E} \to {\ms E}'$, $t \in [0, 1]$ which defines a continuous map from ${\ms E} \times [0, 1]\to {\ms E}'$, such that $\tilde \phi_0$ is equivalent to $\phi_0$ and $\tilde \phi_1$ is equivalent to $\phi_1$. 
\end{enumerate}
\end{defn}

\begin{example}
\begin{enumerate}
    \item If $p: E \to X$ is an ordinary real vector bundle, then together with the zero section $E$ defines a microbundle over $X$. We denote this microbundle by $E_\mu$.
    
    \item Let $X$ be a topological manifold. The {\bf tangent microbundle} is the diagram 
    \begin{align}
        \xymatrix{ X \ar[r]^-{\Delta} & X \times X \ar[r]^-{p_1} & X }
    \end{align}
    where $p_1$ is the projection onto the first factor. We denote the tangent microbundle of $X$ by $T_\mu X$.
\end{enumerate}
\end{example}

If a topological manifold admits a smooth structure, then its tangent microbundle is isomorphic to the associated tangent bundle. In general, a microbundle ${\ms E}$ over a space $X$ may or may not come from a vector bundle. We recall the following definition. 

\begin{defn}
Let ${\ms E} \to X$ be a microbundle. A {\bf vector bundle lift} (or {\bf vector bundle reduction}) of ${\ms E}$ is a vector bundle $E \to X$ together with a microbundle isomorphism $E_\mu \to {\ms E}$. 
\end{defn}

Microbundles have many properties similar to vector bundles. One can construct microbundles by taking certain operations.

\begin{defn}[Direct sums of microbundles]
Let $\xymatrix{ X \ar[r]^{s_i} & {\ms E}_i \ar[r]^{p_i} & X}$, $i = 1, 2$ be two microbundles and let $\Delta: X \to X \times X$ be the diagonal inclusion map. Let 
\beqn
{\ms E}_{12}:= (p_1 \times p_2)^{-1}( \Delta(X)) \subset {\ms E}_1 \times {\ms E}_2.
\eeqn
Then the {\bf direct sum} of ${\ms E}_1$ and ${\ms E}_2$ is the microbundle 
\beqn
\xymatrix{ X \ar[rr]^-{(s_1 \times s_2) \circ \Delta} & & {\ms E}_{12} \ar[rr]^-{p_1 \times p_2} & & \Delta(X) \cong X}.
\eeqn
\end{defn}

On the other hand, the notions of pullbacks and restrictions of microbundles can be defined in the same way as for the case of vector bundles. The correspondence $E \mapsto E_\mu$ commutes with these operations.

For the purpose of discussing the uniqueness of stable (equivariant) smoothings later, we need to compare different vector bundle lifts. 

\begin{defn}(\cite{Lashof_1979})
Let $\phi_i: (E_i)_\mu \to {\ms E}$, $i = 0, 1$ be two vector bundle lifts of a microbundle ${\ms E}$ over $X$. We say that $\phi_0$ and $\phi_1$ are 
\begin{enumerate}

\item {\bf equivalent} if there exists a vector bundle isomorphism $\psi: E_0 \to E_1$ (which induces a microbundle isomorphism $\psi: (E_0)_\mu \to (E_1)_\mu$ such that $\phi_0$ is equivalent to $\phi_1 \circ \psi$,

\item {\bf isotopic} if there exist a vector bundle $E \to X \times [0,1]$ and a continuous family of vector bundle lifts $\phi_{t}: (E |_{X \times \{t\}})_{\mu} \to {\ms E}$ which restricts to $\phi_0$ and $\phi_1$ when $t=0, 1$.
\end{enumerate}
\end{defn}

\begin{example}
Suppose $M$ is a smooth manifold. Then there is a canonical isotopy class of vector bundle reductions of $T_\mu M$. Choose a Riemannian metric $g$ on $M$ and let $\exp$ be the associated exponential map. Then there is an open neighborhood $U(TM)$ of the zero section of $TM$ such that the exponential map is defined over $U(TM)$. Then define 
\beqn
\varphi_g: U(TM) \to M \times M,\ (x, \xi) \mapsto (x, \exp_x \xi)
\eeqn
which is a vector bundle reduction of $T_\mu M$. Moreover, the vector bundle reduction defined by any two Riemannian metrics are isotopic as microbundle isomorphisms. 
\end{example}

Now we discuss the equivariant situation. 
\begin{defn}
Let $G$ be a topological group and $X$ be a $G$-space. A {\bf $G$-microbundle} is a microbundle $\xymatrix{ X \ar[r]^s & {\ms E} \ar[r]^p & X}$ where ${\ms E}$ is a $G$-space such that $p$ and $s$ are both equivariant. 
\end{defn}

All previous notions for non-equivariant microbundles can be easily extended to the equivariant situation. We just emphasize that if $X$ is a $G$-topological manifold, then its tangent microbundle $T_\mu X$ is a $G$-microbundle; if $G$ is a compact Lie group and $X$ is a smooth $G$-manifold, then its tangent microbundle $T_\mu X$ admits a $G$-equivariant vector bundle lift (by its smooth tangent bundle); a concrete lift can be given via the exponential map associated to a $G$-invariant Riemannian metric; moreover, the equivariant isotopy class of such a lift is unique.

\subsubsection{Stable equivariant smoothings}

Let $M$ be a topological manifold (without boundary). A {\bf smoothing} of $M$ is by definition a pair $(N, \phi)$ where $N$ is a smooth manifold and $\phi: N \to M$ is a homeomorphism. A smoothing is also equivalent to a smooth structure $\alpha$ on $M$, i.e., a maximal atlas of $C^\infty$ compatible charts. Two smooth structures $\alpha_0$, $\alpha_1$ on $M$ are called {\bf isotopic} if there is an isotopy (through homeomorphisms) between ${\rm Id}_M$ and a diffeomorphism $\phi: M_{\alpha_0} \to M_{\alpha_1}$, where $M_{\alpha_i}$ is the smooth manifold given by equipping $M$ the smooth structure $\alpha_i$. 

If $G$ is a compact Lie group acting on $M$, then a {\bf $G$-smoothing} is a smooth structure $\alpha$ on $M$ such that the $G$-action is smooth. Two $G$-smoothings $\alpha_0$ and $\alpha_1$ are $G$-isotopic if there is an isotopy through $G$-equivariant homeomorphisms between ${\rm Id}_M$ and a diffeomorphism from $M_{\alpha_0}$ to $M_{\alpha_1}$.

For the non-equivariant case, one can see that a necessary condition for a topological manifold being smoothable is that its tangent microbundle admits a vector bundle lift. In fact, this is almost also a sufficient condition. Once the tangent microbundle of a topological manifold $M$ admits a vector bundle lift, then $M$ is {\bf stably smoothable}, i.e., there is a smoothing $(M \times {\mb R}^k)_{\alpha}$ on $M\times {\mb R}^k$ for some $k \geq 0$ \cite[Section 5]{Milnor_micro_1}. Two stable smoothings $(M\times {\mb R}^{k_i})_{\alpha_i}$, $i = 0, 1$, are {\bf stably isotopic} if there exist $k_0', k_1' \geq 0$ with $k_0 + k_0' = k = k_1 + k_1'$ and an isotopy between $(M\times {\mb R}^{k_0})_{\alpha_0} \times {\mb R}^{k_0'}$ and $(M\times {\mb R}^{k_1})_{\alpha_1} \times {\mb R}^{k_1'}$ as smoothings on $M \times {\mb R}^k$, where the Euclidean spaces are endowed with the standard smooth structure. The work of Lashof \cite{Lashof_1979} basically shows that this correspondence extends to the $G$-equivariant case. Let us introduce the notion of stable $G$-smoothings. 

\begin{defn}\cite{Lashof_1979}
Let $G$ be a compact Lie group and $M$ be a topological $G$-manifold. 

\begin{enumerate}
    \item A {\bf stable $G$-smoothing} of $M$ is a $G$-smoothing on the product $M \times {\bm R}$ where ${\bm R}$ is a finite-dimensional orthogonal representation of $G$. 
    
    \item Two stable $G$-smoothings, $(M \times {\bm R}_i)_{\alpha_i}$, $i = 0, 1$, of $M$, are {\bf stably $G$-isotopic}, if there exist orthogonal representations ${\bm R}_0'$, ${\bm R}_1'$ of $G$ such that ${\bm R}_0 \oplus {\bm R}_0'\cong {\bm R} \cong {\bm R}_1 \oplus {\bm R}_1'$ as $G$-representations and $(M \times {\bm R}_0)_{\alpha_0}\times {\bm R}_0'$ is isotopic to $(M \times {\bm R}_1)_{\alpha_1}\times {\bm R}_1'$ as $G$-smoothings on $M\times {\bm R}$.
    
    \item A {\bf stable $G$-vector bundle lift} of $T_\mu M$ is a $G$-vector bundle lift of $T_\mu M \oplus R$, where $R \to M$ is the trivial bundle $M \times {\bm R}$ with ${\bm R}$ being an orthogonal representation of $G$.
    
    \item Two stable $G$-vector bundle lifts, $\phi_i: (E_i)_\mu \cong T_\mu M \oplus R_i$, $i = 0, 1$, are {\bf stably $G$-isotopic}, if there are orthogonal $G$-representations ${\bm R}_0'$, ${\bm R}_1'$, with ${\bm R}_0 \oplus {\bm R}_0' \cong {\bm R} \cong {\bm R}_1 \oplus {\bm R}_1'$ such that the induced $G$-vector bundle lifts
    \beqn
    \phi_i \oplus {\rm Id}_{R_i'}: (E_i \oplus R_i')_\mu \cong T_\mu M \oplus R_i \oplus R_i' \cong T_\mu M \oplus R,\ i = 0, 1
    \eeqn
    are $G$-isotopic.
    \end{enumerate}
\end{defn}

Any stable $G$-smoothing $(M \times {\bm R})_\alpha$ of $M$ induces a stable $G$-vector bundle lift 
\beqn
T (M \times {\bm R})_\alpha \to T_\mu M \oplus R.
\eeqn
However it is not obvious that the stable $G$-isotopy class of the lift only depends on the stable $G$-isotopy class of the smoothing. By using an intermediate notion of stable sliced concordance among stable $G$-smoothings, Lashof \cite{Lashof_1979} showed that there is indeed such a correspondence sending stable $G$-isotopy classes of stable $G$-smoothings to stable $G$-isotopy classes of stable $G$-vector bundle reductions of $T_\mu M$. In fact, this correspondence is bijective. 

\begin{thm}\cite{Lashof_1979}\label{lashof_theorem} Let $G$ be a compact Lie group and $M$ be a topological $G$-manifold which only has finitely many orbit types. Suppose $E \to M$ is a $G$-equivariant vector bundle and $\varphi: E_\mu \to T_\mu M$ is a $G$-equivariant vector bundle reduction of $T_\mu M$, then there exists an orthogonal $G$-representation ${\bm R}$ and a $G$-smoothing $\alpha$ on $M \times {\bm R}$ satisfying 
\begin{enumerate}
    \item There is a $G$-equivariant vector bundle isomorphism $\rho: p_M^* E \oplus R \cong T(M \times {\bm R})_\alpha$ where $p_M: M \times {\bm R} \to M$ is the natural projection.
    
    \item Denote by $\rho_{\mu}$ the induced map between microbundles from $\rho$. Then the vector bundle reduction $\exp_{M \times R} \circ \rho_\mu: p_M^* E_\mu \oplus R \to T_\mu ( M \times {\bm R})$ is isotopic to the stabilization $p_M^* \varphi \oplus {\rm Id}_R$.\footnote{This claim is not explicitly stated in \cite{Lashof_1979} but can be observed from Lashof's construction.} 
\end{enumerate}
Moreover, the correspondence $\varphi \mapsto \alpha$ induces a bijection between stable isotopy classes of stable $G$-smoothings of $M$ and stable isotopy classes of stable $G$-vector bundle reductions of $T_\mu M$.
\end{thm}

Our construction relies on the following relative version of Lashof's theorem. 

\begin{thm}\label{relative_smoothing}
Let $G$, $M$, $E$, and $\varphi$ be as in Theorem \ref{lashof_theorem}. Let $C \subset M$ be a $G$-invariant closed set and $U \subset M$ be a $G$-invariant open neighborhood of $C$. Suppose 
\begin{enumerate}
    \item $U$ is equipped with a $G$-smoothing $\alpha_0$ and $E|_U$ is equipped with a smooth $G$-equivariant vector bundle structure.
    
    \item $\varphi|_U: (E|_U)_\mu \to T_\mu U$ is a smooth microbundle reduction and is isotopic to the microbundle reduction $TU_{\alpha_0} \to T_\mu U$.
    \end{enumerate}
Then there exists an orthogonal $G$-representation ${\bm R}$ and a $G$-smoothing $\alpha$ on $M \times {\bm R}$ such that over a $G$-invariant open neighborhood $U' \subset U$ of $C$, $(U'\times {\bm R})_\alpha$ is diffeomorphic to the product $(U')_{\alpha_0} \times {\bm R}$.
\end{thm}

\begin{proof}
See Appendix \ref{appendixb}.
\end{proof}

\subsubsection{Topological submersions}

Now we start to prepare for smoothing the K-chart lift of the Hamiltonian Floer flow category following a generalization of the strategy used in \cite{AMS}. First, we discuss topological submersions which are relevant for constructing lifts of tangent microbundles.

\begin{defn}\label{product_neighborhood}
Let $M$ be an $\bA$-space and $B$ be a smooth $\bA'$-manifold. Let $\pi: M \to B$ be a continuous map covering a map between the posets $\bA \to \bA'$.
\begin{enumerate}
    \item Let $p \in M$ and $b = \pi(p) \in B$. A {\bf product neighborhood} of $p$ is a pair $(W, \iota)$ where $W$ is an open neighborhood of $p$ (denoting $W|_b = \pi^{-1}(b) \cap W$) and $\iota: W \to W|_b \times \pi(W)$ is an homeomorphism satisfying 
    \begin{itemize}
        \item $\pi(W)$ is an open neighborhood of $b$;
    
        \item $\pi \circ \iota^{-1}: W|_b \times \pi(W) \to \pi(W)$ is the projection to $\pi(W)$;
        
        \item $\iota|_{W|_b}: W|_b \to W|_b \times \{b\}$ is the identity map.
        
    \end{itemize}
    
    \item The map $\pi$ is called a {\bf topological submersion} if every point $p \in M$ admits a product neighborhood and each fiber is a topological manifold (of a fixed dimension) without boundary. In this case, the {\bf vertical tangent microbundle} $T_\mu^{\rm vt}M$ is the microbundle 
    \beqn
    \xymatrix{ M \ar[r]^-{\Delta} & M \times_B M \ar[r]^-{p_1} & M}
    \eeqn
    where $p_1 : M \times_B M \to M$ is the projection to the first factor, viewing the fiber product $M \times_B M$ as a subset of $M \times M$.
\end{enumerate}
\end{defn}

\begin{rem}
Our primary examples of topological submersions come from the forgetful map $\pi_{pq}: V_{pq} \to B_{pq}$ in \eqref{eqn:pi-V-B}. It is important to notice that the fibers are indeed topological manifolds without boundary, because the complex structure on the domain of the element $(\Sigma, u, F, \eta_1, \dots, \eta_d)$ has been fixed.
\end{rem}

Now we include group actions. Let $G$ be a compact Lie group. Recall that a $G$-action on an $\bA$-space is a continuous $G$-action which preserve each stratum. We assume that $G$ acts continuously on $M$ and smoothly on $B$, and that $\pi: M \to B$ is $G$-equivariant. Then for each $g\in G$, $p\in M$, and a product neighborhood $(W, \iota)$ of $p$, $(g(W), g_* \iota)$ is a product neighborhood of $g(p)$, where 
\beqn
g_* \iota(w) = g( \iota( g^{-1}(w))).
\eeqn
Here $g: W|_b \times \pi(W) \to W|_{g(b)} \times \pi(g(W))$ is the diagonal action.

\begin{defn}\cite[Definition 4.19]{AMS}
For each $p\in M$, denote by $G_p \subset G$ the stabilizer of $p$. A $G_p$-invariant product neighborhood of $p$ is a product neighborhood $(W, \iota)$ of $p$ which is $G_p$-invariant, i.e., $(g(W), g_* \iota)= (W, \iota)$ for any $g \in G_p$.

We say that the $G$-equivariant map $\pi: M \to B$ is a {\bf $G$-equivariant topological submersion} if every point $p \in M$ admits a $G_p$-invariant product neighborhood. We say that the $G$-action is {\bf fiberwise locally linear} if for each $b\in B$, the action of $G_b \subset G$ on $M_b$ is locally linear.
\end{defn}

\begin{rem}
Any $G$-equivariant vector bundle $p: E \to B$ is a fiberwise locally linear equivariant topological submersion.
\end{rem}

Now we recall how to construct smoothings of the total space of a topological submersion with a fiberwise smooth structure.

\begin{defn}\cite[Definition 4.27]{AMS}
Let $G$ be a compact Lie group. Let $M$ be a topological manifold with a $G$-action. Let $B$ be a smooth $G$-manifold and $\pi: M \to B$ be a $G$-equivariant topological submersion. 

\begin{enumerate}

\item A {\bf fiberwise smooth structure} on $M$ (with respect to the map $\pi$) is a choice of smooth structures (by definition a maximal $C^\infty$-compatible atlas) on all fibers of $\pi$.

\item Given a fiberwise smooth structure on $M$, two product neighborhoods $\iota_i: W_i \to W_i|_{b_i} \times \pi(W_i)$, $i = 1, 2$ are {\bf $C_{\rm loc}^1$-compatible} if for each $p \in W_1 \cap W_2$ there exists a product neighborhood $\iota: W \to W|_b \times \pi(W)$, $b = \pi(p)$ satisfying the following conditions. For each $v \in \pi(W)$, the definition of product neighborhoods (see Definition \ref{product_neighborhood}) implies that the map 
\beqn
\iota_v: W|_v \to \{v\}\times W|_b
\eeqn
is a homeomorphism. Consider the family of maps
\beqn
\eta_v: W|_b \to W_i|_{b_i},\ w \mapsto \Pi_i \big(  \iota_i (\iota_v^{-1}(w))\big)
\eeqn
where $\Pi_i: W_i|_{b_i}\times \pi(W_i) \to W_i|_{b_i}$ is the projection. Then for each $v$ and $i = 1, 2$, $\eta_v$ is a smooth map from $W|_b$ to $W_i|_{b_i}$ and varies with $v$ continuously with respect to the $C_{\rm loc}^1$-topology.

\item A {\bf fiberwise smooth $C_{\rm loc}^1$ $G$-bundle} is a $G$-equivariant topological submersion $\pi: M \to B$ together with a collection $(\iota_i)_{i \in I}$ of $C_{\rm loc}^1$-compatible $G_p$-invariant product neighborhoods around a collection of points $(p_i)_{i\in I}$ whose domains cover $M$.

\end{enumerate}

\end{defn}

\begin{rem}
Despite our previous discussions work under the setting that $M$ is an $\bA$-manifold, the above notions is defined for $M$ being a manifold without boundary. In practice, when we perform smoothing inductively, we apply Theorem \ref{relative_smoothing} to the top stratum of a topological $\bA$-manifold $M$. Therefore, it suffices to set up certain parts of the theory for manifolds without boundary or corners.
\end{rem}

\subsubsection{Fiberwise submersions}

A fiberwise smooth $C_{\rm loc}^1$-bundle $\pi: M \to B$ has a well-defined {\bf vertical tangent bundle}, denoted by $T^{\rm vt} M$, by patching up the tangent bundles of the fibers, and an isotopy class of $G$-vector bundle lift
\beqn
T^{\rm vt}M \to T_\mu^{\rm vt}M
\eeqn
by using a $G$-invariant fiberwise Riemannian metric (see \cite[Lemma 4.29]{AMS}).

Now suppose $\pi: M \to B$ is a fiberwise smooth $C_{\rm loc}^1$ $G$-bundle. There is a well-defined $G$-microbundle
\beqn
\pi^* T_\mu B \oplus T^{\rm vt}_\mu M,
\eeqn
a well-defined $G$-vector bundle
\beqn
\pi^* TB \oplus T^{\rm vt} M,
\eeqn
and a well-defined $G$-isotopy class of vector bundle lifts
\beq\label{eqn:mu-split}
(\pi^* TB \oplus T^{\rm vt}M)_\mu \to \pi^* T_\mu B \oplus T^{\rm vt}_\mu M
\eeq
constructed using the exponential map after further choosing a Riemannian metric on $B$. Note that both the space of fiberwise Riemannian metrics on $M$ varying continuously over $B$, and the space of Riemannian metrics are contractible. Therefore, all these choices induce the same isotopy class of the vector bundle reduction for the direct sum $\pi^* T_\mu B \oplus T^{\rm vt}_\mu M$.
To qualify the condition for stable $G$-smoothing (see Theorem \ref{lashof_theorem}), one only needs to identify the right hand side with the tangent microbundle of the total space $M$. However this identification is not canonical even up to isotopy of microbundle isomorphisms. This is what we need to construct during the smoothing process.

\begin{rem}
Different data for constructing the splitting \eqref{eqn:mu-split} from \cite[Lemma 4.24, Proposition 4.26]{AMS} can in fact be interpolated with each other, so the splittings \eqref{eqn:mu-split} arising from this way should be isotopic to each other. Keeping track of these choices are important for showing that the ``invariants" constructed in this paper are indeed independent of all the auxiliary choices, but we do not need such invariance results for our applications.
\end{rem}

\begin{defn}\cite[Definition 4.22]{AMS}\label{splitting_defn}
Suppose $\pi: M \to B$ is a $G$-equivariant topological submersion. Let $W \subset M$ be a $G$-invariant open subset. A {\bf $G$-equivariant fiberwise submersion along $W$} is a continuous map $\phi: \widetilde W \to M$ where $\widetilde W \subset W \times W$ is an open neighborhood of the diagonal satisfying the following properties.
\begin{enumerate}
    \item $\phi_{q, b}:= \phi|_{\widetilde W \cap ( \{q\} \times M_b)}$ is a homeomorphism onto an open subset of $M_{\pi(q)}$ for all $b\in \pi(W)$ and $q \in W$.
    
    \item $\phi_{q, \pi(q)}$ sends each point $(q, q') \in \widetilde W \cap ( \{q\}\times M_{\pi(q)})$ to $q'$.
    
    \item $\widetilde W$ is a $G$-invariant set with respect to the diagonal $G$-action on $W \times W$ and $\phi$ is a $G$-equivariant map.
\end{enumerate}
\end{defn}

\begin{rem}
Suppose $M$ is smooth and $\pi: M \to B$ is a smooth submersion. Then one can construct a fiberwise submersion using a $G$-invariant Riemannian metric and the exponential map in the horizontal direction.
\end{rem}

\begin{rem}
A $G$-equivariant fiberwise submersion along $W = M$ provides a $G$-microbundle morphism 
\beqn
\tau: T_\mu M \to T_\mu^{\rm vt} M
\eeqn
which is defined as 
\beqn
M \times M \supseteq \widetilde M \ni (p, p') \mapsto (p, \phi(p, p'))
\eeqn
which is the identity morphism on the sub-microbundle $T_\mu^{\rm vt} M \subset T_\mu M$. On the other hand, there is a canonical $G$-microbundle morphism 
\beqn
P: T_\mu M \to \pi^* T_\mu B,\ (p, p') \mapsto (p, (\pi(p), \pi(p'))).
\eeqn
Hence the $G$-equivariant fiberwise submersion induces an isomorphism of microbundles 
\beqn
P \oplus \tau: T_\mu M \to \pi^* T_\mu B \oplus T^{\rm vt}_\mu M.
\eeqn
If $\pi$ is a fiberwise smooth $C_{\rm loc}^1$ $G$-bundle, then a $G$-equivariant fiberwise submersion along $M$ induces a well-defined $G$-isotopy class of vector bundle lift 
\beqn
(\pi^* TB \oplus T^{\rm vt} M)_\mu \cong T_\mu M.
\eeqn
\end{rem}

\begin{lemma}\label{splitting_lemma}
Let $C \subset M$ be a $G$-invariant closed set and $U \subset M$ be a $G$-invariant open neighborhood of $C$. Let $D \subset M$ be another $G$-invariant closed set. Suppose $\phi: \widetilde U \to M$ is a $G$-equivariant fiberwise submersion along $U$. Then there exists a $G$-invariant open neighborhood $W$ of $C \cup D$ and a $G$-equivariant fiberwise submersion $\psi: \widetilde W \to M$ along $W$ which coincides with $\phi$ in a small neighborhood of the diagonal of $C \times C$ in $U \times U$.
\end{lemma}

\begin{proof}
This is a restatement of the extension lemma \cite[Lemma 4.24]{AMS} and a relative version of \cite[Proposition 4.25]{AMS}.
\end{proof}

\subsection{Pre-smoothing operations and stable complex structures}

This subsection can be viewed as a continuation of Section \ref{sec-5}, in which we explore the more refined information of the K-charts $K_{pq} = (V_{pq}, E_{pq}, G_{pq}, S_{pq})$. These properties will be used in the construction of smoothing. Furthermore, we also present materials relevant for building the normal complex structure.

\subsubsection{Topological submersions}

We have the following analogue of \cite[Corollary 6.28, 6.29]{AMS}.

\begin{prop}\label{prop:c1loc}
The natural forgetful map 
\beqn
\pi_{pq}: V_{pq} \to B_{d_{pq}}
\eeqn
is a $G_{pq}$-equivariant topological submersion and has the structure of a $C_{\rm loc}^1$ fiberwise smooth $G_{pq}$-bundle.
\end{prop}

\begin{proof}
This statement essentially a reformulation of the standard gluing construction in Floer theory which is presented for instance, in \cite[Appendix C]{pardon-VFC}. The slight difference is, here we are in a Morse--Bott situation because the asymptotic operator of elements in the moduli space ${\mc M}_{J_{\tilde \Psi}, H}(k_1, \ldots, k_d)$ has nontrivial kernel along the ${\mc C}$ and ${\mc E}_i$ direction. However, this does not introduce any problem to the argument.

As a first step, let us show that $\pi_{pq}$ is a $G_{pq}$-equivariant topological submersion. Consider a point $x \in V_{pq}$ represented by $(u, \Sigma, F, 0)$. Using the sheared almost complex structure and Gromov's graph trick, we can view $(u, \Sigma, F, \eta_1, \dots, \eta_d)$ as a pseudo-holomorphic stable cylinder contained in the $0$-section of ${\mc E}_d \to M \times {\mc C}$. Recall that the framing $F$ defines a holomorphic map $\iota_F: \Sigma \to \mb{CP}^d$ of degree $d$. Let us choose $d' = (d+2)d$ generic hyperplanes $H_1, \ldots, H_{d'} \subset \mb{CP}^d$ which intersect the image of $\uds{\mu}$ transversely at points different from nodes and markings, and choose
\beqn
z_i \in v(\Sigma) \cap H_i,\ i = 1, \ldots, d'.
\eeqn
Then the map $[v] \mapsto (z_1, \ldots, z_{d'})$ is a local diffeomorphism from ${\mc F}{}_{0,2}^{\mb R}(d)$ to an open subset $j: U \subset \ov{\mc M}{}_{0,2 +d'}^{\mb R}$. These intersection points play the role as local stabilizing divisors. Namely, we choose the preimage of the divisors $H_1, \dots, H_{d'}$ under the projection map ${\mc E}_d \to M \times {\mc C} \to {\mc C} \to {\mb CP}^d$ to be the local stabilizing divisors. By keeping track of the positions of the marked points obtained from intersecting the framing map with these divisors, we obtain map
$$
\mu_{H}: V_{pq} \to \ov{\mc M}{}_{0,2 +d'}^{\mb R}
$$
defined over a neighborhood of $x \in V_{pq}$. 

Let $K_{\rm map}$ be the kernel of the linearized Cauchy--Riemann operator at $(u, \iota_{F}, 0)$ (without deforming the underlying map to $\mb{CP}^d$). The gluing map
$$
g: U \times K \to V_{pq}
$$
defined by \cite[Equation (C.10.3)]{pardon-VFC} fit into the following diagram:
\beqn
    \vcenter{ \xymatrix{ U \times K \ar[r]^g \ar[d]_{{\rm pr}_U} & V_{pq}   \ar[d]^{\mu_H}\\
                U \ar[r]_j & \ov{\mc M}{}_{0,2 +d'}^{\mb R} }  }.
\eeqn
According to \cite[Section C.12]{pardon-VFC}, $g$ is a homeomorphism onto its image. The commutativity of the above diagram follows from the definition \cite[Equation (C.10.4)]{pardon-VFC}. Furthermore, the restriction of the gluing map $g$ to a fiber $\{a\} \times K$ is smooth, see \cite[Section C.9]{pardon-VFC} (this has already been explored in Section \ref{sec-5}). Note that the topological submersion property holds over a $G_{pq}$-invariant open neighborhood of $S_{pq}^{-1}(0)$ in $V_{pq}$ by the transversality assumption. Up to restricting to such an open subset, we see that $\pi_{pq}: V_{pq} \to B_{pq}$ is indeed a $G_{pq}$-equivariant topological submersion.

It remains to see that $\pi_{pq}$ actually has a $C^1_{\rm loc}$ $G_{pq}$-bundle structure. Note that the gluing construction of $J$-holomorphic curves is based on applying a Newton--Picard iteration scheme to a pre-glued curve. By \cite[Proposition B.11.1, C.11.1]{pardon-VFC}, the $L^2$-norm of the honest solution near the ends of the neck region controls all $C^k$-norms through the region. With this in mind, the rest of the argument follows from \cite[Corollary 6.29]{AMS}.
\end{proof}

As a consequence, we see that the tangent microbundle of the interior of $V_{pq}$ admits a vector bundle lift by \eqref{eqn:mu-split} and Lemma \ref{splitting_lemma}. This suffices for us to construct a $G_{pq}$-equivariant smoothing on the interior of $V_{pq}$, but it falls short of providing a smoothing of $V_{pq}$ as a manifold with corners. This problem will be solved on the outer-collaring $V_{pq}^+$ by performing an inductive argument.

\subsubsection{Vertical stable complex structures}

Following \cite[Section 11]{Abouzaid_Blumberg}, we describe how to construct a stabilization of the vertical tangent bundle $T^{\rm vt}V_{pq}$ of the $C^1_{\rm loc}$ $G_{pq}$-bundle $\pi_{pq}: V_{pq} \to B_{pq}$ such that the resulting vector bundle admits a complex structure. We also explain how to use these stable complex structures to construct normal complex structures on derived orbifold charts.

For any $1$-periodic orbit $\uds{p}$ of $H$, fix a unitary trivialization of the complex vector bundle $\uds{p}^* TM \to S^1$. Moreover, we choose a complex linear connection $\nabla$ defined on the pullback of $TM$ under the projection map $S^1 \times M \to M$, such that its restriction to the graph of any $1$-periodic orbit is induced by the chosen trivializations. For a $1$-periodic orbit $\uds{p}$, using the projection ${\rm pr}_{S^1}: {\mb R} \times S^1 \to S^1$, we consider the complex vector bundle
$$
(\uds{p} \circ {\rm pr}_{S^1})^* TM \to {\mb R} \times S^1.
$$
Denote by $B_{H, \uds{p}}(t): S^1 \to \text{End}(\uds{p}^* TM)$ the section obtained by differentiating the flow of $X_{H_t}$, i.e., 
\beqn
B_{H, \uds p}(t) ( W ) = \nabla_W X_{H_t},\ \forall W \in T_{\uds p(t)} M.
\eeqn
Using a smooth cut-off function, we can consider the Cauchy--Riemann operator
$$
\nabla^{0,1}_{\uds{p}}: \Omega^0((\uds{p} \circ {\rm pr}_{S^1})^* TM) \to \Omega^{0,1}((\uds{p} \circ {\rm pr}_{S^1})^* TM)
$$
such that
$$
\nabla^{0,1}_{\uds{p}} = \left\{ \begin{array}{rr}  (\nabla - B_{H,\uds{p}} \otimes dt)_{J}^{0,1} \text{ near } -\infty\\
     \nabla^{0,1} \text{ near } + \infty.       
            \end{array} 
            \right.
$$
It is a standard fact that $\nabla^{0,1}_{\uds{p}}$ defines an elliptic operator. Therefore, we can find a finite dimensional {\bf complex} vector space $V_{\uds{p}}^-$ and a linear map $\lambda_{\uds{p}}: V_{\uds{p}}^- \to \Omega^{0,1}((\uds{p} \circ {\rm pr}_{S^1})^* TM)$ such that the map
$$
\nabla^{0,1}_{\uds{p}} \oplus \lambda_{\uds{p}}: \Omega^0((\uds{p} \circ {\rm pr}_{S^1})^* TM) \oplus V_{\uds{p}}^- \to \Omega^{0,1}((\uds{p} \circ {\rm pr}_{S^1})^* TM)
$$
is surjective. This map is called the asymptotic operator of $\uds{p}$. We fix a choice of $V_{\uds{p}}^-$ and $\lambda_{\uds{p}}$, and denote the kernel of the above map by $V_{\uds{p}}^+$. The virtual vector space $(V_{\uds{p}}^+ , V_{\uds{p}}^-)$ will be used to construct a stable complex structure on $T^{\rm vt}V_{pq}$.

\begin{lemma}\label{lem:handwaving}
For any $p, q \in T^\floer$ such that $\ov{\mc M}{}^\floer_{pq} \neq \emptyset$, there exists a complex vector bundle $I^{\rm vt}_{pq} \to V_{pq}$ such that there exists an homotopy between the vector bundles
\begin{equation}\label{eqn:stable-complex}
V_{\uds{q}}^+ \oplus T^{\rm vt}V_{pq} \oplus V_{\uds{p}}^{-} \cong V_{\uds{q}}^- \oplus I^{\rm vt}_{pq} \oplus V_{\uds{p}}^{+}.
\end{equation}
\end{lemma}
\begin{proof}
We explain the construction in the case of a single-layered thickening. The vertical tangent bundle of $V_{pq}$ can be identified with the kernel of the linearized operator associated with \eqref{thickening_equation} with the framing $F$ being kept fixed. To be more precise, using the model as in the proof of Theorem \ref{regularity1}, given an element in $V_{pq}$ with representative $(\Sigma, u, \mu, \eta)$, the deformation operator over each component $\Sigma_{\alpha}$ is of the form
\beqn
\begin{aligned}
W^{1,p}(\Sigma_{\alpha}, u^* TM) \oplus W^{1,p}(\Sigma_{\alpha}, (u,\mu)^*E(k)) \to 
L^{p}(\Lambda^{0,1}(\Sigma_{\alpha}, u^* TM)) \oplus L^{p}(\Lambda^{0,1}(\Sigma_{\alpha}, u^* TM)) \\
(\xi, \eta') \mapsto (D_{u}\xi + P(\eta'), D_E \eta').
\end{aligned}
\eeqn
By taking the direct sum of these linear operators ranging over all the components of $\Sigma$, and impose the usual matching condition in the case of nodal curves, we indeed identify $T^{\rm vt}V_{pq}$ over $(\Sigma, u, \mu, \eta)$ with the kernel of a Fredholm operator which is schematically written as
$$
D^{\rm vt}_{u, \mu, \eta}: {\mc E}_{u, \mu, \eta} \to {\mc F}_{u, \mu, \eta}.
$$
Now we consider the concatenation of the three operators $D^{\rm vt}_{u, \mu, \eta} \oplus \nabla^{0,1}_{\uds{q}} \oplus \lambda_{\uds{q}} \oplus 0$
\beqn
{\mc E}_{u, \mu, \eta} \oplus (\Omega^0((\uds{q} \circ {\rm pr}_{S^1})^* TM) \oplus V_{\uds{q}}^-) \oplus V_{\uds{p}}^- \to  {\mc F}_{u, \mu, \eta} \oplus \Omega^{0,1}((\uds{q} \circ {\rm pr}_{S^1})^* TM).
\eeqn
This operator is Fredholm and surjective, with kernel identified with
$$
T^{\rm vt}V_{pq} \oplus V^{+}_{\uds{q}} \oplus V_{\uds{p}}^-.
$$
We will construct a homotopy of this operator following the proof of \cite[Proposition 11.30]{Abouzaid_Blumberg}. To this end, introduce an auxialiary moduli space $\tilde{V}_{pq} \to V_{pq}$. It is described by adding an additional marked point on the compactified lateral line ${\mb R} \cup \{ \pm \infty \} \cong [0,1]$ of the cylindrical components and the map $\tilde{V}_{pq} \to V_{pq}$ is the associated forgetful map. The fiber of the universal family over $\tilde{V}_{pq}$ is the same as the fiber of $V_{pq}$ when the newly-added marked point $\neq \pm \infty$; when it is given by $\pm \infty$, we insert a cylinder between the consecutive two cylindrical components which is mapped to the common asymptotic orbit. In particular, $\tilde{V}_{pq} \to V_{pq}$ is a fiber bundle with fiber $\cong [0,1]$. For an element $(\Sigma, u, \mu, \eta) \in V_{pq}$, the domain of the universal family over $(\Sigma, u, \mu, \eta) \times \{1\}$ coincides with the domain of $D^{\rm vt}_{u, \mu, \eta} \oplus \nabla^{0,1}_{\uds{q}} \oplus \lambda_{\uds{q}} \oplus 0$.

Write 
$$
D^{\rm vt}_{u, \mu, \eta} = D^{{\rm vt}, {\mb C}}_{u, \mu, \eta} + Y^{\rm vt}_{u, \mu, \eta}
$$
where $D^{{\rm vt}, {\mb C}}_{u, \mu, \eta}$ is a complex linear Cauchy--Riemann operator and $Y^{\rm vt}_{u, \mu, \eta}$ is a $0$-th order differential operator which is complex anti-linear. Then we can consider another operator $D^{{\rm vt}, {\mb C}}_{u, \mu, \eta} \oplus \nabla^{0,1}_{\uds{p}} \oplus \lambda_{\uds{p}} \oplus 0 $
$$
{\mc E}_{u, \mu, \eta} \oplus (\Omega^0((\uds{p} \circ {\rm pr}_{S^1})^* TM) \oplus V_{\uds{p}}^-) \oplus V_{\uds{q}}^- \to {\mc F}_{u, \mu, \eta} \oplus \Omega^{0,1}((\uds{p} \circ {\rm pr}_{S^1})^* TM).
$$
If this Fredholm operator is surjective, then there exists a finite dimensional complex vector space $(I^{\rm vt}_{pq})_{u, \mu, \eta}$ such that the kernel of this operator is
$$
(I^{\rm vt}_{pq})_{u, \mu, \eta} \oplus V_{\uds{p}}^+ \oplus V_{\uds{q}}^-.
$$
Note that using the position of the additional marked point on the lateral line $\tilde{V}_{pq} \to V_{pq}$ and a cut-off function defined on the domains of the universal family over $\tilde{V}_{pq}$ whose differential has support on the horizontal component where the marked point lives on, we can construct a homotopy between two Fredholm operators
$$
D^{\rm vt}_{u, \mu, \eta} \oplus \nabla^{0,1}_{\uds{q}} \oplus \lambda_{\uds{q}} \oplus 0 \text{ and } D^{{\rm vt}, {\mb C}}_{u, \mu, \eta} \oplus \nabla^{0,1}_{\uds{p}} \oplus \lambda_{\uds{p}} \oplus 0
$$
which varies smoothly as we vary $(\Sigma, u, \mu, \eta)$ along the fiber of $V_{pq} \to B_{pq}$.

Moreover, $D^{{\rm vt}, {\mb C}}_{u, \mu, \eta}$ can be made surjective if the $k$ as from $E(k)$ is chosen to be large enough. This fact can be derived from exactly the same proof of \cite[Proposition 6.26]{AMS}. Therefore the statement is proved by defining $I^{\rm vt}_{pq}$ to be the vector bundle over $V_{pq}$ by assembling $(I^{\rm vt}_{pq})_{u, \mu, \eta}$.
\end{proof}

\begin{cor}\label{cor:handwaving}
If the almost free action of $G_{pq}$ on $V_{pq}$ is smooth, then the induced orbifold vector bundle $T^{\rm vt}V_{pq}$ on the quotient orbifold $V_{pq}/G_{pq}$ has a normal complex structure.
\end{cor}

\begin{proof}
This follows from Lemma \ref{lem:handwaving} and the follwing fact. For $p \in {\mc P}$ (also for $q \in {\mc P}$), the trivial vector bundles $V_{\uds{p}}^+ \times V_{pq} \to V_{pq}$ and $V_{\uds{p}}^- \times V_{pq} \to V_{pq}$ can be viewed as $G_{pq}$-equivariant vectors for which $G_{pq}$ acts trivially on the factor $V_{\uds{p}}^{\pm}$. In particular, we can use $V_{\uds{p}}^{\pm}$ to stabilize the space $V_{pq}$. Because the induced $G_{pq}$-action on $V_{\uds{p}}^\pm \times V_{pq} \to V_{pq}$ is trivial on the first factor, they do not affect the \emph{normal} directions to the union of $G_{pq}$-orbits with the same isotropy group of both the tangent bundle of $V_{pq}$. Therefore the statement follows from taking the nontrivial part of the representations of \eqref{eqn:stable-complex} after taking the $G_{pq}$-quotient. 
\end{proof}

\begin{comment}

\subsubsection{A system of families of local trivializations}

\begin{prop}
There exist, for each $pq$ and $d \geq d_{pq}$, a family of local trivializations of the bundle $O_{pq}^{(d)} \to V_{pq}$, denoted by 
\beqn
\lambda_{pq}^{(d)}: p_2^* O_{pq}^{(d)}|_{\widetilde U_{pq}^{(d)}} \to p_1^* O_{pq}^{(d)}|_{\widetilde U_{pq}^{(d)}}
\eeqn
(see Definition XXX) satisfying the following conditions. 
\begin{enumerate}
    \item When $d< d'$, one has $\widetilde U_{pq}^{(d')} \subset \widetilde U_{pq}^{(d)}$ and the following diagram commutes
    \beqn
    \xymatrix{  p_2^* O_{pq}^{(d)}|_{\widetilde U_{pq}^{(d')}} \ar[r] \ar[d] &  p_1^* O_{pq}^{(d)}|_{\widetilde U_{pq}^{(d')}} \ar[d] \\
                p_2^* O_{pq}^{(d')}|_{\widetilde U_{pq}^{(d')}} \ar[r]       & p_1^* O_{pq}^{(d')}|_{\widetilde U_{pq}^{(d')}} }.
    \eeqn
    
    \item Collared.
    
    \item 
\end{enumerate}
\end{prop}

\end{comment}

\subsubsection{Stabilizing the K-chart lift}

In the equivariant smoothing theorem of Lashof \cite{Lashof_1979} a smooth structure can only be obtained after stabilizing the manifold. Given a $G$-manifold $M$, a stabilization of $M$ is $M \times {\bm R}$ where ${\bm R}$ is an orthogonal $G$-space.\footnote{We could try to extend the theory to the case of nontrivial bundles but that is not completely necessary.} This operation fits into the stabilization of Kuranishi charts. As we are dealing with infinitely many Kuranishi charts associated to a K-chart lift of the Floer category $T^\floer$ (or its outer-collaring), once we stabilize certain charts (and their products) and equipped them with smooth structures, a ``larger" chart needs to be stabilized by a larger representation in order to include the stabilization of products corresponding to its strata. Thus we need to introduce the following concept which describes how a system of stabilizations should fit together.

\begin{defn}\label{defn_stabilization}
Let $T^{\mc P}$ be a flow category over ${\mc P}$ and suppose the collection $\{ K_{pq}=(G_{pq}, V_{pq}, E_{pq}, S_{pq}) \}_{p<q}$ of Kuranishi charts with the collection $ \{ {\bm \iota}_{prq}\}_{p< r < q}$ of K-chart embeddings define a K-chart lift of $T^{\mc P}$ (see Definition \ref{Klift_defn}). A {\bf stabilization} of consists of the following objects.

\begin{enumerate}
    \item For each $T_{pq}$ a finite-dimensional orthogonal $G_{pq}$-representation ${\bm R}_{pq}$. Denote by $R_{pq}$ the trivial $G_{pq}$-equivariant vector bundle $V_{pq} \times {\bm R}_{pq}$. Then for each stratum $\alpha = pr_1 \cdots r_l q$, define the $G_\alpha$-equivariant vector bundle $R_\alpha$ over $V_\alpha = V_{pr_1} \times \cdots V_{r_l q}$ to be 
    \begin{equation}\label{eqn:direct-sum}
    R_{pr_1} \boxplus \cdots \boxplus R_{r_l q}.
    \end{equation}
    
    \item For any pair of strata $\alpha \leq \beta$, an equivariant bundle embedding
    \beqn
    \xymatrix{  R_\alpha \ar[r]^{\phi_{\beta\alpha}} \ar[d] & R_\beta|_{\partial^\alpha V_\beta} \ar[d]\\
                V_\alpha \ar[r]  & \partial^\alpha V_\beta}.
    \eeqn
\end{enumerate}
Assume that these data satisfy the following compatibility condition.
\begin{enumerate}
    
    \item[(A)] For each triple of strata $\alpha \leq \beta \leq \gamma$ one has $\phi_{\gamma\beta}\circ \phi_{\beta\alpha} = \phi_{\gamma\alpha}$.
    
    \item[(B)] The bundle embeddings are induced from products. Namely, given a stratum $\alpha = pr_1 \cdots r_l q \in \bA^{\mc P}_{pq}$, there exist equivariant isometric linear embeddings 
    $$
    \phi_{r_{i}r_{i+1}}: R_{r_{i}r_{i+1}} \hookrightarrow R_{pq} \text{ for } i=0,\dots,l
    $$
    with respect to the group embeddings $G_{r_{i}r_{i+1}} \hookrightarrow G_{pq}$ such that 
    $$
    \phi_{pq, \alpha} = \phi_{pq, pr_1} \times \cdots \times \phi_{pq, r_l q},
    $$
    and $\phi_{\beta \alpha}$ for a general pair $\beta \leq \alpha$ is constructed similarly using the corresponding factorization.
%    angle_{p,q}$, consider the $G_{P''}$-equivariant vector bundle embedding
%    $$ G_{P''}(\Phi_{PP'}): G_{P''}({\mc R}_{P \to P'}^{\sim}) = {\mc R}_{P \to P''}^{\sim} \to G_{P''}({\mc R}_{P'}) = {\mc R}_{P' \to P''}^{\sim} $$
%    covering the embedding $G_{P''}(G_{P'}(V_P)) = G_{P''}(V_P) \hookrightarrow G_{P''}(V_{P'})$. Then we require that
%    $$ \Phi_{P P''} = \Phi_{P' P''} \circ G_{P''}(\Phi_{PP'}). $$
    
    %\item (\textbf{Orthogonality}) (Although we do not require that the bundle embedding $\phi_{\beta\alpha}$ is isometric,) for each triple of stratum $\alpha \leq \beta \leq \gamma$, as subbundles of $R_\gamma|_{\partial^\alpha V_\gamma}$ one has
    %\beqn
    %R_{\gamma\alpha}^\bot = R_{\gamma\beta}^\bot|_{\partial^\alpha V_\gamma} \oplus %\phi_{\gamma\beta}( G_\gamma( R_{\beta\alpha}^\bot))
    %\eeqn
    %and this decomposition is orthogonal.
    
\end{enumerate}
%Then the \textbf{stabilization} of the K-chart lift by $\{ R_{pq} \}_{p<q}$ is defined to be
%$$\{ \text{Stab}_{{\mc R}_{pq}}(K_{pq}) = \tilde{K}_{pq}=(G_{pq}, V_{pq} \oplus {\mc R}_{pq}, E_{pq} \oplus {\mc R}_{pq}, S_{pq} \oplus \text{id}_{{\mc R}_{pq}}) \}_{p<q}$$
%with \textbf{thickening data} $\{ F_{PP'} \oplus {\mc R}_{P \to P'}^{\sim, \perp}\}_{P < P'}$ and \textbf{transition data} $\{ \phi_{PP'} \times \Phi_{PP'} \}_{P < P'}$.
%\textcolor{red}{Up to energy level $E$...}
\end{defn}

Given a stabilization containing $\{ {\bm R}_{pq}\}$ and $\{ \phi_{\beta\alpha}\}$ we can define another K-chart lift for $T^{\mc P}$. Its collection of Kuranishi charts are 
\beqn
\widehat K_{pq}:= (G_{pq}, \widehat V_{pq}, \widehat E_{pq}, \widehat S_{pq}):= {\rm Stab}_{R_{pq}} (K_{pq});
\eeqn
its collection of chart embeddings
\beqn
\widehat {\bm \iota}_{prq}: \widehat K_{pr} \times \widehat K_{rq} \to \partial^{prq} \widehat K_{pq}
\eeqn
are defined as follows. Given $(x_{pr}, e_{pr}) \in \widehat V_{pr}$, $(x_{rq}, e_{rq}) \in \widehat V_{rq}$, define
\beqn
\widehat \iota_{prq}( ( x_{pr}, e_{pr}), (x_{rq}, e_{rq})) = ( \iota_{prq}(x_{pr}, x_{rq}), \phi_{pq, prq}(e_{pr}, e_{rq}))
\eeqn
which is clearly equivariant with respect to the group map $G_{pr}\times G_{rq} \to G_{pq}$.

\begin{lemma}
The data $\{ \widehat K_{pq} \}_{p<q}$ with $\{ \widehat {\bm \iota}_{prq} \}_{p<r<q}$ is a K-chart lift of $T^{\mc P}$.
\end{lemma}

\begin{proof}
This follows directly from the definition.
\end{proof}

We need to equip a stabilization certain extra structures.

\begin{defn}
Suppose the K-chart lift ${\mf K}^{\mc P}$ is equipped with a collar structure. Then a stabilization is called {\bf collared} if there are equivariant bundle isomorphisms
\beqn
\xymatrix{ R_{pq}|_{\partial^\alpha V_{pq}} \times [0, \epsilon)^{\bF_\alpha} \ar[r] \ar[d]   &       R_{pq} \ar[d]\\
                    \partial^\alpha V_{pq} \times [0, \epsilon)^{\bF_\alpha} \ar[r] & V_{pq} }
                    \eeqn
                    such that the bundle embeddings are collared, i.e., 
                    \beqn
                    \xymatrix{   R_\beta|_{\partial^\alpha V_\beta} \times [0, \epsilon)^{\bF_\alpha \setminus \bF_\beta} \ar[r] \ar[d] &  R_\gamma|_{\partial^\alpha V_\gamma} \times [0, \epsilon)^{\bF_\alpha \setminus \bF_\beta} \ar[d] \\
                                  R_\beta  \ar[r] & R_\gamma|_{\partial^\beta V_\gamma} }.
                    \eeqn

\end{defn}

Note that although our stabilization is constructed using product bundles, the stabilization maps are not necessarily induced from the linear maps between the representations. Therefore, the above definition is not entirely superfluous. 

We will also use not-necessarily-trivial inner products on those trivial bundles $R_{pq}$ in the construction. 

\begin{defn}
Let ${\mf K}^{\mc P}$ be a K-chart lift of $T^{\mc P}$ equipped with a collar structure and a collared stabilization ${\mf R}$. A {\bf collared system of inner products} on ${\mf R}$ is a collection of $G_{pq}$-equivariant inner products on the bundle $R_{pq}$ such that
\begin{enumerate}
    \item All the embeddings $\phi_{\beta\alpha}: R_\alpha \to R_\beta|_{\partial^\alpha V_\beta}$ are isometric.
    
    \item The collar isomorphisms are isometric.
\end{enumerate}
\end{defn}

Suppose ${\mf R}$ is equipped with a collared system of inner products. Then define 
\beqn
R_{\beta\alpha} \subset R_\beta|_{V_\alpha}
\eeqn
to be the orthogonal complement of $R_\alpha$ in $R_\beta|_{V_\alpha}$. Then by basic linear algebra one can see that whenever $\alpha \leq \beta \leq \gamma$, as subbundles of $R_\gamma|_{V_\alpha}$, one has 
\beqn
R_{\gamma\alpha} = R_{\gamma\beta}|_{V_\alpha} \oplus R_{\beta\alpha}.
\eeqn
Then the original scaffolding can be extended to include the extra pieces $R_{\beta\alpha}$. Indeed, define 
\beqn
\widehat F_{\beta\alpha}:= \pi^* F_{\beta\alpha} \oplus \pi^* R_{\beta\alpha}
\eeqn
where $\pi$ denotes (temporarily) the projection $V_{pq}\times {\bm R}_{pq} \to V_{pq}$ and the induced projection $V_\alpha \times {\bm R}_\alpha \to V_\alpha$. The stabilization described in Proposition \ref{bigprop} is also extended to a map
\beqn
{\rm Stab}_{\widehat F_{pq, \alpha}}( \widehat V_\alpha) \to \partial^\alpha \widehat V_{pq}.
\eeqn
After taking the equivariantization, we obtain a map
$$
G_{pq} \times_{G_{\alpha}} ({\rm Stab}_{\widehat F_{pq, \alpha}}( \widehat V_\alpha)) \to \widehat V_{pq}
$$
which defines an equivariant open embedding. We would like to remark that the orthogonal complement of the $G_{pq}$-equivariantization of $R_{pq, \alpha} \to V_{\alpha}$ in $R_{pq}|_{G_{pq}(V_\alpha)}$ is not necessarily a product bundle. Therefore, some extra care needs to be taken in the smoothing process.

\subsubsection{An induction scheme of stabilization}\label{induction_scheme}

In the smoothing process, the construction of successive stabilizations is elementary but complicated. In general, the stabilizations are constructed inductively, similar to many other inductive constructions. We would like to describe how the induction can be carried out. We first look at the simple case when a moduli space has strata of codimension at most 2. Assume $p<q$ in ${\mc P}^\floer$ and let $V_{pq}$ be the domain of a global Kuranishi chart of $\ov{\mc M}{}_{pq}^\floer$. In the following discussion, we assume that the tangent microbundle of the underlying thickened moduli space of the relevant K-charts admits a vector bundle lift, which is the necessary condition for us to apply the smoothing theory. Moreover, in the discussion of extending the smoothing relatively, we assume that the lifts are compatible in a suitable sense. The construction of vector bundle lifts and the precise meaning of compatibility will be detailed in Section \ref{subsection64}.

Note that we also need to endow the vector bundle $E_{pq} \to V_{pq}$ with a smooth structure after suitable stabilization. However, this is rather straightforward after smoothing $V_{pq}$: we can simply approximate the classifying map of $E_{pq}$ by a smooth map which stays in the same homotopy class.

\vspace{0.2cm}

\noindent {\it The case with no boundary or corner.} When $V_{pq}$ has no boundary or corner, then it is the initial step of the induction. One can find a trivial $G_{pq}$-bundle $V_{pq}\times {\bm R}_{pq}$ induced by a suitable representation $R_{pq}$ needed for stable smoothing.

\vspace{0.2cm}

\noindent {\it The case with only codimension one boundary strata.} When $V_{pq}$ has only codimension one boundary strata, without loss of generality, suppose there is only one boundary stratum $V_{prq}$. Using the induction hypothesis, suppose ${\bm R}_{pr}$ and ${\bm R}_{rq}$ have been given. Consider the representation ${\bm R}_{pr}$ resp. ${\bm R}_{rq}$ of $G_{pr}$ resp. $G_{rq}$. By Frobenius reciprocity (see \cite[Theorem 7.47]{Sepanski}), there exist orthogonal representations ${\bm R}_{pr}^{pq}$ resp. ${\bm R}_{rq}^{pq}$ of $G_{pq}$ which contains $G_{pr}$ resp. $G_{rq}$ as subrepresentations of $G_{pr}$ resp. $G_{rq}$. Then define 
\beqn
{\bm R}_{pq} = {\bm R}_{pr}^{pq}\oplus {\bm R}_{rq}^{pq}
\eeqn
which is an orthogonal $G_{pq}$-representation. 

Now we define the bundle embeddings required for stabilization. We introduce more notations. Over the product $V_{prq} = V_{pr} \times V_{rq}$, define 
\begin{align*}
&\ R_{pr\to pq}:= R_{pr}\boxplus \{0\},\ &\ R_{rq \to pq}:= \{0\}\boxplus R_{rq}.
\end{align*}
%Define their $G_{pq}$-equivariantization to be 
%\begin{align*}
%&\ R_{pr \to pq}^\sim:= G_{pq}(R_{pr\to pq}),\ &\ R_{rq\to pq}^\sim:= G_{pq}( R_{rq\to pq}).
%\end{align*}
Then 
\beqn
R_{prq}= R_{pr\to pq}\oplus R_{rq\to pq}. %,\ &\ G_{pq}(R_{prq}) = R_{pr\to pq}^\sim \oplus R_{rq\to pq}^\sim.
\eeqn
We would like the bundle embedding $\phi_{pq, prq}$ to be the sum of 
\begin{align*}
&\ \phi_{pr \to pq}: R_{pr\to pq} \to R_{pq}|_{\partial^{prq} V_{pq}},\ & \ \phi_{rq \to pq}: R_{rq \to pq} \to R_{pq}|_{\partial^{prq} V_{pq}}.
\end{align*}
Indeed, $\phi_{pr \to pq}$ and $\phi_{rq\to pq}$ are the ones canonically induced from the linear inclusions
\begin{align*}
&\ {\bm R}_{pr} \hookrightarrow {\bm R}_{pr}^{pq} \hookrightarrow {\bm R}_{pq},\ &\ {\bm R}_{rq} \hookrightarrow {\bm R}_{rq}^{pq} \hookrightarrow {\bm R}_{pq}.
\end{align*}
Then all the requirements of Definition \ref{defn_stabilization} are automatically satisfied. In addition, for the purpose of stable smoothing, one can also take an additional direct sum to ${\bm R}_{pq}$ by another orthogonal $G_{pq}$-space ${\bm R}_{pq, 0}$. 

Note that in this final step, we need to apply the relative smoothing result Theorem \ref{relative_smoothing}. Indeed, the space $R_{prq}$ admits a smooth structure by taking the product of smooth structures on $R_{pr}$ and $R_{rq}$, whose existence is based on the induction hypothesis. Then the equivariantization $G_{pq}(R_{prq})$ has a smooth structure. Denote by $G_{pq}(R_{prq}^\perp)$ the equivariantization of the orthogonal complement of $R_{prq}$ in $R_{pq} |_{V_{prq}}$. Using the projection map $R_{prq} \to V_{prq}$,  the vector bundle $G_{pq}(R_{prq}^\perp)$ can be pulled back to $G_{pq}(R_{prq})$ and the total space of this vector bundle can be endowed with a smooth structure, by approximating the classifying map. This in turn equips $R_{pq} |_{\partial^{prq}V_{pq}}$ with a smooth structure. Using the collar structure
$$
\partial^{prq}V_{pq} \times [0,\epsilon) \to V_{pq}
$$
which extends to a collar structure of the stabilized charts, we see that an open neighborhood of $R_{pq} |_{\partial^{prq}V_{pq}}$ has a smooth structure using the product decomposition 
$$
R_{pq} |_{\partial^{prq}V_{pq}} \times [0,\epsilon).
$$
By replacing the open interval $[0,\epsilon)$ by a closed interval of the form $[0,\frac{\epsilon}{2}]$, a suitable compatibility between the vector bundle lifts guarantees that we can apply Theorem \ref{relative_smoothing} to obtain the $G_{pq}$ orthogonal representation ${\bm R}_{pq, 0}$ which induces a stabilized smoothing for $V_{pq}$ extending the previously constructed stabilized smoothings on $V_{pr}$ and $V_{rq}$.

\begin{rem}
We hope that the above arguments showcase the importance of applying Frobenius reciprocity, approximations of classifying maps of equivariant vector bundles, and the relative smoothing result in our inductive construction. Moreover, as we can see from above, outer-collaring conveniently provide us with an automatic smoothing near the boundary stratum. Additionally, it is important to keep in mind that the compatibility between the vector bundle lifts of tangent microbundles of the total space and boundary stratum is crucial for the application of Theorem \ref{relative_smoothing}.
\end{rem}

\vspace{0.2cm}

\noindent {\it The case with only codimension one or two strata.} Now we consider the case when $\bA_{pq}^{\mc P} = \{ pq, prq, psq, prsq\}$. Suppose the representations 
\beqn
{\bm R}_{pr},\ {\bm R}_{rs},\ {\bm R}_{sq}, {\bm R}_{ps}, {\bm R}_{rq}
\eeqn
are chosen so that they induce compatible stabilized smoothings except the top stratum $V_{pq}$. Define the bundles $R_{pr\to pq}$ etc. in a way similar to the previous case. Suppose we have also defined bundle embeddings
\beqn
\phi_{pr \to ps},\ \phi_{rs \to ps},\ \phi_{rs \to rq},\ \phi_{sq \to rq}.
\eeqn
In addition, we make the following assumptions, which basically says that for $ps$ and $rq$, the stabilization data are constructed as in the previous case.
\begin{enumerate}
    \item ${\bm R}_{ps}$ is the direct sum
    \beqn
    {\bm R}_{ps} = {\bm R}_{pr}^{ps}\oplus {\bm R}_{rs}^{ps} \oplus {\bm R}_{ps,0}
    \eeqn
    of orthogonal $G_{ps}$-spaces such that ${\bm R}_{pr}^{ps}$ resp. ${\bm R}_{rs}^{ps}$ contains ${\bm R}_{pr}$ resp. ${\bm R}_{rs}$ as subrepresentations of $G_{pr}$ resp. $G_{rs}$. Similarly, there is a direct sum
    \beqn
    {\bm R}_{rq} = {\bm R}_{rs}^{rq} \oplus {\bm R}_{sq}^{rq} \oplus {\bm R}_{rq,0}
    \eeqn
    of orthogonal $G_{rq}$-spaces such that ${\bm R}_{rs}^{rq}$ resp. ${\bm R}_{sq}^{rq}$ contains ${\bm R}_{rs}$ resp. ${\bm R}_{sq}$ as subrepresentations of $G_{rs}$ resp. $G_{sq}$.
    
    \item The bundle embedding 
    \beqn
    \phi_{pr\to ps}: R_{pr\to ps} \to R_{ps}|_{\partial^{prs} V_{ps}}
    \eeqn
    is induced from the linear inclusion ${\bm R}_{pr}\hookrightarrow {\bm R}_{pr}^{ps} \hookrightarrow {\bm R}_{ps}$. Similar requirement applies to other bundle embeddings.
\end{enumerate}

Now we define an orthogonal $G_{pq}$-space ${\bm R}_{pq}$. There are two codimension one strata, $prq$ and $psq$. For the $G_{pr}$-space ${\bm R}_{pr}$, by Frobenius reciprocity, there is a $G_{pq}$-orthogonal space ${\bm R}_{pr}^{pq}$ which contains ${\bm R}_{pr}$ as a subrepresentation of $G_{pr}$. We choose similarly ${\bm R}_{rq}^{pq}$, ${\bm R}_{ps}^{pq}$, ${\bm R}_{sq}^{pq}$. For the purpose of stable smoothing, we choose another $G_{pq}$-orthogonal space ${\bm R}_{pq, 0}$, which is constructed similarly as the case of one boundary stratum by applying relative smoothing. Then define 
\beqn
{\bm R}_{pq}:= \Big( {\bm R}_{pr}^{pq}\oplus {\bm R}_{rq}^{pq}\Big) \oplus \Big( {\bm R}_{ps}^{pq}\oplus {\bm R}_{sq}^{pq}\Big) \oplus {\bm R}_{pq, 0}.
\eeqn
We rewrite this decomposition as 
\beq\label{decomposition55}
{\bm R}_{pq} = {\bm R}_{pq, r}\oplus {\bm R}_{pq, s} \oplus {\bm R}_{pq, 0}.
\eeq
Then we want to define the bundle embeddings into 
\beqn
R_{pq}:= R_{pq, r} \oplus R_{pq, s} \oplus R_{pq, 0}.
\eeqn
This time the bundle embeddings are not purely induced from linear maps between the representations. 

\begin{enumerate}
    \item We define $\phi_{pr \to pq}$ and $\phi_{sq \to pq}$. We only describe $\phi_{pr\to pq}$ in detail. With respect to \eqref{decomposition55}, we write 
    \beqn
    \phi_{pr\to pq} = \phi_{pr, r} \oplus \phi_{pr, s} \oplus 0.
    \eeqn
    We define $\phi_{pr, r}: R_{pr\to pq} \to R_{pq, r}$ to be the one induced from the linear inclusion
    \beqn
    {\bm R}_{pr} \hookrightarrow {\bm R}_{pr}^{pq} \hookrightarrow {\bm R}_{pq, r}.
    \eeqn
    For $\phi_{pr, s}: R_{pr\to pq} \to R_{pq, s}$, notice that it is defined over $V_{prq}$. We first define it over $\partial^{prsq} V_{prq}$, which is induced from the existing bundle map
    \beqn
    \phi_{pr\to ps}: R_{pr\to ps} \to R_{ps}
    \eeqn
    and its equivariantization and stabilization (by the bundle associated to the scaffolding), composed with the linear inclusion
    \beqn
    R_{ps} \hookrightarrow R_{ps}^{pq} \hookrightarrow R_{pq, s}.
    \eeqn
    To extend to the whole $V_{prq}$, we use the collar structure and a cut-off function on $V_{prq}$ which is supported in the collar region near $\partial^{prsq} V_{prq}$ and which only depends on the collar coordinate. 
    
    \item We define $\phi_{rs\to pq}$. Again, we write $\phi_{rs \to pq} = \phi_{rs, r} \oplus \phi_{rs, s} \oplus 0$. The first component $\phi_{rs, r}$ is induced from the bundle map \beqn
    \phi_{rs \to rq}: R_{rs \to rq} \to R_{rq}
    \eeqn
    composed with the linear inclusion ${\bm R}_{rq} \hookrightarrow {\bm R}_{rq}^{pq} \hookrightarrow {\bm R}_{pq, r}$; the second component $\phi_{rs, s}$ is induced from the bundle map 
    \beqn
    \phi_{rs \to ps}: R_{rs \to ps} \to R_{ps}
    \eeqn
    composed with the linear inclusion ${\bm R}_{ps} \hookrightarrow {\bm R}_{ps}^{pq} \hookrightarrow {\bm R}_{pq, s}$.
    
    \item The definitions of $\phi_{ps \to pq}$ and $\phi_{rq \to pq}$ is similar to the case of $\phi_{pr \to pq}$ and $\phi_{sq \to pq}$. We only describe in detail the bundle map $\phi_{ps \to pq}$ (which should be defined over $V_{psq}$). We write $\phi_{ps\to pq}$ as $\phi_{ps, r} \oplus \phi_{ps, s}\oplus 0$. The second component $\phi_{ps, s}$ is naturally induced from the linear inclusion ${\bm R}_{ps} \hookrightarrow {\bm R}_{ps}^{pq} \hookrightarrow {\bm R}_{pq, s}$, For the first component $\phi_{ps, r}$, we decompose it further as 
    \beqn
    \phi_{ps, r}= \phi_{ps, pr} \oplus \phi_{ps, rs}.
    \eeqn
    We first define it over the boundary stratum $\partial^{prsq} V_{psq}$ and then use a cut-off function to turn it off as we leave this stratum towards the interior of $V_{psq}$. Consider the composition of orthogonal projections
    \beqn
    {\bm R}_{ps} \to {\bm R}_{pr}^{ps}\oplus {\bm R}_{rs}^{ps} \to {\bm R}_{pr}\oplus {\bm R}_{rs}
    \eeqn
    which induces a bundle map 
    \beqn
    R_{ps \to pq}|_{V_{prsq}} \to R_{pr \to pq}|_{V_{prsq}} \oplus R_{rs \to pq}.
    \eeqn
    Then using the linear inclusions ${\bm R}_{pr}\to {\bm R}_{pr}^{pq}$ and the map $R_{rs \to rq} \to R_{rq}\to R_{rq}^{pq}$, one defines the map $\phi_{ps, r}$ over the closed set $V_{prsq}$. One can then equivariantize, stabilize to define it over $\partial^{prsq} V_{psq}$. Lastly, use a cut-off function to extend this component to a neighborhood. 
\end{enumerate}

\begin{rem}
We put effort on defining the embeddings maps between charts in the above discussions, which contains certain distinct features than the case with fewer strata because the bundle maps are no longer linear. We would like to remark that the compatibility between the stabilized smoothings of $V_{prq}$ and $V_{psq}$ near their common stratum $V_{prsq}$ is already guaranteed by the inductive nature of our construction.
\end{rem}

The above inductive strategy, especially the case with codimension two strata, indicates that it seems difficult to construct a stabilization of K-chart lifts of a flow category using purely linear maps between representations. This brings in an extra layer of complexity in the following smoothing process because we would like the bundle embeddings also to be smooth.

\subsection{The main theorem about smoothing}
We state the main theorem on smoothing the outer-collared K-chart lift of $(T^\floer)^+$. In the following, a smoothing on a $\bA$-manifold $V$ equipped with a collar structure (Definition \ref{defn:collar}) is called a {\bf collared smoothing} if the smoothing is equal to the product of the standard smooth structure on $[0,\epsilon)^{{\mb F}_{\alpha}}$ and a smoothing on $\partial^{\alpha} V$ over the collar region $\partial^\alpha V \times [0,\epsilon)^{{\mb F}_{\alpha}}$. The same notion is also used for a smoothing of a vector bundle.

\begin{thm}\label{smoothing_theorem}
Given the outer-collared K-chart lift ${\mf K}^+$ of the outer-collared flow category $(T^{\floer})^+$ equipped with a collared scaffolding (see Proposition \ref{bigprop}), there exist the following objects. We revome the ``$+$'' to make the notations more succinct.

\begin{enumerate}

    \item A stabilization of ${\mf K}$ (see Definition \ref{defn_stabilization}), given by a collection of orthogonal $G_{pq}$-spaces ${\bm R}_{pq}$ and a collection of bundle embeddings.
    
    \item A collared $G_{pq}$-smoothing on $\widehat V_{pq}:= V_{pq} \times {\bm R}_{pq}$.
    
    \item A collared bundle $G_{pq}$-smoothing on the $O(n)$-bundle $\widehat O_{pq} \to \widehat V_{pq}$, where $\widehat O_{pq}$ is the pullback of $O_{pq}$ under the natural projection $\widehat V_{pq} \to V_{pq}$.
    
    \item A collared $G_{pq}$-invariant inner product on the trivial bundle $R_{pq} \to V_{pq}$ inducing an inner product on the vector bundle $\widehat R_{pq} \to \widehat V_{pq}$, where $\widehat R_{pq}$ is the pullback of $R_{pq}$ under the natural projection $\widehat V_{pq} \to V_{pq}$.
    
    \item A collared $G_{pq}$-smoothing on the $O(n)$-bundles $\widehat R_{pq}$. 
    
%    \item A perturbation of the zero sections of $\widehat E_{pq} \to \widehat V_{pq}$.
    
\end{enumerate}
These objects satisfy the following conditions.

\begin{enumerate}

    \item[(A)] The smoothing of $\widehat V_{pq}$ induces a smooth bundle structure $\widehat Q_{pq}$ as it is induced from the $G_{pq}$-representation ${\bm Q}_{pq}$. Then the smooth structures on $\widehat O_{pq}$ and $\widehat R_{pq}$ induce a smooth structure on the obstruction bundle 
    \beqn
    \widehat E_{pq} = \widehat O_{pq} \oplus \widehat Q_{pq}\oplus \widehat R_{pq}.
    \eeqn

    \item[(B)] Boundary smoothings are given by products and stabilizations. More precisely, this means the following. For any stratum $\alpha = pr_1 \cdots r_l q$ of $\bA^\floer_{pq}$, the smoothings on $\widehat V_{pr_1}$, $\ldots$, $\widehat V_{r_l q}$ induce a smooth structure on the prodcut $\widehat V_\alpha = \widehat V_{pr_1}\times \cdots \times \widehat V_{r_l q}$ and its equivariantization $\widehat V_\alpha^\sim$; the smooth bundle structures on $\widehat O_{pr_1}$, $\ldots$, $\widehat O_{r_l q}$ induce a smooth structure on the bundle $\widehat O_\alpha^\sim  \to \widehat V_\alpha^\sim$; there is also a canonical smooth bundle structure on $R_\alpha^\sim$ as it is induced from trivial bundles. Then 
    
    \begin{enumerate}
    
    \item The embedding $\widehat V_\alpha \hookrightarrow \partial^\alpha \widehat V_{pq}$ is smooth.

    \item The bundle embedding $\widehat O_\alpha \hookrightarrow \widehat O_{pq}|_{\partial^\alpha \widehat V_{pq}}$ and the bundle embedding $\widehat R_\alpha \hookrightarrow \widehat R_{pq}|_{\partial^\alpha \widehat V_{pq}}$ are  smooth bundle embeddings. It follows that the orthogonal complement $\widehat O_{pq, \alpha}$ and the orthogonal complement $\widehat R_{pq, \alpha}$ are smooth bundles over $\widehat V_\alpha$. It follows that the bundle
    \beqn
    \widehat F_{pq,\alpha} = \widehat O_{pq,\alpha} \oplus \widehat Q_{pq, \alpha} \oplus \widehat R_{pq, \alpha} \to \widehat V_\alpha
    \eeqn
    is smooth, where the smooth structure on $\widehat Q_{pq, \alpha}$ is induced from the orthogonal complement of the $G_{pq}$-subrepresentation ${\bm Q}_{\alpha} \hookrightarrow {\bm Q}_{pq}$. 
        
    \item The stabilization map 
        \beqn
        \widehat \theta_{pq, \alpha}: {\rm Stab}_{\widehat F_{pq, \alpha}} \Big( \widehat V_\alpha \Big) \to \partial^\alpha V_{pq}
        \eeqn
        (which is the $\theta_{pq,\alpha}$ in \eqref{eqn:sca-K} lifted) is a smooth embedding.
        
        \item Let $\widehat \pi_{pq, \alpha}: \partial^\alpha \widehat V_{pq} \to \widehat V_\alpha^\sim$ be the projection map induced from the stabilization map above, which is smooth. Then bundle isomorphism
        \beqn
        \widehat \vartheta_{pq, \alpha}: \widehat \pi_{pq, \alpha}^* \Big( \widehat O_{pq} |_{\widehat V_\alpha}  \Big) \to \widehat O_{pq}|_{\partial^\alpha \widehat V_{pq} }
        \eeqn
        (included in the scaffolding) is smooth.
    \end{enumerate}
    
    %\item[(C)] The smoothings respect the collar structure. More precisely, for all $rs$ and all stratum $\gamma$ of $rs$, the canonical map 
    %\beq\label{collar}
    %\widehat V_{rs, \gamma} \times [0, \epsilon)^{{\rm codim}(\gamma)} \to \widehat V_{rs}
    %\eeq
    %is a diffeomorphism onto its image and the bundle isomorphisms (given before, need a {\bf %reference}) 
    %\beqn
    %\widehat O_{rs}^{(d)} |_{V_{rs, \gamma}} \times [0, \epsilon)^{{\rm codim}(\gamma)} \to \widehat %O_{rs}^{(d)}
    %\eeqn
    %is smooth.

\end{enumerate}

\end{thm}

Assuming this theorem, it is straightforward to construct a derived orbifold lift of the flow category $(T^\floer )^+$ by taking group quotients. Remove the $\widehat{\cdot}$ from the notation. Define
\beqn
C_{pq} = ({\mc U}_{pq}, {\mc E}_{pq}, {\mc S}_{pq}, \psi_{pq})
\eeqn
where ${\mc U}_{pq} = V_{pq}/G_{pq}$, ${\mc E}_{pq} = E_{pq}/ G_{pq}$, ${\mc S}_{pq}: {\mc U}_{pq} \to {\mc E}_{pq}$ is the induced orbibundle section, and $\psi_{pq}:{\mc S}_{pq}^{-1}(0) \to \ov{\mc M}{}_{pq}^+$ is the induced footprint map. Moreover, for $\alpha = pr_1 \cdots r_l q$, one has the natural identification
\beqn
K_\alpha/G_\alpha \cong C_\alpha:= C_{pr_1}\times \cdots \times C_{r_l q}.
\eeqn

The scaffolding in the K-chart sense also descends to a scaffolding in the D-chart sense. Indeed, for $\alpha, \beta \in \bA_{pq}^\floer$ with $\alpha \leq \beta$, the bundle
\beqn
{\mc F}_{\beta\alpha}:= F_{\beta\alpha}/G_\alpha \to {\mc U}_\alpha
\eeqn
is a smooth orbifold vector bundle over ${\mc U}_\alpha$. Then the stabilization map ${\rm Stab}_{F_\beta\alpha} (V_\alpha) \to \partial^\alpha V_\beta$ induces 
\beqn
{\bm \theta}_{\beta\alpha}: {\rm Stab}_{{\mc F}_{\beta\alpha}}(C_\alpha) \to \partial^\alpha C_\beta.
\eeqn
One checkes directly that the data $({\mc F}_{\beta\alpha}, {\bm \theta}_{\beta\alpha})_{\alpha \leq \beta}$ form a scaffolding in the D-presentation of $\ov{\mc M}{}_{pq}^+$. It should be pointed out that these derived orbifold charts are effective due to the geometric construction.

\begin{thm}\label{thm:normal-c-orient}
The D-chart lift of $(T^\floer)^+$ induced from $\{ C_{pq} \}$ equipped with the compatible collar structure and scaffolding can be upgraded to an oriented and normally complex derived orbifold lift by doing a further stabilization of the corresponding K-chart lift.
\end{thm}

\begin{proof}
In the case of a single chart $({\mc U}_{pq}, {\mc E}_{pq}, {\mc S}_{pq}, \psi_{pq})$, recall that ${\mc U}_{pq}$ is obtained from a quotient $\widehat V_{pq} / G_{pq}$, such that
\begin{enumerate}
    \item The space $\widehat V_{pq}$ is equal to the product $V_{pq} \times {\bm R}_{pq}$.
    \item There is a $C^1_{\rm loc}$ $G_{pq}$-bundle structure $\pi_{pq}: V_{pq} \to B_{pq}$ which ensures that the tangent space of the smoothing of $\widehat V_{pq}$ is isotopic to the direct sum
    \beq\label{eqn:direct-sum-C}
    R_{pq} \oplus T^{\rm vt}V_{pq} \oplus \pi^* TB_{pq},
    \eeq
    where we abuse the notation to denote the corresponding bundles obtained by the pullback under the natural projection map $\widehat V_{pq} \to V_{pq}$.
\end{enumerate}
The first simple obervation is that the stabilizing factors ${\bm R}_{pq}$ can be taken to be unitary representations. We can simply use ${\bm R}_{pq} \oplus {\bm i} {\bm R}_{pq}$ instead of a possibly real representation ${\bm R}_{pq}$. Moreover, the bundle inclusions $R_\alpha \oplus {\bm i} R_\alpha \hookrightarrow R_{pq}\oplus {\bm i}R_{pq}$ are then complex linear and the inner products on the trivial bundles $R_{pq}\oplus {\bm i} R_{pq}$ become Hermitian. 

We would like to do a further stabilization to the (already smoothed) K-chart lift. Indeed, we stabilize by the $G_{pq}$-representation ${\bm Q}_{d_{pq}}$. Let $\widehat Q_{pq}' \to \widehat V_{pq}$ be the trivial bundle (to be distinguished from the original $Q$-bundle). Notice that the bundle inclusions $Q_\delta \hookrightarrow Q_d$ over $B_\delta \subset B_d$ for all $\delta \in \bA_d$ is defined via a linear inclusion ${\bm Q}_\delta \to {\bm Q}_d$ the corresponding bundle. Then there are natural bundle inclusions 
\beqn
\widehat Q_\alpha' \to \widehat Q_{pq}'|_{\widehat V_\alpha}.
\eeqn
As this bundle map is induced from linear map between representation spaces, it is automatically smooth. Hence we obtained another smooth K-chart lift whose K-charts are
\beqn
\widehat{\widehat{K}}_{pq} = {\rm Stab}_{\widehat Q_{pq}'} \widehat K_{pq}.
\eeqn
It induces a new D-chart lift. 

Now we consider the normal complex structures on this new D-chart lift. Indeed, the obstruction bundles are now all complex. Indeed, 
\beqn
\widehat{\widehat{E}}_{pq} = \widehat{\widehat{O}}_{pq} \oplus \widehat{\widehat{R}}_{pq} \oplus \widehat{\widehat{Q}}_{pq} \oplus \widehat{\widehat{Q}}{}_{pq}'
\eeqn
and we give the last two summands the natural complex structure. 

Now the quotient orbifold 
\beqn
\widehat{\widehat{V}}_{pq}/ G_{pq}
\eeqn
is normally complex. Indeed, \eqref{eqn:direct-sum-C} becomes
\beqn
T \widehat{\widehat{V}}_{pq} \cong \big( R_{pq}\otimes {\mb C}\big) \oplus T^{\rm vt} V_{pq} \oplus (\pi_{pq}^* TB_{d_{pq}} \oplus \widehat{\widehat{Q}}{}_{d_{pq}}' ).
\eeqn
As shown in Proposition \ref{base_normal_complex}, the quotient of the last summand above is normally complex. Moreover, by Lemma \ref{lem:handwaving} and Corollary \ref{cor:handwaving}, the quotient of $T^{\rm vt} V_{pq}$ is also normally complex.

%All of the three components of the decomposition in fact have a natural normal complex structure after passing to the quotient orbifold $\widehat V_{pq} / G_{pq}$. Indeed, without loss of generality, we can always take the orthogonal representation ${\bm R}_{pq}$ in the smoothing process to be a complex vector space, e.g., using the complexification of the representation produced from Lashof's stabilized smoothing theorem. As for $T^{\rm vt}V_{pq}$, Lemma \ref{lem:handwaving} endows it with a normal complex structure, which in turn is due to the nature of Cauchy--Riemann operators. For the third part, it follows from the description in Lemma \ref{lem:r-blowup} that $B_{pq} / G_{pq}$ is a real blowup of a complex orbifold, which implies that the bundle $T B_{pq}$ has a normal complex structure after passing to the quotient orbifold.

%The equivariant bundle inducing ${\mc E}_{pq}$ is $\widehat E_{pq}$, which is the stabilization of $E_{pq}$ by the product bundle $R_{pq}$ (already assumed to be complex). Recall that we have the direct sum decomposition 
%$$
%E_{pq} = O_{pq} \oplus Q_{pq}.
%$$
%The obstruction bundle $O_{pq}$, as defined by \eqref{eqn:obstruction-O}, is in fact a complex vector bundle. On the other hand, we need to show that $Q_{pq}/G_{pq} \to B_{pq}/G_{pq}$ has a normal complex structure. \textcolor{red}{This is true because all the summand consisting of all nontrivial subrepresentations of ${\bm Q}^*_{pq}$, viewed as a representation of any finite group of $G_{pq} \cong U(d_{pq})$, is complex.}

To obtain a normal complex structure on the D-chart lift, we further check:
\begin{enumerate}
    \item A complex structure on the scaffolding. Indeed, the original difference bundles
    $$
    \widehat F_{pq,\alpha} = \widehat O_{pq,\alpha} \oplus \widehat Q_{pq, \alpha} \oplus \widehat R_{pq, \alpha}
    $$
    is already complex as explained below. First, the bundle $\widehat O_{pq,\alpha}$ is complex because we can choose the metric on the $O$-bundles to be Hermitian. The factor $Q_{pq, \alpha}$ is complex because it consists of ``off-diagonal" Hermitian matrices. As mentioned above, the stabilization in the smoothing can be chosen to be complex representations, therefore $\widehat R_{pq, \alpha}$ also has a complex structure. Then the further stabilization by ${\bm Q}_{d_{pq}}$ gives an extra copy of $\widehat Q_{pq, \alpha}$ which is complex as well.
    
    \item Compatibility between the information on the product boundary charts and the information on the restriction of the chart to the boundary. This follows from the constructions \cite[Section 11]{Abouzaid_Blumberg} by incorporating the geometry of the broken trajectories from the thick-and-thin decomposition of the domains.
    \item The compatibility with collars follows from the construction.
\end{enumerate}
By examing conditions of Definition \ref{defn:normal-C} we obtained a normal complex structure on this (further stabilized) D-chart lift of $(T^\floer)^+$.

For any capped orbit $p \in T^\floer$, we define its associated virtual vector space to be
$$
(V_p^+, V_p^-) := (V_{\uds p}^+, V_{\uds p}^-),
$$
where the right hand side is constructed before Lemma \ref{lem:handwaving}. Then the existence of orientation (Definition \ref{defn:orient}) is a classical result in Floer theory \cite{Floer_Hofer_Orientation}. For a modern exposition, the reader could refer to \cite[Section C.13]{pardon-VFC}.
\end{proof}

\begin{comment}
\begin{defn}
Given $p, q \in T^\floer$ such that $\ov{\mc M}^\floer_{pq} \neq \emptyset$, define the complex vector space
$$
W_{pq}^{{\mb C}, {\rm vt}} := \bigoplus_{p < r < q} V_{\uds{r}}^-.
$$
Abusing the notation, we also denote the trivial vector bundle $W_{pq}^{{\mb C}, {\rm vt}} \times V_{pq} \to V_{pq}$ by $W_{pq}^{{\mb C}, {\rm vt}}$. For $\alpha = pr_1 \cdots r_l q \in \bA^\floer_{pq}$, define 
$$
W_{\alpha}^{{\mb C}, {\rm vt}} := W_{p r_1}^{{\mb C}, {\rm vt}} \boxplus \cdots \boxplus W_{r_l q}^{{\mb C}, {\rm vt}} \to V_{pr_1} \times \cdots \times V_{r_l q}.
$$
\end{defn}
\end{comment}

% only in the topological category {\it a priori}. However, due to the concrete construction, in fact $F_{\uds{pr_1 \cdots r_k q}}^\sim$ is naturally equipped with a smooth bundle structure. In fact, once $\widehat V_{rs}$ is equipped with a $G_{rs}$-smoothing, one can make the classifying map smooth via an approximation process.

\subsubsection{Constructing straightening}\label{subsubsec:straight}

We sketch how to construct a straightening on the D-chart lift which is compatible with the existing collar and scaffolding structures. This construction relies on the concrete geometric feature of the obstruction bundles. We do not know if instead the straightening can be constructed in a more abstract level.

We start with moduli spaces without boundary or corners. Let $C_{pq}$ be the D-chart. Then by the method of \cite[Lemma 3.15, Lemma 3.20]{Bai_Xu_2022}, one can construct a straightening on the pair $({\mc U}_{pq}, {\mc E}_{pq})$. Indeed, notice that there are a sequence of obstruction bundles ${\mc E}_{pq}^{(d)}$ for $d \geq d_{pq}$. We can actually construct, not just a connection on ${\mc E}_{pq}^{(d_0)}$ for the lowest $d_0 = d_{pq}$ (which is required for straightening), but actually a sequence of metrics on all ${\mc E}_{pq}^{(d)}$ and metric connections such that the natural inclusion ${\mc E}_{pq}^{(d)} \hookrightarrow {\mc E}_{pq}^{(d')}$ (for $d< d'$) preserves the metric and connection.

Inductively, suppose we have constructed such compatible structures for all $C_{rs}$ for $d_{rs} < d_{pq}$. These structures induce metrics and metric connections on the difference bundles ${\mc F}_{\beta\alpha}$. Now consider the chart $C_{pq}$. Over a lowest stratum $\alpha = pr_1 \cdots r_l q$, the existing structures induce metrics and metric connections on ${\mc E}_\beta^{(d)}|_{{\mc U}_\alpha}$ for all $d\geq d_{pq}$ and proper strata between $\alpha$ and $pq$. We would like to extend the metric and connection to ${\mc E}_{pq}^{(d)}$ over ${\mc U}_\alpha$ for all $d\geq d_{pq}$. Starting from the lowest $d_0 = d_{pq}$. Properties proved in Lemma \ref{obundle_property} shows that we can first extend the metric and metric connection to the sum ${\mc E}_{pr_l \to pq}^{(d_0)} + {\mc E}_{r_1 q \to pq}^{(d_0)}$. This will be compatible with all existing metrics and connections. Then we extend (arbitrarily but satisfying the requirement for straightening) to ${\mc E}_{pq}^{(d_0)}|_{{\mc U}_\alpha}$. Then it induces a metric and connection on the difference bundle ${\mc F}_{pq, \alpha}$. Via the stabilization map from ${\mc U}_\alpha$ to $\partial^\alpha {\mc U}_{pq}$, one obtains a metric on $\partial^\alpha {\mc U}_{pq}$ which is straightened. Next, we can inductively construct metrics and connections on ${\mc E}_{pq}^{(d)}$ for higher $d$. Once we finished for the lowest stratum $\alpha$, we can use the collar structure to extend to all nearby higher strata. For a next stratum $\beta$, we carry out the induction from $d_0$ above again. We omit the details.

\subsection{Proof of Theorem \ref{smoothing_theorem}}\label{subsection64}

To save notations, we remove all the superscript ``${}^+$'' which indicates the outer-collaring. Instead, we keep in mind that all objects and structures have corresponding collar structures or respect the collar structure.

\subsubsection{Smoothing charts without boundary}

Consider a moduli space $\ov{\mc M}_{pq}$ of Floer trajectories which has no codimension one strata. Let $K_{pq} = (G_{pq}, V_{pq}, E_{pq}, S_{pq})$ be the global Kuranishi chart constructed. Then $V_{pq}$ has no boundary or corner and one can apply Lashof's smoothing theory directly. In this case there is no outer-collaring to the chart. The projection $\pi_{pq}: V_{pq} \to B_{pq}$ is a $G_{pq}$-equivariant $C^1_{\rm loc}$ fiberwise smooth topological submersion. Then by the existence of $G$-equivariant fiberwise submersions (Lemma \ref{splitting_lemma}), one can choose a $G$-equivariant fiberwise submersion of $\pi_{pq}$ which induces an isotopy class of $G_{pq}$-vector bundle reductions of $T_\mu V_{pq}$. By Lashof's theorem (Theorem \ref{lashof_theorem}) on stable $G$-smoothings, there exists a stable $G_{pq}$-smoothing on $V_{pq}$, i.e., a finite-dimensional orthogonal representation ${\bm R}_{pq}$ of $G_{pq}$ and a $G_{pq}$-invariant smooth structure on the product
\beqn
\widehat V_{pq}:= V_{pq}\times {\bm R}_{pq}.
\eeqn
Let $\widehat \pi_{pq}: \widehat V_{pq} \to B_{pq}$ be composition $\widehat V_{pq} \to V_{pq} \to B_{pq}$. Then there is an isomorphism of $G_{pq}$-equivariant vector bundles
\beqn
T\widehat V_{pq} \cong T^{\rm vt} \widehat V_{pq} \oplus \widehat \pi_{pq}^* T B_{pq}. 
\eeqn
(However, the projection $\widehat \pi_{pq}$ may not be smooth and the fibers may not be smooth submanifolds.)

Moreover one can give smooth structures on the obstruction bundles. Indeed, one can inductively construct a structure of smooth equivariant $O(n)$-bundles on $\widehat O_{pq}^{(d)} \to \widehat V_{pq}$ for all $d \geq d_{pq}$ such that $\widehat O_{pq}^{(d)}$ is a smooth subbundle of $\widehat O_{pq}^{(d+1)}$ by smoothing the relevant equivariant classifying map. The bundle $\widehat R_{pq} \to \widehat V_{pq}$, which is the trivial bundle with fiber ${\bm R}_{pq}$ automatically has a smooth structure. We also use the inner product on ${\bm R}_{pq}$ to equip the bundle $\widehat R_{pq}$ with an inner product structure, which is a smooth inner product.

\subsubsection{Smoothing charts with boundary and corners}

Now consider a chart $K_{pq} = (G_{pq}, V_{pq}, E_{pq}, S_{pq})$. We have assumed that $K_{pq}$ has a collar structure which was actually constructed via outer-collaring. We need to construct a stable smoothing which is of ``product type'' near the boundary and which extends the existing stable smoothings. We first state the induction hypothesis. As an initial remark, we induct on the energy of the moduli spaces $d_{pq}$. 

\vspace{0.2cm}

\noindent {\it Induction Hypothesis I. Stabilization.} Suppose we have the following data.
\begin{enumerate}
\item An orthogonal $G_{rs}$-representation ${\bm R}_{rs}$ for each $rs$ with $d_{rs} < d_{pq}$. We define a bundle 
\beqn
R_{uv \to rs}^\sim \to \partial^{ruvs} V_{rs}\ {\rm whenever}\ ruvs \in \bA_{rs}^\floer
\eeqn
as follows. Consider the product $V_{ruvs} = V_{ru}\times V_{uv}\times V_{vs}$. Then the trivial bundle $R_{uv} = V_{uv}\times {\bm R}_{uv} \to V_{uv}$ is pulled back to $V_{ruvs}$. Via the stabilization map ${\rm Stab}_{F_{rs, ruvs}} (V_{ruvs}) \hookrightarrow \partial^{ruvs} V_{rs}$ this trivial bundle is pulled back to a trivial bundle. Then using the $G_{rs}$-equivariantization we obtain a not-necessarily-trivial vector bundle 
\beqn
R_{uv \to rs}^\sim \to \partial^{ruvs} V_{rs}\cong G_{rs}\times_{G_{ruvs}} ({\rm Stab}_{F_{rs, ruvs}} (V_{ruvs})).
\eeqn

\item A $G_{rs}$-equivariant bundle embedding 
\beqn
    \phi_{uv\to rs}: R_{uv \to rs} \to R_{rs}|_{\partial^{ruvs} V_{rs}}.
    \eeqn
    
\item For each $\alpha = ru_1 \cdots u_l s\in \bA_{rs}^\floer$, define a bundle map
\beqn
\phi_{rs, \alpha}: R_\alpha := R_{ru_1} \boxplus \cdots \boxplus R_{u_l s} \to R_{rs}
\eeqn
to be the sum of the maps $\phi_{ru_1 \to rs}, \ldots, \phi_{u_l s\to rs}$. We require that the bundle map is actually a bundle embedding. This induces bundle embeddings 
\beqn
\phi_{\beta\alpha}: R_\alpha \to R_\beta\ \alpha \leq \beta,\ \alpha, \beta \in \bA_{rs}^\floer.
\eeqn
We require that the collection $\{R_{rs}\}$ and the collection $\{\phi_{\beta\alpha}\}$ satisfy the requirements for stabilization of a K-chart lift (see Definition \ref{defn_stabilization}) up to level $d_{pq}$, i.e., the K-charts for moduli spaces $\ov{\mc M}_{rs}^\floer$ with $d_{rs} < d_{pq}$ are equipped with stabilizations induced from $R_{rs}$ with compatibility conditions satisfied.

\item We also assume we have a $G_{rs}$-invariant inner product on the trivial bundle $R_{rs} \to V_{rs}$ for all $d_{rs} < d_{pq}$ such that the bundle maps $\phi_{rs, \alpha}$ are all isometric.
\end{enumerate}
Moreover, we assume that the representations ${\bm R}_{rs}$ and bundle embeddings $\phi_{uv \to rs}$ are of the following particular forms. (One can see the discussion in Section \ref{induction_scheme} for why we impose the following requirement.)
\begin{enumerate}
    \item We have a $G_{rs}$-equivariant orthogonal decomposition 
    \beqn
    {\bm R}_{rs} = {\bm R}_{rs,-}\oplus {\bm R}_{rs, 0}
    \eeqn
    and a $G_{rs}$-equivariant orthogonal decomposition
    \beqn
    {\bm R}_{rs,-} = \bigoplus_{r< w < s} {\bm R}_{rs, w} = \bigoplus_{r< w < s} {\bm R}_{rw}^{rs} \oplus {\bm R}_{ws}^{rs}.
    \eeqn
    Intuitively, ${\bm R}_{rs, -}$ comes from lower stratum via Frobenius reciprocity, and ${\bm R}_{rs, 0}$ is chosen so that a relative smoothing exists. Then one has the (trivial) bundles $R_{rs, -}$, $R_{rs, 0}$, $R_{rw}^{rs}$, $R_{ws}^{rs}$ over $V_{rs}$. 
    
    \item Whenever $r<w <s$, the representations ${\bm R}_{rw}^{rs}$ resp. ${\bm R}_{ws}^{rs}$ contains ${\bm R}_{rw}$ resp. ${\bm R}_{ws}$ as orthogonal subrepresentations of $G_{rw}$ resp. $G_{ws}$. 
    
    \item For all $uv\prec rs$ (which implies $r \leq u < v \leq s$), abbreviate $\phi_{uv} = \phi_{uv \to rs}$. Then the range of $\phi_{uv}$ is contained in $R_{rs, -}$ and can be written as 
    \beqn
    \phi_{uv} = \bigoplus_{r<w<s} \phi_{uv, w},\ \phi_{uv, w}: R_{uv\to rs} \to R_{rs, w}.
    \eeqn
    Moreover, for each $w$, the map $\phi_{uv, w}$ satisfies the following conditions. We write $\phi_{uv, w}: R_{uv \to rs} \to R_{rs, w} = R_{rw}^{rs} \oplus R_{ws}^{rs}$ as the sum $\phi_{uv, rw} \oplus \phi_{uv, ws}$.
    \begin{enumerate}
        \item If $w\leq u$, then $\phi_{uv, rw} = 0$. Moreover, over the stratum $V_{rwuvs} \subset V_{ruvs}$, $\phi_{uv, ws}$ is induced from the bundle map 
        \beqn
        \xymatrix{ R_{uv \to ws} \ar[rr]^{\phi_{uv \to ws}} \ar[d] & & R_{ws} \ar[d] \\
                    V_{wuvs} \ar[rr] & & V_{ws} }
        \eeqn
        and the linear inclusion ${\bm R}_{ws} \hookrightarrow {\bm R}_{ws}^{rs}$.
        
        \item If $w \geq v$, then $\phi_{uv, ws} = 0$. Moreover, over the stratum $V_{ruvws} \subset V_{ruvs}$, $\phi_{uv, rw}$ is induced from the bundle map
        \beqn
        \xymatrix{ R_{uv \to rw} \ar[rr]^{\phi_{uv \to rw}} \ar[d] & & R_{rw} \ar[d] \\
                    V_{ruvw} \ar[rr] & & V_{rw}}
        \eeqn
        and the linear inclusion ${\bm R}_{rw}\hookrightarrow {\bm R}_{rw}^{rs}$.
        
        \item If $u < w < v$, then the following is true. Consider the inclusion
        \beqn
        {\bm R}_{uw} \oplus {\bm R}_{wv} \hookrightarrow {\bm R}_{uw}^{uv} \oplus {\bm R}_{wv}^{uv} \hookrightarrow {\bm R}_{uv}.
        \eeqn
        The orthogonal projection 
        \beqn
        {\bm R}_{uv} \to {\bm R}_{uw} \oplus {\bm R}_{wv}
        \eeqn
        is $G_{uwv}$-equivariant, hence induces a bundle map
        \beqn
        R_{uv}|_{V_{uwv}} \to R_{uw} \boxplus R_{wv}.
        \eeqn
        On the other hand, there are bundle embeddings
        \begin{align*}
        &\ R_{uw \to rw} \hookrightarrow R_{rw} \hookrightarrow R_{rw}^{rs},\ &\ R_{wv \to ws} \hookrightarrow R_{ws} \hookrightarrow R_{ws}^{rs}.
        \end{align*}
        Then over the stratum $V_{ruwvs} \subset V_{ruvs}$, the map $\phi_{uv, w}$ is the composition 
        \beqn
        R_{\uds{uv\to rs}}|_{V_{\uds{ruwvs}}} \to R_{\uds{uw\to rs}}\oplus R_{\uds{wv \to rs}} \to R_{rw}^{rs} \oplus R_{ws}^{rs}.
        \eeqn
        
        \item If none of the above happends, then $\phi_{uv, w} = 0$.
    \end{enumerate}
\end{enumerate}
Moreover, the bundles and bundle maps will be added with $\widehat{\cdot}$ when lifted to the stabilizations $\widehat V_{rs}$. Such lifts are canonical.

\vspace{0.2cm}

\noindent {\it Induction Hypothesis II. Fiberwise submersion.} Whenever $d_{rs}< d_{pq}$, we have chosen a $G_{rs}$-equivariant fiberwise submersion (Definition \ref{splitting_defn}) of $\widehat \pi_{rs}: \widehat V_{rs} \to B_{rs}$, which is a microbundle isomorphism
\beqn
T_\mu \widehat V_{rs} \cong T^{\rm vt}_\mu \widehat V_{rs} \oplus \widehat \pi_{rs}^* T_\mu B_{d_{rs}}
\eeqn
which is the canonical projection map on the second component. Together with the structure of $C_{\rm loc}^1$ fiberwise smooth $G_{rs}$-bundle on $\widehat \pi_{rs}$, this provides an isotopy class of vector bundle lifts 
\beq\label{lift}
T^{\rm vt} \widehat V_{rs} \oplus \widehat \pi_{rs}^* TB_{d_{rs}} \to T_\mu \widehat V_{rs},
\eeq
where the isotopy class is fixed by Equation \eqref{eqn:mu-split}.

\vspace{0.2cm}

\noindent {\it Induction Hypothesis III. Stable smoothings and bundle smoothings.} Whenever $d_{rs} < d_{pq}$, we have chosen a $G_{rs}$-invariant smoothing on $\widehat V_{rs}:= V_{rs}\times {\bm R}_{rs}$ in the $G_{rs}$-isotopy class corresponding to stable isotopy class of the vector bundle lift \eqref{lift}. Moreover, we have chosen a structure of smooth $G_{rs}$-equivariant $O(n)$-bundle on $\widehat O_{rs}^{(d)} \to \widehat V_{rs}$ for all $d \geq d_{rs}$ and a structure of smooth $G_{rs}$-equivariant $O(n)$-bundle on $\widehat R_{rs}$.

Before we state the conditions about these structures, we summarize a few consequences of these induction hypotheses. First, the stable smoothings induce, for each stratum $\alpha = pr_1 \cdots r_l q$, a $G_{\alpha}$-equivariant smooth structure on the product
\beqn
\widehat V_\alpha = \widehat V_{pr_1} \times \cdots \times \widehat V_{r_k q}
\eeqn
and hence a $G_{pq}$-equivariant smooth structure on $\widehat V_\alpha^\sim$ by taking the equivariantization. Second, the smooth structures on the vector bundles induce smooth structures on the product bundles $\widehat O_\alpha$ and $\widehat R_\alpha$. Third, as the bundle $Q_{rs}\to V_{rs}$ is trivial, the pullback $\widehat Q_{rs} \to \widehat V_{rs}$ is automatically smooth. Similarly, the bundle 
\beqn
\widehat Q_{rs, \alpha} = (\widehat Q_\alpha)^\bot \to \widehat V_\alpha
\eeqn
is smooth because it comes from a product bundle.

Now we state the conditions.  
\begin{enumerate}
    \item The bundle embedding
\beqn
\widehat \phi_{\beta\alpha}: \widehat R_\alpha \to \widehat R_\beta|_{\partial^\alpha \widehat V_\beta}
\eeqn
is smooth.\footnote{This is not automatic as this bundle embedding is not purely induced from linear maps between representation spaces.} This implies that the orthogonal complement 
\beqn
\widehat R_{rs, \alpha} \to \widehat V_\alpha
\eeqn
is a smooth bundle.

\item Whenever $d' > d \geq d_{rs}$, the natural bundle embedding $\widehat O_{rs}^{(d)} \to \widehat O_{rs}^{(d')}$ induced from \eqref{obstruction_inclusion} is smooth.

\item For all $d \geq d_{rs}$, the bundle embedding 
\beqn
\widehat \phi_{uv \to rs}: \widehat O_{uv \to rs}^{(d)} \to \widehat O_{rs}^{(d)} |_{\partial^{ruvs} V_{rs}}
\eeqn
is smooth. This implies that (when $d = d_{rs}$) the orthogonal complement 
\beqn
\widehat O_{rs,\alpha} \to V_\alpha
\eeqn
is a smooth bundle.

\item It follows that  
\beqn
\widehat F_{rs,\alpha}:= \widehat O_{rs, \alpha} \oplus \widehat Q_{rs,\alpha} \oplus \widehat R_{rs,  \alpha}
\eeqn
over the smooth manifold with corners $\widehat V_\alpha$ is smooth. Then we require that the stabilization map 
\beqn
\widehat \theta_{rs, \alpha}: {\rm Stab}_{\widehat F_{rs, \alpha}} ( \widehat V_\alpha)   \to \partial^\alpha \widehat V_{rs}
\eeqn
is a germ of diffeomorphisms onto an open set. These structures ensure that the projection map 
\beqn
\widehat \pi_{rs, \alpha}: \partial^\alpha V_{rs} \to \widehat V_\alpha^\sim
\eeqn
is smooth.

\item For all $d \geq d_{rs}$, the bundle isomorphism 
\beqn
\widehat \psi_{rs, \alpha}: \widehat \pi_{rs, \alpha}^* \Big( \widehat O_{rs}^{(d)}|_{\widehat V_\alpha^\sim} \Big) \to \widehat O_{rs}^{(d)} |_{\partial^\alpha V_{rs}}
\eeqn
is smooth.

%\noindent {\it Induction Hypothesis VI. Compatibility of the smooth bundle structure.} Whenever $d_{rs} < d_{pq}$, for each stratum $rt_1 \cdots t_l s$, the bundle isomorphism 
%%\beqn
%\widehat O_{rs}|_{\widehat V_{\uds{rt_1 \cdots t_l s}}} \cong \Big( \widehat O_{rt_1} \boxplus \cdots \boxplus \widehat O_{t_l s} \Big) \oplus \widehat O_{\uds{rt_1 \cdots r_k s}}^\bot
%\eeqn
%is smooth. Moreover, for each intermediate stratum $rt_{\beta_1} \cdots t_{\beta_m} s$, the bundle isomorphism
%\beqn
%\widehat O_{\uds{rt_1 \cdots t_l s}}^\bot \cong \Big( \widehat O_{\uds{rt_1 \cdots t_{\beta_1}}}^\bot \boxplus \cdots \boxplus \widehat O_{\uds{t_{\beta_m}\cdots t_l s}}^\bot \Big)\oplus \widehat O_{\uds{rt_{\beta_1} \cdots r_{\beta_m} s}}^\bot|_{\widehat V_{\uds{rt_1 \cdots t_l s}}}
%\eeqn
%is smooth.

%\vspace{0.2cm}

\item The existing smoothings (on both domains and bundles) respect the corner structures. Namely, for all $rs$ with $d_{rs} < d_{pq}$, item (B) of Theorem \ref{smoothing_theorem} is satisfied.
\end{enumerate}
    
\vspace{0.2cm}

Now we start to construct the stabilized smoothings corresponding to the Kuranishi chart $K_{pq}$. 

\vspace{0.2cm}

\noindent {\it Inductive Construction I. We construct an orthogonal $G_{pq}$-space ${\bm R}_{pq}$ (which defines the trivial bundle $R_{pq} \to V_{pq}$) and bundle embeddings 
\beqn
\phi_{rs \to pq}: R_{rs \to pq} \to R_{pq}|_{\partial^{prsq} V_{pq}}
\eeqn
which, together with the existing stabilizations and bundle embeddings, satisfy the conditions required for stabilizations.}

\vspace{0.2cm}

\begin{enumerate}

\item Firstly, we need to find an orthogonal $G_{pq}$-representation ${\bm R}_{pq}$ such that the stabilization $R_{pq} = V_{pq} \times {\bm R}_{pq}$ receives embeddings from $R_\alpha$ for all $\alpha \in \bA^\floer_{pq}$.

To start, for each $p<w<q$, consider the $G_{pw}$-representation ${\bm R}_{pw}$ and the $G_{wq}$-representation ${\bm R}_{wq}$ granted by the induction hypothesis. By Frobenius reciprocity, there exist an orthogonal $G_{pq}$-representation ${\bm R}_{pw}^{pq}$ which contains ${\bm R}_{pw}$ as a subrepresentation of $G_{pw}$, and an orthogonal $G_{pq}$-representation ${\bm R}_{wq}^{pq}$ which contains ${\bm R}_{wq}$ as a subrepresentation of $G_{wq}$. Then define 
\beq\label{decomposition}
{\bm R}_{pq}:= \bigoplus_{p<w<q} {\bm R}_{pw}^{pq}\oplus {\bm R}_{wq}^{pq}
\eeq
which is an orthogonal $G_{pq}$-space. Now we define the bundle embeddings 
\beqn
\phi_{rs\to pq}: R_{rs \to pq} \to R_{pq}|_{\partial^{prsq} V_{pq}}.
\eeqn
We abbreviate the decomposition \ref{decomposition} as 
\beqn
{\bm R}_{pq} = \bigoplus_w {\bm R}_{pq, w} = \bigoplus_{p<w<q} {\bm R}_{pw}^{pq}\oplus {\bm R}_{wq}^{pq}.
\eeqn
Abbreviate the bundle embedding to be defined by $\phi_{rs}$. Then we define it to be the direct sum
\beqn
\phi_{rs} = \bigoplus_w \phi_{rs, w} = \bigoplus_w \phi_{rs, pw} \oplus \phi_{rs, wq}\ {\rm where}\ \phi_{rs, w}: R_{rs\to pq} \to R_{pq, w}|_{\partial^{prsq} V_{pq}}.
\eeqn
In fact, for each $w$, one can define the restriction of $\phi_{rs, w}$ to the corresponding stratum involving $w$, such as $pwrsq$, $prswq$, or $prwsq$ in the form described in {\it Induction Hypothesis I}. These are all codimension-1 boundary strata of $prsq$. Then using the collar structure of existing objects and a cut-off function which only depends on the collar coordinate, one can turn off the corresponding maps $\phi_{rs, w}$ once we go away from this corner. On the other hand, if $w = r$ or $w = s$, we just have a bundle embedding coming from the induction hypothesis. One can check that the newly constructed objects still satisfy the requirement for a stabilization up to level $d_{pq}$ (see Definition \ref{defn_stabilization}). 

\item Secondly, we would like to construct an inner product on $R_{pq}$. 

We start from a minimal $\alpha = pr_1 \cdots r_l q$. Notice that inside $R_{pq}|_{\partial^\alpha V_{pq}}$, there are a collection of embedded subbundles $R_{\beta\to pq} |_{\partial^\alpha V_{pq}}$ for all $\beta$ with $\alpha \leq \beta < pq$ with an inner product equipped. We first check that, for any pair $\beta, \beta'$ of such strata, their inner products agree on the overlap 
\beqn
R_{\beta\to pq} \cap R_{\beta' \to pq}\subset R_{pq}.
\eeqn
In fact the intersection $\partial^\beta V_{pq}\cap \partial^{\beta'} V_{pq}$ is a deeper stratum $\partial^{\beta\# \beta'} V_{pq}$. Here $\beta\# \beta'$ stands for the stratum described by the word which includes all the intermediate capped orbits in $\beta$ and $\beta'$ between $p$ and $q$. We can in fact check from the explicit construction that 
\beqn
R_{\beta \to pq} \cap R_{\beta' \to pq} = R_{\beta\# \beta' \to pq}
\eeqn
Hence by induction hypothesis, the inner products agrees on overlaps. 

Now there are two special codimension one strata, $\beta_1 = pr_1 q$ and $\beta_l = pr_l q$. We can check that 
\beqn
R_{\beta \to pq} \subset R_{\beta_1 \to p q} + R_{\beta_l \to pq}.
\eeqn
Hence we can define an inner product on the sum $R_{\beta_1 \to pq} + R_{\beta_l \to pq}$ which extends the existing ones. Then extend arbitrarily to $R_{pq}$ over this stratum. 

Inductively, we can use the collar structure near $\alpha$ to extend the inner product on $\partial^\alpha V_{pq}$ to a neighborhood. Then the same argument above can be applied to construct an inner product inductively on $R_{pq}$. 

\end{enumerate}

\vspace{0.2cm}

\noindent {\it Inductive Construction II. We construct smooth bundle structures on the obstruction bundles $\widehat O_{pq}^{(d)} \to \widehat V_{pq}$ for all $d\geq d_{pq}$ over $\partial \widehat V_{pq}$.} The construction is inductive on strata $\alpha = pr_1 \cdots r_l q$ of $pq$. We start with a minimal $\alpha$. Then there are the vector bundles
\beqn
O_\alpha^{(d)} \to V_\alpha = V_{pr_1} \times \cdots \times V_{r_l q},\ \forall d \geq d_{pq}.
\eeqn
Denote by its pullback to $\widehat V_\alpha$ by 
\beqn
\widehat O_\alpha^{(d)} \to \widehat V_\alpha.
\eeqn
The induction hypothesis granted a smooth $O(n)$-bundle structure on it which is the product of each individual factor. Notice that this is a subbundle of $\widehat O_{pq}^{(d)} |_{\widehat V_\alpha}$. We would like to extend this smooth structure to a smooth structure on $\widehat O_{pq}|_{\widehat V_\alpha}$. 

Start with $d = d_0 = d_{pq}$ and remember $\widehat O_{pq} = \widehat O_{pq}^{(d_0)}$. Consider any intermediate stratum $\beta = ps_1 \cdots s_m q$. Then one has the inclusion induced by \eqref{eqn:o-embedding}
\beqn
\widehat O_\alpha^{(d_0)} \subset \widehat O_\beta^{(d_0)} |_{\widehat V_\alpha} \subset \widehat O_{pq}|_{\widehat V_\alpha}.
\eeqn
Notice that by the induction hypothesis, the smooth structure on $\widehat O_\beta^{(d_0)} |_{\widehat V_\alpha}$ extends the smooth structure on $\widehat O_\alpha^{(d_0)}$. This requires that the smooth structure on $\widehat O_{pq}|_{\widehat V_\alpha}$ needs to be an extension of all such smooth structures. Indeed, consider two codimension one strata, $pr_1 q$ and $pr_l q$. Consider the two bundles 
\begin{align*}
&\ \widehat O_{pr_l \to pq}^{(d_0)} |_{\widehat V_\alpha},\ &\ \widehat O_{r_1 q \to pq}^{(d_0)} |_{\widehat V_\alpha}.
\end{align*}
Abbreviate them temporarily by $\widehat O_{r_l}^{(d_0)}$ and $\widehat O_{r_1}^{(d_0)}$ respectively. Then it is easy to see that for any intermediate stratum $\beta$, one has 
\beq\label{eqn:sum}
\widehat O_\beta^{(d_0)} |_{\widehat V_\alpha} \subset \widehat O_{r_l}^{(d_0)} + \widehat O_{r_1}^{(d_0)}.
\eeq
Then Lemma \ref{lemma637} below provides a smooth structure on the sum $\widehat O_{r_l}^{(d_0)} + \widehat O_{r_1}^{(d_0)}$. Then we extend this extension to $\widehat O_{pq}|_{\widehat V_\alpha}$ by choosing an arbitrary smooth structure on its orthogonal complement. 

\begin{lemma}\label{lemma637}
Let $M$ be a smooth manifold and let $E \to M$ be a topological vector bundle equipped with an inner product. Let $E_1, E_2 \subset E$ be subbundles such that $E = E_1 + E_2$. Moreover, suppose $E_1$ and $E_2$ are equipped with structures of smooth vector bundles with respect to which the restriction of the inner product is smooth, such that the intersection $E_1 \cap E_2$ is both a smooth subbundle of $E_1$ and a smooth subbundle of $E_2$. Then there exists a unique smooth bundle structure of $E$ which extends the smooth bundle structures on $E_1$ and $E_2$ such that the inner product is smooth.
\end{lemma}

\begin{proof}
Define a smooth vector bundle $E'\to M$ from $E_1$ and $E_2$ as 
\beqn
E':= E_1' \oplus E_2' \oplus (E_1 \cap E_2)
\eeqn
where $E_1'$ resp. $E_2'$ is the orthogonal complement of $E_1 \cap E_2$ in $E_1$ resp. $E_2$. Then there is an canonical isometric bundle isomorphism $E' \to E$. 
\end{proof}

As $\widehat O_{pq, \alpha}$ is the orthogonal complement of $\widehat O_\alpha^{(d_0)}$, the above smooth structure induces a smooth structure on $\widehat O_{pq} |_{\widehat V_\alpha} \to \widehat V_\alpha$, hence one on $\widehat O_{pq,\alpha} \to \widehat V_{\alpha}$. On the other hand, run the same argument as above, one can cook up a smooth $O(n)$-bundle structure on $\widehat R_{pq}|_{\widehat V_\alpha}$ which extends existing ones. Hence the orthogonal complement $\widehat R_{pq, \alpha}$ is a smooth $O(n)$-bundle. The bundle $\widehat Q_{pq, \alpha}$, which is essentially trivial, automatically carries a smooth structure. Hence the bundle 
\beq\label{eqn:smooth-scaffold}
\widehat F_{pq, \alpha} = \widehat O_{pq, \alpha} \oplus \widehat Q_{pq, \alpha} \oplus \widehat R_{pq, \alpha}
\eeq
is smooth. Then the stabilization map (see \eqref{eqn:sca-K})
\beqn
\widehat \theta_{pq, \alpha}: {\rm Stab}_{\widehat F_{pq, \alpha}} ( \widehat V_\alpha)  \to \partial^\alpha \widehat V_{pq}
\eeqn
and its equivariantization induce a smooth structure on the stratum $\partial^\alpha \widehat V_{pq}$. 

Next we can inductively construct smooth structures on the bundle $\widehat O_{pq}^{(d)}|_{V_{pq, \uds\alpha}}$ for all $d \geq d_0 = d_{pq}$. This is similar to the construction of inner products. We omit the details. Using equivariantizations and the bundle isomorphisms coming from the stabilization map associated to $\widehat F_{pq, \alpha}$, one obtains smooth structures on the bundle $\widehat O_{pq}^{(d)} |_{\partial^\alpha \widehat V_{pq}}$ for all $d \geq d_{pq}$. Using the collar structure, we extend slightly to a neighborhood of $\partial^\alpha \widehat V_{pq}$ inside $\widehat V_{pq}$. 

We can carry on the induction for all strata $\alpha = pr_1 \cdots r_l q$. Suppose we have constructed a smooth structure near the boundary of $\partial^\alpha \widehat V_{pq}$ and smooth bundle structures on $\widehat O_{pq}^{(d)}$ over this neighborhood. Now start with $d  = d_0 = d_{pq}$ and we would like to construct the smooth structure of $\partial^\alpha \widehat V_{pq}$ and the smooth bundle structure on $\widehat O_{pq}^{(d_0)}|_{\partial^\alpha \widehat V_{pq}}$. First, the induction hypothesis granted a smooth structure on the product $\widehat V_{\alpha}$. This is compatible with all the smooth structures near the boundary of $\partial^\alpha \widehat V_{pq}$. Then similar to above, one can have a smooth structure on the sum (the right hand side of \eqref{eqn:sum}). Then extend to a smooth structure on $\widehat O_{pq}|_{\partial^\alpha \widehat V_{pq}}$. The stabilization map then grants a smooth structure on $\partial^\alpha \widehat V_{pq}$ which extends the existing one near the boundary. The bundle isomorphism provides a smooth bundle structure of $\widehat O_{pq}|_{\partial^\alpha \widehat V_{pq}}$. Inductively, one extends the smooth bundle structure to $\widehat O_{pq}^{(d)}|_{\partial^\alpha \widehat V_{pq}}$ for all $d \geq d_0$. Then use the collar structure we extend the structures slightly into the interior of $\widehat V_{pq}$.

\vspace{0.2cm}

\noindent {\it Inductive Construction III. We construct a $G_{pq}$-equivariant vector bundle reduction of $T_\mu \widehat V_{pq}$ which is smooth near the boundary and is in the the isotopy class of the vector bundle reduction induced from the existing smoothings.} The previous construction induces a $G_{pq}$-smoothing on a neighborhood $N_\epsilon(\partial \widehat V_{pq})$ of the boundary, using the smooth structures on bundles of the form \eqref{eqn:smooth-scaffold} and equivariantization. Hence there is a corresponding (isotopy class) of vector bundle lift of the tangent microbundle $T_\mu \widehat V_{pq}$ restricted to this region. From the construction we can see that 
\beqn
T \left( N_\epsilon (\partial \widehat V_{pq}) \right) \cong \left( T^{\rm vt} \widehat V_{pq} \oplus \widehat \pi_{pq}^* T B_{d_{pq}}\right)|_{N_\epsilon( \partial \widehat V_{pq})}.
\eeqn 
On the other hand, we know that via a fiberwise submersion on $\widehat V_{pq}$ one can obtain another vector bundle lift of $T_\mu \widehat V_{pq}$ by the same vector bundle.

We claim that these two vector bundle lifts, when restricted to their common domains, are in the same stable $G_{pq}$-isotopy class of vector bundle lifts. We show why this is the case if $\alpha$ is a codimension 1 stratum. In this case, we know that 
\beqn
\partial^\alpha V_{pq} = {\rm Stab}_{F_{pq, \alpha}} (V_\alpha)
\eeqn
i.e., the boundary $\partial^\alpha V_{pq}$ is a disk bundle over $V_\alpha$. We can choose an (invariant) connection on $F_{pq, \alpha}$ to induce a smooth splitting 
\beqn
T F_{pq, \alpha} \cong T^{\rm vt} F_{pq, \alpha} \oplus \pi_{F_{pq, \alpha}}^* TV_\alpha.
\eeqn
It also induces a corresponding splitting on the microbundle level. Although the fibers of the projection onto $B_{\delta(\alpha)}$ are no longer smooth, the microbundle version of the above splitting can still be restricted to each fiber. Then combining with the microbundle splitting of $T_\mu V_\alpha$ one obtains
\beqn
T_\mu (\partial^\alpha V_{pq}) \cong \pi_{F_{pq,\alpha}}^* (F_{pq, \alpha})_\mu \oplus \pi_{F_{pq, \alpha}}^* T_\mu V_\alpha \cong T^{\rm vt}_\mu (\partial^\alpha V_{pq}) \oplus \pi^* T_\mu B_{\delta(\alpha)}.
\eeqn
Notice that this is an equivariant fiberwise submersion of $\partial^\alpha V_{pq} \to \partial^{\delta(\alpha)} B_d$ (see Definition \ref{splitting_defn}). The collar structure can be used to extend this fiberwise submersion trivially into the collar region. Using Lemma \ref{splitting_lemma} one can extend it globally to $V_{pq}$. Then using the fiberwise smooth structure we obtained a vector bundle lift of $T_\mu V_{pq}$. 

On the other hand, the vector bundle lift on $V_\alpha$ induced from smoothing is in the same stable isotopy class of vector bundle lifts as the one from fiberwise smooth structure and microbundle splitting. Therefore, one can find another orthogonal $G_{pq}$-space ${\bm R}_{pq}'$ such that the ${\bm R}_{pq}'$-stabilization of the interior vector bundle lift of $T_\mu\widehat V_{pq}$ is isotopic to the one from the boundary smooth structure. Also notice that one can choose a $G_{pq}$-invariant continuous cut-off function on the overlap of the two domains of the vector bundle lifts to interpolate this two vector bundle lifts using the isotopy. Note that this is one salient feature of the existence of collar structures: we are free to interpolate between different data over the collar regions. Therefore, we have obtained a vector bundle lift of $T_\mu \widehat V_{pq}$ which is smooth near the boundary. 

In general, when $\alpha$ is not a codimension 1 stratum, there are many intermediate stratum. However, one can choose connections on all the intermediate difference bundles $F_{\beta\alpha}$ compatibly, utilizing Lemma \ref{obundle_property}. This still allows us to build the interior microbundle splitting which extends the boundary ones. We omit the details.

\vspace{0.2cm}

\noindent {\it Inductive Construction IV. Smoothing for $pq$.} In the previous steps, we have obtained a smooth structure on a neighborhood of $\partial \widehat V_{pq}$, a smooth $O(n)$-bundle structure of $\widehat O_{pq}^{(d)}$ over this neighborhood, and a $G_{pq}$-equivariant vector bundle reduction 
\beqn
\rho_{pq}: (T^{\rm vt} \widehat V_{pq} \oplus \widehat \pi_{pq}^* B_{pq} )_\mu \to T_\mu \widehat V_{pq}
\eeqn
which is smooth near $\partial \widehat V_{pq}$. Then one can apply the relative version of Lashof's $G$-smoothing theorem (see Appendix \ref{appendixb} and Theorem \ref{thmb3}) to obtain another $G_{pq}$-orthogonal space ${\bm R}_{pq, 0}$ and a $G_{pq}$-smoothing on ${\rm Int} \widehat{V_{pq}}\times {\bm R}_{pq, 0}$ which coincides with the stabilization of the existing smoothing. Redefine ${\bm R}_{pq}$ by taking direct sum with ${\bm R}_{pq, 0}$ and redefine $\widehat V_{pq}$, $\widehat O_{pq}^{(d)}$ etc. Moreover, one can extend the smooth structures on $\widehat O_{pq}^{(d)}$ inductively to the interior of $\widehat V_{pq}$. Furthermore, we extend the existing inner product on the trivial bundle $R_{pq}$ to the whole $\widehat V_{pq}$ smoothly.

\vspace{0.2cm}

Finally, we see that we have constructed a smoothing for any K-chart $K_{pq}$ with energy $d_{pq}$. Furthermore, they satisfy the induction hypotheses listed at the beginning. Therefore, Theorem \ref{smoothing_theorem} holds. \qed

\section{Constructions for PSS, SSP, and the homotopy}\label{sec:pss}

In this subsection, we discuss K-charts for PSS, SSP, and continuation type moduli spaces. The constructions here are carried out almost verbatim as the case of moduli spaces of Floer trajectories. Therefore, our main purpose here is to fix the notations, and only the key modifications are presented in detail.

\subsection{PSS bimodule and SSP bimodule}

\subsubsection{Thimbles and auxiliary moduli spaces}

\begin{defn}\label{def:thimble}\hfill
\begin{enumerate}
    \item A \textbf{prestable PSS thimble} is a triple $(\Sigma, \Sigma_{\text{PSS}}, {\bf L})$ where $\Sigma$ is a genus $0$ prestable curve with two marked points $z_-, z_+$, with $\Sigma_{\text{PSS}}$ being a distinguished horizontal irreducible component of $\Sigma$, and ${\bf L}=(\L_i)$ where $\L_i$ is a lateral line on each horizontal component lying between $\Sigma_{\text{PSS}}$ (included!) and the marked point $z_+$ (see Figure \ref{fig:my_label}).
    
    \begin{figure}[h]
        \centering
        \includegraphics{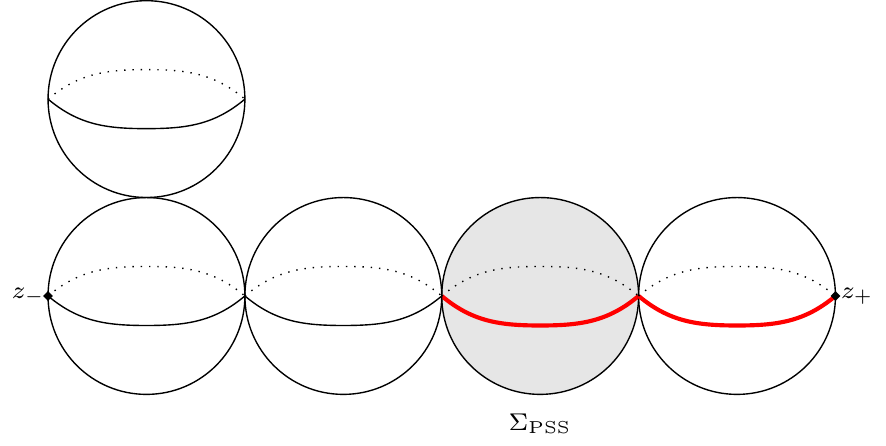}
        \caption{An example of prestable PSS thimble. The red curves are the lateral lines and the gray sphere is $\Sigma_{\rm PSS}$ on which we interpolate between the Hamiltonian $H$ and a constant function.}
        \label{fig:my_label}
    \end{figure}
    \item A \textbf{prestable SSP thimble} is a triple $(\Sigma, \Sigma_{\text{SSP}}, {\bf L})$ where $\Sigma$ is a genus $0$ prestable curve with two marked points $z_-, z_+$, with $\Sigma_{\text{SSP}}$ being a distinguished horizontal irreducible component of $\Sigma$, and ${\bf L}=(\L_i)$ where $\L_i$ is a lateral line on each horizontal component lying between the marked point $z_-$ and $\Sigma_{\text{SSP}}$.
\end{enumerate}
\end{defn}

By imposing extra marked points and stability conditions, one can define the notion of stable marked PSS/SSP thimble as in Definition \ref{stablecylinder}. Given an integer $d>0$, the moduli space of stable PSS resp. SSP thimbles in $\mb{CP}^d$ is the set of equivalence classes of objects
\beqn
(\Sigma, \Sigma_{\rm PSS}, {\bf L}, u),\ \ \ \ \ {\rm resp.}\ \ \ \ \ (\Sigma, \Sigma_{\rm SSP}, {\bf L}, u)
\eeqn
where $(\Sigma, \Sigma_{\text{PSS}} / \Sigma_{\text{SSP}}, {\bf L})$ is a prestable PSS/SSP thimble and $u: \Sigma \to \mb{CP}^d$ is a stable map. We define ${\mc F}^{\text{PSS}}_{0,2}(d)$ to be the subset of moduli spaces of stable PSS thimbles such that the underlying map $u: \Sigma \to \mb{CP}^d$ is not contained in any hyperplane, and ${\mc F}^{\text{SSP}}_{0,2}(d)$ can be defined similarly. These two moduli spaces are smooth manifolds with corners. Just as Lemma \ref{lem:r-blowup}, the following statements hold.
\begin{enumerate}
    \item The natural forgetful map ${\mc F}^{\text{PSS}}_{0,2}(d), {\mc F}^{\text{SSP}}_{0,2}(d) \to {\mc F}_{0,2}(d)$ is smooth.
    \item The evaluation maps $\text{ev}_{\pm}$ at the marked points $z_{\pm}$ are smooth and stratified submersive onto $\mb{CP}^d$.
    \item By intersecting with $d(d+2)$ generic hyperplanes in $\mb{CP}^d$, the moduli spaces ${\mc F}^{\text{PSS}}_{0,2}(d), {\mc F}^{\text{SSP}}_{0,2}(d)$ are locally diffeomorphic to a real blowup of the Deligne--Mumford space $\ov{\mc{M}}_{0,d(d+2)}$.
\end{enumerate}
Note that we do not necessarily equip all the horizontal components of $(\Sigma, {\bf L})$ with a lateral line, so the real blowup in (3) is different from the one in Lemma \ref{lem:r-blowup}.

The spaces ${\mc F}_{0,2}^{\rm PSS}(d)$ and ${\mc F}_{0,2}^{\rm SSP}(d)$ are stratified by the same poset $\bA_d$ of ordered partitions of $d$ (with the same depth function, see Notation \ref{notation:G}). However, the geometric meaning of an element $\delta = (d_0, \ldots, d_l) \in \bA_d$ is different. For example, in the PSS case, the integer $d_0$ represents the total degree of the components between $z_-$ and $\Sigma_{\rm PSS}$ (included). 

\begin{notation}

\begin{enumerate}

\item Let $p,q$ be capped 1-periodic capped Hamiltonian orbits. Introduce the integers
\begin{align*}
&\ d_{\bullet q}:= {\mc A}_H(q) - C^\pss,\ &\ d_{p\bullet}:= C^\ssp - {\mc A}_H(p)
\end{align*}
where $C^\pss$ resp. $C^\ssp$ is the constant chosen to define the $\pss$ resp. $\ssp$ moduli spaces (see Subsection \ref{subsection_pss}). 

\item Extra symmetry groups
\beqn
G^\pss_{\bullet q}:= \{ g \in PU(d_{\bullet q}+1) \ |\ g([1, 0, \ldots, 0]) = [1, 0, \ldots, 0]\in \mb{CP}^{d_{\bullet p}}  \} \cong U(d_{\bullet p})
\eeqn
\beqn
G_{p \bullet}^\ssp:= \{ g \in PU(d_{p \bullet} +1)\ |\ g([1, 0, \ldots, 0]) = [1, 0, \ldots, 0]\in \mb{CP}^{d_{p \bullet}} \} \cong U(d_{p\bullet} )
\eeqn
where the isomorphisms are induced from \eqref{eqn52}.

\item Homogeneous posets with partial order induced from refinement of words
$$
\bA_{\bullet q}^\pss := \{ \bullet r_1 \cdots r_l q | r_1 < \cdots < r_l < p \text{ as objects in } T^\floer\}
$$
$$
\bA_{p \bullet}^\ssp := \{ p r_1 \cdots r_l  \bullet | p < r_1 < \cdots < r_l \text{ as objects in } T^\floer\},
$$
with depth functions
$$
\dep(\bullet r_1 \cdots r_l q) = l, \ \dep(p r_1 \cdots r_l  \bullet) = l.
$$
The concatenation of words induces maps between homogeneous posets
$$
\bA_{\bullet p}^\pss \times \bA^\floer_{pq} \to \bA_{\bullet q}^\pss
$$
$$
\bA^\floer_{pq} \times \bA_{q \bullet}^\ssp \times  \to \bA_{p \bullet}^\ssp.
$$
There are natural maps of posets 
$$
\delta^\pss: \bA_{\bullet q}^\pss \to \bA_{d_{\bullet q}}, \ \delta^\ssp: \bA_{p \bullet}^\ssp \to \bA_{d_{p \bullet}}
$$
defined by
$$
\delta^\pss(\bullet r_1 \cdots r_l q) := (d_{\bullet r_1}, \dots, d_{r_l q}), \ \delta^\ssp(p r_1 \cdots r_l \bullet) := (d_{p r_1}, \dots, d_{r_l \bullet}),
$$
such that the following diagrams commute.
\beqn
    \vcenter{ \xymatrix{ \bA_{\bullet p}^\pss \times \bA^\floer_{pq} \ar[r] \ar[d]_{\delta^\pss \times \delta} & \bA_{\bullet q}^\pss  \ar[d]^{\delta^\pss}\\
                \bA_{d_{\bullet p}} \times \bA_{d_{pq}} \ar[r] & \bA_{d_{\bullet q}} }  }
\eeqn
\beqn
    \vcenter{ \xymatrix{ \bA^\floer_{pq} \times \bA_{q \bullet}^\ssp \ar[r] \ar[d]_{\delta \times \delta^\ssp} & \bA_{p \bullet}^\ssp  \ar[d]^{\delta^\pss}\\
                \bA_{d_{pq}} \times \bA^\ssp_{d_{q \bullet}} \ar[r] & \bA_{d_{p \bullet}} }  }
\eeqn

%\item A poset $\langle \mc{P} \rangle_{\bullet q}$ with elements given by strings of the form $\bullet r_1 \cdots r_l q$ for $r_1 < \cdots < r_l < q$ in ${\mc P}$ with partial order induced by contraction of words. In particular, $\langle \mc{P} \rangle_{\bullet q}$ has a maximal element $\bullet q$. Similarly, a poset $\langle \mc{P} \rangle_{p \bullet}$ with elements of the form $p r_1 \cdots r_l \bullet$ where $p<r_1 < \cdots < r_l$ in ${\mc P}$, and the partial order is also induced by contraction of words. $\langle \mc{P} \rangle_{p \bullet}$ has maximal element $p \bullet$.

\item For an index $\alpha = \bullet r_1 \cdots r_l p \in \bA^\pss_{\bullet q}$, there is an embedding of Lie groups
$$ G^{\text{PSS}}_{\alpha} := G^{\text{PSS}}_{\bullet r_1} \times \cdots \times G_{r_l q} \hookrightarrow G^{\text{PSS}}_{\bullet q} = G^{\text{PSS}}_{\bullet q} $$
as from Notation \ref{notation:G}, which further induces an embedding $G^{\text{PSS}}_{\alpha} \hookrightarrow G^{\text{PSS}}_{\beta}$ for any pair $\alpha \leq \beta$ in $\bA_{\bullet p}^\pss$. Analogously, there are embeddings of Lie groups $G^{\text{SSP}}_{\alpha} \hookrightarrow G^{\text{SSP}}_{\beta}$ where $\alpha \leq \beta$ are in $\bA^\ssp_{p \bullet}$.
\item The system of auxiliary moduli spaces
\beqn
B^{\text{PSS}}_{d}:= \big\{ x = [\Sigma, \Sigma_{\text{PSS}}, {\bf L},  u] \in {\mc F}^{\text{PSS}}_{0,2}(d)\ |\ \ev_-(x) = [1, 0, \ldots, 0] \in \mb{CP}^{d} \big\},
\eeqn
\beqn
B^{\text{SSP}}_{d}:= \big\{ x = [\Sigma, \Sigma_{\text{SSP}}, {\bf L},  u] \in {\mc F}^{\text{SSP}}_{0,2}(d)\ |\ \ev_-(x) = [1, 0, \ldots, 0] \in \mb{CP}^{d} \big\},
\eeqn
where $\ev_{-}(x)$ is the image of the evaluation map of $u$ at the marked point $z_-$.
Each $B^{\text{PSS}}_d$ resp. $B^{\text{SSP}}_d$ is a smooth $\bA_d$-manifold and has a $G_d$-action. 

\item For $\alpha = \bullet r_1 \cdots r_l q \in \bA^\pss_{\bullet q}$ (resp. $p r_1 \cdots r_l \bullet \in \bA^\ssp_{p \bullet}$), define
$$
B^\pss_{\alpha} := B^\pss_{d_{\bullet r_1}} \times \cdots \times B_{d_{r_l q}}
$$
$$
\text{resp. } B^\ssp_{\alpha} := B_{d_{p r_1}} \times \cdots \times B^\ssp_{d_{r_l \bullet}}.
$$
Then $B^\pss_{\alpha}$ (resp. $B^\ssp_{\alpha}$) has a smooth $G_{\alpha}^\pss$ (resp. $G_{\alpha}^\ssp$) action.

\item Just as Equation \eqref{eqn:embedding}, under the group embedding $G^\pss_{\alpha} \hookrightarrow G_{\bullet q}^\pss$, there is an equivariant embedding
$$
B^\pss_{\alpha} \hookrightarrow B^\pss_{\bullet q},
$$
which is part of the data of the embeddings of the form
$$
B^\pss_{\alpha} \hookrightarrow B^\pss_\beta
$$
for any $\alpha \leq \beta$ under the group embedding $G_{\alpha}^\pss \hookrightarrow G_{\beta}^\pss$. This fact is also true after replacing PSS by SSP.
\end{enumerate}
\end{notation}
 
\subsubsection{Description of global charts}
The following statement allows to construct framings from solutions to the PSS/SSP equations. The proof is the same as the proof of Corollary \ref{HK2} using the integrality assumptions.

\begin{prop}
Let $(\Sigma, \Sigma_{\rm PSS}, {\bm L})$ be a prestable PSS thimble, and suppose that
$$
u: \Sigma \to M
$$
is a {\bf stable PSS thimble}, i.e. the restriction of $u$ to the irreducible component $\Sigma_{\rm PSS}$ solves the PSS equation \eqref{eqn:pss}, its restriction to the cylindrical components on the right of $\Sigma_{\rm PSS}$ solves the Floer equation \eqref{Floereqn}, while its restriction to other components solves the genuine $J$-holomorphic curve equation, and $u$ has finite automorphism group. Consider the $2$-form $\Omega_{u, \pss}$ defined by
\beq\label{eqn:pss-2-form}
\Omega_{u, \pss} := \left\{ \begin{array}{lr} 
\displaystyle u^* \omega - d( H_{s,t}^\pss (u) dt) & \text{ on the distinguished component $\Sigma_\pss$} \\
\displaystyle  u^* \omega - d( H_t (u) dt) & \text{ on cylindrical components on the right of $\Sigma_\pss$} \\
\displaystyle u^* \omega & \text{ on other irreducible components.} 
\end{array} \right.
\eeq
Then there exists a holomorphic Hermitian line bundle on $\Sigma$ whose curvature form is $-2 \pi {\bm i}\Omega_{u, \pss}$. Moreover, this line bundle is unique up to isomorphism.

Similarly, if $(\Sigma, \Sigma_{\rm SSP}, {\bm L})$ is a prestable SSP thimble and $u: \Sigma \to M$ is a {\bf stable SSP thimble} which solves the SSP equation \eqref{eqn:ssp} on $\Sigma_{\ssp}$, there exists a Hermitian line bundle over $\Sigma$, unique up to isomorphism, such that the curvature of its Chern connection is equal to $-2 \pi {\bm i}\Omega_{u, \pss}$, where
\beqn
\Omega_{u, \ssp} := \left\{ \begin{array}{lr} 
\displaystyle u^* \omega - d( H_{s,t}^\ssp (u) dt) & \text{ on the distinguished component $\Sigma_\ssp$} \\
\displaystyle  u^* \omega - d( H_t (u) dt) & \text{ on cylindrical components on the left of $\Sigma_\ssp$} \\
\displaystyle u^* \omega & \text{ on other irreducible components.} \qed
\end{array} \right.
\eeqn
\end{prop}

Given a stable PSS (resp. SSP) thimble, let us denote the line bundle constructed as above by $L_{u, \pss}$ (resp. $L_{u,\ssp}$). Now we can introduce the concept of framed curves in the context of PSS/SSP map.

\begin{defn}
Given a moduli space $\ov{\mc M}{}^\pss_{\bullet q}$ of stable PSS thimbles and a stratum $\alpha = \bullet r_1 \cdots r_l q$ which corresponds to thimbles breaking at the orbits $r_1, \dots, r_l$, denoting $d=d_{\bullet q}$, a {\bf framed PSS thimble} (of type $\alpha$) is a tuple $(u, \Sigma, F)$ where 
\begin{enumerate}
    \item $\Sigma$ is a prestable PSS thimble with $l$ horizontal levels on the right of $\Sigma_\pss$. (see Definition \ref{def:thimble}). 
    
    \item $u: \Sigma \to M$ is a smooth map whose restriction to each cylindrical component on the right of $\Sigma_\pss$ (included) converges to periodic orbits prescribed by the capped orbits $p, r_1,\dots, r_l, q$ at $\pm\infty$ in an exponential rate, and the topological energy of these levels are prescribed accordingly. Moreover, the $2$-form $\Omega_{u, \pss}$ defined in Equation \eqref{eqn:pss-2-form} is non-negative and is strictly positive on each unstable component of $u$.
    
    %\item A Hermitian metric on $L_u$ whose curvature form is equal to $-2\pi {\bm i} \Omega_{u, H}$ (see \eqref{2form}) (the choice is unique up to a constant factor). We suppress this datum from the notation.
    
    \item $F = (f_0, f_1, \ldots, f_d)$ is basis %\footnote{We do not require that the list is a basis at this moment.}
    of global sections of the line bundle $L_{u,\pss}$ constructed using the $2$-form \eqref{eqn:pss-2-form}. Moreover, the induced holomorphic map 
    \beq
    \iota_{F, \pss}: \Sigma \to \mb{CP}^d,\ w \mapsto [f_0(w), \ldots, f_d(w)]
    \eeq
    is a stable map, which represents a point $[\iota_{F,\pss}] \in {\mc F}^\pss_{0,2}(d)$. Denote by 
    \beqn
    \tilde \iota_{F, \pss}: \Sigma \to {\mc C}^\pss = {\mc C}^\pss_{0,2}(d)
    \eeqn
    to be the identification between $\Sigma$ and the fiber of the universal curve ${\mc C}^\pss_{0,2}(d) \to {\mc F}^\pss_{0,2}(d)$ over $[\iota_{F,\pss}]$. This will also be called the domain map.
\end{enumerate}
\end{defn}

By replacing PSS with SSP, we obtain the definition of {\bf framed PSS thimble}. Isomorphisms of framed PSS/SSP thimbles are defined in the same way as Definition \ref{defn:isomorphism_framed_cylinder}. The compact Lie group $PU(d+1)$ acts on the space of framed PSS/SSP thimbles by changing the framing. The stabilizer of this $PU(d+1)$ action at $(u, \Sigma, F)$ agrees with the automorphism group of the map $u: \Sigma \to M$ if $u$ is a stable PSS/SSP thimble.

Now we can repeat the global thickening construction in the setting of PSS/SSP maps as we did in the case for Floer trajectories. Let $p$ be a capped $1$-periodic orbit. Fix $d = d_{\bullet q}$. Let 
\beqn
k_1 < \cdots < k_d
\eeqn
be a sequence of positive integers. Define $V_{\bullet q}^\pss:= V_{\bullet q}^\pss (k_1, \ldots, k_d)$ be the moduli space of tuples
\beqn
(\Sigma, u, F, \eta_1, \ldots, \eta_d)
\eeqn
where $\Sigma$ is a prestable PSS thimble, $u: \Sigma \to M$ is a smooth map with topological energy $d$ converging to $q$ at the marked point $z_+$, $F$ is a frame of the line bundle $L_{u,\pss}\to \Sigma$ inducing a holomorphic map $\iota_{F,\pss}: \Sigma \to \mb{CP}^d$ such that the framing $F=(f_0, \dots, f_d)$ satisfies $[f_0(z_-): \cdots : f_d(z_-)] = [1:0:\cdots :0]$, and the corresponding domain map $\tilde \iota_{F,\pss}: \Sigma \to {\mc C}^\pss$, and (denoting by $\tilde u = (u, \tilde \iota_{F, \pss}): \Sigma \to M \times {\mc C}^\pss$)
\beqn
\eta_i\in \tilde{u}^* E_i = \Gamma \left( \ov{\rm Hom} \big( \iota_{F,\pss}^* T\mb{CP}^d, u^* TM \oplus \bigoplus_{j=1}^{i-1} \tilde u^* E_j \big) \otimes \iota_{F,\pss}^* {\mc O}(k_i) \right) \otimes \ov{ H^0(\iota_{F,\pss}^* {\mc O}(k_i))_0} 
\eeqn
satisfying the following equations
\begin{equation}
\left\{
\begin{aligned}
\ov\partial_{J, H^\pss} u +  \pi_{u^* TM} \left( \sum_{i=1}^d \langle \eta_i \rangle \circ d\iota_{F,\pss} \right) &\ = 0,\\
\ov\partial \eta_i + \pi_{E_i} \left( \sum_{j = i+1}^d \langle \eta_j \rangle \circ d\iota_{F,\pss} \right) &\ = 0,\ i = 1, \ldots, d,
\end{aligned}\right.
\end{equation}
where $E_i$ is defined in the same way as \eqref{eqn:e-i}, and the $\ov{\partial}$-operator $\ov\partial_{J, H^\pss}$ is read off from the PSS equation \ref{eqn:pss}.

By replacing PSS with SSP, we can similarly define the thickened moduli space $V_{p \bullet}^\ssp = V_{p \bullet}^\ssp (k_1, \dots, k_d)$ for a capped orbit $p$.

The space $V_{\bullet q}^\pss$ (resp. $V_{p \bullet}^\ssp$) admits a $G_{\bullet q}^\pss = G_{d_{\bullet q}}$ (resp. $G_{p \bullet}^\ssp = G_{d_{p \bullet}}$) action extending the action on the framing. The obstruction bundle over $E_{\bullet q}^\pss \to V_{\bullet q}^\bullet$ also has a direct sum decomposition
\beqn
E_{\bullet q}^\pss = O_{\bullet q}^\pss \oplus Q_{\bullet q}^\pss
\eeqn
where $O_{\bullet q}^\pss$ is specified by \eqref{eqn:obstruction-O}, and $Q_{\bullet q}^\pss$ is the trivial vector bundle from ${\bm Q}_{d_{\bullet q}}$ as in \eqref{eqn57}. This is a $G^\pss_{\bullet q}$-equivariant vector bundle. As for the Kuranishi section, it is defined by
\beqn
\begin{split}
S_{\bullet q}^\pss: V_{\bullet q}^\pss &\ \to E_{\bullet q}^\pss \\
        [\Sigma, u, F, \eta] &\ \mapsto ( \eta, H_F)
\end{split}
\eeqn
where $H_F\in Q_{\bullet q}^\pss$ is represented by the Hermitian matrix whose entries are
\beqn
\int_\Sigma \langle f_i, f_j \rangle_{L_{u, \pss}} \Omega_{u, \pss},
\eeqn
where $f_0, \dots, f_d$ is a basis representing the frame $F$. It is obvious that there exists a homeomorphism 
\beqn
\psi_{\bullet q}^\pss: (S_{\bullet q}^\pss)^{-1}(0) / G^\pss_{\bullet q} \xrightarrow{\sim} \ov{\mc M}{}_{\bullet q}^\pss.
\eeqn
This finishes the description of the K-chart
\beqn
K_{\bullet q}^\pss := K_{\bullet q}^\pss (k_1, \ldots, k_d): = (G^\pss_{\bullet q}, V_{\bullet q}^\pss, E^\pss_{\bullet q}, S_{\bullet q}^\pss, \psi_{\bullet q}^\pss).
\eeqn
The same description applies to the SSP moduli space and we obtain the K-chart
\beqn
K_{p \bullet}^\ssp := K_{p \bullet}^\ssp = (G^\ssp_{p \bullet}, V_{p \bullet}^\ssp, E_{p \bullet}^\ssp, S_{p \bullet}^\ssp, \psi_{p \bullet}^\ssp)
\eeqn
for the moduli space $\ov{\mc M}{}_{p \bullet}^\ssp$.

The following statement follows from the proof of Theorem \ref{regularity2}.
\begin{prop}
For fixed $k_1, \dots, k_{d-1}$, as long as $k_d$ is sufficiently large, the thickened moduli spaces $V_{\bullet q}^\pss = V_{\bullet q}^\pss(k_1, \ldots, k_d)$ and $V_{p\bullet}^\ssp = V_{p \bullet}^\ssp(k_1, \ldots, k_d)$ is regular near the zero locus of the Kuranishi section and $E_{\bullet q}^\pss \to V_{\bullet q}^\pss$, $E_{p \bullet}^\ssp \to V_{p \bullet}^\ssp$ are indeed vector bundles. \qed
\end{prop}

\subsubsection{The K-chart lift of the flow bimodule}

We would like the system of K-charts obtained above provide a K-chart lift of the flow bimodule $M^\pss$ resp. $M^\ssp$. Still we describe the PSS case with more details. Recall that the moduli space \eqref{eqn:pss-moduli} is a $\bA^\pss_{\bullet q}$-space. Given a stratum $\alpha = \bullet r_1 \cdots r_l q$, the moduli space 
\beqn
\partial^\alpha \ov{\mc M}{}_{\bullet q}^\pss = \ov{\mc M}{}_{\bullet r_1}^\pss \times \ov{\mc M}{}_{r_1 r_2}^\floer \times \cdots \times \ov{\mc M}{}^{\floer}_{r_l q}
\eeqn
has a product K-chart
\beqn
K_{\alpha}^\pss := K_{\bullet r_1}^\pss \times K_{r_1 r_2} \times \cdots \times K_{r_l q}
\eeqn
where $K_{r_i r_{i+1}}$ is the K-chart constructed in \eqref{eqn:k-chart-floer} for $\ov{\mc M}{}_{r_i r_{i+1}}^\floer$. Similarly, for $\alpha = p r_1 \cdots r_l \bullet$, the moduli space
\beqn
\partial^{\alpha} \ov{\mc M}{}_{p \bullet}^\ssp = \ov{\mc M}{}^\floer_{pr_1} \times \cdots \times \ov{\mc M}{}_{r_{l-1} r_l}^\floer \times \ov{\mc M}{}^\ssp_{r_l \bullet}
\eeqn
has a product K-chart
\beqn
K_{\alpha}^\ssp := K_{pr_1} \times \cdots \times K_{r_{l-1} r_l} \times K_{r_l \bullet}^\ssp.
\eeqn
Recall that the flow category $T^\floer$ is endowed with a weak K-chart lift as from Theorem \ref{thm:flow-chart} of the form
\beqn
    \big\{ K_{pq} = (G_{pq}, V_{pq}, E_{pq}, S_{pq}, \psi_{pq} ) \big\}_{p< q},
\eeqn
together with the collection of weak K-chart embeddings 
\beqn
\big\{ {\bm \iota}_{prq}: K_{pr}\times K_{rq}  \wembed  \partial^{prq} K_{pq}\big\}_{p< r < q}.
\eeqn

\begin{thm}
For a capped $1$-periodic orbits $p$ resp. $q$, there exist a collection of weak K-chart embeddings 
\begin{equation*}
    \big\{ {\bm \iota}_{\bullet r q}: K_{\bullet r}^\pss \times K_{rq}  \wembed  \partial^{\bullet rq} K_{\bullet q}^\pss \big\}_{r<q}
\end{equation*}
\begin{equation*}
    \text{resp. }\big\{ {\bm \iota}_{pr \bullet}: K_{pr} \times K^\ssp_{r \bullet}  \wembed  \partial^{pr \bullet} K_{p \bullet}^\ssp \big\}_{p<r},
\end{equation*}
such that the following statements are true.
\begin{enumerate}
    \item The following diagrams are commutative under the chart embeddings:
     \beqn
        \xymatrix{  &    K^\pss_{\bullet r} \times K_{rs}\times K_{sq} \ar[ld] \ar[rd] & \\
                      \partial^{\bullet rs} K_{\bullet s}^\pss \times K_{sq} \ar[rd] & & K_{\bullet r}^\pss \times \partial^{rsq} K_{rq} \ar[ld] \\
                      & \partial^{\bullet rsq} K^\pss_{\bullet q}, & }
        \eeqn
    \beqn
        \xymatrix{  &    K_{rs} \times K_{sp}\times K^\ssp_{p \bullet} \ar[ld] \ar[rd] & \\
                      \partial^{rsp} K_{rp} \times K^\ssp_{p \bullet} \ar[rd] & & K_{rs} \times \partial^{sp \bullet} K^\ssp_{s \bullet} \ar[ld] \\
                      & \partial^{\bullet rsp} K^\ssp_{\bullet p}. & }
        \eeqn
        
    \item For each stratum $\alpha = \bullet r_1 \cdots r_l q \in \bA_{\bullet q}^\pss$ resp. $\alpha = p r_1 \cdots r_l \bullet \in \bA_{p\bullet}^\ssp$, define the product chart
    \beqn
    K_\alpha^\pss := K_{\bullet r_1}^\pss \times \cdots \times K_{r_l q}
    \eeqn
    \beqn
    {\rm resp.} K^\ssp_{\alpha} := K_{pr_1} \times \cdots \times K^\ssp_{r_l \bullet}.
    \eeqn
    Then for any pair $\alpha \leq \beta$ in $\bA_{\bullet q}^\pss$ resp. $\bA_{p\bullet}^\ssp$, there are induced product weak embeddings
    $$ {\bm \iota}^\pss_{\beta \alpha}: K^\pss_{\alpha} \wembed \partial^{\alpha} K^\pss_{\beta} \text{\ resp. } {\bm \iota}^\ssp_{\beta \alpha}: K^\ssp_{\alpha} \wembed \partial^{\alpha}K^\ssp_{\beta}.$$
    %Then there exist vector bundles $F^{\text{PSS}}_{\beta \alpha} \to V^{\text{PSS}}_{\alpha}$ resp. $F^{\text{SSP}}_{\beta \alpha} \to V^{\text{SSP}}_{\alpha}$ for all pairs $\alpha \leq \beta$ together with germs of weak open K-chart embeddings 
    %$${\bm \theta}^{\text{PSS}}_{\beta \alpha}:{\rm Stab}_{F^{\text{PSS}}_{\beta\alpha}}( K^{\text{PSS}}_\alpha) \wembed \partial^\alpha K^{\text{PSS}}_\beta$$
    %$$\text{resp. } {\bm \theta}^{\text{SSP}}_{\beta \alpha}:{\rm Stab}_{F^{\text{SSP}}_{\beta\alpha}}( K^{\text{SSP}}_\alpha) \wembed \partial^\alpha K^{\text{SSP}}_\beta$$
    Then the system 
    $$ \{ K^\pss_{\alpha} \}_{\alpha \in \bA_{\bullet q}^\pss} \text{  resp. } \{ K^\ssp_\alpha  \}_{\alpha \in \bA_{p\bullet}^\ssp} $$
    defines a weak K-chart presentation of the $\bA_{\bullet q}^\pss$-space $\ov{\mc M}{}_{\bullet q}^\pss$ resp. the $\bA_{p \bullet}^\ssp$-space $\ov{\mc{M}}{}_{p \bullet}^\ssp$.
    
    \item After outer-collaring, one can match up the Kuranishi sections so that the weak K-chart presentation is a genuine K-chart presentation. 
    
    \item The natural projection $\pi_{\bullet q}^\pss: V_{\bullet q}^\pss \to B_{d_{\bullet q}}^\pss$ resp. $\pi_{p \bullet}^\ssp: V_{p \bullet}^\ssp \to B_{d_{p \bullet}}^\ssp$ has a $C_{\rm loc}^1$ equivariant fiberwise smooth bundle structure. 

\end{enumerate}
\end{thm}

\begin{proof}[Sketch of proof]
The embeddings ${\bm \iota}_{\bullet rq}^\pss$ and ${\bm \iota}_{pr\bullet}^\ssp$ are defined in a similar way as in Section \ref{subsection54} when we embed a product of K-charts of a list of Floer moduli spaces. It is routine to check that these embeddings satisfy the listed commutative diagrams and give weak K-chart presentations of the corresponding PSS and SSP moduli spaces. Notice that for the same reason, the product Kuranishi section only matches with the boundary restriction of the larger Kuranishi chart in the $O$-summand. We can use outer-collaring and an interpolation to match their $Q$-summand. Lastly, the projection $\pi_{\bullet q}^\pss$ and $\pi_{p \bullet}^\ssp$ have the structure of $C_{\rm loc}^1$ equivariant fiberwise smooth bundles for the same reason as in the case of Floer trajectories.
\end{proof}

\subsubsection{Scaffolding}

We could state and prove a proposition similar to Proposition \ref{bigprop}. However, instead of doing the tedious thing, we only describe how the existing scaffoldings for the K-chart lift of $(T^\floer)^+$ can be extended. Notice that we will not use any object or structure we chose in Section \ref{sec-6}. Essentially (for the PSS case), we would like to choose subbundles
\beqn
F_{\bullet q, \alpha}^\pss \subset O_{\bullet q}|_{V_\alpha^\pss}
\eeqn
and define stabilization maps
\beqn
{\rm Stab}_{F_{\bullet q, \alpha}}( V_\alpha^\pss) \to \partial^\alpha V_{\bullet q}^\pss
\eeqn
whose $G_{\bullet q}$-equivariantization is an open embedding. (We skip the discussion of the companion bundle isomorphisms.) This bundle $F_{\bullet q, \alpha}$ is indeed the orthogonal complement of $E_\alpha^\pss$ inside $E_{\bullet q}^\pss$; a good choice of an inner products can be made inductively as in the proof of Proposition \ref{bigprop} which extend the inner products we have already chosen on the obstruction bundles over the thickend Floer moduli spaces. Then the stabilization map is still constructed using the implicit function theorem which depends on choosing a family of local Banach manifold charts, a family of approximate solutions, and a family of right inverses. These objects can all be extended from the existing ones we have chosen for the Floer moduli spaces. Therefore a compatible system of stabilization maps can be constructed. Of course, the collar structure is necessary for the construction.

\subsubsection{Stable smoothing}

We would like to construct a system of stable equivariant smoothings of $V_{\bullet q}^\pss$ resp. $V_{p\bullet}^\ssp$. Notice that we have chosen stabilizations by $G_{pq}$-representations ${\bm R}_{pq}$ for all $p<q$ in ${\mc P}^\floer$. The method we used in Subsection \ref{subsection64} can be used again here to construct stabilizations ${\bm R}_{\bullet q}^\pss$ and to construct smoothings on $\widehat V_{\bullet q}^\pss:= V_{\bullet q}^\pss \times {\bm R}_{\bullet q}^\pss$ whose boundary is diffeomorphic to products of stabilized thickenings.

We can also require that the stabilizations satisfy the following property. Namely, there exist smooth submersive evaluation maps
\beqn
\ev_\bullet: \widehat V_{\bullet q}^\pss \to M
\eeqn
which coincides with the original evaluation map on $V_{\bullet q}^\pss$. Indeed, by observation of Abouzaid--McLean--Smith \cite[Lemma 4.5]{AMS}, for a single smooth Kuranishi chart, on a further stabilization there exists a smooth submersive extension of the evaluation map. Then when we inductively construct stable smoothings, we can always achieve this extra condition.

It follows that by intersecting with the unstable manifolds of the given Morse function $f: M \to {\mb R}$, one obtains smooth Kuranishi charts 
\beqn
\widehat K_{xq}^\pss = (G_{\bullet q}^\pss, \widehat V_{xq}^\pss, \widehat E_{xq}^\pss, \widehat S_{xq}^\pss)
\eeqn
where $\widehat V_{xq}^\pss\subset \widehat V_{\bullet q}^\pss$ is the intersection with the (outer-collared) compactified unstable manifold $\ov{W^u(\uds{x})}\subset M$. It is straightforward to check that we obtained a smooth K-chart lift of the flow bimodule $(M^\pss)^+$. 

Then after taking quotient, we obtained a D-chart lift of $(M^\pss)^+$ equipped with a collar structure and a scaffolding. The straightening that is compatible with the collar structure and the scaffolding can be constructed in a way similar to the case of Floer trajectories.

\subsubsection{Normal complex structures and orientaions}

The construction of normal complex structure and orientation on $\widehat K_{xq}^\pss$ essentially follows from the proof of Theorem \ref{thm:normal-c-orient}. and the relevant discussions in \cite[Section 11]{Abouzaid_Blumberg}. Two remarks should be pointed out. First, the stabilization from \cite[Lemma 4.5]{AMS} and the fiber product with unstable submanifolds does not affect the \emph{normal} complex structure, because the changes on the tangent bundles are all included in the trivial summand of the corresponding representation of the isotropy group. Second, when constructing orientation structures, the virtual vector bundle associated with an object $x \in T^\morse$ is defined by $(T W^u(\uds{x}), 0)$.

\subsection{Pearly bimodule}

Consider the pearly bimodule defined in Section \ref{sec-4}. Here we describe a K-chart lift of it and show how to use the same argument to construct a stable smoothing. First, for each $d$, consider an auxiliary moduli space which is similar to $B_d$, $B_d^\pss$, and $B_d^\ssp$ defined before. Indeed, let ${\mc F}_{0,2}^\pearl(d)$ be the moduli space of parametrized 2-marked holomorphic spheres into $\mb{CP}^d$ which have degree $d$ and whose images are not entirely contained in a hyperplane. In particular, each point of ${\mc F}_{0,2}^\pearl(d)$ is represented by a stable map with 2 marked points that have a distinguished component whose parametrization is fixed. ${\mc F}_{0,2}^\pearl(d)$ is a smooth manifold (with no boundary or corners). Moreover there are holomorphic submersions
\beqn
\ev_\pm: {\mc F}_{0,2}^\pearl(d) \to \mb{CP}^d.
\eeqn
Let
\beqn
B_d^\pearl \subset {\mc F}_{0,2}^\pearl(d)
\eeqn
be the subset of elements whose evaluation at $z_-$ is $[1, 0, \ldots, 0]\in \mb{CP}^d$. Then $B_d^\pearl$ is a smooth manifold. There is a $G_d \cong U(d)$-action on $B_d^\pearl$.

Now consider the corresponding space $\ov{\mc M}{}_{\bullet\bullet}^\pearl(d)$ of stable parametrized holomorphic spheres in $M$ whose degree is $d$. Using multi-layered thickening associated to a sequence of integers $k_1<k_2 < \cdots < k_d < \cdots$ one can produce a similar thickened moduli space
\beqn
K_{\bullet \bullet}^\pearl(d) = ( G_d, V_{\bullet\bullet}^\pearl(d), E_{\bullet\bullet}^\pearl(d), S_{\bullet\bullet}^\pearl).
\eeqn
The domain $V_{\bullet\bullet}^\pearl(d)$ admits a $G_d$-action and an equivariant projection map
\beqn
\pi_{\bullet\bullet}^\pearl: V_{\bullet\bullet}^\pearl(d) \to B_d^\pearl.
\eeqn
Moreover, by taking intersections with stable and unstable submanifolds of the Morse function $f: M \to {\mb R}$, one obtains stratified charts
\beqn
K_{xy}^\pearl = (G_{xy}, V_{xy}^\pearl, E_{xy}^\pearl, S_{xy}^\pearl)
\eeqn
This collection of charts obviously defines a K-chart lift of the pearly bimodule $M^\pearl$.

\begin{thm}
There exists a ``single-layered'' normally complex derived orbifold lift of the outer-collared pearly bimodule $(M^\pearl)^+$.
\end{thm}

\begin{proof}[Sketch of Proof]
We can construct stable smoothings of $K_{\bullet\bullet}^\pearl(d)$ individualy for each $d\geq 1$ to obtain smooth Kuranishi charts. To save notations, still denote them by $K_{\bullet\bullet}^\pearl(d)$. We can do a further stabilization such that the evaluations at $z_\pm$ are smooth submersions onto $M$. Hence each $K_{xy}^\pearl$ becomes smooth stratified Kuranishi charts. Its outer-collaring is again smooth Kuranishi charts. The scaffoldings are trivial as the Morse theory part is already regular. Then taking quotient by $G_{xy}$ one obtains a D-chart lift of $(M^\pearl)^+$, denoted by $({\mf D}^\pearl)^+$. As the scaffolding is trivial, this is a single-layered lift. Lastly, there is a normal complex structure and an orientation on $({\mf D}^\pearl)^+$ for the same reason as the case of \cite{AMS}. 
\end{proof}

\subsection{The homotopy}

Lastly, we explain the collection of D-chart lifts of the system of moduli spaces referred to as the ``homotopy moduli spaces.'' Although this {\it ad hoc} collection lacks a formal package, the constructions are almost identical to previous cases. The global chart construction, including the multi-layered strategy, is the same as the case of Floer trajectories or PSS/SSP cases. Then one obtains a weak K-chart lift of the collection of moduli spaces. By taking outer-collaring and interpolating between the product Kuranishi section and the boundary restriction of Kuranishi sections, one obtains a K-chart lift of the outer-collared system. The construction of scaffolding (see Proposition \ref{bigprop}) has completely the same induction procedure. We can also run the smoothing procedure without essential differences.

\appendix

\section{Product of canonical Whitney stratifications}\label{appendixa}

\subsection{Basics of Whitney stratifications}

We review the basic notions of (Whitney) prestratifications and stratification. Following Mather \cite{Mather_1973}, our convention is the same as our previous work \cite{Bai_Xu_2022}. 

\begin{defn}\label{defna1}
Let $S$ be a topological space. A {\bf prestratification} on $S$ is a decomposition 
\beqn
S = \bigsqcup_{\alpha \in {\mf A}} S_\alpha
\eeqn
of $S$ into the disjoint union of locally closed subsets satisfying the following condition
\begin{enumerate}
    \item The decomposition is locally finite.
    
    \item (Axiom of frontier) If $\alpha, \beta \in {\mf A}$, $S_\alpha \cap \ov{S_\beta} \neq \emptyset$, then $S_\alpha \subset \ov{S_\beta}$.
\end{enumerate}
Each $S_\alpha$ is called a {\bf stratum} of the prestratification. The axiom of frontier induces a partial order among strata: $S_\alpha \leq S_\beta$ if $S_\alpha \subset \ov{S_\beta}$. We use the symbol ${\mf A}$ to denote the prestratification as well as the partially ordered set of strata. A space equipped with a prestratification ${\mf A}$ is called an ${\mf A}$-stratified space.
\end{defn}

Now we consider the notion of stratifications. Two subsets $A, B \subseteq M$ are called {\bf equivalent} at $x \in M$ if there exists an open neighborhood $U_x$ of $x$ such that $A\cap U_x = B \cap U_x$. An equivalence class is called a {\bf setgerm} at $x$. Given a prestratification ${\mc A}$ on $S\subseteq M$, it assigns to each $x \in S$ a setgerm represented by the unique stratum that contains $x$. A {\bf stratification} of a subset $S \subset M$ is a rule which assigns to each point $x \in S$ a set-germ $S_x$, such that for each $x\in S$, there is an open neighborhood $U_x$ of $x$ and a prestratification of $U_x \cap S$ such that the setgerm-valued function restricted to $U_x$ is induced from this prestratification.

\begin{defn}[Whitney stratification]\label{defna2}
Let $M$ be a smooth manifold. 
\begin{enumerate}
    \item Given two disjoint smooth submanifolds $S$, $S'$, we say that the pair $(S, S')$ satisfies {\bf Whitney's condition (b)} at $x \in S' \cap \ov{S}$ if the following is true: suppose $x_i \in S$, $x_i' \in S'$ are two sequences converging to $x \in \ov{S} \cap S'$. Suppose the sequence of tangent spaces $T_{x_i} S$ converges to a subspace $H \subset T_x M$ and the sequence of secant lines $\ov{x_i x_i'}$ converges to a line $L \subset T_x M$, then $L \subset H$. 

\item Let $S \subset M$ be a subset. A prestratification on $S$ is called a {\bf Whitney prestratification} if each stratum $S_\alpha$ is a smooth submanifold and each pair $(S_\alpha, S_\beta)$ of strata with $S_\beta \subsetneq \ov{S_\alpha}$ satisfies Whitney's condition (b) at every point of $S_\beta$.

\item A stratification on $S \subset M$ is called a {\bf Whitney stratification} if for each point $x \in S$ there exists an open neighborhood $U_x \subset S$ of $x$ and a Whitney prestratification on $U_x$ which induces this stratification inside $U_x$. We use the symbol ${\mc S}$ to denote a Whitney stratification on $S$, which is a setgerm-valued function on $S$.
\end{enumerate}
\end{defn}

On the same set there could be many different Whitney stratifications. One can use a partial order among them to compare. Suppose ${\mc S}_1$ and ${\mc S}_2$ are two Whitney stratifications, which assign to each $x$ setgerms $S_{1,x}$ and $S_{2,x}$ respectively. Then define 
\beqn
S_i^{(k)}= \{ x \in S\ |\ {\rm dim}_{\mb R} S_{i, x}  \leq k\},\ i = 1, 2.
\eeqn
Then one has descending sequence of closed sets 
\beqn
S_i = S_i^{(m)} \supseteq S_i^{(m-1)} \supseteq \cdots \supseteq S_i^{(0)}
\eeqn
where $m= {\rm dim}_{\mb R}M$.

\begin{defn}\label{minimal}
${\mc S}_1  < {\mc S}_2$ if there exists an integer $k$ such that 
\beqn
S_1^{(l)} = S_2^{(l)}\ \forall l > k\ {\rm and}\ S_1^{(k)} \subsetneq S_2^{(k)}.
\eeqn
\end{defn}

If a Whitney stratification is minimal, then it is unique. Moreover, a minimal Whitney stratification on $S\subseteq M$ is invariant under diffeomorphisms of $M$ which preserve $S$ set-wise (see \cite[Lemma A.11]{Bai_Xu_2022}).

\begin{thm}\label{thm:Whitney_1965}\cite{Whitney_1965}
Given a smooth complex algebraic variety and $S \subseteq M$ a constructible set. There exists a minimal Whitney stratification which is induced from a Whitney prestratification whose strata are all smooth complex algebraic submanifolds. 
\end{thm}

We sketch the constructive proof following Mather \cite{Mather_1973} as this construction will be needed for further discussions. Suppose the (real) dimension of $M$ is $m$. Then we construct inductively a sequence of closed algebraic subsets 
\beqn
S = S^{(m)} \supseteq S^{(m-1)} \supset \cdots \supseteq S^{(0)}
\eeqn
satisfying the following conditions.
\begin{enumerate}
    \item Each $S^{(l)}$ is a closed algebraic subset of $M$ of real dimension at most $l$.

    \item Each $\mathring{S}^{(l)}:=S^{(l)} \setminus S^{(l-1)}$ is a smooth complex algebraic submanifold of (real) dimension $l$. It can empty, for example, when $l$ is odd.
    
    \item For each pair $l > k$, the pair $(\mathring S^{(l)}, \mathring S^{(k)})$ satisfies Whitney's condition (b) at every point of $\mathring S^{(k)}$.
\end{enumerate}
Indeed, suppose $S^{(m)}, \cdots, S^{(k)}$ have been constructed. Then we define $S^{(k-1)} \subset S^{(k)}$ to be the subset of points $x \in S^{(k)}$ satisfying one of the following conditions.
\begin{enumerate}
    \item $x$ is a singular point of $S^{(k)}$ or a regular point of $S^{(k)}$ with local (real) dimension strictly less than $k$.
    
    \item $x$ is a regular point of $S^{(k)}$ with local dimension $k$ and there exists $l>k$ such that the pair $(\mathring S^{(l)}, S^{(k)}_{reg})$ does not satisfy Whitney's condition (b) at $x$. Here $S_{reg}^{(k)} \subseteq S^{(k)}$ is the Zariski open subset of regular points of local dimension $k$.
\end{enumerate}
Basic algebraic geometry implies that the set of points satisfying the first condition above is a closed algebraic set of real dimension strictly less than $k$. On the other hand, it is a fundamental theorem of Whitney that the set of points satisfying the second condition above is also a closed algebraic set of real dimension strictly less than $k$. Then this produces a Whitney prestratification whose strata are all connected components of all $\mathring S^{(k)}$. It is not difficult to prove that this Whitney stratification is the minimal one. 

\subsection{Product of canonical Whitney stratifications}

Now we consider if the product of canonical Whitney stratifications on a product algebraic set is the canonical one. Let $M_1, M_2$ be smooth complex algebraic varieties and $S_1 \subseteq M_1$, $S_2 \subseteq M_2$ be complex algebraic subsets, equipped with the canonical Whitney stratification of with levels $S_i^{(k)}$. Denote
\begin{align*}
&\  M = M_1 \times M_2,\ &\ S = S_1 \times S_2.
\end{align*}
For each $k\geq 0$, denote  
\beqn
\mathring S_{\times}^{(k)}:= \bigcup_{k_1 + k_2 = k} \mathring S_1^{(k_1)} \times \mathring S_2^{(k_2)}.
\eeqn
It is easy to see that the prestratification on $S$ whose strata are all connected components of all $\mathring S_\times^{(k)}$ is a Whitney prestratification. Denote the associated Whitney stratification by ${\mc S}_\times$. It is easy to see the induced dimension filtration on $S$ is given by 
\beqn
S_\times^{(k)} = \bigcup_{k_1 +k_2 = k} S_1^{(k_1)} \times S_2^{(k_2)}.
\eeqn

\begin{prop}\label{propa4}
The Whitney stratification ${\mc S}_\times$ on $S$ is the minimal (canonical) one.
\end{prop}

\begin{proof}
We know that the canonical Whitney stratification, denoted by ${\mc S}$, is the minimal one. Then one has ${\mc S} \leq {\mc S}_\times$. Suppose this is not an equality. Then by definition (see Definition \ref{minimal}), there exists $l$ such that 
\beqn
S^{(k)} = S_\times^{(k)}\ \forall k > l,\ S^{(l)} \subsetneq S_\times^{(l)}.
\eeqn
Therefore, there exists a point 
\beqn
x = (x_1, x_2) \in S^{(l+1)} = S_\times^{(l+1)} = \bigcup_{l_1 + l_2 = l+1} S_1^{(l_1)} \times S_2^{(l_2)},
\eeqn
$x \in S_\times^{(l)}$, $x \notin S^{(l)}$. Then by the construction of the canonical Whitney stratification, $x$ is a regular point of $S^{(l+1)}$ of local dimension $l+1$. Then $x$ is a regular point of the union of $S_1^{(l_1)} \times S_2^{(l_2)}$ for all possible $l_1 + l_2 = l+1$. Then for some $l_1, l_2$ with $l_1 + l_2 = l+1$, $x = (x_1, x_2)$ is a regular point of $S_1^{(l_1)} \times S_2^{(l_2)}$ with local dimension $l+1$. This implies that $x_i$ is a regular point of local dimension $l_i$ in each factor. 

However, $x \in S_\times^{(l)}$ also implies that either $x_1 \in S_1^{(l_1 -1)}$ or $x_2 \in S_2^{(l_2-1)}$. Without loss of generality, we assume $x_1 \in S_1^{(l_1-1)}$. Then as $x_1$ is a regular point of $S_1^{(l_1)}$, it implies that for some $k_1>l_1$, Whitney's condition (b) for $(\mathring S_1^{(k_1)}, S_{1,  reg}^{(l_1)})$ fails at $x_1$. Then there exists a sequence $y_{1, \nu} \in \mathring S_1^{(k_1)}$ and a sequence $x_{1, \nu}\in S_{1, reg}^{(l_1)}$, both of which converging to $x_1$ such that $T_{y_{1, \nu}} \mathring S_1^{(k_1)}$ converges to a subspace $H_1 \subset T_{x_1} M_1$, the secant line $\ov{x_{1, \nu} y_{1,\nu}}$ converges to a line $L_1 \subset T_{x_1} M$, but $L_1 \nsubseteq H_1$. Now we separate the discussion in two scenarioes
\begin{enumerate}
\item If $x_2 \in \mathring S_2^{(l_2)}$, then consider the two sequences $(x_{1, \nu}, x_2) \in S_{1, reg}^{(l_1)} \times S_{2, reg}^{(l_2)} \subset S^{(l+1)}_{reg}$, $(y_{1, \nu}, x_2) \in \mathring S_1^{(k_1)} \times \mathring S_2^{(l_2)} \subset \mathring S^{(k_1 + l_2)}$. Then one can see that Whitney's condition (b) fails for $(\mathring S^{(k_1 + l_2)}, S_{reg}^{(l + 1)})$ at the limit $x = (x_1, x_2)$, which is a contradiction that $x \in \mathring S^{(l+1)}$. 

\item Suppose $x_2\notin \mathring S_2^{(l_2)}$. Then by construction, there exist some $k_2 > l_2$, a sequence $y_{2, \nu}\in \mathring S_2^{(k_2)}$ converging to $x_2$, a sequence $x_{2, \nu}\in S_{2, reg}^{(l_2)}$ converging to $x_2$, such that the sequence of tangent spaces $T_{y_{2,\nu}} \mathring S_2^{(k_2)}$ converges to $H_2$ and the sequence of secant lines $\ov{x_{2, \nu} y_{2,\nu}}$ converges to a line $L_2$ but $L_2$ is not contained in $H_2$. Then consider the sequence of points $y_\nu = (y_{1, \nu}, y_{2, \nu}) \in \mathring S_1^{(k_1)}\times \mathring S_2^{(k_2)} \subset \mathring S^{(k_1+ k_2)}$, the sequence of points $x_\nu = (x_{1, \nu}, x_{2, \nu})\in S_{1, reg}^{(l_1)} \times S_{2, reg}^{(l_2)}\subset S_{reg}^{(l+1 + l_2)} = S_{reg}^{(l+1)}$. This breaks the hypothesis that $x \in \mathring S^{(l+1)}$. 
\end{enumerate}

\end{proof}

\subsection{Relative case}

Now we consider the relative case. Let $S$ be equipped with a prestratification ${\mf A}$
\beqn
S = \bigsqcup_{\alpha \in {\mf A}} S_\alpha
\eeqn
such that each stratum $S_\alpha$ is algebraic. A Whitney stratification ${\mc S}$ on $S$ is said to respect the given prestratification if for each $x \in S$, the germ $S_x$ is contained in the stratum $S_\alpha$ which contains $x$. We call such Whitney stratifications ${\mf A}$-Whitney stratifications. We proved in \cite{Bai_Xu_2022} that there exists a unique minimal ${\mf A}$-Whitney stratification, which we call the canonical ${\mf A}$-Whitney stratification. 

For the purpose of this paper, we need to verify that such canonical Whitney stratification is natural with respect to products. Let $M, N$ be two smooth complex algebraic varieties and $S \subset M$, $T \subset N$ be constructible subsets. Given algebraic prestratifications
\begin{align*}
&\ S = \bigsqcup_{\alpha \in {\mf A}} S_\alpha,\ &\ T = \bigsqcup_{\beta \in {\mf B}} T_\beta,
\end{align*}
one has an associated decomposition 
\beqn
R:=S \times T = \bigsqcup_{(\alpha, \beta) \in {\mf A} \times {\mf B}} T_{(\alpha, \beta)} = \bigsqcup_{(\alpha, \beta) \in {\mf A}\times {\mf B}} S_\alpha \times T_\beta.
\eeqn
We can check that this is still a prestratification with algebraic strata. Moreover, the induced partial order on the set of strata ${\mf A} \times {\mf B}$ is the product one:
\beqn
(\alpha, \beta) \leq (\alpha', \beta') \Longleftrightarrow \alpha \leq \alpha'\ {\rm and}\ \beta \leq \beta'.
\eeqn
Let the prestratification be denoted by ${\mf A}\times {\mf B}$. On $S$ resp. $T$ there are minimal ${\mf A}$- resp. ${\mf B}$-Whitney stratification, whose product is also a Whitney stratification on $R = S \times T$.

\begin{prop}\label{propa5}
The product Whitney stratification on $R = S \times T$ is the minimal ${\mf A}\times {\mf B}$-Whitney stratification.
\end{prop}

\begin{proof}
We prove by induction. Proposition \ref{propa4} implies that for each top stratum $(\alpha_0, \beta_0) \in {\mf A}\times {\mf B}$, the product Whitney stratification coincides with the minimal Whitney stratification on $S_{\alpha_0} \times T_{\beta_0}$. Now given a stratum $(\alpha, \beta) \in {\mf A} \times {\mf B}$. Suppose for any $(\alpha', \beta') > (\alpha, \beta)$, we have proved that the product Whitney stratification on $S_{\alpha'} \times T_{\beta'}$ coincides with the restriction of the minimal $({\mf A}\times {\mf B})$-Whitney stratification. We would like to show that it is still the case for the pair $(\alpha, \beta)$. 

Suppose our claim is false. Then there exists $k\geq 0$ such that 
\beqn
(S_\alpha \times T_\beta)^{(l)} = \bigcup_{p + q = l} S_{\alpha}^{(p)} \times T_\beta^{(q)},\ \forall l \geq k+1,\ (S_\alpha \times T_\beta)^{(k)} \subsetneq \bigcup_{p + q = k} S_\alpha^{(p)} \times T_\beta^{(q)}.
\eeqn
Then one can choose a point 
\beq
x \in (S_\alpha \times T_\beta)^{(k+1)} \setminus (S_\alpha \times T_\beta)^{(k)}
\eeq
such that
\beq\label{yyy}
x\in \bigcup_{p + q = k} S_\alpha^{(p)} \times T_\beta^{(q)}. 
\eeq
By the construction of the minimal Whitney stratifications on constructible sets (sketched after Theorem \ref{thm:Whitney_1965}), one knows that $x = (y, z)$ is a regular point of $(S_\alpha \times T_\beta)^{(k+1)}$ of local dimension $k+1$. This implies that for some $p+q = k+1$, $x$ is a regular point of $S_{\alpha}^{(p)} \times T_{\beta}^{(q)}$ of local dimension $k+1$. Then $y$ resp. $z$ is a regular point of $S_{\alpha}^{(p)}$ resp. $T_{\beta}^{(q)}$ of dimension $p$ resp. $q$. Moreover, \eqref{yyy} implies that either $y \in S_\alpha^{(p-1)}$ or $z\in T_\beta^{(q-1)}$.

We claim that $y$ is not in the boundary of any $\mathring S_{\alpha'}^{(p')}$ with $\alpha' > \alpha$ and $p' \leq p$. If it is the case, assume that $z \in \mathring T_\beta^{(q')}$ for some $q' \leq q$. Then we see that 
\beqn
x = (y, z) \in \partial \mathring S_{\alpha'}^{(p')} \times \mathring T_\beta^{(q')} \subset \partial \left( \mathring S_{\alpha'}^{(p')} \times \mathring T_\beta^{(q')} \right) \subset \partial \mathring R_{(\alpha', \beta)}^{(p'+q')}.
\eeqn
Notice that $p' + q' \leq k+1$ and $(\alpha', \beta) > (\alpha, \beta)$. This contradicts the fact that $x \notin R_{(\alpha, \beta)}^{(k)}$. Similarly, $z$ is not in the boundary of any $\mathring T_{\beta'}^{(q')}$ with $\beta'> \beta$ and $q' \leq q$.

Without loss of generality, assume we are in one of the following two scenarios.

\begin{enumerate}

\item $y \in S_\alpha^{(p-1)}$ and $z \in \mathring T_\beta^{(q)}$. Then we know that for some $\alpha' \geq \alpha$ and $p' \geq p+1$, Whitey's condition (b) fails for the pair $(\mathring S_{\alpha'}^{(p')}, S_{\alpha, reg}^{(p)})$ at $y$. Then this implies that Whitney's condition (b) fails for the pair 
\beqn
\left( \mathring R_{(\alpha', \beta)}^{(p'+q)}, R_{(\alpha, \beta), reg}^{(p+q)} \right)
\eeqn
at $x = (y, z)$. This contradicts the assumption that $x \notin R_{(\alpha, \beta)}^{(k)}$.
 
\item $y \in S_\alpha^{(p-1)}$ and $z \in T_\beta^{(q-1)}$. Then we know for some $\alpha' \geq \alpha$, $\beta' \geq \beta$, $p' \geq p+1$, $q' \geq q+1$, Whitney's condition (b) fails for the pair $(\mathring S_{\alpha'}^{(p')} \times \mathring T_{\beta'}^{(q')}, S_{\alpha, reg}^{(p)} \times T_{\beta, reg}^{(q)})$ at $(y, z)$. It follows that Whitney's condition (b) fails for the pair 
\beqn
\left( \mathring R_{(\alpha', \beta')}^{(p'+q')}, R_{(\alpha, \beta), reg}^{(p+q)} \right).
\eeqn
at $x = (y, z)$. This contradicts the assumption that $x \notin R_{(\alpha, \beta)}^{(k)}$.
\end{enumerate}
\end{proof}

\section{Relative stable equivariant smoothing}\label{appendixb}

Our construction in this paper relies on equivariant smoothings of infintely many moduli spaces. It is necessary to have an extension of the stable $G$-smoothing results of Lashof \cite{Lashof_1979} to a relative setting. Namely, if a $G$-manifold $M$ has a $G$-smoothing over an open set, then one can extend the existing smoothing to $M$ once the tangent microbundle admits a $G$-vector bundle reduction which is compatible with the existing smoothing. 

Lashof's original construction relies on two technical results. The first one is Mostow's embedding theorem \cite{Mostow_1957} and the second is Jaworowski's extension theorem \cite{Jaworowski_1976}. We need to recall these two theorems in order to obtain a generalization of Lashof's results on stable $G$-smoothing to the relative setting.

\begin{lemma}\cite[Lemma 5.2]{Mostow_1957} \label{Mostow_embedding}
Let $G$ be a compact Lie group of transformations on a metric space $E$, and let $T_1, T_2$ be invariant subsets with $E = T_1 \cup T_2$ and $T_2$ be closed. Assume that there exist $G$-equivariant topological embeddings $\varphi_i: T_i \to {\bm R}_i$, $i = 1, 2$ where ${\bm R}_1, {\bm R}_2$ are orthogonal $G$-spaces. Then there exists a $G$-equivariant topological embedding $\varphi:E \to {\bm R}_1 \oplus {\bm R}_2$ such that
\beqn
\varphi|_{T_2} = (0, \varphi_2).
\eeqn
\end{lemma}

\begin{thm}\cite[Theorem 2.2]{Jaworowski_1976}\label{Jaworowski}
Let $X$ be a locally compact $G$-space which can be equivariantly embedded into some finite-dimensional orthogonal $G$-vector space. Suppose $X$ has only finitely many orbit types. Let $A \subset X$ be a closed $G$-subspace and $f: A \to Y$ be a continuous $G$-map to a locally compact separable metrizable $G$-space $Y$. If for each $x \in X \setminus A$ the fixed point set $Y^{G_x}$ is an ANR (absolute neighborhood retract), then $f$ can be extended to a continuous $G$-map from a neighborhood of $A$ in $X$. 
\end{thm}

Now we want to extend Lashof's construction to a relative setting. 

\begin{thm}\label{thmb3}
Let $M$ be a topological $G$-manifold with only finitely many orbit types. Let $O \subset M$ be an open subset such that its closure $\ov{O}$ is compact. %submanifold with boundary. 
Suppose that
\begin{enumerate}
    \item the tangent microbundle $T_\mu M$ of $M$ admits a $G$-vector bundle reduction $\varphi: E_\mu \to T_\mu M$ for a vector bundle $E \to M$.
    
    \item A $G$-invariant open neighborhood $U$ of $\ov{O}$ is equipped with a $G$-smoothing and $E|_U$ is equipped with a smooth $G$-bundle structure such that the reduction $\varphi$ is smooth over $U$. Moreover, $\varphi$ is in the same isotopy class of the canonical $TU \to T_\mu U$.
\end{enumerate}
Then there exists a finite-dimensional orthogonal representation ${\bm R}$ of $G$ and a $G$-smoothing $\alpha$ on $M \times {\bm R}$ satisfying 
\begin{enumerate}
    \item The induced stable isotopy class of $G$-vector bundle reductions of $T_\mu( M \times R)$ coincides with the ${\bm R}$-stabilization of $\varphi$. 
    
    \item There exists a $G$-invariant open neighborhood $U'$ of $\ov{O}$ with $\ov{U'} \subset U$, such that the restriction of the $G$-smoothing $\alpha$ on $U'\times {\bm R}$ is diffeomorphic to the given $G$-smoothing on $U' \times {\bm R}$. 
\end{enumerate}
\end{thm}

\begin{proof}
We generalize the construction of Lashof \cite{Lashof_1979} to the current setting. We first consider the case that $E \to M$ is trivial. Assume that 
\beqn
E = M \times {\bm W}
\eeqn
where ${\bm W}$ is an orthogonal $G$-space. The vector bundle lift is given by a map 
\beqn
\exp: M \times {\bm W} \to M \times M
\eeqn
which is smooth over $U \times {\bm W}$. For each $x\in M$, denote the restriction of $\exp$ to $\{x\}\times {\bm W}$ by 
\beqn
\exp_x: {\bm W} \to \{x\}\times M
\eeqn
which is a homeomorphism onto an open neighborhood of $x$. 

Now we prove the following claims.

\vspace{0.2cm}

\noindent {\it Claim A. There exist an orthogonal $G$-representation ${\bm R}$ and a $G$-equivariant topological embedding $\iota:M \to {\bm R}$ which is smooth near $\ov{O}$. }

\vspace{0.2cm}

\noindent {\it Proof of Claim A.} As $\ov{O}$ is compact, by a theorem of Palais (see \cite[Theorem III]{Palais_1957}), there is a smooth $G$-equivariant embedding from an open neighborhood of $\ov{O}$ into an orthogonal $G$-space ${\bm R}_1$. We may assume that this embedding is defined and smooth over $U$, denoted by $\iota_U: U \to {\bm R}_1$. Let $U' \subset U$ be a $G$-invariant open neighborhood of $\ov{O}$ such that $\ov{U'}\subset U$. Then $M':=M \setminus \ov{U'}$ is a topological $G$-manifold with only finitely many orbit types. Hence by Mostow's embedding theorem (Theorem \ref{Mostow_embedding}), there exists a $G$-orthogonal representation ${\bm R}_2$ and a $G$-equivariant topological embedding $\iota_{M'}: M' \to {\bm R}_2$. Then by Lemma \ref{Mostow_embedding}, writing $M = \ov{U'} \cup M'$, there exists a $G$-equivariant topological embedding 
\beqn
\iota: M \to {\bm R}_1 \oplus {\bm R}_2
\eeqn
such that its restriction to $\ov{U'}$ coincides with $(\iota_U, 0)$. In particular, the restriction of $\iota$ to $U'$ is smooth. Now define ${\bm R}:= {\bm R}_1 \oplus {\bm R}_2$. \hfill {\it End of the proof of Claim A.}

\vspace{0.2cm}

\noindent {\it Claim B. There is a $G$-invariant open neighborhood $N_\iota \subset {\bm R}$ and a continuous $G$-equivariant retraction $r: N_\iota \to \iota(M)$ which is smooth near $\iota(\ov{O})$. }

\vspace{0.2cm}

\noindent {\it Proof of Claim B.} Note that there exist a $G$-invariant open neighborhood $N_{U'} \subset {\bm R}_1$ of $\iota_U(U')$ and a smooth $G$-equivariant retraction 
\beqn
r_{U'}: \ov{N_{U'}} \to \iota_U(U').
\eeqn
Then $r_{U'}$ induces a retraction from $\ov{N_{U'}} \times {\bm R}_2$ to $\iota(U')$. Choose a smaller open neighborhood $U''$ of $\ov{O}$ with $\ov{U''}\subset U'$. By applying Jaworowski's extension theorem (Theorem \ref{Jaworowski}) to $X = {\bm R}_1 \oplus {\bm R}_2$, $A = ( r_{U'}^{-1}( \ov{U''}) \times {\bm R}_2 )\cup \iota(M)$, and $Y = \iota(M)$, one obtains an extension, i.e., a $G$-invariant open neighborhood $N_\iota$ of $\iota(M)$ and a $G$-invariant retraction $r: N_\iota \to \iota(M)$ which coincides with $r_{U'}$ near $\iota(\ov{U''})$. Indeed, such an extension is a retraction because it extends the identity map $Y \to Y$. In particular, it is smooth near $\iota(\ov{O})$. \hfill {\it End of the proof of Claim B.}

\vspace{0.2cm}

\noindent {\it Claim C. There is a $G$-equivariant continuous map 
\beqn
\psi: M \times {\bm R} \to {\bm R} \times {\bm W}
\eeqn
satisfying 
\begin{enumerate}
    \item $\psi$ is a homeomorphism onto an open subset.
    
    \item $\psi$ is a diffeomorphism near $\ov{O}\times \{0\}$ with respect to the given $G$-smoothing near $\ov{O}$.
\end{enumerate}
}

\noindent {\it Proof of Claim C.} Consider the $G$-equivariant map 
\beqn
\theta: M \times {\bm R} \to {\bm R},\ (x, y) \mapsto  x+ y.
\eeqn
There exists a continuous function $\delta: M \to (0, +\infty)$ such that
\beqn
|y| < \delta(x) \Longrightarrow \theta(x, y) \in N_\iota.
\eeqn
We can choose $\delta$ such that it is smooth over $U$. Let $y \mapsto y'$ be a $G$-invariant smooth map which maps ${\bm R}$ diffeomorphically onto the unit ball ${\bm R}(1)$. Then define 
\beqn
\begin{split}
    \psi: M \times {\bm R} \to &\ {\bm R} \times {\bm W},\\
          (x, y) \mapsto &\ \Big( \theta (x, \delta(x)y'), \exp_x^{-1} ( r(\theta(x, \delta(x) y')) ) \Big).
          \end{split}
\eeqn
This map is of the same form of the one used in \cite[Theorem 1.3]{Lashof_1979}. Hence this map is a homeomorphism onto an open subset of ${\bm R} \times {\bm W}$. Moreover, by the explicit construction one can verify that $\psi$ is a smooth embedding near $\ov{O}\times \{0\}\subset \ov{O} \times {\bm R}$. \hfill {\it End of the proof of Claim C.}

\vspace{0.2cm}

Therefore, $\psi$ pulls back the standard $G$-invariant smooth structure on ${\bm R} \times {\bm W}$ to a smooth structure on $M \times {\bm R}$, giving a stable $G$-smoothing on $M$. Moroever, as $\psi$ is smooth near $\ov{O}\times \{0\}$, this stable $G$-smoothing is diffeomorphic to the ${\bm R}$-stabilization of the given one. Lastly, one can check easily that the induced stable $G$-vector bundle reduction of $T_\mu( M \times {\bm R})$ is in the same isotopy class of the given one. 

Now we consider the general case where $E \to M$ is not necessarily trivial. We know that $E$ has a stable inverse, which is a $G$-equivariant vector bundle $\pi_F: F \to M$ such that $E \oplus F \cong M \times {\bm W}$ for some orthogonal $G$-space ${\bm W}$. We may assume that both $E$ and $F$ are smooth $G$-equivariant vector bundles near $\ov{O}$ and the isomorphism $E \oplus F \cong M \times {\bm W}$ is smooth there. We may assume assume that the $G$-vector bundle reduction, which is equivalent to a map
\beqn
\exp: E \to M \times M
\eeqn
is smooth near $\ov{O}$.

Consider the total space $F$. It is a topological $G$-manifold hence has a $G$-tangent microbundle $T_\mu F$. By \cite[Lemma 1.6]{Lashof_1979}, the $G$-vector bundle reduction
\beqn
\exp: E \to M \times M
\eeqn
of $T_\mu M$ induces a $G$-vector bundle reduction 
\beqn
\widehat\varphi: F \times {\bm W} \cong \pi_F^* (E \oplus F) \to F \times F
\eeqn
of $T_\mu F$. Notice that there is an induced $G$-smoothing in a neighborhood of $\pi_F^{-1}(\ov{O})$ in the total space of $F$. The construction of $\widehat\varphi$ in \cite[Lemma 1.6]{Lashof_1979} can be carried out to guarantee that with respect to the smooth structures on $E$ and $F$ near $\ov{O}$, $\widehat\varphi$ is smooth near $\pi_F^{-1}(\ov{O})$. Then by the special case when $E$ is trivial, one can construct a $G$-smoothing on $F \times {\bm W}$ for some orthogonal $G$-space ${\bm R}$ which coincides with the induced smooth structure near $\pi_F^{-1}(\ov{O}) \times {\bm R}$. Then consider 
\beqn
M \times {\bm W} \times {\bm R} = ( E \oplus F) \times {\bm R} = \pi_F^* E \times {\bm R}
\eeqn
which is the total space of a $G$-equivariant vector bundle over $F$. Hence we can choose a smooth bundle structure on this bundle making the total space $M \times {\bm W} \times {\bm R}$ smooth. As $E$ and $F$ are already smooth near $\ov{O}$, we can guarantee that the smooth bundle structure is induced from the original smooth bundle structures on $E$ and $F$ near $\ov{O}$. 
\end{proof}

\section{A proof sketch of Proposition \ref{bigprop}}\label{appendixc}

\subsection{Analytical setup}

%\textcolor{purple}{One line proof: construct compatible families of right inverses.}

We provide a systematic setup which generalizes the case of a single stabilization map in Subsection \ref{subsection55}. For each $pq$ and each integer $d \geq d_{pq}$ we describe a topological space ${\mc C}_{pq}^{(d)}$ which is morally a stratified Banach fiber bundle over $B_{d_{pq}}$. Indeed, for each $\uds x \in B_{d_{pq}}$, its fiber over the universal curve is a prestable cylinder $\Sigma_{\uds x}$ over which there is a stable map $\uds\mu_{\uds x}: \Sigma_{\uds x} \to \mb{CP}^{d_{pq}}$ representing $\uds x$. Then there are bundles 
\beqn
E_{\uds x}^{(d)} \to \Sigma_{\uds x} \times M
\eeqn
whose fibers are similar to the one \eqref{fiberwise_bundle} in the single-layered case. Fixing $a>2$, there is then a space ${\mc C}^{1,a} (\Sigma_{\uds x}, E_{\uds x}^{(d)})$ of $W^{1, a}$-sections $(u, \eta): \Sigma_{\uds x} \to E_{\uds x}^{(d)}$, where $u: \Sigma_{\uds x} \to M$ is a map asymptotic to possible periodic orbits at cylindrical nodes or the two marked points and $\eta$ is a section of $\Sigma_{\uds x}^{(d)}|_{{\rm graph} u}$ such the total map $(u, \eta)$ is of class $W^{1,a}$. Define 
\beqn
{\mc C}_{pq}^{(d)} = \bigsqcup_{\uds x \in B_{d_{pq}}} {\mc C}^{1,a}( \Sigma_{\uds x}, E_{\uds x}^{(d)}).
\eeqn
It is stratified by $\alpha \in \bA_{pq}^\floer$. There are natural inclusions
\beqn
{\mc C}_{pq}^{(d)} \hookrightarrow {\mc C}_{pq}^{(d')} \ \forall d'>d.
\eeqn
There are also inclusions
\beqn
\partial^\alpha V_{pq} \subset \partial^\alpha {\mc C}_{pq}^{(d_{pq})} \subset \partial^\alpha {\mc C}_{pq}^{(d)}\ \forall d \geq d_{pq}.
\eeqn

Notice that there is a $G_{pq}^{\mb C}$-action on ${\mc C}_{pq}^{(d)}$.

Over each ${\mc C}^{1,a}(\Sigma_{\uds x}, E_{\uds x}^{(d)})$ there is also a Banach space bundle denoted by $Y_{\uds x}$ whose fiber at $(u, \eta)$ is the space of $L^a$-sections of certain bundles along $(u, \eta)$. We omit its detailed definition. Indeed there is a section ${\mc F}_{\uds x}^{(d)}$ of this Banach space bundle whose zero locus in the $d = d_{pq}$ case is the fiber of $V_{pq}$ over $\uds x$. 

One can consider the linearization of the (multi-layered) thickened Floer equation at any $x \in \partial^\alpha V_{pq}$. Indeed, there are sequences of linear operators
\beqn
D_x^{(d)}: T_x {\mc C}^{1,a}(\Sigma_{\uds x}, E_{\uds x}^{(d)}) \to Y_{\uds x}|_x.
\eeqn
In fact, by the nature of the thickening equation (perturbations with lower $d$ does not deform the equation for perturbations with higher $d'$), the linear operators are always block upper-triangular, i.e., when $d< d'$, one has 
\beqn
D_x^{(d')} = \left[ \begin{array}{cc} D_x^{(d)} & * \\ 0 & D_x^{[d,d']}\end{array}\right]
\eeqn
where the last block is a linear Cauchy--Riemann operator on some vector bundle. 

Our construction guarantees that all these linearizations are surjective. Hence we can consider their bounded right inverses. By the above block upper-triangular form, we can have right inverses which are also block upper-triangular. We always assume this without further clarification in the rest of this appendix.

The construction of a compatible system of stabilization maps (as well as other structures included in Proposition \ref{bigprop}), as we showed in the case of a single stabilization map, depends on the choices of three kinds of structures. They are 1) continuous families of right inverses, 2) families of approximate solution maps, and 3) families of local Banach manifold charts and local Banach vector bundle trivializations. A crucial point of our proof presented in this appendix is that these three kinds of objects can always be obtained from local to global and can always be interpolated between each other. 

We first define these notions under the current setting.

\begin{defn}
An {\bf approximate solution map} (for level $d$) is a germ of continuous maps
\beqn
O_{pq}^{(d)} \to {\mc C}_{pq}^{(d)}
\eeqn
which satisfies the following conditions.
\begin{enumerate}
    \item It is level preserving, i.e., if $d'<d$, then it sends $O_{pq}^{(d')}\subset O_{pq}^{(d)}$ into ${\mc C}_{pq}^{(d')}$.
    
    \item It is fiberwise smooth. It means the following. For each $\uds x \in B_{d_{pq}}$, the fiber $\pi_{pq}^{-1}(\uds x) \subset V_{pq}$ of the thickened moduli is smooth, the fiberwise restriction of the level-$d$ obstruction bundle $O_{pq}^{(d)}$ is smooth, and the Banach manifold ${\mc C}^{1,a}(\Sigma_{\uds x}, O_{\uds x}^{(d)})$ is also smooth. Then the fiberwise restriction of the approximate solution map is smooth.
    
    \item Over $V_\alpha\subset V_{pq}$ for any $\alpha \in \bA_{pq}^\floer$, the approximate solution map is equal to the product.
\end{enumerate}
\end{defn}

\begin{rem}
\begin{enumerate} 
\item An obvious choice of such approximate solution map is just to add sections of the obstruction bundle linearly in the infinite-dimensional space. This obvious option does not deform the underlying map into $M$. However in certain stages of our construction we must use more general kinds of approximate solution maps which deforms the underlying map.

\item To construct the stabilization map from a product $V_\alpha$ to $\partial^\alpha V_{pq}$, we do not need to use the whole approximate solution map for all $d$, but only its restriction to some subbundle of $O_{pq}^{(d_{pq})}$ restricted to $V_\alpha$. However, using the approximate solution map to $O_{pq}^{(d)}$ for all $d$ helps one to extend the map we actually need from a lower to a higher stratum.

\item Remember that the obstruction bundle also contains the quadratic form part $Q_{pq}$. When constructing the actual approximate solutions to the thickened equation, we also need to change the approximate solution by variables from $Q_{pq}$ via the $G_{pq}^{\mb C}$-actions.

\end{enumerate}
\end{rem}

We can also define the notion of continuous family of bounded right inverses, for all levels, generalizing the case of Subsection \ref{subsection55}. 

We can also define the notion of families of local charts/trivializations. The obvious choices are the one coming from parallel transport along shortest geodesics starting at any $x \in V_{pq}$. We should regard this kind of choices as ``continuous'' when the domain changes. However, we may also need to use the chart induced from a map $x'$ near $x$. We can also define a notion of continuity.

\subsection{The inductive proof of Proposition \ref{bigprop}}

We have the following induction procedure to construct the structures we need.

\vspace{0.2cm}

\noindent {\bf Step 1.} We first consider any minimal thickened moduli space $V_{pq}$ which does not have boundary or corners. 
\begin{enumerate}
    \item We can choose $G_{pq}$-invariant inner products on $O_{pq}^{(d)} \to V_{pq}$ inductively on $d\geq d_{pq}$ such that the embeddings $O_{pq}^{(d)} \to O_{pq}^{(d+1)}$ for all $d\geq d_{pq}$ are isometric. 
    
    \item We can also inductively construct continuous families of vertical right inverses of all levels $d \geq d_{pq}$ inductively.
    
    \item We can also construct a collection of approximate solution map. For each $d\geq d_{pq}$, we can choose a germ of maps 
    \beqn
    O_{pq}^{(d)} \to {\mc C}_{pq}^{(d)}
    \eeqn
    defined by 
    \beqn
    [\Sigma, u, \mu, \eta, \zeta]\mapsto [\Sigma, u, \mu, \eta + \zeta].
    \eeqn
    The summation is defined because $\eta$ and $\zeta$ are contained in the same vector bundle.
\end{enumerate}

\vspace{0.2cm}

\noindent {\bf Step 2.} Suppose we have constructed the three kinds of structures up to an energy level $d_{pq}$. There are certain compatibility conditions we need to require. For example, they must respect the collars. We will describe the required conditions we we need them. Now we start to consider the structures over the thickened moduli space $V_{pq}$ and the stabilizations maps from each $V_\alpha$ to $\partial^\alpha V_{pq}$. In this step we describe how to define these structures over a minimal stratum $\alpha = pr_1 \cdots r_l q$. 
\begin{enumerate}

\item We first consider the metric on $O_{pq}^{(d)}$ restricted to $V_\alpha$. Start with $d_0 = d_{pq}$. We view this stratum as contained in $V_{pr_1 r_l q}$. Denote $r_1 = r$, $r_l = s$. Then there are the two subbundles 
\begin{align*}
&\ O_{ps\to pq}^{(d_0)},\ &\ O_{r q \to pq}^{(d_0)}
\end{align*}
on which the expected properties and inner products we have already constructed induce unique inner products. We verify that these two inner products on these two subbundles agree on their intersection. By the canonical decomposition \eqref{obstruction_splitting}, one has 
\beqn
O_{ps\to pq}^{(d_0)} \cap O_{rq \to pq}^{(d_0)} \subset O_{pq; rs}^{(d_0)}.
\eeqn
In fact we can show that 
\beqn
O_{ps\to pq}^{(d_0)} \cap O_{rq \to pq}^{(d_0)} = O_{rs \to pq}^{(d_0)}.
\eeqn
Then over this intersection the existing inner products agree. Moreover $O_{ps \to pq}^{(d_0)} + O_{rq\to pq}^{(d_0)}$ contains all existing subbundles of $O_{pq}^{(d_0)}|_{V_\alpha}$. Hence they induce an inner product on the sum $O_{ps \to pq}^{(d_0)} + O_{rq \to pq}^{(d_0)}$ which agrees with all existing induced ones. Then extend it arbitrarily to $O_{pq}^{(d_0)}|_{V_\alpha}$.

\item Now we obtained an orthogonal complement 
\beqn
O_{pq, \alpha} \subset O_{pq}|_{V_\alpha}.
\eeqn
Define 
\beqn
F_{pq, \alpha}:= O_{pq, \alpha} \oplus Q_{pq, \alpha} \to V_\alpha.
\eeqn
At this moment we can define the stabilization map 
\beqn
\theta_{pq, \alpha}: {\rm Stab}_{F_{pq, \alpha}}(V_\alpha) \to \partial^\alpha V_{pq}
\eeqn
In fact, the existing approximate solution maps can be extended to one on $O_{pq}^{(d_0)}$ in the same fashion as we extend the inner product in the above. Then we can restrict the approximate solution map to $O_{pq, \alpha}$. Then as we did in \eqref{eqn:app-solution}, we can define a map from $F_{pq, \alpha} = O_{pq, \alpha} \oplus Q_{pq, \alpha}$ to $\partial^\alpha {\mc C}_{pq}^{(d_0)}$. One can also extend the family of right inverse and the existing charts, in the same fashion as we extended the inner products above. Then it produces the way to correct the approximate solutions to exact solutions canonically. Then the corresponding projection map 
\beqn
\pi_{pq,\alpha}: \partial^\alpha V_{pq} \to V_\alpha^\sim
\eeqn
is obtained. From the construction, we can see that the projections are compatible with all existing ones when restrict to $\partial^\alpha V_\beta$ for an intermediate stratum $\beta$. The bundle isomorphism
\beqn
\widehat \theta_{pq, \alpha}: \pi_{pq, \alpha}^* \left( O_{pq}^{(d)}|_{V_\alpha^\sim} \right) \to O_{pq}^{(d)}
\eeqn
for all $d$ can be obtained.

\item Now we inductively construct inner products on $O_{pq}^{(d)}$ over $V_\alpha$ for all $d \geq d_0$. Suppose we have done this for some $d$. Then similar to the above case, one can obtain a unique inner product on the sum 
\beqn
O_{ps \to pq}^{(d)} + O_{rq \to pq}^{(d)}.
\eeqn
We can verify that this metric agrees with the one on $O_{pq}^{(d-1)}$ on the intersection $O_{pq}^{(d-1)} \cap ( O_{ps \to pq}^{(d)} + O_{rq \to pq}^{(d)})$. Indeed, 
\beqn
O_{pq}^{(d-1)} \cap ( O_{ps \to pq}^{(d)} + O_{rq \to pq}^{(d)}) = O_{ps \to pq}^{(d-1)} + O_{rq \to pq}^{(d-1)}.
\eeqn
Hence there exists a further extension to $O_{pq}^{(d)}$ over the set $V_\alpha$.

Then using the bundle isomorphisms we can extend the inner products to $O_{pq}^{(d)}|_{\partial^\alpha V_{pq}}$. At last, the approximate solution map, the right inverse map, and the local charts can be extended to $\partial^\alpha V_{pq}$.

\end{enumerate}

\vspace{0.2cm}

\noindent {\bf Step 3.} Using the collar structure one can extend all we have constructed over $\partial^\alpha V_{pq}$ to a small neighborhood. Now we would like to carry out the induction for a higher stratum $\beta$. Indeed, for the higher stratum one has its own construction which can be irrelevant to what we just did. What we need to do here is to connect the structures over the collar region near $\partial^\alpha V_\beta$. 

In fact, the interpolation between bundle metrics can be easily done. The interpolation between right inverses can also be done using convex combinations (notice that convex combinations of block upper-triangular operators are still block upper-triangular). The more nonlinear interpolation is the one for approximate solution and the one for local charts/trivializations. For approximate solutions, we can first interpolate (arbitrarily) between the underlying maps into $M$, then interpolate over the fiber (linear) direction of $E_{\uds x}^{(d)}$. For local charts/trivializations, as they are induced from parallel transport ``centered'' at families of maps into $M$, we only needs to interpolate the certer maps. Indeed, the maps we want to interpolate between are close enough to each other. We can do the same construction for all stratum of $V_{pq}$ inductively.

Lastly, the construction of the stabilizations maps implies that 
\beqn
\pi_{pq, \alpha} = \pi_{\beta\alpha}\circ \pi_{pq, \beta}
\eeqn
The bundle isomorphisms $\widehat \theta_{pq, \beta}$ for all stratum can be obtained as well.

\bibliography{reference}

\bibliographystyle{amsalpha}

\end{document}